\documentclass[aos]{arxiv-imsart} %
\RequirePackage{amsthm,amsmath,amsfonts,amssymb,color,xcolor}
\RequirePackage[numbers]{natbib}
\RequirePackage{graphicx}

\usepackage[colorlinks=True,citecolor=blue,urlcolor=blue]{hyperref}
\usepackage[postexp]{optional} %
\startlocaldefs
\usepackage{upgreek} %
\usepackage{xspace} %
\usepackage{enumitem} %
\usepackage{etoc} %
\usepackage{algorithm}
\usepackage[noend]{algpseudocode}
\usepackage[linesnumbered,ruled,algo2e]{algorithm2e}%

\usepackage{multirow}

\usepackage{booktabs}
\usepackage{dsfont}
\usepackage{amsmath}
\usepackage{thmtools}
\usepackage{nicefrac}

\usepackage{thmtools,thm-restate}
\usepackage[capitalize]{cleveref}  
\usepackage{crossreftools} %

\usepackage{autonum}

\theoremstyle{plain}

\newtheorem{theorem}{Theorem}
\newtheorem{lemma}{Lemma}
\newtheorem{definition}{Definition}
\newtheorem{example}{Example}
\newtheorem{remark}{Remark}

\newtheorem{assumption}{Assumption}
\newtheorem{corollary}{Corollary}
\newtheorem{proposition}{Proposition}

\newtheorem{myexample}{Example}

\newtheorem{myproposition}{Proposition}

\newtheorem{mycorollary}{Corollary}

\newtheorem{mylemma}{Lemma}

\newtheorem{mydefinition}{Definition}

\newtheorem{myremark}{Remark}

\newtheorem{myassumption}{Assumption}

\crefname{appendix}{App.}{Apps.}
\crefname{subsubsubappendix}{App.}{Apps.}
\crefname{equation}{}{}
\crefname{lemma}{Lem.}{Lems.}
\crefname{theorem}{Thm.}{Thms.}
\Crefname{theorem}{THM.}{THMS.}
\crefname{Corollary}{Cor.}{Cors.}
\crefname{algorithm}{Alg.}{Algs.}

\crefname{section}{Sec.}{Secs.}
\crefname{table}{Tab.}{Tabs.}
\crefname{remark}{Rem.}{Rems.}
\crefname{definition}{Def.}{Defs.}
\crefname{Proposition}{Prop.}{Props.}
\crefname{myremark}{Rem.}{Rems.}
\crefname{mylemma}{Lem.}{Lems.}
\Crefname{mylemma}{LEM.}{LEMS.}
\crefname{mydefinition}{Def.}{Defs.}
\crefname{myproposition}{Prop.}{Props.}
\Crefname{myproposition}{PROP.}{PROPS.}
\crefname{mycorollary}{Cor.}{Cors.}
\Crefname{mycorollary}{COR.}{CORS.}
\crefname{myassumption}{Assum.}{Assums.}
\crefname{figure}{Fig.}{Figs.}
\crefname{myexample}{Ex.}{Exs.}
\crefname{enumi}{}{}
\crefname{name}{}{} %
\newcommand{\idx}{\mc I}
\DeclareMathAlphabet{\mathbsf}{OT1}{cmss}{bx}{n}%
\DeclareMathAlphabet{\mathssf}{OT1}{cmss}{m}{sl}%

\newcommand{\mE}{\mathbb{E}}

\newcommand{\numperm}{B}

\newcommand{\defeq}{\triangleeq}

\newcommand{\set}[1]{\mc{#1}}

\newcommand{\eventnotag}{\mc{E}}
\newcommand{\event}[1][]{\eventnotag_{#1}}

\newcommand{\dirac}{\delta}

\newcommand{\mmd}[1][\kernel]{\mathrm{MMD}_{#1}}
\newcommand{\xset}{\mc{X}}
\newcommand{\yset}{\mc{Y}}
\newcommand{\zset}{\mc{Z}}
\newcommand{\kernel}{\mbf{k}}

\newcommand{\rkhs}{\mc{H}_{\kernel}}

\newcommand{\knorm}[2][\kernel]{\Vert{#2}\Vert_{#1}}
\newcommand{\kinner}[3][\kernel]{\inner{#2}{#3}_{#1}} %

\newcommand{\braces}[1]{\left\{ #1 \right \}}
\newcommand{\sbraces}[1]{\{ #1  \}}

\newcommand{\sig}{\sigma}

\def\balign#1\ealign{\begin{align}#1\end{align}}
\def\baligns#1\ealigns{\begin{align*}#1\end{align*}}
\def\balignat#1\ealign{\begin{alignat}#1\end{alignat}}
\def\balignats#1\ealigns{\begin{alignat*}#1\end{alignat*}}
\def\bitemize#1\eitemize{\begin{itemize}#1\end{itemize}}
\def\benumerate#1\eenumerate{\begin{enumerate}#1\end{enumerate}}

\newenvironment{talign*}
 {\let\displaystyle\textstyle\csname align*\endcsname}
 {\endalign}
\newenvironment{talign}
 {\let\displaystyle\textstyle\csname align\endcsname}
 {\endalign}

\def\balignst#1\ealignst{\begin{talign*}#1\end{talign*}}
\def\balignt#1\ealignt{\begin{talign}#1\end{talign}}
\newcommand{\qtext}[1]{\quad\text{#1}\quad} 
\newcommand{\stext}[1]{\text{ #1 }} 
\newcommand{\sstext}[1]{\ \ \text{#1}\ \ }

\let\originalleft\left
\let\originalright\right
\renewcommand{\left}{\mathopen{}\mathclose\bgroup\originalleft}
\renewcommand{\right}{\aftergroup\egroup\originalright}

\def\Holder{H\"older\xspace}

\def\vs{vs.\xspace}

\def\tinycitep*#1{{\tiny\citep*{#1}}}
\def\tinycitealt*#1{{\tiny\citealt*{#1}}}
\def\tinycite*#1{{\tiny\cite*{#1}}}
\def\smallcitep*#1{{\scriptsize\citep*{#1}}}
\def\smallcitealt*#1{{\scriptsize\citealt*{#1}}}
\def\smallcite*#1{{\scriptsize\cite*{#1}}}

\def\blue#1{\textcolor{blue}{{#1}}}

\def\mbf#1{\mathbf{#1}}
\def\mbb#1{\mathbb{#1}}
\def\mc#1{\mathcal{#1}}
\def\mrm#1{\mathrm{#1}}
\def\trm#1{\textrm{#1}}

\def\reals{\mathbb{R}} %
\def\R{\mathbb{R}}
\def\Z{\mathbb{Z}}
\def\Q{\mathbb{Q}}
\def\<{\left\langle} %
\def\>{\right\rangle}

\def\iff{\Leftrightarrow}
\def\implies{\quad\Longrightarrow\quad}

\def\defeq{\triangleq} %
\def\half{\frac{1}{2}}
\def\quarter{\frac{1}{4}}

\newcommand{\floor}[1]{\lfloor{#1}\rfloor}
\newcommand{\ceil}[1]{\lceil{#1}\rceil}
\newcommand{\boldone}{\mbf{1}} %
\def\norm#1{\left\|{#1}\right\|} %
\newcommand{\twonorm}[1]{\norm{#1}_2} %
\newcommand{\opnorm}[1]{\staticnorm{#1}_{\mathrm{op}}} %
\newcommand{\fronorm}[1]{\staticnorm{#1}_{F}} %
\def\staticnorm#1{\|{#1}\|} %
\newcommand{\statictwonorm}[1]{\staticnorm{#1}_2} %
\newcommand{\inner}[2]{\langle{#1},{#2}\rangle} %
\newcommand{\indicator}{\mbf{1}}
\def\indic#1{\indicator_{#1}} %
\def\E{\mbb{E}} %

\def\P{\mbb{P}} %

\def\Var{\mrm{Var}} %

\newcommand{\Gsn}{\mathcal{N}}

\newcommand{\Unif}{\textnormal{Unif}}

\newcommand{\iid}{\textrm{i.i.d.}\ }
\newcommand{\dist}{\sim}
\newcommand{\distiid}{\overset{\textup{\tiny\iid}}{\dist}}

\ifdefined\nonewproofenvironments\else
\ifdefined\ispres\else
\newenvironment{proof-sketch}{\noindent\textbf{Proof Sketch}
  \hspace*{1em}}{\qed\bigskip\\}
\newenvironment{proof-idea}{\noindent\textbf{Proof Idea}
  \hspace*{1em}}{\qed\bigskip\\}
\newenvironment{proof-of-lemma}[1][{}]{\noindent\textbf{Proof of Lemma {#1}}
  \hspace*{1em}}{\qed\\}
\newenvironment{proof-of-theorem}[1][{}]{\noindent\textbf{Proof of Theorem {#1}}
  \hspace*{1em}}{\qed\\}
\newenvironment{proof-attempt}{\noindent\textbf{Proof Attempt}
  \hspace*{1em}}{\qed\bigskip\\}

\newcommand{\nullhypo}{\mc H_{0}}
\newcommand{\althypo}[1][1]{\mc H_{#1}}

\newcommand{\feat}{f} %
\newcommand{\Feat}{F} %

\newcommand{\ncref}[1]{\cref{#1}: \nameref*{#1}} %

\newcommand{\pcref}[1]{%
  \NoCaseChange{Proof of 
  \texorpdfstring{%
    \ncref{#1}%
  }{%
    \crtcref{#1}: \protect\nameref*{#1}%
  }}%
} %

\newcommand{\X}{\mbb X}
\newcommand{\Y}{\mbb Y}
\newcommand{\Yi}[1][i]{\Y^{(#1)}}
\newcommand{\Zi}[1][i]{\Z^{(#1)}}
\renewcommand{\Xi}[1][i]{\X^{(#1)}}

\newcommand{\PX}{\widehat{\pjnt}}
\newcommand{\PY}{\widehat{\P}}
\newcommand{\PZ}{\widehat{\Q}}
\newcommand{\pjnt}{\Pi} %
\newcommand{\Phat}{\widehat{\P}}
\newcommand{\twotest}{U_{n_1, n_2}}

\newcommand{\bin}{\mc I}

\newcommand{\itag}{\mathrm{in}}
\newcommand{\htag}{\mathrm{ho}}
\newcommand{\cheaptag}{\trm{cheap}}
\newcommand{\hcomp}{\mc S_{\htag}}
\newcommand{\hcheap}{\mc S_{\htag,\cheaptag}}
\newcommand{\icomp}[1][]{\mc S_{#1}}

\renewcommand{\set}{\mc{S}}

\newcommand{\gymean}[1][i]{\bar{g}_{Y,#1}}
\newcommand{\gzmean}[1][k]{\bar{g}_{Z,#1}}
\newcommand{\gyz}[1][ijk\ell]{G_{#1}}
\newcommand{\swap}{\texttt{swap}}
\newcommand{\swaps}{\swap_{s}}

\newcommand{\hbarho}{\overline{h}_{\htag}}
\newcommand{\hho}{h_{\htag}}
\newcommand{\Hho}{H_{\htag,m}}
\newcommand{\hbarin}{\overline{h}_{\itag}}
\newcommand{\hin}{h_{\itag}}
\newcommand{\Hin}{H_{\itag,m}}
\newcommand{\hinY}{h_{\itag,\mathrm{Y}}}
\newcommand{\hinZ}{h_{\itag,\mathrm{Z}}}
\newcommand{\eps}{\epsilon}
\newcommand{\teps}{\tilde{\epsilon}}
\newcommand{\tconst}{\mathcal{T}}

\newcommand{\subg}[1][\P]{\sig_{#1}} %
\newcommand{\mom}[1][2,\P]{\mathrm{M}_{#1}} %

\newcommand{\astar}{\alpha^{\star}}
\newcommand{\txiwild}{\tilde{\xi}_{\textup{wild}}}

\newcommand{\Lp}[1][p]{L^{#1}}

\newcommand{\wildxi}{\tilde{\xi}_{\textup{wild}}}

\newcommand{\nratio}{\rho_{n_1n_2}}

\graphicspath{{figs/},{../figs/}}

\newcommand{\papertitle}{Cheap Permutation Testing}

\endlocaldefs

\begin{document}
\etocdepthtag.toc{mtchapter}
\etocsettagdepth{mtchapter}{section}

\makeatletter
\patchcmd{\@algocf@start}%
  {-1.5em}%
  {0pt}%
  {}{}%
\makeatother

\begin{frontmatter}
\title{\papertitle}
\runtitle{Cheap Permutation Testing}

\begin{aug}
\author[A]{\fnms{Carles}~\snm{Domingo-Enrich}\ead[label=e1]{carlesd@microsoft.com}}, 
\author[B]{\fnms{Raaz}~\snm{Dwivedi}\ead[label=e2]{dwivedi@cornell.edu}},
\and
\author[A]{\fnms{Lester}~\snm{Mackey}\ead[label=e3]{lmackey@microsoft.com}}
\address[A]{Microsoft Research New England\printead[presep={,\ }]{e1,e3}}
\address[B]{Cornell Tech\printead[presep={,\ }]{e2}}
\end{aug}

\begin{abstract}
Permutation tests are a popular choice for distinguishing distributions and testing independence, due to their exact, finite-sample control of false positives and their minimax optimality when paired with U-statistics. However, standard permutation tests are also expensive, requiring a test statistic to be computed hundreds or thousands of times to detect a separation between distributions. In this work, we offer a simple approach to accelerate testing: group your datapoints into bins and permute the bin labels. For U and V-statistics, we prove that these \emph{cheap permutation tests} have two remarkable properties. First, by storing appropriate sufficient statistics, a cheap test  can be run in time comparable to evaluating a \emph{single} test statistic. Second, cheap permutation power closely approximates standard permutation power. As a result, cheap tests inherit the exact false positive control and minimax optimality of standard permutation tests while running in a fraction of the time. We complement these findings with improved power guarantees for standard permutation testing and experiments demonstrating the benefits of cheap permutations over standard maximum mean discrepancy (MMD), Hilbert-Schmidt independence criterion (HSIC), random Fourier feature, Wilcoxon-Mann-Whitney, aggregated MMD, cross-MMD, and cross-HSIC tests. 
\end{abstract}



\end{frontmatter}
\section{Introduction}
\label{sec:introduction}

Permutation tests \citep{fisher1925,dwass1957modified} are commonly used in statistics, machine learning, and the sciences 
to test for independence,
detect distributional discrepancies, 
evaluate fairness, and
assess model quality \citep{good2013permutation,PhipsonSmyth2010,diciccio2020evaluating,lindgren1996model}.
By simply permuting the labels assigned to datapoints, these tests offer both exact, non-asymptotic control over false positives \citep{hoeffding1952large} and minimax optimal error rates when paired with popular U and V test statistics \citep{kim2022minimax,berrett2021optimal,schrab2023mmd,schrab2022efficient}.

However, standard permutation tests are also expensive, 
requiring the recomputation of a test statistic hundreds or thousands of times. 
This work offers a simple remedy: 
divide your data into $s$ bins and permute the bin labels. 
In \cref{sec:power} we show that this \emph{cheap permutation} strategy recovers the best known error rates for U and V-statistic permutation testing, even when the bin count $s$ is independent of the sample size $n$.
Moreover, we establish in \cref{sec:cheap} that, after computing the initial test statistic, the same cheap tests can be run in time independent of $n$ by keeping track of simple sufficient statistics.
Together, these findings give rise to new exact tests that inherit the minimax rate optimality of standard permutation testing 
while running in 
time comparable to the cost of a single test statistic.

\subsection{Related work}\label{sec:related}
A number of %
approaches to speeding up permutation testing have appeared in the literature.  
\citet{zmigrod2022exact} %
showed that, for a special class of $G$-bounded, additively-decomposable, and integer-valued test statistics, one can run an independence permutation test using all $n!$ distinct permutations of $n$ points in time $O(Gn\log(Gn) \log(n))$.
Unfortunately, this class does not cover the more general U and V-statistics studied in this work.  

For large sample sizes, one typically avoids the prohibitive cost of enumerating $n!$ permutations by instead iterating over a small subgroup of permutations \citep{chung1958randomization} or sampling $\numperm$ permutations uniformly at random \citep{dwass1957modified}. 
In \cref{sec:cheap}, we establish a  complementary cost savings for binned permutations and quadratic test statistics: each bin-permuted statistic can be computed in time independent of the sample size.

Recently, \citet{koning2024moreefficient,koning2024morepower} revealed another surprising benefit of restricted permutation: certain small subgroups yield greater power than testing with the full permutation group. 
Their analysis focuses on the case of generalized location models, and an interesting open question is whether those gains transfer to the more general settings and subgroups studied in this work.

In the context of homogeneity testing, \citet{chung2019rapid} 
replaced random permutations with transpositions of random pairs of elements showing that, for certain test statistics, 
each transposed statistic could be computed in time independent of the sample size $n$. 
However, for the higher-order U and V-statistics studied in this work, the cost of computing a transposed statistic grows linearly in $n$, while our 
cheap permutation statistics can still be computed in time independent of $n$. 
In addition, \citet{chung2019rapid} did not analyze the power of their proposed test. 
\citet{domingo2023compress} used binned permutation in combination with sample compression to reduce the runtime of homogeneity testing with a bounded kernel maximum mean discrepancy~\citep[MMD,][]{gretton2012akernel} test statistic.
Here we show that binned permutation provides benefits for a much broader range of tests, including independence tests and homogeneity tests based on general quadratic test statistics (see \cref{def:qts}).

\citet{anderson2023exact} 
computed the exact mean and variance of the permutation distribution for a class of multivariate generalizations of the Mann-Whitney test called generalized pairwise comparison (GPC) statistics.  However, the authors noted, ``In order to construct a hypothesis test for the GPC statistics, any use of the means and variances would require the assumption of asymptotic normality,'' an assumption which is not needed for our proposed methodology.
Relatedly, a wide variety of works accelerate testing by fitting an approximation to the permutation distribution \citep[see, e.g.,][]{zhou2009efficient,knijnenburg2009fewer,larson2015moment,segal2018fast,he2019permutation}.
However, these works do not analyze the power of their approximate tests, and, in each case, the approximation sacrifices finite-sample validity. %

Sequential permutation testing \citep[see, e.g.,][]{besag1991sequential,fay2007using,SILVA201333,fischer2024sequential} reduces the total number of permutations required for a powerful test by processing permuted statistics sequentially and applying an early stopping rule. 
This strategy normally incurs the full cost of recomputing a test statistic for each permutation but can be combined with cheap permutation statistics to yield testing overhead independent of the sample size. 
\citet{ramdas2023permutation} 
established the \emph{validity} of a wide variety of generalized permutation tests (involving arbitrary subsets of the permutation distribution and nonuniform distributions over them). However, their work does not discuss power or computational complexity, two matters of primary importance in the present work. %

Finally, the pioneering work of \citet{kim2022minimax} established strong finite-sample power guarantees and minimax rate-optimality results for standard U-statistic permutation tests with $\numperm$ permutations and order $\Theta(n^2 + \numperm n^2)$ runtime. Perhaps surprisingly, we find that binned permutation tests enjoy equivalent power guarantees and minimax rate optimality even when the bin count $s$ is independent of the sample size. Since bin-permuted U-statistics can be computed with only $O(s^2)$ overhead, this opens the door to powerful exact testing in $O(n^2 + \numperm s^2)$ time, that is, in time comparable to computing a single test statistic.

\subsection{Contributions}
The primary contribution of this work lies in the development of practical cheap testing procedures that, on the one hand, dramatically reduce the computational requirements of permutation testing and, on the other, retain the high-quality power properties of standard permutation testing both in theory and in practice. More precisely:
\begin{enumerate}
    \item We develop cheap tests of homogeneity and independence with only $\numperm s^2$ permutation overhead thanks to binned permutation and compact sufficient statistics. 
    \item We adapt the arguments of \citet{kim2022minimax} and \citet{domingo2023compress} to develop quantitative power guarantees for cheap permutation testing with U and V-statistics. Remarkably, the guarantees for $s$ bins closely match the standard guarantees for $n$ bins, even when the bin count is independent of $n$.
    \item \cref{sec:minimax} builds on these results to establish the minimax rate-optimality of cheap permutation tests for the four classic testing problems studied by \citet{kim2022minimax}: 
    discrete homogeneity testing, \Holder homogeneity testing, discrete independence testing, and \Holder independence testing. Notably, this optimality is achievable with a bin count and hence a permutation time {independent} of the sample size. 
    \item Moving beyond the classic testing settings, we use the lower bounds of \citet{kim2023differentially} to establish the minimax-rate optimality of cheap testing for (a) MMD homogeneity testing and (b) Hilbert-Schmidt independence criterion \citep[HSIC,][]{gretton2005measuring} independence testing. This optimality is again achievable with a bin count independent of $n$.
    \item In \cref{sec:experiments}, we complement these findings with experiments demonstrating the practical benefits of cheap testing over standard MMD, HSIC, random Fourier feature~\citep{rahimi2008random}, Wilcoxon-Mann-Whitney~\citep{wilcoxon1945individual,mann1947on}, aggregated MMD, cross-MMD~\citep{shekhar2022permutation}, and cross-HSIC tests~\citep{shekhar2023apermutation}. 
    \item Finally, we outline the extension of cheap testing to broader testing problems and test statistics in \cref{sec:discussion}.
\end{enumerate}
\subsection{Notation} \label{subsec: notation}
Given a permutation $\pi$ of the indices $[n] \defeq \{1,\dots,n\}$ and an index $i \in [n]$, we write either $\pi(i)$ or $\pi_i$ to denote the index to which $i$ is mapped. 
For integers $p\geq q\geq 1$, we define $(p)_q \defeq p(p-1)\cdots{(p-q+1)}$ and 
let $\mathbf{i}_q^p$ denote the set of all $q$-tuples that can be drawn without replacement from the set $[p]$. 
For a function $g$ integrable under a probability measure $\P$, we use 
the notation $\P g \defeq \E_{X\sim \P} [g(X)]$.
We use $\indicator$ for the indicator function.
For two nonnegative sequences $(a_n)_{n=1}^\infty$ and $(b_n)_{n=1}^\infty$, we write $a_n = O(b_n)$ if there exists a constant $C>0$ such that $a_n \le C b_n$ for all sufficiently large $n$ and $a_n = \Omega(b_n)$ if $b_n = O(a_n)$. We write $a_n = \Theta(b_n)$ if both $a_n = O(b_n)$ and $a_n = \Omega(b_n)$ and write $a_n = \omega(b_n)$ if $a_n/b_n \to \infty$ as $n\to\infty$.

\section{Background on hypothesis testing}\label{sec:background}

We define a test as any measurable procedure that takes as input a dataset $\X\defeq(X_1,\ldots,X_n)$ and outputs a probability in $[0,1]$.
Given two competing hypotheses ($\nullhypo$ and $\althypo$)  concerning the data generating distribution, we interpret $\Delta(\X)$ as the probability of rejecting the \emph{null hypothesis} $\nullhypo$ in favor of the \emph{alternative hypothesis} $\althypo$. 
The quality of a test is measured by 
its \emph{size}, its probability of (incorrectly) rejecting the null when the null is true, and its \emph{power}, its probability of (correctly) rejecting $\nullhypo$ when $\althypo$ is true.
A test often comes paired with a \emph{nominal level}, a desired upper bound $\alpha\in(0,1)$ on the Type I error. 
We say $\Delta$ is a \emph{valid level-$\alpha$} test if its size never exceeds $\alpha$ when $\nullhypo$ holds and an \emph{exact level-$\alpha$} test if its size exactly equals $\alpha$ whenever $\nullhypo$ holds.
All of the tests studied in this work are exact and take the form 
\begin{talign}
\label{eq:test}
    \Delta(\X) = 
    \begin{cases}
        1 &\text{if } T(\X) > c(\X), \\ 
        \xi(\X) &\text{if } T(\X) = c(\X), \stext{and}\\ 
        0 &\text{if } T(\X)  < c(\X), \\ 
    \end{cases}
\end{talign}
where $T(\X)$ is called the \emph{test statistic}, $c(\X)$ is called the \emph{critical value},  and $\xi(\X)$ is a rejection probability chosen to ensure exactness.
We will consider two broad classes of tests in this work: tests of homogeneity and tests of independence.

\subsection{Homogeneity testing}

\emph{Homogeneity tests} (also known as \emph{two-sample tests}) test for equality of distribution. 
More precisely, given a sample $\Y = (Y_i)_{i=1}^{n_1}$ drawn \iid from an unknown distribution $\P$ and an independent sample $\Z = (Z_j)_{j=1}^{n_2}$ drawn \iid from an unknown distribution $\Q$, a homogeneity test judges the hypotheses
\begin{talign}
    \nullhypo: \P = \Q \qtext{versus} \althypo: \P \neq \Q.
\end{talign}

A common approach to testing homogeneity is to threshold a \emph{quadratic test statistic (QTS)}.

\begin{definition}[Quadratic test statistic (QTS)]
\label{def:qts}
Given samples $\Y$ and $\Z$ of size $n_1$ and $n_2$ respectively, we call $T(\Y, \Z)$ a \emph{quadratic test statistic}, if $T$ is a quadratic form of the empirical distributions $\PY \defeq  \frac{1}{n_1}\sum_{y\in\Y}\dirac_{y}$ and $\PZ \defeq \frac{1}{n_2}\sum_{z\in\Z}\dirac_{z}$. That is, there exist base functions $\phi_{n_1},\phi_{n_2}$, and $\phi_{n_1, n_2}$ for which $T$ satisfies
\begin{talign}
\label{eq:qts}
    T(\Y, \Z) 
    &= (\PY\times\PY)\phi_{n_1} + (\PY\times\PZ) \phi_{n_1, n_2} + (\PZ\times\PZ)\phi_{n_2}.
\end{talign}
\end{definition}
An important example of a QTS the canonical \emph{homogeneity U-statistic}.
\begin{definition}[Homogeneity U-statistic] \label{def:homogeneity_u_stat}
Given samples $\Y$ and $\Z$ of size $n_1$ and $n_2$ respectively, we call $U(\Y, \Z)$, or $\twotest$ for short, a \emph{homogeneity U-statistic} if 
    \begin{talign} \label{Eq: Two-Sample U-statistic}
          U_{n_1,n_2} &= %
        \frac{1}{(n_1)_2 (n_2)_2}
        \sum_{(i_1,i_2) \in \mathbf{i}_2^{n_1}} \sum_{(j_1,j_2) \in \mathbf{i}_2^{n_2}} h_{\htag}(Y_{i_1},Y_{i_2}; Z_{j_1},Z_{j_2}), \qtext{and}\\
    \label{Eq: Two-Sample kernel}
    h_{\htag}(y_1,y_2;z_1,z_2)  &=  g(y_1,y_2) + g(z_1,z_2) - g(y_1,z_2) - g(y_2,z_1)
    \end{talign}
    for some symmetric 
    real-valued base function $g$. 
\end{definition}
    Notably, any homogeneity U-statistic %
    can be written as a QTS \cref{eq:qts} %
    via the mappings
    \begin{talign}
        \phi_{w}(y, z) &= \frac{1}{1-1/w} (g(y, z) - \frac{g(y, y)+g(z, z)}{w+1}
        )
        \stext{for} w \in \{n_1,n_2\}
        \stext{and}
        \phi_{n_1, n_2}(y, z) = - 2g(y, z).
    \end{talign}
Moreover, many commonly used test statistics can be recovered as homogeneity U-statistics with the right choice of $g$.

\begin{example}[Kernel two-sample test statistics~\citep{gretton2012akernel}]\label{ex:kernelts}
    When $g$ is chosen to be a positive-definite kernel $\kernel$ \citep[Def.~4.15]{steinwart2008support}, then $U_{n_1,n_2}$ \eqref{Eq: Two-Sample U-statistic} is a \emph{kernel two-sample test statistic} 
    and an unbiased estimator of the squared \emph{maximum mean discrepancy (MMD)} between the distributions $\P$ and $\Q$:
    \begin{talign}
    \mmd^2(\P, \Q)&\defeq (\P\times\P)\kernel + (\Q\times\Q)\kernel - 2(\P\times\Q)\kernel.
    \label{eq:kernel_mmd_distance}
    \end{talign}
\end{example}

\begin{example}[Energy statistics~\citep{baringhaus2004onanew,szekely2004testing,sejdinovic2013equivalence}]\label{ex:energy}
When $g$ is the negative Euclidean distance on $\reals^d$, $U_{n_1,n_2}$~\cref{Eq: Two-Sample U-statistic} is called an \emph{energy statistic}. 
More generally, if $g(y,z) = -\rho(y,z)$ for a semimetric $\rho$ of negative type \citep[Defs.~1-2]{sejdinovic2013equivalence}, %
then $U_{n_1,n_2}$ is a \emph{generalized energy statistic} and an unbiased estimate of the \emph{energy distance} between $\P$ and $\Q$:
\begin{talign}
\mathcal{E}_\rho(\P,\Q) = 
2(\P\times\Q)\rho - (\P\times\P)\rho - (\Q\times\Q)\rho.
\end{talign}
\end{example}

In both of these examples, 
the homogeneity U-statistic \cref{Eq: Two-Sample U-statistic} is {degenerate} under the null with  $(n_1+n_2)\, U_{n_1,n_2}$ converging to a complicated, $\P$-dependent, non-Gaussian limit as $n_1,n_2\to\infty$ \citep[Sec.~5.5.2]{serfling2009approximation}.   
As a result, one typically  
turns to 
permutation testing, as described in \cref{sec:perm-homogeneity}, to 
select the critical values for 
these statistics~\citep{sutherland2017generative,liu2020learning,szekely2004testing,baringhaus2004onanew}. 
\subsection{Independence testing}
\emph{Independence tests} test for dependence between paired variables. 
More precisely, given a paired sample $(\Y, \Z) \defeq (Y_i, Z_i)_{i=1}^{n}$ drawn 
\iid from an unknown distribution $\pjnt$ with marginals $\P$ and $\Q$, independence testing judges the hypotheses 
\begin{talign}
    \nullhypo: \pjnt = \P \times \Q \qtext{versus} \althypo: \pjnt \neq \P \times \Q.
\end{talign}

A common approach to testing independence is to threshold an \emph{independence V-statistic}.

\begin{definition}[Independence V-statistic] \label{def:independence_v_statistic}
Given a paired sample $(\Y, \Z)$ of size $n$, we call $V(\Y, \Z)$, or $V_n$ for short, an \emph{independence V-statistic} if 
    \begin{talign} \label{Eq: V-statistic for independence testing}
        V_n &= \frac{1}{n^4} \sum_{(i_1,i_2,i_3,i_4) \in [n]^4} h_{\itag}((Y_{i_1}, Z_{i_1}),
        (Y_{i_2}, Z_{i_2}),
        (Y_{i_3}, Z_{i_3}),
        (Y_{i_4}, Z_{i_4})),
        \qtext{for}\\
    \label{Eq: Independence kernel}
    h_{\itag}&((y_1,z_1),(y_2,z_2),(y_3,z_3),(y_4,z_4))  
    \\ &\defeq \big(g_Y(y_1,y_2) + g_Y(y_3,y_4)  - g_Y(y_1,y_3) - g_Y(y_2,y_4) \big) \\
    &
    \quad \times  \big(g_Z(z_1,z_2) + g_Z(z_3,z_4) - g_Z(z_1,z_3) - g_Z(z_2,z_4) \big).
    \end{talign}
 for some symmetric real-valued base functions $g_Y$ and $g_Z$.
\end{definition}

Many commonly used test statistics can be recovered as independence V-statistics  with the right choices of $g_Y$ and $g_Z$.

\begin{example}[Hilbert-Schmidt independence criterion (HSIC)~\cite{gretton2005measuring}]\label{ex:hsic}
     When $g_Y$ and $g_Z$ are positive-definite kernels, $\quarter V_n$ \cref{Eq: V-statistic for independence testing} is the empirical \emph{Hilbert-Schmidt independence criterion}, i.e., the squared maximum mean discrepancy $\mmd^2(\PX,\PY\times \PZ)$ \cref{eq:kernel_mmd_distance}, 
     where $\kernel = g_Y g_Z$, 
     $\PX \defeq \frac{1}{n} \sum_{i=1}^n \dirac_{(Y_i, Z_i)}$, $\PY \defeq \frac{1}{n} \sum_{i=1}^n \dirac_{Y_i}$, and $\PZ \defeq \frac{1}{n} \sum_{i=1}^n \dirac_{Z_i}$ \citep[App.~B]{schrab2022efficient}. 
\end{example}

\begin{example}[Distance covariance  ~\cite{szekely2007measuring,sejdinovic2013equivalence}]
    When $g_Y(y_1, y_2) = \twonorm{y_1 - y_2}$ and $g_Z(z_1, z_2) = \twonorm{z_1 - z_2}$ for $y_1\in\reals^{d_1}$ and $z_1\in\reals^{d_2}$, $\quarter V_n$ \cref{Eq: V-statistic for independence testing} is called the \emph{distance covariance} and forms the basis for the \emph{energy test of independence}. More generally, if $g_Y$ and $g_Z$ are semimetrics of negative type \citep[Defs.~1-2]{sejdinovic2013equivalence}, then $V_n$ is a \emph{generalized distance covariance}.
\end{example}

In both of these examples, the independence V-statistic \cref{Eq: Independence kernel} is degenerate under the null with $n V_n$ converging to a complicated, $\P$ and $\Q$-dependent, non-Gaussian limit as $n\to\infty$ \citep[Sec.~5.5.2]{serfling2009approximation}.   
As a result, one typically turns to permutation testing, for example, as described in \cref{sec:wild-independence}, to select   critical values for these statistics~\citep{szekely2007measuring,chwialkowski2014wild}. %

\section{Cheap permutation testing}\label{sec:cheap}
Given a test statistic, permutation testing provides an automated procedure for selecting the critical value and rejection probability of a test \cref{eq:test}.
In this section, we review standard (Monte Carlo) permutation tests for homogeneity %
and independence %
and introduce new cheap permutation tests that reduce the runtime and memory demands. The results of this section are summarized in \cref{table:tests_complexities}.

\subsection{Permutation tests of homogeneity}\label{sec:perm-homogeneity} 
Let 
$\X \defeq (X_{i})_{i=1}^{n} = (Y_1, \ldots, Y_{n_1}, Z_1, \ldots, Z_{n_2})$ 
represent a concatenation of the observed samples $\Y =(Y_i)_{i=1}^{n_1}$ and $ \Z=(Z_i)_{i=1}^{n_2}$ with 
$n \defeq n_1+n_2$.
In a standard permutation test of homogeneity, one 
samples $\numperm$ independent permutations $(\pi_b)_{b=1}^{\numperm}$ of the indices $[n]$,  
computes the permuted test statistics $(T(\X^{\pi_b}))_{b=0}^{\numperm}$
for $\pi_0$ the identity permutation and 
$\X^{\pi} \defeq (X_{\pi(i)})_{i=1}^{n}$,
and finally selects a $1-\alpha$ quantile of the empirical  distribution $\frac{1}{\numperm+1}\sum_{b=0}^{\numperm} \dirac_{T(\X^{\pi_b})}$ as the critical test value $c(\X)$.  

For a quadratic test statistic (\cref{def:qts}), this permutation test only depends on the data via the $3n^2$ sufficient statistics
\begin{talign}
    \hcomp \defeq \big\{\phi(x_i, x_j) \mid i, j\in[n], \phi\in\sbraces{\phi_{n_1}, \phi_{n_2},\phi_{n_1, n_2}}\big\}.
\end{talign}
Hence, if we let $c_{\phi}$ denote the maximum time required to evaluate a single element of $\hcomp$, %
then one can precompute and store the elements of $\hcomp$ given  $\Theta(c_{\phi}n^2)$ time and $\Theta(n^2)$ memory and, thereafter,  compute all permuted test statistics using $\Theta(\numperm n^2)$ elementary operations.
When the quadratic memory cost is unsupportable, the elements of $\hcomp$ can instead be recomputed for each permutation at a total cost of $\Theta(c_{\phi}\numperm n^2)$ time and $\Theta(n)$ memory.

\subsection{Cheap permutation tests of homogeneity}
To alleviate both the time and memory burden of standard permutation testing, we propose to permute datapoint bins instead individual points. 
More precisely, we partition the indices $[n]$ into $s$ consecutive sets $\idx_1, \dots, \idx_s$ of equal size, 
identify the corresponding datapoint bins $\X_j\defeq (X_i)_{i\in \idx_{j}}$, and define the bin-permuted samples $\Y^{\pi} \defeq (\X_{\pi(i)})_{i=1}^{s_1}$ and $\Z^{\pi} \defeq (\X_{\pi(i)})_{s_1+1}^{s}$ for each permutation $\pi$ on $[s]$ and $s_1\defeq s\frac{n_1}{n}$. 
Importantly, the bin-permuted QTS,
\begin{talign}
T(\Y^{\pi}, \Z^{\pi}) 
    &= 
\frac1{n_1^2}\sum_{i,j=1}^{s_1} 
   \Phi^{n_1}_{ij}   
   +
\frac1{n_1n_2}\sum_{i=1}^{s_1} \sum_{j=s_1+1}^{s} 
   \Phi^{n_1,n_2}_{ij}  
   +
\frac1{n_2^2}\sum_{i,j=s_1+1}^{s}
   \Phi^{n_2}_{ij},
\end{talign}
only depends on the data via the $3s^2$  sufficient statistics
\begin{talign}
    \hcheap \defeq\braces{ \Phi^{w}_{ij} \defeq\sum_{a\in\idx_i}\sum_{b\in\idx_j}\phi_{w}(X_a, X_b) \mid i,j \in [s], w\in\sbraces{n_1,n_2,(n_1, n_2)}}.
\end{talign}
Hence, one can precompute and store the elements of $\hcheap$ using only $\Theta(s^2)$ memory and $\Theta(c_{\phi}n^2)$ time and subsequently compute all bin-permuted statistics using only $\Theta(\numperm s^2)$ elementary operations. 
\cref{algo:cheap_homogeneity_test_qts} 
details the steps of this \emph{cheap homogeneity test} which recovers the standard permutation test when $s=n$.

Notably, the overhead of cheap permutation is negligible whenever $\numperm s^2 \ll c_\phi n^2$. 
In \cref{sec:power,sec:experiments}, we will see that it suffices to choose $s$ as a slow-growing or even constant function of $n$, 
yielding total runtime comparable to that of evaluating the original test statistic. 

\begin{table} %
    \caption{Time and memory costs for cheap and standard permutation tests with $n$ datapoints, $\numperm$ permutations, and $s$ bins. The multipliers $c_{\phi}$ and $c_g$  denote the cost of evaluating a single homogeneity QTS base function (see \cref{def:qts}) and a single independence V-statistic base function (see  \cref{def:independence_v_statistic}). 
    All costs are reported up to a universal constant.}
 \label{table:tests_complexities}
    \centering
    \begin{tabular}{ccccc}
        \toprule
        {\bf Permutation test}
        &  {\bf Test statistic $\pmb{T(\Y,\Z)}$}
        & {\bf $\pmb{T(\Y,\Z)}$ time} 
        & {\bf Permutation time} 
        & {\bf Memory} 
        \\[0mm]
        \midrule 

        \multirow{2}{*}{Homogeneity} 
        & \multirow{2}{*}{QTS (\cref{def:qts})}
        & \multirow{2}{*}{$c_{\phi} n^2$}
        & $\numperm n^2$
        & $n^2$
        \\[0mm]
        &
        & 
        & $c_{\phi}\numperm n^2$
        & $n\ \,$
        \\[2mm]
        Cheap homogeneity (\cref{algo:cheap_homogeneity_test_qts})
        & QTS (\cref{def:qts})
        & $c_{\phi} n^2$
        & $\blue{\numperm s^2}$
        & $\blue{s^2}$
        \\[2mm]
        
        \multirow{2}{*}{Independence}
        & \multirow{2}{*}{V-statistic (\cref{def:independence_v_statistic})}
        & \multirow{2}{*}{$c_{g} n^2$}
        & $\numperm n^2$
        & $n^2$
        \\[0mm]
        &
        & 
        & $c_{g}\numperm n^2$
        & $n\ \,$
        \\[2mm]

        Cheap independence (\cref{algo:cheap_independence_permutation_test})
        & V-statistic (\cref{def:independence_v_statistic})
        & $c_g n^2$
        & $\blue{\numperm s^2}$
        & $\blue{s^2}$
        \\[1ex] \bottomrule \hline
    \end{tabular}
\end{table}

\begin{algorithm2e} %
\small{    
  \KwIn{Samples $(Y_i)_{i=1}^{n_1}$, $(Z_i)_{i=1}^{n_2}$, bin count $s$, base functions $\phi_{n_1},\phi_{n_2},\phi_{n_1,n_2}$ (\cref{def:qts}), level $\alpha$, permutation count $\numperm$}
  \vspace{5pt}
  
  Define $(X_i)_{i=1}^{n} \defeq (Y_1, \ldots, Y_{n_1}, Z_1, \ldots, Z_{n_2})$ for $n \gets n_1+n_2$ and $s_1 \gets {s\, n_1/}{n}$ \\ 
  Split the indices $\sbraces{1, \ldots, n}$ into $s$ consecutive bins $\bin_1, \ldots, \bin_s$ of equal size \\ 
\textit{// Compute sufficient statistics using $\Theta(c_\phi n^2)$ time and $\Theta(s^2)$ memory} \\
    \lFor{$i,j\in[s]$ and $w\in\{n_1,n_2,(n_1,n_2)\}$}{
    $\Phi^{w}_{ij} \gets \sum_{a \in \bin_i, b\in \bin_j }\phi_{w}(X_a, X_b) $
    \label{line:K_ij_cheap_two_sample}
  }

    \textit{// Compute original and  permuted test statistics using $\Theta(\numperm s^2)$ elementary operations}\\    
  \For{$b=0, 1, 2, \ldots, \numperm$}
    {%
   
    $\pi \gets$ identity permutation if $b=0$ else uniform permutation of $[s]$ \\ 

   $T_{b} \gets  \frac1{n_1^2}\sum_{i,j=1}^{s_1} 
   \Phi^{n_1}_{\pi(i)\pi(j)} +  \frac1{n_1n_2}\sum_{i=1}^{s_1}\sum_{j=s_1+1}^{s} \Phi^{n_1,n_2}_{\pi(i)\pi(j)}$ $+  \frac1{n_2^2}
   \sum_{i,j=s_1+1}^{s}
   \Phi^{n_2}_{\pi(i)\pi(j)} 
   $
   }
   \textit{// Return rejection probability} \\
   $R \gets 1+$ number of permuted statistics $(T_{b})_{b=1}^{\numperm}$
   smaller than $T_0$ if ties are broken at random %
   \label{line:fourth_to_last_homogeneity}\\
    \KwRet{$\Delta(\X) \defeq \min(1,\max(0,R \!-\! (1\!-\!\alpha)(\numperm\!+\!1)))$}     %
    \label{line:last} 
}
  \caption{Cheap homogeneity testing}
  \label{algo:cheap_homogeneity_test_qts}
\end{algorithm2e}

\subsection{Wild bootstrap tests of independence}\label{sec:wild-independence}
To ease notation, we will focus on the case of even $n$.  
In a standard {wild bootstrap} permutation test of independence \citep{chwialkowski2014wild}, one begins by sampling   
$\numperm$ independent \emph{wild bootstrap permutations} $(\pi_b)_{b=1}^{\numperm}$ of the indices $[n]$. 
Each wild bootstrap permutation $\pi$ swaps the indices 
$i$ and 
$\swap(i) \defeq i + \frac{n}{2} - n\indic{i  > \frac{n}{2}}$ 
with probability $\half$ and does so independently for each $i \leq n/2$.
Then, one computes the permuted test statistics $(T(\Y,\Z^{\pi_b}))_{b=0}^{\numperm}$
for $\pi_0$ the identity permutation and  
$\Z^{\pi} \defeq (Z_{\pi(i)})_{i=1}^{n}$
and finally selects a $1-\alpha$ quantile of the empirical  distribution $\frac{1}{\numperm+1}\sum_{b=0}^{\numperm} \dirac_{T(\Y,\Z^{\pi_b})}$ as the critical test value $c(\Y,\Z)$.  

Now consider the independence V-statistic $V(\Y,\Z)$ of \cref{def:independence_v_statistic}, and introduce the shorthand 
\begin{talign}
(\gymean,\, \gzmean[i],\, \bar{g}_Y,\, \bar{g}_Z)
    &\defeq
(\frac{1}{n}\sum_{j=1}^ng_Y(Y_i, Y_j),\, \frac{1}{n}\sum_{j=1}^ng_Z(Z_i, Z_j),\, \frac{1}{n}\sum_{i=1}^n\gymean,\,
\frac{1}{n}\sum_{i=1}^n\gzmean[i])
\\
\qtext{and} \gyz[ij,k\ell] 
    &\defeq
(g_Y(Y_i,Y_j)-\gymean-\gymean[j]+\bar{g}_Y)
(g_Z(Z_k,Z_\ell)-\gzmean-\gzmean[\ell]+\bar{g}_Z)
\end{talign}
for all indices $i,j,k,\ell\in[n]$. With $V(\Y,\Z)$ as its test statistic, the standard wild bootstrap test only depends on the data via the $n^2/2$ sufficient statistics   
\begin{talign}
\set_1
    \defeq
\big\{
&\sum_{a,b\in \{i,\swap(i)\}} 
    (\gyz[ab,ba],\,
    \gyz[ab,\swap(b)\swap(a)])
    \mid 
 i\in[n/2]\big\} 
    \stext{and}\\ 
\set_2
    \defeq 
\big\{
&\sum_{a\in \{i,\swap(i)\},b\in \{j,\swap(j)\}}
(\gyz[ab,ba],\, \gyz[ab,\swap(b)\swap(a)]), \\
&\sum_{a\in \{i,\swap(i)\},b\in \{j,\swap(j)\}}
(\gyz[ab,b\swap(a)],\, \gyz[ab,\swap(b)a])
    \mid 
j < i \in [n/2]\big\}.
\end{talign}
since, for each wild bootstrap permutation $\pi$ and $j\in[n]$, we have $\pi(j)\in\{j,\swap(j)\}$ and 
\begin{talign}
4V(\Y, \Z^{\pi}) 
    &=  
\frac{1}{n^2}\sum_{i=1}^{n/2} \sum_{a,b\in \{i,\swap(i)\}}  \gyz[ab,\pi(a)\pi(b)] \\
    &+ 
\frac{2}{n^2}\sum_{i=1}^{n/2}\sum_{j =1}^{i-1} \sum_{a\in \{i,\swap(i)\}, b\in \{j,\swap(j)\}} \gyz[ab,\pi(a)\pi(b)] .
\end{talign}

Hence, if we let $c_{g}$ denote the maximum time required to evaluate $g_Y$ or $g_Z$, 
then one can precompute and store the elements of $\icomp[1]\cup\icomp[2]$ given  $\Theta(c_{g}n^2)$ time and $\Theta(n^2)$ memory and, thereafter,  compute all permuted test statistics using $\Theta(\numperm n^2)$ elementary operations.
When the quadratic memory cost is unsupportable, the elements of $\icomp[2]$ can instead be recomputed for each permutation at a total cost of $\Theta(c_{g}\numperm n^2)$ time and $\Theta(n)$ memory.

\subsection{Cheap wild bootstrap tests of independence}
To alleviate both the time and memory burden of standard wild bootstrap testing, we again propose to permute datapoint bins instead of individual points. 
To simplify our presentation, we will assume that $s$ is even and divides $n$ evenly.
As before, we partition the indices $[n]$ into $s$ consecutive sets $\idx_1, \dots, \idx_s$ of equal size, 
identify the corresponding datapoint bins $\Z_j\defeq (Z_i)_{i\in \idx_{j}}$, and define the bin-permuted sample $\Z^{\pi} \defeq (\Z_{\pi(i)})_{i=1}^{s}$ for each wild bootstrap permutation $\pi$ on $[s]$. 
If we further define the combination bins $\idx_i' \defeq \idx_i\cup\idx_{i+s/2}$ for each $i\in[\frac{s}{2}]$, then
the bin-wild-bootstrapped V-statistic 
\begin{talign} \label{eq:cheap_permuted_V_statistic}
    4V(\Y, \Z^{\pi}) 
    &=  
\frac{1}{n^2}\sum_{i=1}^{s/2}
D_{i\pi(i)}
    +
\frac{2}{n^2}\sum_{i=1}^{s/2}\sum_{j =1}^{i-1} 
S_{ij\pi(j)\pi(i)} %
\end{talign}
only depends on the data via the $s^2/2$  sufficient statistics %
\begin{talign}
\set_1'
    \defeq
\big\{
&(D_{ii},\, D_{i(i+s/2)})
    \defeq 
\sum_{a,b\in \idx_{i}'} 
    (\gyz[ab,ba],\,
    \gyz[ab,\swap(b)\swap(a)])
    \mid 
 i\in[\frac{s}{2}]\big\} 
    \stext{and}\\
\set_2'
    \defeq 
\big\{
&(S_{ijji},\,  S_{ij(j+s/2)(i+s/2)})
    \defeq
\sum_{a\in\idx_{i}',\,b\in\idx_j'}
(\gyz[ab,ba],\, \gyz[ab,\swap(b)\swap(a)]),
\\
&(S_{ijj(i+s/2)},\, S_{ij(j+s/2)i})
    \defeq
\sum_{a\in\idx_{i}',\,b\in\idx_j'}
(\gyz[ab,b\swap(a)],\, \gyz[ab,\swap(b)a])
    \mid 
j < i \in [\frac{s}{2}]\big\}.
\end{talign}
Hence, one can precompute and store the elements of $\set_1'\cup\set_2'$ using only $\Theta(s^2)$ memory and $\Theta(c_{g}n^2)$ time and subsequently compute all bin-permuted statistics using only $\Theta(\numperm s^2)$ elementary operations. 
\cref{algo:cheap_independence_permutation_test} 
details the steps of this \emph{cheap independence test} which recovers the standard wild bootstrap permutation test when $s=n$. %

As in the homogeneity setting, we  find that the overhead of cheap permutation is negligible whenever $\numperm s^2 \ll c_\phi n^2$, and we will see in \cref{sec:power,sec:experiments} that a slow-growing or constant setting of $s$ suffices to maintain power while substantially reducing runtime. 

A natural alternative to using wild bootstrap permutations is to instead employ uniformly random permutations $\pi$ on $[s]$ \citep[see, e.g.,][]{szekely2007measuring}. Unfortunately, while $s^2/2$ scalar statistics suffice to compute any bin-wild-bootstrapped V-statistic, at least $\Theta(s^4)$ statistics of the form 
\begin{talign}
\sum_{a\in\idx_{i}',\,b\in\idx_{j}',c\in\idx_{k}',d\in\idx_{l}'}\gyz[ab,cd]
    \qtext{for}
i,j,k,l\in[s]
\end{talign}
are required to compute all bin-permuted V-statistics, $4V(\Y, \Z^{\pi})$. 
As a result, the cheap wild bootstrap test of independence enjoys substantial speed and memory benefits over the analogous cheap (uniform) permutation test of independence. 
In \cref{sec:minimax,sec:power}, we will also establish power guarantees for cheap wild bootstrapping that match the best known rates for standard permutation testing and the minimax optimal rates for independence testing. We therefore recommend the cheap wild bootstrap independence test of \cref{algo:cheap_independence_permutation_test} as a fast and powerful replacement for both wild boostrap and uniform permutation tests.

\begin{algorithm2e} %
\small{
  \KwIn{Paired sample ${((Y_i, Z_i))}_{i=1}^n$,
  bin count $s$, base functions $g_Y$, $g_Z$ (\cref{def:independence_v_statistic}), level $\alpha$, wild bootstrap count $\numperm$}
  \vspace{5pt}
  Split the indices $\sbraces{1, \ldots, n}$ into $s$ consecutive bins $\bin_1, \ldots, \bin_s$ of equal size \\ %

\textit{// Compute sufficient statistics using $\Theta(c_g n^2)$ time and $\Theta(s^2)$ memory} \\
  \lFor{$a\in[n]$}{
    $(\gymean[a],\gzmean[a]) \gets \frac{1}{n}\sum_{b=1}^n(g_Y(Y_a, Y_b),
g_Z(Z_a, Z_b))$}
$(\bar{g}_Y, \bar{g}_Z) \gets\frac{1}{n}\!\sum_{a=1}^n(\gymean[a],\gzmean[a])$  \\
  Define the mappings 
 $\swap(a) \defeq a + \frac{n}{2} - n\indic{a  > \frac{n}{2}}$ 
  and $G_{ab,cd}\defeq (g_Y(Y_a,Y_b)-\gymean[a]-\gymean[b]+\bar{g}_Y)
(g_Z(Z_c,Z_d)-\gzmean[c]-\gzmean[d]+\bar{g}_Z)$\\
  \For{$i\in[s]$}{
    $(D_{ii},\,D_{i(i+s/2)})  \gets \sum_{a, b\in \idx_{i}} 
    (G_{ab,ba},\,  G_{ab,\swap(b)\swap(a)})$ \\
    \For{$j\in[i-1]$}{
        $(S_{ijji},\, S_{ij(j+s/2)(i+s/2)}) \gets \sum_{a\in\bin_i,b\in\bin_j} (G_{ab,ba},\, G_{ab,\swap(b)\swap(a)})$ \\
        $(S_{ij(j+s/2)i},\, S_{ijj(i+s/2)}) \gets \sum_{a\in\bin_i,b\in\bin_j} (G_{ab,\swap(b)a},\, G_{ab,b\swap(a)})$
        \label{line_S_17}
    }
  }

    \textit{// Compute original and  permuted test statistics using $\Theta(\numperm s^2)$ elementary operations}\\  
  \For{$b=0, 1, 2, \ldots, \numperm$}
    {%
    $\pi \gets$ identity permutation if $b=0$ else wild bootstrap permutation of $[s]$ \\ 
    $T_b \gets \frac{1}{n^2}\sum_{i=1}^s
D_{i\pi(i)}
    +
\frac{2}{n^2}\sum_{i=1}^s\sum_{j =1}^{i-1} 
S_{ij\pi(j)\pi(i)}$
   }
   
   \textit{// Return rejection probability} \\
   $R \gets 1 +$ number of permuted statistics $(T_{b})_{b=1}^{\numperm}$
   smaller than $T_0$ if ties are broken at random 
   \label{line:fourth_to_last_independence}\\
    \KwRet{$\Delta(\X) \defeq \min(1,\max(0,R \!-\! (1\!-\!\alpha)(\numperm\!+\!1)))
    $}
}
  \caption{Cheap independence testing\label{algo:cheap_independence_permutation_test}}
\end{algorithm2e}

\subsection{Finite-sample exactness}%
We conclude this section by noting that the cheap permutation tests of \cref{algo:cheap_homogeneity_test_qts,algo:cheap_independence_permutation_test} are \emph{exact}, that is, their sizes exactly match the nominal level $\alpha$ for all sample sizes, data distributions, and permutation counts $\numperm$.
This is a principal advantage of permutation testing over asymptotic tests that (a) require precise knowledge of a test statistic's asymptotic null distribution and (b) can violate the level constraint at any given sample size.

\begin{proposition}[Cheap permutation tests are exact] \label{prop:exactness}
    Under their respective null hypotheses, the cheap permutation tests of \cref{algo:cheap_homogeneity_test_qts,algo:cheap_independence_permutation_test} reject the null with probability $\alpha$.
\end{proposition}
\begin{proof}
Since cheap permutation tests are standard permutation tests over bins of datapoints and \cref{algo:cheap_homogeneity_test_qts,algo:cheap_independence_permutation_test} both employ the randomized rejection rule of \citet{hoeffding1952large},  exactness follows immediately from Prop.~3 of \citet{hemerik2018exact}.
\end{proof}
\section{Power of cheap testing}
\label{sec:power}
We now turn our attention to the power of cheap testing with homogeneity U-statistics (\cref{def:homogeneity_u_stat}) and independence V-statistics (\cref{def:independence_v_statistic}). 
All of our results apply to any choice of permutation count $\numperm$ and are non-vacuous exactly when $\numperm \geq \frac{1}{\alpha}-1$.

\subsection{Power guarantees for homogeneity testing} \label{sec:ts_finite}
For the homogeneity U-statistic of \cref{def:homogeneity_u_stat}, define the symmetrized version of 
$\hho$,
\begin{talign} \label{Eq: symmetrized kernel}
\hbarho (y_1,y_2;z_1,z_2) \defeq \frac{1}{2!2!} \sum_{(i_1,i_2) \in \mathbf{i}_2^{2}} \sum_{(j_1,j_2) \in \mathbf{i}_2^{2}} \hho(y_{i_1},y_{i_2}; z_{j_1},z_{j_2}), 
\end{talign} 
and the \emph{variance components}
\begin{talign}
\psi_{Y,1}&\defeq \max\{ \Var(\E[\hbarho(Y_1,Y_2;Z_1,Z_2)|Y_1]), \Var(\E[\hbarho(Z_3,Y_2;Z_1,Z_2)|Z_3]) \}, \\  \label{Eq: definitions of psi functions}
\psi_{Z,1}&\defeq \max\{ \Var(\E[\hbarho(Y_1,Y_2;Z_1,Z_2)|Z_1]), \Var(\E[\hbarho(Y_1,Y_2;Y_3,Z_2)|Y_3]) \}, \stext{and} \\
\psi_{YZ,2}&\defeq \max \{ \E[g^2(Y_1,Y_2)], \  \E[g^2(Y_1,Z_1)], \ \E[g^2(Z_1,Z_2)] \}.
\end{talign}
Our first theorem, proved in \cref{subsec:proof_ts_finite}, bounds the minimum separation between $\P$ and $\Q$ required for standard and cheap homogeneity tests to reject with power at least $1-\beta$. 
Here, separation is measured by the mean of the test statistic $\E[U_{n_1,n_2}]$, which is zero under the null ($\P=\Q$) but may be non-zero under the alternative ($\P\neq\Q$). 
\begin{theorem}[Power of cheap homogeneity: finite variance] %
 \label{thm: Full Two Sample Tests}
Consider a homogeneity U-statistic $U_{n_1,n_2}$ (\cref{def:homogeneity_u_stat}) with finite variance components \cref{Eq: definitions of psi functions}, sample sizes $n_2 \geq n_1 \geq 4$, and $n\defeq n_1+n_2$.  
A cheap homogeneity test (\cref{algo:cheap_homogeneity_test_qts} with $s \geq 4$) %
based on $U_{n_1,n_2}$ has power at least $1-\beta$ whenever
    \begin{talign} 
    \label{eq:lb_cheap_2s}
\mathbb{E}[U_{n_1,n_2}] 
    &\geq 
\gamma_{n_1,n_2,s} 
    \defeq 
\gamma_{n_1,n_2} 
    +
\indic{s < n}\frac{\gamma_{n_1,n_2}
    +
\sqrt{\frac{201s^2}{4n^2} \psi_{YZ,2} 
    + 
{\frac{12 s}{n} \max( \psi_{Y,1}, \psi_{Z,1})}{}}}{ \sqrt{s^3\frac{\beta\astar}{24(1-\astar)}\nratio}-1} %
\\ \label{eq:lb_full_2s}
\qtext{for}  \gamma_{n_1,n_2} &\defeq
        \sqrt{\frac{3-\beta}{\beta}\big(\frac{4  \psi_{Y,1}}{n_1} \! + \! \frac{4  \psi_{Z,1}}{n_2} \big)} + \sqrt{\frac{1-\astar}{\astar}\frac{(36 n_1^2 + 36 n_2^2 + 198 n_1 n_2) \psi_{YZ,2}}{\beta n_1 (n_1 - 1) n_2 (n_2 - 1)}},
\end{talign}
$\astar\defeq \frac{\alpha}{2e} \big(\frac{\beta}{3}\big)^{1/\floor{\alpha(\numperm+1)}}$,
and
$\nratio \defeq \frac{ n_1(n_1 - 1)^2 n_2(n_2 - 1)^2}{n^6}$.
\end{theorem}
\begin{remark}[Properties of $\astar$]
The parameter $\astar$ lies in the range $[\frac{\alpha}{2e} \frac{\beta}{3}, \frac{\alpha}{2e})$ for all $\numperm \geq \frac{1}{\alpha}-1$ with $\astar \geq \frac{\alpha^2}{2e}$ for $\numperm \geq \frac{\log(3/\beta)}{\alpha\log(1/\alpha)}-1$
and
$\astar \geq \frac{\alpha}{2e^2}$ for $\numperm \geq \frac{\log(3/\beta)}{\alpha}-1$.
\end{remark}
\begin{remark}[Equivalent separation rates and thresholds]\label{rem:sep_rates_homogeneity}
We have $\gamma_{n_1,n_2,s} = \Theta(\gamma_{n_1,n_2})$ whenever $s\geq \sqrt[3]{\frac{24(1-\astar)(1+\Omega(1))}{\beta\astar \nratio}}$
and $\gamma_{n_1,n_2,s} = \gamma_{n_1,n_2}(1 + o(1))$ whenever 
$s=\omega(\frac{n_2}{n_1})$.
\end{remark}

The separation threshold $\gamma_{n_1,n_2}$ for standard permutation testing is a refined, quantitative version of the one derived in \cite[Thm.~4.1]{kim2022minimax}. Here, we keep track of all numerical factors and, thanks to the quantile comparison method introduced in \cref{sec:quantile_comparison}, 
establish a tighter dependency on the number of permutations $\numperm$.
The separation threshold $\gamma_{n_1,n_2,s}$ for cheap permutation testing is entirely new and implies two remarkable properties summarized in \cref{rem:sep_rates_homogeneity}.
First, the cheap and standard thresholds are \emph{asymptotically equivalent} (in the strong sense that $\gamma_{n_1,n_2,s}/\gamma_{n_1,n_2} = 1 + o(1)$) whenever $s=\omega(\frac{n_2}{n_1})$.
For $n_1 = \Theta(n_2)$, this occurs whenever $s$ grows unboundedly with $n$, even if the growth rate is arbitrarily slow.
Second, the cheap and standard \emph{separation rates} are identical (i.e., $\gamma_{n_1,n_2,s} = \Theta(\gamma_{n_1,n_2})$) whenever $s \geq \sqrt[3]{\frac{24(1-\astar)(1+\Omega(1))}{\beta\astar \nratio}}$.  Hence, when $n_1 = \Theta(n_2)$,  cheap testing can match the separation rate of standard permutation testing using a \emph{constant} number of bins independent of the sample size $n$.

\subsection{Power guarantees for independence testing} \label{sec:ind_finite}

For the independence V-statistic of \cref{def:independence_v_statistic}, 
we introduce the shorthand $X_i \defeq (Y_i,Z_i)$ and 
define a symmetrized version of 
$\hin$,
\begin{talign}
\hbarin(x_1,x_2,x_3,x_4) &\defeq \frac{1}{4!} \sum_{(i_1,i_2,i_3,i_4) \in \mathbf{i}_4^{4}} \hin(x_{i_1},x_{i_2},x_{i_3},x_{i_4}), 
\end{talign}
the \emph{variance components}
\begin{talign} 
\begin{split}
\label{Eq: definition of psi prime functions}
\psi_1' 
    &\defeq 
\Var(\E[\hbarin(X_1,X_2,X_3,X_4)\mid X_1]), 
    \\%[.5em]
\quad\tilde{\psi}_1' 
    & \defeq 
\Var( \E[\hbarin((Y_1,Z_{\pi_1}),(Y_2,Z_{\pi_2}),(Y_3,Z_{\pi_3}),(Y_4,Z_{\pi_4})) \mid (Y_1,Z_{\pi_1}, \pi_1)]  ), 
\quad\text{and} 
    \\
\psi_2' 
    & \defeq 
16 \max \big\{ \E[g_Y^2(Y_1,Y_2)g_Z^2(Z_i,Z_j)] : (i,j) \in \{(1,2),(1,3),(3,4)\}\big\},
\end{split}
\end{talign}
where $\pi$ is an independent wild bootstrap permutation of $[n]$, and the \emph{mean component}
\begin{talign}\label{Eq: mean component}
\tilde{\xi} 
    &\defeq 
\max\{ | \E[\hin((Y_{i_1},Z_{i'_1}),(Y_{i_2},Z_{i'_2}),(Y_{i_3},Z_{i'_3}),(Y_{i_4},Z_{i'_4}))] |
: 
i_k, i'_k \in [8], \forall k\in[4]\}.
\end{talign}

Our next theorem, proved in \cref{subsec:proof_independence_finite_variance}, bounds the minimum separation between $\pjnt$ and $\P\times\Q$ required for standard and cheap independence tests to reject with power at least $1-\beta$. 
Here, separation is measured by 
\begin{talign}\label{eq:mcU}
\mathcal{U} 
    &\defeq 
\E[\hin(X_1,X_2,X_3,X_4)],
\end{talign}
the expected U-statistic corresponding to the V-statistic $V_n$. 
Notably,  $\mathcal{U}$  is zero under the null ($\pjnt = \P\times\Q$) but may be non-zero under the alternative ($\pjnt \neq \P\times\Q$).
\begin{theorem}[%
Power of cheap independence: finite variance] \label{thm: Full Independence Tests}
Consider an independence V-statistic $V_{n}$ (\cref{def:independence_v_statistic}) with finite variance components \cref{Eq: definition of psi prime functions}, 
$\astar \defeq \frac{\alpha}{2e} \big(\frac{\beta}{3}\big)^{1/\floor{\alpha(\numperm+1)}}$, 
and $n\geq 8$. 
    A cheap independence test (\cref{algo:cheap_independence_permutation_test}) %
    based on $V_n$ has power at least $1-\beta$ whenever 
\begin{talign} \label{eq:ind_cheap_suff}
\mathcal{U} 
    &\geq 
\gamma_{n,s} 
    \defeq 
 \gamma_n
    +
 \indic{s<n}
 \frac{\gamma_n(s+1) + 4s^2\sqrt{\big( \frac{12}{\astar \beta} - 2\big)
    \big(\tilde{\psi}_1' \big( \frac{2304}{n s} + \frac{460800}{n s^2}
    \big) + \psi_2' \frac{4147200}{n^2 s} \big)}}{ 3s^2 - s - 1}\!\! \\
    \label{eq:ind_full_suff}
        \text{for}\ 
        \gamma_n &\defeq 
        \frac{4}{3}
        \big(\sqrt{(\frac{3}{\beta} \!-\! 1) \big(\frac{16 \psi_1'}{n}\! +\! \frac{144 \psi_2'}{n^2} \big)}
        \!+\! 
        \sqrt{\big( \frac{12}{\astar \beta} \!- \!2\big) \big( \frac{32 \tilde{\psi}_1'}{n} \!+ \!\frac{960 \psi_2'}{n^2} \big)}
        \!+ \!\frac{16}{n} \tilde{\xi}\! +\! 
        \frac{24 %
        }{n}\sqrt{\psi_2'}
        \big).
    \end{talign}
\end{theorem}
\begin{remark}[Equivalent separation rates and thresholds]\label{rem:sep_rates_independence}
We have $\gamma_{n,s} = \Theta(\gamma_{n})$ for \textbf{any} choice of $s$ 
and $\gamma_{n,s} = \gamma_{n}(1 + o(1))$ whenever $s=\omega(1)$.
\end{remark}

The separation threshold $\gamma_{n}$ is a quantitative analogue of the threshold derived in \citep[Thm.~5.1]{kim2022minimax}. The main structural difference is that \cref{thm: Full Independence Tests} analyzes the wild-bootstrapped V-statistic $V_n$ while Thm.~5.1 of \citep{kim2022minimax} analyzes a permutation test built on the corresponding independence U-statistic; this accounts for the two quantities in $\gamma_n$ that have no counterpart in \citep[Thm.~5.1]{kim2022minimax}. The term $\tilde{\xi}$ \cref{Eq: mean component} collects the diagonal contributions retained by the V-statistic but discarded by the U-statistic and enters only at the lower order $O(1/n)$. The term $\tilde{\psi}_1'$ \cref{Eq: definition of psi prime functions} is the first-order variance component measured under the wild-bootstrap resampling distribution rather than under $\pjnt$, entering at the same leading order $O(1/\sqrt{n})$ as $\psi_1'$. We will see in \cref{lem:tilde_psi_1_prime_bound} that $\tilde{\psi}_1'$ is bounded by a constant multiple of $\psi_1'$ for positive-definite kernels so that $\gamma_n$ matches the separation rate of \citep[Thm.~5.1]{kim2022minimax} up to constant factors for such kernels. \cref{thm: Full Independence Tests} additionally keeps track of all numerical constants and, thanks to the quantile comparison method of \cref{sec:quantile_comparison}, yields a tighter dependency on the number of permutations $\numperm$.

The separation threshold $\gamma_{n,s}$ for cheap permutation testing is entirely new and satisfies two remarkable properties (\cref{rem:sep_rates_independence}). First, the cheap and standard thresholds are again asymptotically equivalent whenever $s=\omega(1)$. Second, the cheap and standard separation \emph{rates} are identical (i.e., $\gamma_{n,s} = \Theta(\gamma_{n})$) for any choice of $s$, even $s=2$. Hence, cheap independence testing can match the separation rate of standard testing using a \emph{constant} number of bins independent of the sample size $n$.

\subsection{Quantile comparison method}\label{sec:quantile_comparison}
At the heart of our power guarantees is a quantile comparison method (\cref{Lemma: Generic Two Moments Method}) suitable for bounding the power of Monte Carlo permutation tests. %
Related results have appeared in the literature under the name \textit{two moments method},  as they control quantiles using the first two moments of the test statistic.
This approach to assessing non-asymptotic power was pioneered by \citet{fromont2012thetwosample} in the context of homogeneity testing and later developed by \citet{kim2022minimax,schrab2023mmd,domingo2023compress,kim2023differentially}, among others.  
Earlier guarantees due to \citep{kim2022minimax,schrab2023mmd,kim2023differentially} placed undue constraints on the number of permutations $\numperm$. For example, Prop.~D.1 of \citep{kim2022minimax} requires $\numperm \geq \frac{8}{\alpha^2}\log(\frac{4}{\beta})$,  Thm.~5 of \citep{schrab2023mmd} requires $\numperm \geq \frac{3}{\alpha^2}(\log(\frac{8}{\beta}) + \alpha(1-\alpha))$, and Lem.~21 of \citep{kim2023differentially} requires $\numperm \geq \frac{1}{\alpha}(\log(\frac{1}{\beta})+1-\alpha)$.
Here we adapt a tighter argument developed for homogeneity testing with a bounded positive-definite kernel \citep[Lem.~6]{domingo2023compress} and extend it to all forms of permutation test and, more generally, to any test that thresholds using the quantiles of \iid perturbations.
The result, proved in  \cref{Section: Proof of Lemma: Two Moments Method SubG}, 
provides non-vacuous power guarantees for all permutation counts $\numperm \geq \frac{1}{\alpha} -1$.

\begin{proposition}[Quantile comparison method] \label{Lemma: Generic Two Moments Method}
Consider a test statistic $T(\X)$ and an auxiliary sequence $(T_b)_{b=1}^\numperm$ generated \iid and independently of $T(\X)$ given $\X$.
Suppose that a test $\Delta(\X)$ rejects whenever $T(\X)$ exceeds the  $b_{\alpha} = \ceil{(1-\alpha)(\numperm+1)}$ smallest values in $(T_b)_{b=1}^\numperm$.  
Then $\E[\Delta(\X)] \geq 1-\beta$ if any population parameter $\mathcal{T}$ satisfies 
\begin{talign} \label{eq:separation}
&\tconst
    \geq
\Psi_{n}(\astar,\frac{\beta}{3}) + \Phi_{n}(\frac{\beta}{3}) 
\qtext{for}
\astar 
    \defeq 
\frac{\alpha}{2e} \big(\frac{\beta}{3}\big)^{1/\floor{\alpha(\numperm+1)}}, 
\end{talign}
where $(\Phi_{n}, \Psi_{\X}, \Psi_{n})$ are complementary quantile bounds satisfying
\begin{talign} \label{eq:Phi_P_n}
&\Pr(
\tconst
- T(\X) \geq \Phi_{n}(\frac{\beta}{3})) 
    \leq 
\frac{\beta}{3},
    \\
&\Pr( T_1 %
    \geq 
\Psi_{\X}(\astar) \mid  \X) \leq \astar  \qtext{almost surely,}
\stext{and} \label{eq:Psi_X_n} \\
&\Pr ( %
    \Psi_{\X}(\astar) 
    \geq 
\Psi_{n}(\astar,\frac{\beta}{3})) \leq \frac{\beta}{3}.\label{eq:Psi_P_n}
\end{talign}
\end{proposition}
When the test statistic $T(\X)$ has finite variance and the auxiliary variable $T_1$ has almost-surely finite conditional variance, Cantelli's inequality yields the following corollary, a refinement of  the two moments method of \citep[Prop.~D.1]{kim2022minimax}. 
See \cref{Section: Proof of Lemma: Refined Two Moments Method} for the proof.
\begin{corollary}[Refined two moments method] \label{Lemma: Refined Two Moments Method}
If $\Var(T(\X)),\E[ \Var (T_1 \mid \X)] <\infty$, then the following functions satisfy the requirements \cref{eq:Phi_P_n,eq:Psi_X_n,eq:Psi_P_n} of \cref{Lemma: Generic Two Moments Method} for $\tconst = \E[T(\X)]$:
\begin{talign}
\Phi_{n}(\beta) 
    &= \sqrt{(\frac1\beta - 1) \mathrm{Var}(T(\X))}, 
    \quad
    \Psi_{\X}(\alpha) = \sqrt{(\frac1\alpha-1) \Var(T_1\mid \X) } + \E[T_1 \mid \X], \\
\label{eq:refined_Psi_n_zero}
\Psi_{n}(\alpha,\beta) 
    &= \sqrt{(\frac1{\alpha\beta}\!-\!\frac1\beta) \E[ \Var (T_1 \mid \X )] } \qtext{when} \E[T_1 \mid \X] = 0 \qtext{almost surely,} \text{and}\\
\label{eq:refined_Psi_n_nonzero}
\Psi_{n}(\alpha,\beta) 
    &= \sqrt{(\frac2{\alpha\beta}\!-\!\frac2\beta) \E[ \Var (T_1 \mid \X) ] }  %
    + \E[T_1] %
+ \!\sqrt{\!(\frac{2}{\beta} -1) \Var (\E[T_1 \mid \X])}
    \ \text{otherwise.}
\end{talign}
\end{corollary}

\section{Minimax optimality of cheap testing}\label{sec:minimax}

In this section, we apply the power results of \cref{sec:power} to establish the minimax rate optimality of cheap testing in a variety of settings. 
Below, for discrete and absolutely continuous signed measures $\mu$, we make use of the standard $\Lp[2]$ norm on densities, 
$\twonorm{\mu}^2 \defeq \int (\frac{d\mu}{d\nu}(x))^2d\nu(x)$, where $\nu$ is the counting measure if $\mu$ is supported on a discrete set $\xset$ and $\nu$ is the Lebesgue measure when $\mu$ has a Lebesgue density.
Note that 
\begin{talign}
\twonorm{\P-\Q}^2
    =
\begin{cases}
\sum_{x\in\xset} (p(x)-q(x))^2
    &
\text{if } \P,\Q \stext{have probability mass functions $p,q$ on a  set $\xset$} \\
\int (p(x)-q(x))^2 dx
    &
\text{if } \P,\Q  \stext{have Lebesgue densities $p,q$.}
\end{cases}
\end{talign}
\newcommand{\holderorder}{\tau}
We also define the \Holder-space ball, 
\begin{talign}
H^d_{\holderorder}(L)
    \defeq
    \big\{
f: [0,1]^d \mapsto \reals   
    \,\big| 
\,&|f^{(\floor{{\holderorder}})}(x) - f^{(\floor{{\holderorder}})}(x')|
    \leq 
L \twonorm{x - x' }^{{\holderorder} - \floor{{\holderorder}}}, \quad \forall x,x'\in [0,1]^d, \\
&|f^{({\holderorder}')}(x)| 
    \leq 
L,
\quad
\forall x\in[0,1]^d, 
\quad{\holderorder}' \in \{1,\dots,\floor{\holderorder}\}
    \big\}.
\label{eq:holder}
\end{talign}

\subsection{Discrete homogeneity testing} %

Our first result, proved in \cref{subsec:pf_minimax_homogeneity_multinomial}, shows that cheap homogeneity testing is minimax rate optimal for detecting discrete $\Lp[2]$ separations.

\begin{proposition}[Power of cheap testing: discrete {$\Lp[2]$} homogeneity] \label{Proposition: Multinomial Two-Sample Testing}
Suppose $\P$ and $\Q$ are discrete distributions on a finite set $\xset$ 
with $b_{(1)}\defeq \max (\twonorm{\P}^2,\twonorm{\Q}^2)$. %
A cheap homogeneity test (\cref{algo:cheap_homogeneity_test_qts}) based on a homogeneity U-statistic (\cref{def:homogeneity_u_stat}) with $n_1\le n_2$ and base function $g(y,z) = \indic{y=z}$ 
has power at least $1-\beta$ whenever 
\begin{talign} 
\twonorm{\P-\Q} 
    &\geq
\frac{Cb_{(1)}^{1/4}}{\sqrt[6]{\beta\astar}\sqrt{\min(\beta,\astar)\, n_1}}
\bigg(1+\frac{1}{
\sqrt{s^3\frac{\beta\astar}{24(1-\astar)} \nratio}-1}\bigg)
\end{talign}
for a universal constant $C$,  $\astar\defeq \frac{\alpha}{2e} \big(\frac{\beta}{3}\big)^{1/\floor{\alpha(\numperm+1)}}$,
and 
$\nratio\defeq \frac{n_1(n_1 - 1)^2 n_2(n_2 - 1)^2}{n^6}$.
\end{proposition}

\begin{remark}[Cheap optimality: discrete {$\Lp[2]$} homogeneity] %
The cheap test of \cref{Proposition: Multinomial Two-Sample Testing} achieves the minimax optimal separation rate of order $b_{(1)}^{1/4} /\sqrt{n_1}$ \citep[Prop.~4.4]{kim2022minimax}
whenever 
$s \geq \sqrt[3]{\frac{24(1-\astar)(1+\Omega(1))}{\beta\astar \nratio}}$.
\end{remark}

\subsection{\Holder homogeneity testing} \label{subsec:minimax_homogeneity_holder}
Our second result, proved in \cref{subsec:pf_minimax_homogeneity_holder}, shows that cheap homogeneity testing is minimax rate optimal for detecting \Holder $\Lp[2]$ separations.
The test statistic, based on discretization of the unit cube, is the same used in \citep[Prop.~4.6]{kim2022minimax}.

\begin{proposition}[Power of cheap testing: \Holder { $\Lp[2]$} homogeneity] \label{Proposition: Two-Sample Testing for Holder Densities}
Suppose $\P$ and $\Q$ are distributions on $[0,1]^d$ with Lebesgue densities %
belonging to the \Holder ball $H^d_{\holderorder}(L)$ \cref{eq:holder}.
Let $\{B_1, \dots, B_K \}$ be a partition of $[0,1]^d$ into $d$-dimensional hypercubes with side length $1/\floor{n_1^{2/(4\holderorder+d)}}$. 
A cheap homogeneity test (\cref{algo:cheap_homogeneity_test_qts}) based on a homogeneity U-statistic (\cref{def:homogeneity_u_stat}) with $n_1\le n_2$ and base function $g(y,z) = \sum_{k=1}^K\indic{y,z\in B_k}$ 
has power at least $1-\beta$ whenever 
\begin{talign} 
\twonorm{\P-\Q} 
    &\geq
(\frac{1}{n_1})^{2\holderorder/(4\holderorder+d)}\,
\frac{C\sqrt[4]{L}}{\sqrt[6]{\beta\astar}\sqrt{\min(\beta,\astar)}}\,
\bigg(1+\frac{1}{
\sqrt{s^3\frac{\beta\astar}{24(1-\astar)} \nratio}-1}\bigg)
\end{talign}
for a universal constant $C$,  $\astar\defeq \frac{\alpha}{2e} \big(\frac{\beta}{3}\big)^{1/\floor{\alpha(\numperm+1)}}$, and
$\nratio\defeq \frac{n_1(n_1 - 1)^2 n_2(n_2 - 1)^2}{n^6}$.
\end{proposition}

\begin{remark}[Cheap optimality: \Holder { $\Lp[2]$} homogeneity]
The cheap test of \cref{Proposition: Two-Sample Testing for Holder Densities} achieves the minimax optimal separation rate of order $n_1^{-2\holderorder/(4\holderorder+d)}$ \citep[Thm.~3]{arias2018remember}
whenever
$s
    \geq 
\sqrt[3]{\frac{24(1-\astar)(1+\Omega(1))}{\beta\astar \nratio}}$.
\end{remark}

\subsection{MMD homogeneity testing} \label{subsec:minimax_homogeneity_mmd}

Recently, \citet[Thm.~7]{kim2026differentially} established a lower-bound of order $n_1^{-1/2}$ on the minimax-optimal detectable separation in maximum mean discrepancy ($\mmd(\P,\Q)$, \cref{eq:kernel_mmd_distance}) when $\P$ and $\Q$ are distributions on $\reals^d$ and the positive-definite kernel $\kernel$ is bounded, translation-invariant (i.e., $\kernel(x, y) = \kappa(x-y)$ for some $\kappa$), and non-constant.
To show that cheap homogeneity testing can recover this optimal rate, we first establish two refined guarantees for testing with positive-definite kernels.

When the base function $g$ is a positive-definite kernel $\kernel$, 
the expected homogeneity U-statistic (\cref{def:homogeneity_u_stat}) is nonnegative and equal to the squared MMD. %
In this case, \cref{thm: Full Two Sample Tests} also yields a detectable separation threshold  of order $\sqrt{1/n_1+1/n_2}$ for the root mean test statistic,  $\mmd(\P,\Q)=\sqrt{\E[U_{n_1,n_2}]}$.

\begin{corollary}[Power of cheap homogeneity: finite variance, PD]%
\label{cor:ts_separation_rate}
Under the assumptions of \cref{thm: Full Two Sample Tests} and notation of \cref{ex:kernelts}, suppose that $g$ is a positive-definite kernel $\kernel$ with 
\begin{talign}
\xi_{\Q},\,
\xi_{\P} < \infty
    \qtext{for}
\xi_{\mu}
    \defeq
\max(\E[\mmd^2(\delta_{Y_1},\mu)],\,  \E[\mmd^2(\delta_{Z_1},\mu)]).
\end{talign}
A cheap homogeneity test based on $U_{n_1,n_2}$ has power at least $1-\beta$ whenever 
\begin{talign}
\mmd(\P,\Q) 
    &\geq 
 \eps_{n_1,n_2,s}
 	\defeq
\eps_{n_1,n_2}
    +
\indic{s<n}\frac{\eps_{n_1,n_2} + \sqrt{\frac{12s}{n}\max(\xi_\Q,\xi_\P)} + \sqrt[4]{ \frac{201\beta \astar s^5 \nratio}{96 (1-\astar)n^2 } \psi_{YZ,2}}}{\sqrt{s^3\frac{\beta\astar}{24(1-\astar)}\nratio}-1}
    \label{eq:psd_ts_cheap}
    \\
\label{eq:psd_ts_full}
    \qtext{for}
\eps_{n_1,n_2}
    &\defeq 
\sqrt{\frac{3-\beta}{\beta}\big(\frac{4  \xi_{\Q}}{n_1} \! + \! \frac{4  \xi_{\P}}{n_2} \big)} + \sqrt[4]{\frac{1-\astar}{\astar}\frac{(36 n_1^2 + 36 n_2^2 + 198 n_1 n_2) \psi_{YZ,2}}{\beta n_1 (n_1 - 1) n_2 (n_2 - 1)}}.
\end{talign}
\end{corollary}
\begin{remark}[Equivalent separation rates and thresholds]\label{rem:mmd_sep_rates_homogeneity}
We have $\eps_{n_1,n_2,s} = \Theta(\eps_{n_1,n_2})$ whenever $s\geq \sqrt[3]{\frac{24(1-\astar)(1+\Omega(1))}{\beta\astar \nratio}}$
and $\eps_{n_1,n_2,s} = \eps_{n_1,n_2}(1 + o(1))$ whenever $s=\omega(\frac{n_2}{n_1})$.
\end{remark}

The proof of  \cref{cor:ts_separation_rate} can be found in  \cref{sec:proof_cor_ts_separation_rate}. 
Our next result establishes improved power guarantees for cheap testing with only logarithmic dependence on $\beta$ and $\astar$ under a commonly-satisfied sub-Gaussianity assumption. 
\begin{theorem}[Power of cheap homogeneity: sub-Gaussian, PD] \label{thm:final_2s}
Under the assumptions of \cref{thm: Full Two Sample Tests} and notation of \cref{ex:kernelts}, suppose that $g$ is a positive-definite kernel $\kernel$ with finite sub-Gaussian parameters
\begin{talign}
\subg,\ \subg[\Q] < \infty 
    \qtext{for} 
\subg[\mu] \defeq \sup_{p\geq 1} \frac{(\E_{X\sim\mu} |\E_{\tilde{X}\sim \mu}\mmd(\delta_X,\delta_{\tilde{X}})|^p)^{1/p}}{\sqrt{p}}
\end{talign}
and MMD moments
\begin{talign}
\mom[4,\P],\ \mom[4,\Q] < \infty 
    \qtext{for} 
\mom[p,\mu] \defeq \E_{X\sim \mu}\mmd^p(\delta_X,\mu).
\end{talign}
Suppose moreover that $s_1 \defeq \frac{s n_1}{n_1+n_2} > 
2\log(4/\astar)+1$
for $\astar \defeq \frac{\alpha}{2e} \big(\frac{\beta}{3}\big)^{1/\floor{\alpha(\numperm+1)}}$. 
A cheap homogeneity test based on $U_{n_1,n_2}$ has power at least $1-\beta$ whenever 
\begin{talign}
&\mmd(\mathbb{P},\mathbb{Q}) 
    \geq
\frac{1}{1 - 
\frac{2 \log(4/\astar) + 1}{s_1}
}
\bigg( 
\frac{\sqrt{40\sqrt{2 \log(\frac{2}{\astar})/3} (\tilde{\mathcal{A}}_4 + \tilde{\mathcal{B}}(\frac{\beta}{9})^2) + 4 (\tilde{\mathcal{A}}_2^2 + \tilde{\mathcal{B}}(\frac{\beta}{9})^2)}}{\sqrt{n_1}} 
    +
\frac{3\tilde{\Lambda}_\P(\frac{\beta}{3})}{\sqrt{n_1-1}} 
 +
\frac{3\tilde{\Lambda}_\Q(\frac{\beta}{3})}{\sqrt{n_2-1}} 
\\ \label{eq:Power of cheap homogeneity sub-Gaussian PD}
&\qquad\qquad\qquad\qquad\qquad\qquad\quad
+
\frac{40\sqrt{\log(\frac{2}{\astar})} (\tilde{\mathcal{A}}_2 + \tilde{\mathcal{B}}(\frac{\beta}{9}))}{\sqrt{3 s_1 n_1}} 
    + 
\frac{2 (\sqrt{2} \tilde{\mathcal{A}}_2 + \tilde{\mathcal{B}}(\frac{\beta}{9}))}{s_1 \sqrt{n_1}} \bigg)
    \qtext{for}\\
&\tilde{\mathcal{A}}_p \!
    \defeq \!
\frac{p}{2} \sqrt{
    \mom[p,\P]
    \! + \! \mom[p,\Q] },\ \ 
\tilde{\mathcal{B}}(\delta) 
    \! \defeq \! 
8\sqrt{ e \log(\frac{1}{\delta}) ( \subg[\P]^2 \! + \! \subg[\Q]^2 )},\ \ 
\tilde{\Lambda}_\mu(\delta) 
    \! \defeq \!
\sqrt{2\mom[2,\mu]} \! + \! 
    8\subg[\mu]\sqrt{ 2 e\log(\frac{1}{\delta})}.
\end{talign}
\end{theorem}

\begin{remark}[Bounded kernels are sub-Gaussian]
If $\kernel$ is bounded, then the sub-Gaussian parameter $\subg[\mu]\leq 2\sup_{x\in\xset}\sqrt{\kernel(x,x)}$ for all distributions $\mu$.
\end{remark}

\begin{remark}[Cheap optimality: MMD homogeneity]
The cheap test of \cref{thm:final_2s} with a bounded, translation-invariant, and non-constant kernel $\kernel$ achieves the minimax optimal separation rate of order $n_1^{-1/2}$ \citep[Thm.~7]{kim2026differentially}  
whenever
$s=(2+\Omega(1))(2\log(4/\astar)+1){n}/{n_1}$.
\end{remark}

For balanced sample sizes ($n_2=\Theta(n_1)$),  \cref{thm:final_2s}, proved in \cref{sec:proof_cor:final_2s}, 
establishes that cheap testing powerfully detects sub-Gaussian MMD separations of order $\frac{\sqrt{\log(1/\beta)}\sqrt[4]{\log(1/\astar)}}{\sqrt{n_1}}$
using only $s=\Theta(\log(4/\astar))$ bins. 
This represents an exponential improvement over the order
$\frac{\sqrt{1/\beta}\sqrt[4]{1/(\astar\beta)}}{\sqrt{n_1}}$ threshold and $s=\Omega\left(\sqrt[3]{\frac{1}{\beta\astar}}\right)$ bin size requirement of \cref{cor:ts_separation_rate} and recovers the minimax optimal MMD separation rate of \citep[Thm.~7]{kim2026differentially} with a constant bin count independent of the sample size.
In cases of extreme imbalance ($n_2 = \omega(n_1)$), one can choose $s = \Theta(\log(4/\astar)n/{n_1})$ bins for minimax optimality and still obtain a speed-up of order $n_1^2$ over the $\numperm n^2$ standard permutation testing overhead.

For bounded kernels bounded below by $0$, \citet[Thm.~6]{kim2026differentially} recently established a minimum separation rate upper bound of order $\frac{\sqrt{\max(\log(1/\beta),\log(1/\alpha))}}{\sqrt{n_1}}$ using a standard permutation test with $\numperm \geq \frac{6}{\alpha}\log(2/\beta)$ permutations, $n_2 = \Theta(n_1)$, and  $\Theta(\numperm n^2)$ permutation overhead. 
By \cref{thm:final_2s}, cheap testing in this setting can detect order $\frac{\sqrt{\log(1/\beta)}\sqrt[4]{\log(1/\alpha)}}{\sqrt{n_1}}$ separations using $\numperm = \Omega(\frac{\log(3/\beta)}{\alpha\log(2e/\alpha)})$ permutations and $s = \Theta(\log(1/\alpha))$ bins. Hence, binned permutation can match both the rate and the power dependence of state-of-the-art permutation guarantees while running in time independent of $n$. By \cref{thm:final_2s}, the same guarantees also hold more generally for unbounded sub-Gaussian kernels (like distance-induced kernels \citep{sejdinovic2013equivalence}) and kernels that take on negative values (like the popular $\textup{sinc}(x-y)$ kernel).

To interpret the sub-Gaussian bound \cref{eq:Power of cheap homogeneity sub-Gaussian PD}, it helps to isolate the roles of its three building blocks. The moment terms $\tilde{\mathcal{A}}_p$ aggregate the $p$-th moments $\mom[p,\P]$ and $\mom[p,\Q]$ of the MMD deviation $\mmd(\delta_X,\mu)$ of a sample from its own distribution and thus quantify how spread out $\P$ and $\Q$ are in the kernel geometry. The sub-Gaussian term $\tilde{\mathcal{B}}(\delta)$ controls the tails of this deviation through the sub-Gaussian parameters $\subg[\P],\subg[\Q]$ and is the sole source of the logarithmic confidence factors $\sqrt{\log(1/\delta)}$. The term $\tilde{\Lambda}_\mu(\delta)$ combines the second moment $\mom[2,\mu]$ with the sub-Gaussian parameter $\subg[\mu]$ to bound the per-sample fluctuations of the statistic under $\mu$.

With this notation in hand, we can unpack the terms comprising \cref{eq:Power of cheap homogeneity sub-Gaussian PD}. The first three terms, of orders $n_1^{-1/2}$, $(n_1-1)^{-1/2}$, and $(n_2-1)^{-1/2}$, are independent of $s$ and capture the irreducible cost of distinguishing the MMD signal from the sampling fluctuations of $U_{n_1,n_2}$ under $\P$ and $\Q$; this is the same detection cost incurred by standard permutation testing. The last two terms, of orders $(s_1 n_1)^{-1/2}$ and $s_1^{-1} n_1^{-1/2}$, are the additional price of cheap testing with $s$ bins, and both vanish as $s_1$ grows, so that enough bins recover the standard threshold. Finally, the prefactor $(1-(2\log(4/\astar)+1)/s_1)^{-1}$ encodes the requirement $s_1 > 2\log(4/\astar)+1$ that there be enough bins to permute and decreases to one as $s_1$ increases.

\subsection{Discrete independence testing}
Our fourth result, proved in \cref{subsec:pf_minimax_independence_multinomial}, shows that cheap independence testing is minimax rate optimal for detecting discrete $\Lp[2]$ dependence.

\begin{proposition}[Power of cheap testing: discrete {$\Lp[2]$} independence] \label{Proposition: Multinomial Independence Testing}
Suppose that $\P$ and $\Q$ are discrete distributions on finite sets $\yset$ and $\zset$ respectively with  $b_{(2)}\defeq \max (\twonorm{\pjnt}^2,\twonorm{\P\times\Q}^2)$, and consider the augmented samples 
\begin{talign}
\tilde{\Y} \defeq ((Y_i,A_i))_{i=1}^n
    \qtext{and}
\tilde{\Z} \defeq ((Z_i,B_i))_{i=1}^n
    \qtext{for}
(A_i)_{i=1}^n, (B_i)_{i=1}^n\distiid \Unif((0,1)). 
\end{talign}
A cheap independence test (\cref{algo:cheap_independence_permutation_test}) based on an independence V-statistic (\cref{def:independence_v_statistic}) with paired sample $(\tilde{\Y},\tilde{\Z})$ and base functions  $g_Y((y_1,a_1),(y_2,a_2)) = \indic{y_1=y_2, a_1\neq a_2}$ and $g_Z((z_1,b_1),(z_2,b_2)) = \indic{z_1=z_2,b_1\neq b_2}$ 
has power at least $1-\beta$ whenever 
\begin{talign} 
\twonorm{\pjnt-\P\times\Q} 
    &\geq
\frac{C\,b_{(2)}^{1/4}}{\sqrt{\beta\astar n}},
\end{talign}
for a universal constant $C$
and $\astar\defeq \frac{\alpha}{2e} \big(\frac{\beta}{3}\big)^{1/\floor{\alpha(\numperm+1)}}$. 
\end{proposition}

\begin{remark}[Cheap optimality: discrete {$\Lp[2]$} independence] %
The cheap test of \cref{Proposition: Multinomial Independence Testing} achieves the minimax optimal separation rate of order $b_{(2)}^{1/4} /\sqrt{n}$ \citep[Prop.~5.4]{kim2022minimax}
for \textbf{any} choice of  $s$.
\end{remark}

\subsection{\Holder independence testing}
Our fifth result, proved in \cref{subsec:pf_minimax_independence_holder}, shows that cheap independence testing is minimax rate optimal for detecting \Holder $\Lp[2]$ dependence. 
The test statistic base functions, based  discretization of the unit cube, are the same used in \citep[Prop.~5.5]{kim2022minimax}.

\begin{proposition}[Power of cheap testing: \Holder { $\Lp[2]$} independence] \label{Proposition: Independence Testing for Holder Densities}
Suppose that $\pjnt$ is a distribution on $[0,1]^{d_1+d_2}$ and that the Lebesgue densities of $\pjnt$ and $\P\times\Q$ belong to the \Holder ball  $H^{d_1+d_2}_{\holderorder}(L)$ \cref{eq:holder}
For $a\in\{1,2\}$, let $\{B_{a,1}, \dots, B_{a,K_a} \}$ be a partition of $[0,1]^{d_a}$ into $d_a$-dimensional hypercubes with side length $1/\floor{n^{2/(4\holderorder+d_1+d_2)}}$, and consider the augmented samples 
\begin{talign}
\tilde{\Y} \defeq ((Y_i,A_i))_{i=1}^n
    \qtext{and}
\tilde{\Z} \defeq ((Z_i,B_i))_{i=1}^n
    \qtext{for}
(A_i)_{i=1}^n, (B_i)_{i=1}^n\distiid \Unif((0,1)). 
\end{talign}
A cheap independence test (\cref{algo:cheap_independence_permutation_test}) based on an independence V-statistic (\cref{def:independence_v_statistic}) with paired sample $(\tilde{\Y},\tilde{\Z})$ and  base functions  
\begin{talign}
g_Y((y_1,a_1),(y_2,a_2)) 
    &= 
\sum_{k=1}^{K_1}\indic{y_1,y_2\in B_{1,k}, a_1\neq a_2} \qtext{and} \\
g_Z((z_1,b_1),(z_2,b_2)) 
    &=
\sum_{k=1}^{K_2}\indic{z_1,z_2\in B_{2,k}, b_1\neq b_2}
\end{talign} 
has power at least $1-\beta$ whenever 
\begin{talign}
\twonorm{\pjnt-\P\times\Q} 
    &\geq
\frac{C\,\sqrt[4]{L}}{\sqrt{\beta\astar}}\cdot
\left(\frac{1}{n}\right)^{{2\holderorder/}{(4\holderorder+d_1+d_2)}},
\end{talign}
for a universal constant $C$
and $\astar\defeq \frac{\alpha}{2e} \big(\frac{\beta}{3}\big)^{1/\floor{\alpha(\numperm+1)}}$. 
\end{proposition}

\begin{remark}[Cheap optimality: \Holder { $\Lp[2]$} independence] %
The cheap test of \cref{Proposition: Independence Testing for Holder Densities} achieves the minimax optimal separation rate of order $n^{-2\holderorder/(4\holderorder+d_1+d_2)}$ \citep[Prop.~5.6]{kim2022minimax}
for \textbf{any} choice of  $s$.
\end{remark}

\subsection{HSIC independence testing}\label{subsec:minimax_independence_hsic}
Recently, \citet[Thm.~S.4]{kim2026differentially} established a lower-bound of order $n^{-1/2}$ on the minimax-optimal detectable separation in the (square-root) Hilbert-Schmidt independence criterion, $\mmd(\pjnt, \P \times \Q )$ (\cref{ex:hsic}), when $\pjnt$ is a distribution on $\reals^d$ and the positive-definite kernels $g_Y$ and $g_Z$ are bounded, translation-invariant, and non-constant.
To show that cheap independence testing can recover this optimal rate, we first establish two refined guarantees for testing with positive-definite kernels.

When the base functions $(g_Y,g_Z)$ are positive-definite kernels, 
the scaled independence V-statistic $\quarter V_n$ (\cref{def:independence_v_statistic}) is equal to the squared sample MMD, $\mmd^2(\PX,\PY\times\PZ)$. 
In this case, the proof of \cref{thm: Full Independence Tests} also yields a detectable separation threshold of order $\frac{1}{\sqrt{n}}$ for the population parameter, $\mmd(\pjnt,\P\times\Q)$.

\begin{corollary}[Power of cheap independence: finite variance, PD] \label{cor:psd_gs}
Under the assumptions of \cref{thm: Full Independence Tests} and notation of \cref{ex:hsic}, suppose that $n\geq 14$ and that $g_Y,g_Z$ are positive-definite kernels with 
$\kernel = g_Y g_Z$, 
$\xi \defeq \E[\hin(X_1,X_1,X_3,X_4)]$, and
\begin{talign}
\wildxi
    \defeq
\max \big\{ \E_{(Y,Z) \sim \mu} \mmd^2\big(\frac{\delta_{(Y,Z)} + \nu}{2},\frac{\delta_{Y} \times \Q + \P \times \delta_{Z}}{2} \big) : \mu, \nu \in \{\pjnt, \P\times\Q\}\big\}.
\end{talign}
A cheap independence test based on $V_n$ has power at least $1-\beta$ whenever
\begin{talign}
\label{eq:psd_ind_cheap} 
 &\mmd(\pjnt, \P\!\times\! \Q) 
    \!\geq \! \\
\eps_{n,s}
    \defeq
&\frac{
    (3  - \sqrt{( \frac{12}{\astar \beta} - 2) \frac{32}{n}})\,\eps_n
        +
    \frac{\indic{s<n}}{\sqrt{n}}
    \left(\!16\sqrt{ (\frac{12}{\alpha^* \beta} - 2) \txiwild (\frac{72}{s} + \frac{14400}{s^2}) } 
        +
    \sqrt[4]{
    (\frac{12}{\astar \beta} - 2)
    \psi_2' \frac{4147200}{s}}
    \sqrt{3  - \sqrt{( \frac{12}{\astar \beta} - 2) \frac{32}{n}}}\right)
}{
    3 -  \indic{s<n}\frac{s+1}{s^2} - \sqrt{
    (\frac{12}{\astar \beta} - 2)\frac{32}{n} (1 + \indic{s<n} (\frac{72}{s} +  \frac{14400}{s^2}) )}
} \\
\label{eq:psd_ind_full}
    &\qtext{for}
\eps_{n}
    \defeq
    \frac{
    \sqrt{ 
    ( 12(\astar \beta)^{-1} - 2) (32 \, \xi + 256 \, \txiwild)}}{ (3 -
    \sqrt{( 12(\astar \beta)^{-1} - 2) \frac{32}{n}}) \sqrt{n}} + \frac{\sqrt[4]{ ( 12(\astar \beta)^{-1} - 2) 960\, \psi_2' } + \sqrt{10\tilde{\xi}}}{(3 -
    \sqrt{( 12(\astar \beta)^{-1} - 2) \frac{32}{n}} )^{1/2} \sqrt{n}}.
\end{talign}
\end{corollary}
\begin{remark}[Equivalent MMD separation rates and thresholds]\label{rem:mmd_sep_rates_independence}
We have $\eps_{n,s} = \Theta(\eps_{n})$ for \textbf{any} choice of $s$ 
and $\eps_{n,s} = \eps_{n}(1 + o(1))$ whenever $s=\omega(1)$.
\end{remark}

The proof of  \cref{cor:psd_gs} can be found in  \cref{subsec:cor_psd_ind}. 
Our next result establishes improved power guarantees for cheap independence testing with only logarithmic dependence on $\beta$ and $\astar$ whenever $\kernel$ is bounded.

\begin{theorem}[Power of cheap independence: bounded, PSD] \label{thm:final_ind}
Under the assumptions of \cref{thm: Full Independence Tests} and notation of \cref{ex:hsic}, suppose that $g_Y,g_Z$ are positive-definite kernels with 
$\kernel = g_Y g_Z$ and $K \defeq \sup_{x} |\kernel(x,x)| <\infty$. 
If $\frac{1}{4} + \frac{\log (6/\astar)}{s} + \sqrt{\frac{\log (6/\astar)}{s}} < 1$ for $\astar = \frac{\alpha}{2e} \big(\frac{\beta}{3}\big)^{1/\floor{\alpha(\numperm+1)}}$,  then  
a cheap independence test based on $V_n$ has power at least $1-\beta$ whenever 
\begin{talign}
 \mmd(\pjnt, \P \times \Q) 
  	&\geq 
\frac{
\sqrt{K}\left( 30 + 22 \sqrt{\log({12}{/\beta})} \right) 
    \left( \frac{7}{2} + \frac{1+12\log({3}{/\astar})}{6\sqrt{s}}
    + 2\sqrt{\log({3}{/\astar})} \right)
    }{
\sqrt{n}\,\big(\frac{3}{4} - \frac{\log (6/\astar)}{s} - \sqrt{\frac{\log (6/\astar)}{s}} \big) 
	}.
\end{talign}
\end{theorem}

\begin{remark}[Cheap optimality: HSIC independence]
The cheap test of \cref{thm:final_ind} with a bounded, translation-invariant, and non-constant kernels $g_Y$ and $g_Z$ achieves the minimax optimal separation rate of order $n^{-1/2}$ \citep[Thm.~S.4]{kim2026differentially}  
whenever $s\geq 5\log(6/\astar)$.
\end{remark}

\cref{thm:final_ind}, proven in \cref{sec:proof_thm:final_ind}, 
establishes that cheap independence testing powerfully detects bounded MMD separations of order $\frac{\sqrt{\log(1/\beta)\log(1/\astar)}}{\sqrt{n}}$
using only  $s=\Theta(\log(1/\astar))$
bins.
This represents an exponential improvement over the order
$\frac{\sqrt{1/(\beta\astar)}}{\sqrt{n}}$ threshold of \cref{cor:psd_gs} and recovers the minimax optimal separation rate of \citep[Thm.~S.4]{kim2026differentially} with a constant bin count independent of the sample size.

For bounded kernels bounded below by $0$, \citet[Thm.~S.3]{kim2026differentially} recently established a minimum separation rate upper bound of order $\frac{\sqrt{\max(\log(1/\beta),\log(1/\alpha))}}{\sqrt{n}}$ using a standard permutation test with $\numperm \geq \frac{6}{\alpha}\log(2/\beta)$ permutations and $\Theta(\numperm n^2)$ permutation overhead. 
By \cref{thm:final_ind}, cheap testing in this setting can detect order $\frac{\sqrt{\log(1/\beta)}\sqrt[4]{\log(1/\alpha)}}{\sqrt{n}}$ separations using $\numperm = \Omega(\frac{\log(3/\beta)}{\alpha\log(2e/\alpha)})$ permutations and $s = \Theta(\log(1/\alpha))$ bins. Hence, binned wild bootstrapping can match both the rate and the power dependence of state-of-the-art permutation guarantees while running in time independent of $n$.

\section{Experiments}\label{sec:experiments}
We now complement our methodological development and theoretical analysis %
with a series of synthetic experiments exploring the practical benefits of cheap testing. 
We implemented all tests in Cython using nominal level $\alpha=0.05$ and a common interface and library of functions to enable a consistent comparison of computational complexity. 
For each experiment, we display average rejection rates and $95\%$ Wilson confidence intervals over $2500$ independent replications with $B=299$ permutations for Higgs boson testing and $10000$ independent replications with $B=1279$ permutations otherwise. 
Python and Cython code replicating all experiments can be found at \url{https://github.com/microsoft/cheap-permutations}, and supplementary experimental details can be found in \cref{sec:additional_plots}.

\subsection{Homogeneity testing} \label{subsec:homogeneity_experiments}
We benchmark the power and runtime of homogeneity U-statistic (\cref{def:homogeneity_u_stat}) tests with $n_1 =n_2$ in three settings. 
\begin{enumerate}[itemsep=.5\baselineskip]
\item \textsc{MMD testing}: 
$\P = \mathcal{N}(-0.03 e_1,\mathrm{I})$ and $\Q = \mathcal{N}(0.03 e_1,\mathrm{I})$ on $\reals^{d}$ for $d=10$, and 
the base function $g$ is the Gaussian kernel $\kernel_{\lambda}(x,y) \defeq e^{ -{\twonorm{x-y}^2}{/(2 \lambda^2)}}$ with bandwidth $\lambda$ equal to the median pairwise distance among sample points.
\label{item:mmd_testing}
\item \textsc{{RFF testing}}: $\P$, $\Q$, and $\lambda$ are as above, and the base function is the \emph{random Fourier feature} \citep[RFF,][]{rahimi2008random}  kernel $g(x,y) = 
\frac{4}{r^2}\sum_{k=1}^r
\cos( \omega_k^{\top} x + b_k ) \cos ( \omega_k^{\top} y + b_k)$
for $(\omega_k,b_k)$ sampled \iid from $\mathcal{N}(0,\frac{1}{\lambda^2})\times \Unif(0,2\pi)$. In this case, because the kernel $g$ has rank $r < n$, each test statistic can be computed in time $\Theta(dnr)$, and cheap permutation can be carried out in $\Theta(\numperm s r)$ additional time with $\Theta(sr)$ memory, as described in \cref{algo:cheap_permutation_wild_bootstrap_two_sample_features}.
\label{item:rff_testing}
\item \textsc{Aggregated MMD testing}: 
An alternative to selecting a single kernel MMD for testing is to construct an aggregated test across a weighted collection of MMD test statistics indexed by kernel choice.  We use the permutation-based MMDAgg aggregation procedure of \citet{schrab2023mmd} atop a uniformly-weighted collection of Gaussian kernel MMDs with bandwidths $\lambda\in \{ 2^i \lambda_0 \}_{i=0}^{-4}$ where $\lambda_0$ represents the median pairwise distance among sample points. 
Our data comes from the Higgs boson dataset of \citet{baldi2014searching} where $\Q$ is distribution over the $\phi$-momenta of the two leading jets for a signal process that produces Higgs bosons and $\P$ is the corresponding distribution for a background process that does not produce Higgs bosons.
\item \textsc{WMW testing}:    $\P=\Gsn(-0.015,1)$,  $\Q=\Gsn(0.015,1)$, and our test statistic is the popular {Wilcoxon-Mann-Whitney (WMW) test statistic} of \cref{ex:wilcoxon-mann-whitney}. 
\label{item:wilcoxon_testing}
\end{enumerate}

\begin{example}[Wilcoxon-Mann-Whitney test statistic \citep{wilcoxon1945individual,mann1947on}] \label{ex:wilcoxon-mann-whitney}
The \emph{Wilcoxon-Mann-Whitney test statistic} (also known as the Wilcoxon rank sum statistic and the Mann-Whitney U-statistic) is defined as $T(\Y,\Z) = \frac{1}{n_1 n_2} \sum_{i=1}^{n_1} \sum_{j=1}^{n_2} \indic{Y_{i} \leq Z_{j}} - \frac{1}{2}$, where %
$\Y$ and $\Z$ are assumed to have no elements in common. This statistic may be recast into the U-statistic form \eqref{Eq: Two-Sample U-statistic} by selecting
    \begin{talign} \label{eq:g_wilcoxon}
        g(y,z) = 
        -\half\indic{y<z}
        -\quarter\indic{y=z}.
    \end{talign}
\end{example}
Technically, the WMW statistic is an \emph{asymmetric} homogeneity U-statistic as $g(y,z)\neq g(z,y)$. Nevertheless, it remains a quadratic test statistic (\cref{def:qts}) and hence is amenable to cheap testing.

\begin{figure}
    \centering
    
    \includegraphics[width=0.49\textwidth]{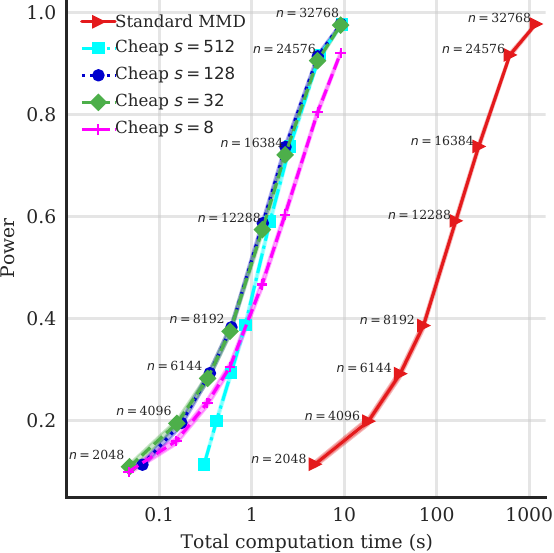}
    \includegraphics[width=0.49\textwidth]{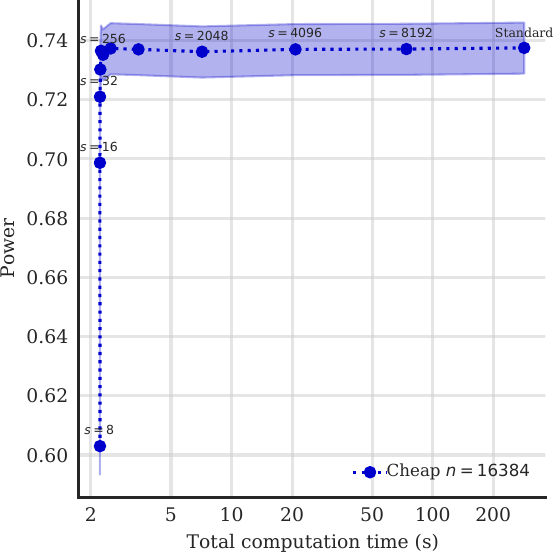}
    \caption{\textbf{Power \vs runtime for cheap and standard MMD tests of homogeneity} as total sample size $n=n_1+n_2$ and bin count $s$ vary.} 
\label{fig:cheap_homogeneity}
\end{figure}
\subsubsection{Cheap \vs standard MMD testing} 
\cref{fig:cheap_homogeneity} (left) compares cheap MMD testing to standard MMD permutation testing 
as a function of the total sample size $n = n_1+n_2$ and bin count $s$. 
We find that $s=128$ bins suffice to closely match the power of standard testing for higher power levels ($1-\beta \geq 0.3$) and that $s=32$ bins suffice for smaller power levels. 
In both cases, cheap testing is consistently $100$ times faster than the standard permutation test.

\cref{fig:cheap_homogeneity} (right) traces out the power-vs.-runtime trajectory of cheap testing for varying $s$ and fixed  $n=16384$.
We find that runtime remains essentially constant for $s\leq 128$ and that power remains essentially constant for $s\geq 128$.

\begin{figure}
    \centering
    \includegraphics[width=0.49\textwidth]{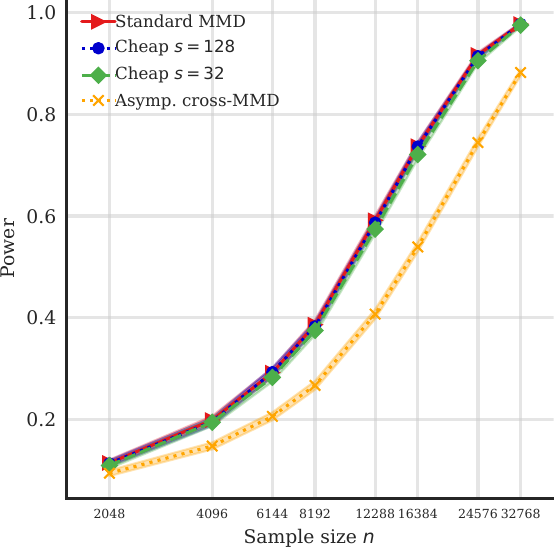}
    \includegraphics[width=0.49\textwidth]{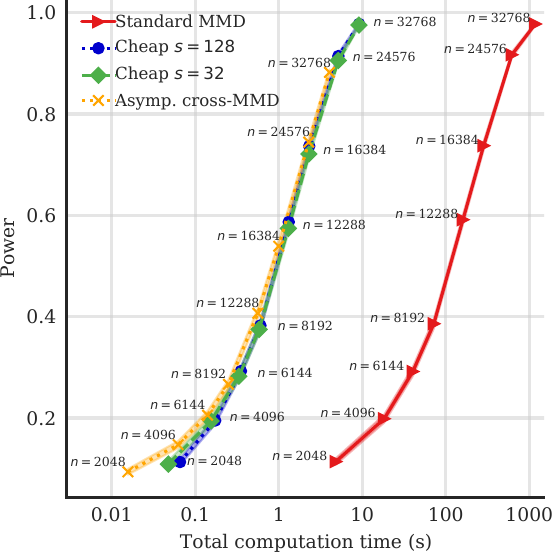}
    \caption{\textbf{Cheap MMD \vs asymptotic cross-MMD tests of homogeneity.} (Left) Power as a function of the total sample size $n=n_1+n_2$.
    (Right) Time-power trade-off curves as $n$ and bin count $s$ vary.
    }
    \label{fig:cheap_homogeneity_n_plot}
\end{figure}

\subsubsection{Cheap MMD \vs asymptotic cross-MMD testing}
A recently-proposed alternative to the MMD permutation test is the \emph{asymptotic cross-MMD test} \citep{shekhar2022permutation} which avoids permutation by thresholding a related, asymptotically-normal test statistic using a Gaussian quantile. 
By avoiding permutation, the asymptotic cross-MMD test achieves a substantial speed-up over a standard MMD permutation test.  However, as \citet{shekhar2022permutation} note, this comes at the cost of reduced power and the loss of finite-sample validity.  
In \cref{fig:cheap_homogeneity_n_plot}, we see that cheap MMD achieves nearly the same speed-up benefits as asymptotic cross-MMD (right) while also retaining the full power and final-sample validity of standard MMD testing for any given sample size (left).

\begin{figure}
    \centering    \includegraphics[width=\textwidth]{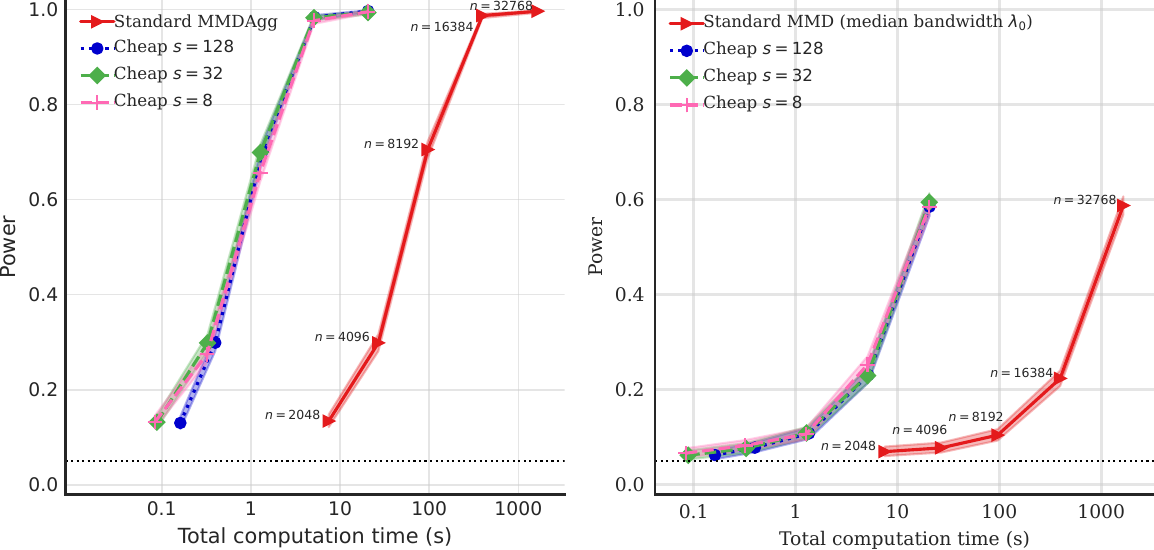}
    \caption{\textbf{Power \vs runtime for cheap and standard MMDAgg tests (left) and MMD tests (right) of homogeneity} as total sample size $n=n_1+n_2$ and bin count $s$ vary.} 
\label{fig:cheap_MMDAgg}
\end{figure}
\subsubsection{Cheap \vs standard aggregated MMD testing} 
\cref{fig:cheap_MMDAgg} (left) compares cheap MMDAgg to standard MMDAgg  
as a function of the total sample size $n = n_1+n_2$ and bin count $s$. 
We find that $s=32$ bins suffice to closely match the power of standard testing for all power levels.  With this setting, cheap testing is consistently $100$ times faster than the standard MMDAgg permutation test.

\cref{fig:cheap_MMDAgg} (right) highlights the primary benefit of aggregated MMD testing: testing the same data with a single kernel bandwidth, like the popular popular median bandwidth $\lambda_0$, results in a substantial degradation in power for standard testing and a matched degradation for cheap testing, albeit in $100\times$ less time.

\subsubsection{Cheap vs. standard RFF testing} 
\cref{fig:cheap_RFF} compares cheap RFF testing to standard RFF permutation testing 
as a function of the number of random Fourier features $r$ and bin count $s$. 
We find that $s=256$ bins suffice to closely match the power of standard testing for all power levels. 
With this setting, cheap testing is also $100$ to $1000$ times faster than the standard permutation test, depending on the choice of $r$.

\begin{figure}
    \centering
    \includegraphics[width=0.7\textwidth]{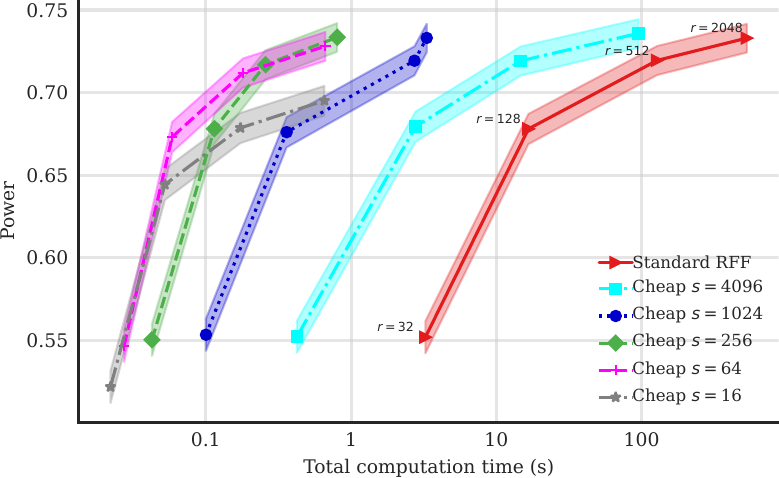}
    \caption{\textbf{Power \vs runtime for cheap and standard RFF tests of homogeneity} as the bin count $s$ and number of random Fourier features $r$ vary. Here, the total sample size is $n=16384$.
    }
    \label{fig:cheap_RFF}
\end{figure}

\subsubsection{Cheap \vs standard WMW testing}
\cref{fig:mann_whitney-wilcoxon} (left) compares cheap WMW testing to standard WMW permutation testing 
as a function of the total sample size $n = n_1+n_2$ and bin count $s$. 
We find that $s=128$ bins suffice to closely match the power of standard testing for higher power levels ($1-\beta \geq 0.3$) and that $s=64$ bins suffice for smaller power levels. 
In both cases, cheap testing is consistently $300$ to $900$ times faster than the standard permutation test.

\cref{fig:mann_whitney-wilcoxon} (right) traces out the power-vs.-runtime trajectory of cheap testing for varying $s$ and fixed  $n=16384$.
We find that runtime remains essentially constant for $s\leq 64$ and that power increases minimally beyond $s= 128$.

\begin{figure}
    \centering
    \includegraphics[width=0.49\textwidth]{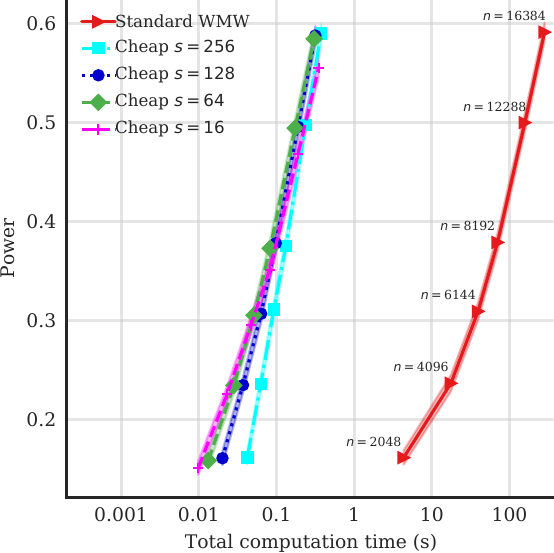}
    \includegraphics[width=0.49\textwidth]{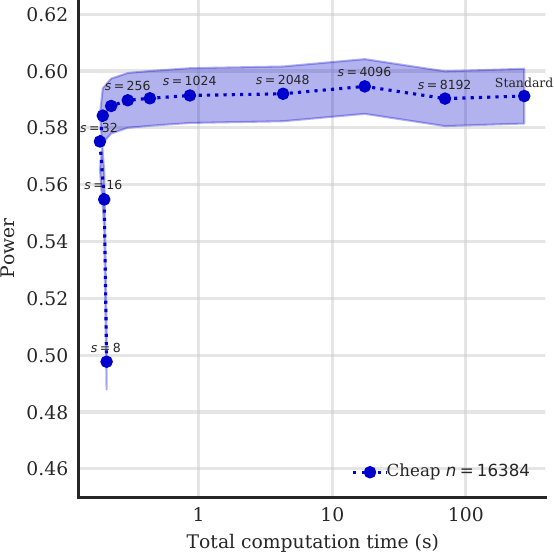}
    \caption{\textbf{Power \vs runtime for cheap and standard Wilcoxon-Mann-Whitney tests of homogeneity} as total sample size $n=n_1+n_2$ and bin count $s$ vary.
    }
    \label{fig:mann_whitney-wilcoxon}
\end{figure}

\subsection{Independence testing} \label{subsec:independence_experiments}
We additionally benchmark the power and runtime of independence V-statistic (\cref{def:independence_v_statistic}) tests in the following setting. 
\begin{enumerate}[itemsep=.5\baselineskip]
\item[4.] \textsc{HSIC testing}: 
$\pjnt$ a multivariate Gaussian distribution with mean zero and covariance  
    \begin{talign}
        \begin{pmatrix}
            \mathrm{I} & 0.047\mathrm{I} \\
            0.047\mathrm{I} & \mathrm{I}
        \end{pmatrix} \in \mathbb{R}^{20 \times 20}.
    \end{talign}
The base functions $g_Y$ and $g_Z$ are  Gaussian kernels $\kernel_{\lambda_Y}$ and $\kernel_{\lambda_Z}$ with bandwidths $\lambda_Y$ and $\lambda_Z$ equal to the median pairwise distance among $\Y$ and $\Z$ points respectively.
\label{item:hsic_testing}
\end{enumerate}

\subsubsection{Cheap \vs standard HSIC testing}
\cref{fig:cheap_independence} (left) compares cheap HSIC testing to standard HSIC wild bootstrap testing 
as a function of the total sample size $n = n_1+n_2$ and bin count $s$. 
We find that $s=128$ bins suffice to closely match the power of standard testing for higher power levels ($1-\beta \geq 0.3$) and that $s=64$ bins suffice for smaller power levels. 
In both cases, cheap testing is consistently $15$ times faster than the standard wild bootstrap test.

\cref{fig:cheap_independence} (right) traces out the power-vs.-runtime trajectory of cheap testing for varying $s$ and fixed  $n=2048$.
We find that runtime remains essentially constant for $s\leq 128$ and that power increases minimally beyond $s=256$.

\begin{figure}
    \centering
    \includegraphics[width=0.49\textwidth]{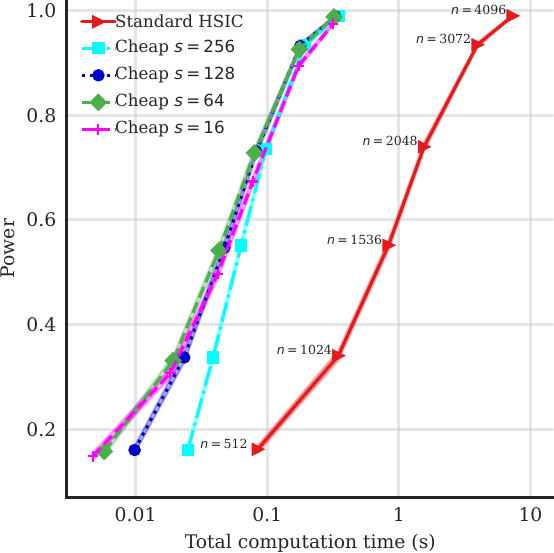}
    \includegraphics[width=0.49\textwidth]{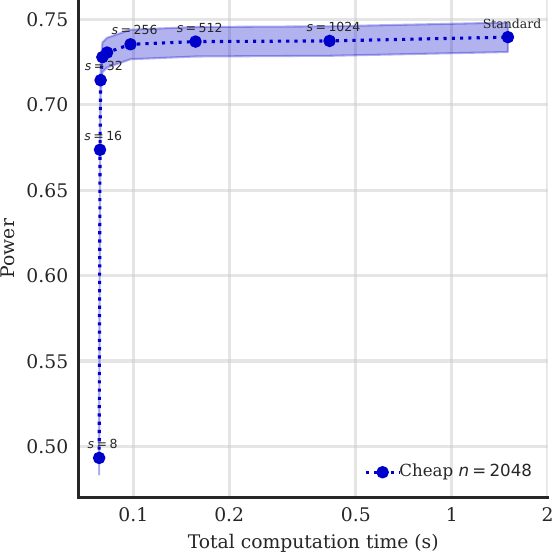}
    \caption{\textbf{Power \vs runtime for cheap and standard HSIC tests of independence} as total sample size $n=n_1+n_2$ and bin count $s$ vary.}
    \label{fig:cheap_independence}
\end{figure}

\begin{figure}
    \centering
    \includegraphics[width=0.49\textwidth]{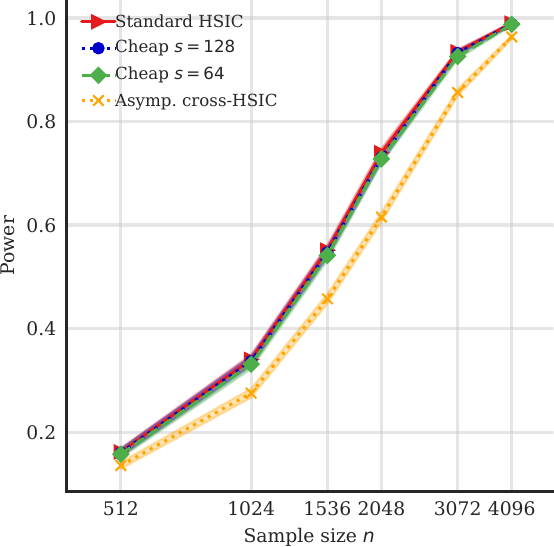}
    \includegraphics[width=0.49\textwidth]{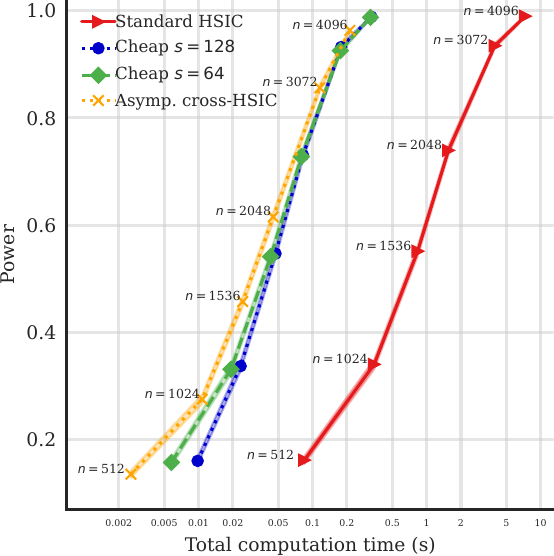}
    \caption{\textbf{Cheap HSIC \vs asymptotic cross-HSIC tests of independence.} (Left) Power as a function of the total sample size $n=n_1+n_2$.
    (Right) Time-power trade-off curves as $n$ and bin count $s$ vary.
    }
    \label{fig:cheap_independence_n_plot}
\end{figure}
\subsubsection{Cheap HSIC \vs asymptotic cross-HSIC testing} \label{subsubsec:cross_hsic}
A recently-proposed alternative to the HSIC wild bootstrap test is the \emph{asymptotic cross-HSIC test} \citep{shekhar2023apermutation} which avoids wild bootstrapping by thresholding a related, asymptotically-normal test statistic using a Gaussian quantile. 
By avoiding wild bootstrapping, the asymptotic cross-HSIC test achieves a substantial speed-up over a standard HSIC wild bootstrap test.  However, as \citet{shekhar2023apermutation} note, this comes at the cost of reduced power and the loss of finite-sample validity.  
In \cref{fig:cheap_independence_n_plot}, we see that cheap HSIC achieves speed-up benefits comparable to those of asymptotic cross-HSIC (right) while also preserving the higher power and final-sample validity of standard HSIC testing for any given sample size (left).

\section{Discussion and extensions}\label{sec:discussion}

We have presented a general framework, \emph{cheap permutation testing}, for accelerating hypothesis testing while preserving the high power and exact false positive control of standard permutation tests. %
The core idea is to group datapoints into bins and to exploit bin-wise sufficient statistics to avoid the standard sample-size-dependent overhead of permutation testing. 
A notable limitation of this approach is that compact sufficient are not available for all unstructured or black-box test statistics. 
However, they are available for the quadratic test statistics popular in two-sample homogeneity and independence testing and for a variety of more general structured classes:

\paragraph{Generalized quadratic test statistics} 
Many popular test statistics can be recovered as simple transformations 
of quadratic statistics.  For example, \emph{studentized} U and V-statistics, which are known to imbue permutation tests with additional robustness properties \citep{neuhaus1993conditional,janssen1997studentized,chung2013exact,chung2016asymptotically}, divide a standard U or V-statistic by an estimate of its standard error and commonly take the form $U/\sqrt{V}$ where both $U$ and $V$ are quadratic test statistics in the sense of \cref{def:qts}.
All functions of one or more quadratic test statistics are amenable to the binned sufficient statistic constructions of \cref{sec:cheap} and hence can benefit from the speed-ups of cheap testing.

\paragraph{Multisample / $K$-sample testing}

In multisample homogeneity testing one observes $K \geq 2$ independent samples drawn \iid from unknown distribution $\P_1,\dots,\P_K$ and seeks to test whether all $K$ distributions match. 
Similarly, in multisample independence testing one observes tuples of $K \geq 2$ variables and seeks to test whether the components are jointly independent.  
In each case, one can generalize the cheap permutation tests of  \cref{algo:cheap_homogeneity_test_qts,algo:cheap_independence_permutation_test} to accommodate the multisample U and V-statistics commonly employed \citep[see, e.g.,][]{pfister2018kernel,panda2025universally,chung2013exact}.

\paragraph{Higher-order test statistics} 
Cheap testing is also compatible with  \emph{polynomial test statistics}, that is, degree $p$ polynomial functionals of the sample empirical distributions $\Phat_1,\dots,\Phat_K$: 
\begin{talign}
T
    =
\sum_{j_1,\dots,j_p \in[K]} (\Phat_{j_1}\times\dots\times\Phat_{j_p} )\phi_{j_1,\dots,j_p}.
\end{talign}
In this case, $O(s^p)$ sufficient statistics would suffice to run a cheap test with $O(\numperm s^p)$ permutation overhead, a substantial improvement over the standard  $\Theta(\numperm n^p)$ permutation overhead.

\begin{acks}[Acknowledgments]
We thank Nick Koning, Antonin Schrab, Ilmun Kim, Aaditya Ramdas, and Art Owen for their helpful feedback on an earlier version of this manuscript.
\end{acks}

\begin{funding}
RD acknowledges support by National Science Foundation under Grant No. DMS2023528 for the Foundations of Data Science Institute (FODSI). 
\end{funding}

\bibliographystyle{imsart-nameyear} %
\bibliography{refs}

\newpage

\numberwithin{theorem}{section}
\numberwithin{mylemma}{section}
\numberwithin{mydefinition}{section}
\numberwithin{myproposition}{section}
\numberwithin{mycorollary}{section}
\numberwithin{figure}{section}
\numberwithin{table}{section}
\numberwithin{algocf}{section}

\begin{appendix}
\begin{supplement}
\centering
\stitle{Supplement to ``\papertitle''}
\sdescription{This supplement contains proofs of all results as well as  supplementary experiment details.\\[\baselineskip]}
\end{supplement}

\etocsettocstyle{}{}
    \etocdepthtag.toc{mtappendix}
    \etocsettagdepth{mtchapter}{none}
    \etocsettagdepth{mtappendix}{section}
    \etocsettagdepth{mtappendix}{subsection}
    \etocsettagdepth{mtappendix}{subsubsection}
    {\tableofcontents}

\section{Additional notation}
For any dataset $\X = {(X_i)}_{i=1}^{n}$, we let $\delta_{\mathbb{X}}$ denote the empirical distribution $\frac{1}{n} \sum_{i=1}^{n} \delta_{X_i}$ and 
$\X_{-k} = (X_1, \dots, X_{k-1},X_{k+1},\dots, X_n)$ for each $k\in[n]$.
For a given positive-definite kernel $\kernel$, we let $\rkhs$ represent the associated reproducing kernel Hilbert space \citep[RKHS,][Def.~4.18]{steinwart2008support} and $\knorm{\cdot}$ and $\kinner{\cdot}{\cdot}$ represent the RKHS norm and inner product \citep[Thm.~4.21]{steinwart2008support}.

\subsection{Additional notation for cheap homogeneity testing (\cref{algo:cheap_homogeneity_test_qts})} \label{subsec:additional_homogeneity}
Let $\X$ be the concatenation of $\Y$ and $\Z$. We define the number of elements per bin $m \defeq \frac{n}{s} = \frac{n_1 + n_2}{s}$, the number of bins for each sample $s_1 \defeq \frac{n_1}{m}$ and $s_2 \defeq \frac{n_2}{m}$, and the datapoint bins $\Y^{(i)}=(Y_j)_{j\in\bin_i}$ and $\Z^{(i)}=(Z_j)_{j\in\bin_i}$.

For for the homogeneity U-statistic of \cref{def:homogeneity_u_stat}, we define the bin kernel 
\begin{talign} 
\begin{split} \label{eq:H_ts_def}
    \Hho(\Yi,\Yi[j];\Zi,\Zi[j]) = %
    \frac{1}{m^4} \sum_{y\in\Yi,y'\in\Yi[j],z\in\Zi,z'\in\Zi[j]} \hho(y,y';z,z').
\end{split}
\end{talign}
Moreover, by the definition of $\hho$, we can rewrite $\Hho$ as
\begin{talign} 
    \Hho(\Yi,\Yi[j];\Zi,\Zi[j]) &= G_m(\Yi,\Yi[j]) + G_m(\Zi,\Zi[j]) - G_m(\Yi,\Zi[j]) - G_m(\Yi[j],\Zi), \\
    \text{where} \qquad 
    &G_m(\Yi,\Yi[j]) \defeq \frac{1}{m^2} \sum_{y\in\Yi,y'\in\Yi[j]} g(y, y').
\label{eq:H_ts_G_m}
\end{talign}
Observe that $\Hho$ can also be written in terms of $\hbarho$:
\begin{talign} 
\begin{split} \label{eq:H_ts_def_barho}
\Hho(\Yi,\Yi[j];\Zi,\Zi[j])
    = 
\frac{1}{m^4} \sum_{y\in\Yi,y'\in\Yi[j],z\in\Zi,z'\in\Zi[j]} \hbarho(y,y';z,z').
\end{split}
\end{talign}

Let $(\mathbb{X}^{(i)})_{i=1}^{s}$ denote the concatenation of $(\mathbb{Y}^{(i)})_{i=1}^{s_1}$ and $ (\mathbb{Z}^{(i)})_{i=1}^{s_2}$. 
For an independent, uniformly random permutation $\pi$ of $[s]$, we also define the \emph{cheaply-permuted two-sample V and U-statistics} as
\begin{talign} 
    V_{n_1,n_2}^{\pi,s} &= \frac{1}{s_1^2 s_2^2} \sum_{i_1,i_2=1}^{s_1} \sum_{j_1,j_2=1}^{s_2} \Hho(\mathbb{X}^{(\pi_{i_1})},\mathbb{X}^{(\pi_{i_2})};\mathbb{X}^{(\pi_{s_1 + j_1})},\mathbb{X}^{(\pi_{s_1 + j_2})})
    \quad\text{and}\!\!\!\!\!\!\!
    \label{eq: Two Sample Cheap V_statistic} \\
\label{eq: Two Sample Cheap U_statistic}
    U_{n_1,n_2}^{\pi,s} &= \frac{1}{(s_1)_{(2)} (s_2)_{(2)}} \sum_{(i_1,i_2) \in \mathbf{i}_2^{s_1}} \sum_{(j_1,j_2) \in \mathbf{i}_2^{s_2}} \Hho(\mathbb{X}^{(\pi_{i_1})},\mathbb{X}^{(\pi_{i_2})};\mathbb{X}^{(\pi_{s_1 + j_1})},\mathbb{X}^{(\pi_{s_1 + j_2})}).\!\!\!\!\!\!\!
\end{talign}
These expressions along with the substitutions $n \gets s$, $n_1 \gets s_1$, $n_2 \gets s_2$, $\hho \gets \Hho$, and $g \gets G_m$ help us to frame a cheap homogeneity U-statistic permutation test as a standard homogeneity U-statistic permutation test operating on bins of datapoints.

\subsection{Additional notation for cheap independence testing (\cref{algo:cheap_independence_permutation_test})}

We define the number of elements per bin $m \defeq \frac{n}{s}$ and the datapoint bins $\Y^{(i)}=(Y_j)_{j\in\bin_i}$,  $\Z^{(i)}=(Z_j)_{j\in\bin_i}$, and 
$\Xi = ((Y_j,Z_j))_{j\in\bin_i}$.
When convenient, we will also refer to the $m$ elements of these bins as $(Y^{(i)}_j)_{j=1}^m$,  $(Z^{(i)}_j)_{j=1}^m$, 
and 
$(X^{(i)}_j)_{j=1}^m$.

For the independence V-statistic of \cref{def:independence_v_statistic}, we 
define the bin kernel $\Hin$ as
\begin{talign}
    \Hin(\mathbb{X}^{(1)},\mathbb{X}^{(2)},\mathbb{X}^{(3)},\mathbb{X}^{(4)}) = \frac{1}{m^4} \sum_{i_1,i_2,i_3,i_4 = 1}^{m} \hin(X^{(1)}_{i_1},X^{(2)}_{i_2},X^{(3)}_{i_3},X^{(4)}_{i_4})
\end{talign}
where, for
\begin{talign}
g((y,z),(y',z'))\defeq g_Y(y,y')g_Z(x,z'), 
\end{talign}
we can write
\begin{talign}
    \hin(X_1,X_2,X_3,X_4) = &\iint g((y,z),(y',z')) \, d(\delta_{(Y_1,Z_1)} + \delta_{(Y_4,Z_4)} - \delta_{(Y_1,Z_4)} - \delta_{(Y_4,Z_1)})(y,z) \\ &\qquad\qquad\qquad\qquad d(\delta_{(Y_2,Z_2)} + \delta_{(Y_3,Z_3)} - \delta_{(Y_2,Z_3)} - \delta_{(Y_3,Z_2)})(y',z').
\end{talign}
An alternative way to express $\hin(X_1,X_2,X_3,X_4)$ is the following:
\begin{talign} \label{eq:h_in_Y_Z_def_1}
    \hin(X_1,X_2,X_3,X_4) 
    &= \hinY(Y_1,Y_2,Y_3,Y_4) \hinZ(Z_1,Z_2,Z_3,Z_4),
\qtext{for}\\
\label{eq:h_in_Y_Z_def_2}
    \hinY(Y_1,Y_2,Y_3,Y_4) &\defeq g_Y(y_1,y_2) + g_Y(y_3,y_4)  - g_Y(y_1,y_3) - g_Y(y_2,y_4), \\
    \hinZ(Z_1,Z_2,Z_3,Z_4) &\defeq g_Z(z_1,z_2) + g_Z(z_3,z_4) - g_Z(z_1,z_3) - g_Z(z_2,z_4).
\end{talign}
Using this, we can rewrite $\Hin$ as
\begin{talign}
    &\Hin(\mathbb{X}^{(1)},\mathbb{X}^{(2)},\mathbb{X}^{(3)},\mathbb{X}^{(4)}) \\ &=  
     \frac{1}{m^4} \sum_{i_1,i_2,i_3,i_4 = 1}^{m} \iint g((y,z),(y',z')) \, d(\delta_{(Y^{(1)}_{i_1},Z^{(1)}_{i_1})} + \delta_{(Y^{(4)}_{i_4},Z^{(4)}_{i_4})} - \delta_{Y^{(1)}_{i_1},Z^{(4)}_{i_4}} - \delta_{Y^{(4)}_{i_4},Z^{(1)}_{i_1}})(y,z) \\ &\qquad\qquad\qquad d(\delta_{(Y^{(2)}_{i_2},Z^{(2)}_{i_2})} + \delta_{(Y^{(3)}_{i_3},Z^{(3)}_{i_3})} - \delta_{(Y^{(2)}_{i_2},Z^{(3)}_{i_3})} - \delta_{(Y^{(3)}_{i_3},Z^{(2)}_{i_2})})(y',z') \\ &= \iint g((y,z),(y',z')) \, d(\frac{1}{m^2} \sum_{i_1,i_4=1}^{m} \delta_{(Y^{(1)}_{i_1},Z^{(1)}_{i_1})} + \delta_{(Y^{(4)}_{i_4},Z^{(4)}_{i_4})} - \delta_{Y^{(1)}_{i_1},Z^{(4)}_{i_4}} - \delta_{Y^{(4)}_{i_4},Z^{(1)}_{i_1}})(y,z) \\ &\qquad\qquad\qquad d(\frac{1}{m^2} \sum_{i_2,i_3=1}^{m} \delta_{(Y^{(2)}_{i_2},Z^{(2)}_{i_2})} + \delta_{(Y^{(3)}_{i_3},Z^{(3)}_{i_3})} - \delta_{(Y^{(2)}_{i_2},Z^{(3)}_{i_3})} - \delta_{(Y^{(3)}_{i_3},Z^{(2)}_{i_2})})(y',z')
     \\ &= \iint g((y,z),(y',z')) \, d(\delta_{(\mathbb{Y}^{(1)},\mathbb{Z}^{(1)})} + \delta_{(\mathbb{Y}^{(4)},\mathbb{Z}^{(4)})} - \delta_{\mathbb{Y}^{(1)}} \times \delta_{\mathbb{Z}^{(4)}} - \delta_{\mathbb{Y}^{(4)}} \times \delta_{\mathbb{Z}^{(1)}})(y,z) \\ &\qquad\qquad\qquad d(\delta_{(\mathbb{Y}^{(2)},\mathbb{Z}^{(2)})} + \delta_{(\mathbb{Y}^{(3)},\mathbb{Z}^{(3)})} - \delta_{\mathbb{Y}^{(2)}} \times \delta_{\mathbb{Z}^{(3)}} - \delta_{\mathbb{Y}^{(3)}} \times \delta_{\mathbb{Z}^{(2)}})(y',z')
\end{talign}
For a permutation $\pi$ of $[s]$, we can also define the \emph{cheaply-permuted independence V-statistic}, 
\begin{talign}
    V_n^{\pi,s} \defeq \frac{1}{s^4} \sum_{(i_1,i_2,i_3,i_4) \in [s]^4} \Hin \big(&(\mathbb{Y}^{(i_1)},\mathbb{Z}^{(\pi_{i_1})}),(\mathbb{Y}^{(i_2)},\mathbb{Z}^{(\pi_{i_2})}), \label{eq:cheaply-permuted-independence-V}\\
    &(\mathbb{Y}^{(i_3)},\mathbb{Z}^{(\pi_{i_3})}),(\mathbb{Y}^{(i_4)},\mathbb{Z}^{(\pi_{i_4})}) \big).
\end{talign}
These expressions along with the substitutions $s$ for $n$ and $\Hin$ for $\hin$ help us to frame a cheap independence V-statistic permutation test as a standard independence V-statistic permutation test operating on bins of datapoints.

\section{\pcref{thm: Full Two Sample Tests}
} 
\label{subsec:proof_ts_finite}
We will establish both claims using the refined two moments method (\cref{Lemma: Refined Two Moments Method}) with the auxiliary sequence  $(U_{n_1,n_2}^{\pi_b,s})_{b=1}^\numperm$ of cheaply-permuted two-sample U-statistics \cref{eq: Two Sample Cheap U_statistic}. 
The expression \cref{eq:H_ts_G_m} for $\Hho$ ensures that 
$\E[U_{n_1,n_2}^{\pi,s}\mid \X] = 0$ surely for all $s\geq 4$, so, by \cref{Lemma: Refined Two Moments Method}, it suffices to bound  $\Var(U_{n_1,n_2})$ and $\E[\Var(U_{n_1,n_2}^{\pi,s}\mid \X)]$.
\cref{lem:variance_2s}, proved in \cref{subsec:variance_2s_proofs}, gives the unconditional variance bound. 
\begin{lemma}[Variance of homogeneity U-statistics] \label{lem:variance_2s}
Under the assumptions of \cref{thm: Full Two Sample Tests},
    \begin{talign} \label{Eq: Variance upper bound}
    \Var(U_{n_1,n_2})
    \leq %
    \frac{4  \psi_{Y,1}}{n_1} + \frac{4  \psi_{Z,1}}{n_2} + \frac{(17 n_1 n_2 + 4 n_2^2+ 4 n_1^2) \psi_{YZ,2}}{n_1 (n_1 - 1) n_2 (n_2 - 1)}.
    \end{talign}
\end{lemma}

\cref{lem:exp_variance_2s}, proved in \cref{subsec:exp_variance_2s_proofs}, provides the conditional variance bounds.
\begin{lemma}[Conditional variance of permuted homogeneity U-statistics] \label{lem:exp_variance_2s}
Under the assumptions of \cref{thm: Full Two Sample Tests}, 
\begin{talign} \label{eq:kim_thm_4.1}
    \mE[ \Var (U_{n_1,n_2}^{\pi,n} &| \mbb{X}) ]
    \leq \frac{2n_1^2 + 2n_2^2 + 16 n_1 n_2}{n_1 (n_1 - 1) n_2 (n_2 - 1)} \tilde{\psi}_{YZ,2} \qtext{and} \\
    \mE[ \Var (U_{n_1,n_2}^{\pi,s} | \mbb{X}) ] &\leq \big( \frac{ 2 n_2^2 + 2 n_1^2 + 16 n_1 n_2 + 400 (n^2_1 + n^2_2)/s}{n_1(n_1-1) n_2(n_2-1)} \! + \! \frac{352 n^4 / s}{n_1(n_1-1)^2 n_2(n_2-1)^2} \big) \tilde{\psi}_{YZ,2} \\ &\qquad + \frac{8 n^3 m^3 \mathbb{E}[U_{n_1,n_2}]^2 + 96 n^3 m^2 \max\{ \psi_{Y,1}, \psi_{Z,1} \}}{n_1(n_1 - 1)^2 n_2(n_2 - 1)^2}. \label{eq:kim_thm_4.1_cheap}
\end{talign}
\end{lemma}
\cref{Lemma: Refined Two Moments Method,lem:variance_2s,lem:exp_variance_2s} together imply power at least $1-\beta$ for standard testing if
\begin{talign} 
    \begin{split} \label{eq:strong_lb_full_2s}
    \mathbb{E}[U_{n_1,n_2}] &\geq 
        \sqrt{\frac{3-\beta}{\beta}\big(\frac{4  \psi_{Y,1}}{n_1} + \frac{4  \psi_{Z,1}}{n_2} + \frac{(17 n_1 n_2 + 4 n_2^2+ 4 n_1^2) \psi_{YZ,2}}{n_1 (n_1 - 1) n_2 (n_2 - 1)} \big)}  + \sqrt{\frac{1-\astar}{\astar}\frac{(6 n_1^2 + 6 n_2^2 + 48 n_1 n_2) \psi_{YZ,2}}{\beta n_1 (n_1 - 1) n_2 (n_2 - 1)}},
    \end{split} 
\end{talign}
a condition implied by our assumption \cref{eq:lb_full_2s}, and power at least $1-\beta$ for cheap testing if
\begin{talign}
    \begin{split}
        \mathbb{E}[U_{n_1,n_2}] &\geq 
        \sqrt{\frac{3-\beta}{\beta}\big(\frac{4  \psi_{Y,1}}{n_1} + \frac{4  \psi_{Z,1}}{n_2} + \frac{(17 n_1 n_2 + 4 n_2^2+ 4 n_1^2) \psi_{YZ,2}}{n_1 (n_1 - 1) n_2 (n_2 - 1)} \big)} \\ &+
        \big(\frac{(\astar)^{-1} - 1}{\beta} \big( \frac{ 6 n_2^2 + 6 n_1^2 + 48 n_1 n_2}{n_1(n_1-1) n_2(n_2-1)} \! + \! \frac{1200 (n_1-1) (n_2-1) (n^2_1 + n^2_2)/s + 1056 n^3 m}{n_1(n_1-1)^2 n_2(n_2-1)^2} \big) \tilde{\psi}_{YZ,2} \\ &+ \frac{288 ((\astar)^{-1} - 1) n^3 m^2 \max\{ \psi_{Y,1}, \psi_{Z,1} \}}{\beta n_1(n_1 - 1)^2 n_2(n_2 - 1)^2} \big)^{1/2} + \sqrt{\frac{24 ((\astar)^{-1} - 1) n^3 m^3 \mathbb{E}[U_{n_1,n_2}]^2}{\beta n_1(n_1 - 1)^2 n_2(n_2 - 1)^2}}.
        \label{eq:strong_lb_cheap_2s}
    \end{split}
\end{talign}
We can rewrite \cref{eq:strong_lb_cheap_2s} equivalently as 
\begin{talign} 
\begin{split} \label{eq:strong_lb_cheap_2s_2}
    \mathbb{E}[U_{n_1,n_2}] &\geq \big( 1 - \sqrt{\frac{24 ((\astar)^{-1} - 1) n^3 m^3}{\beta n_1(n_1 - 1)^2 n_2(n_2 - 1)^2}} \big)^{-1}  \bigg( \sqrt{\frac{3-\beta}{\beta}\big(\frac{4  \psi_{Y,1}}{n_1} + \frac{4  \psi_{Z,1}}{n_2} + \frac{(17 n_1 n_2 + 4 n_2^2+ 4 n_1^2) \psi_{YZ,2}}{n_1 (n_1 - 1) n_2 (n_2 - 1)} \big)} \\ &\ 
    + \big(\frac{(\astar)^{-1} - 1}{\beta} \big( \frac{ 6 n_2^2 + 6 n_1^2 + 48 n_1 n_2 + 1200 (n^2_1 + n^2_2)/s}{n_1(n_1-1) n_2(n_2-1)} \! + \! \frac{1056 n^3 m}{n_1(n_1-1)^2 n_2(n_2-1)^2} \big) \tilde{\psi}_{YZ,2} \\ &\qquad + \frac{288 ((\astar)^{-1} - 1) n^3 m^2 \max\{ \psi_{Y,1}, \psi_{Z,1} \}}{\beta n_1(n_1 - 1)^2 n_2(n_2 - 1)^2} \big)^{1/2}
    \bigg).
\end{split}
\end{talign}
To conclude, we note that our assumption \cref{eq:lb_cheap_2s} is a sufficient condition for  \cref{eq:strong_lb_cheap_2s_2} as
\begin{talign}
n^4 %
&= (n_1^2 + n_2^2)^2 + 4 n_1 n_2 (n_1^2 + n_2^2) + 4 n_1^2 n_2^2 \geq 8 n_1 n_2 (n_1^2 + n_2^2) \geq 8 (n_1 - 1) (n_2 - 1) (n_1^2 + n_2^2).
\end{talign}

\subsection{\pcref{lem:variance_2s}} \label{subsec:variance_2s_proofs}
By \cite[p.~38]{lee1990ustatistics}, 
\begin{talign}  
\Var(U_{n_1,n_1}) 
    &\leq 
\sum_{i=0}^2 \sum_{j=0}^2 {2 \choose i} {2 \choose j} {n_1 - 2 \choose 2 - i} {n_2 - 2 \choose 2 - j} {n_1 \choose 2}^{-1} {n_2 \choose 2}^{-1}\hat{\sigma}_{i,j}^2 %
 \\ &= 16 \sum_{i=0}^2 \sum_{j=0}^2 \frac{(n_1 - 2)!^2}{n_1! i! (2-i)!^2 (n_1 - 4 + i)!} \frac{(n_2 - 2)!^2}{n_2! j! (2-j)!^2 (n_2- 4 + j)!} \hat{\sigma}_{i,j}^2 \qtext{where} 
 \label{eq:var_lee}
\\
\hat{\sigma}_{i,j}^2 
    &\defeq 
\Var( \E [ \hbarho (Y_1,Y_2; Z_1, Z_2) | Y_{i+1}, \dots, Y_2, Z_{j+1}, \dots, Z_2 ])
\qtext{for}
i, j \in [2].
\end{talign}
Note that for $n_1 \geq 4$ (and analogously for $n_2 \geq 4$),
\begin{talign}
	\frac{(n_1 - 2)!^2}{n_1! i! (2-i)!^2 (n_1 - 4 + i)!} = 
	\begin{cases}
		\frac{(n_1 - 2)!^2}{4 n_1! (n_1 - 4)!} = \frac{(n_1 - 2)(n_1 - 3)}{4 n_1 (n_1 - 1)} \leq \frac{1}{4} &\text{if } i = 0 \\ 
		\frac{(n_1 - 2)!^2}{n_1! (n_1 - 3)!} = \frac{n_1 - 2}{n_1(n_1 - 1)} \leq \frac{1}{n_1} &\text{if } i = 1 \\
		\frac{(n_1 - 2)!^2}{2 n_1! (n_1 - 2)!} = \frac{1}{2 n_1 (n_1 -1)} &\text{if } i = 2
	\end{cases}.
\end{talign}
Plugging this into the right-hand side of \cref{eq:var_lee}, we obtain that for $n_1, n_2 \geq 4$,
\begin{talign} 
\begin{split} \label{eq:var_lee_conclusion}
	\Var(U_{n_1,n_2}) &\leq 16 \big( \frac{\hat{\sigma}_{0,0}^2}{4^2} + \frac{ \hat{\sigma}_{0,1}^2}{4 n_2}  +   \frac{ \hat{\sigma}_{1,0}^2 }{4 n_1} + \frac{\hat{\sigma}_{1,1}^2}{n_1 n_2}  + \frac{ \hat{\sigma}_{0,2}^2}{8 n_2 (n_2 -1)}  + \frac{\hat{\sigma}_{2,0}^2}{8 n_1 (n_1 -1)} \\ &\qquad + \frac{\hat{\sigma}_{1,2}^2 }{2 n_1 n_2 (n_2 - 1)} + \frac{\hat{\sigma}_{2,1}^2}{2 n_2 n_1 (n_1 - 1)}  +  \frac{\hat{\sigma}_{2,2}^2}{4 n_1 (n_1 - 1) n_2 (n_2 - 1)} \big) \\ &\leq \frac{4  \hat{\sigma}_{1,0}^2}{n_1} + \frac{4  \hat{\sigma}_{0,1}^2}{n_2} + \big(16 + \frac{4}{9} \big) \frac{\hat{\sigma}_{2,2}^2}{n_1 n_2} + \big(2 + 2 \big) \frac{\hat{\sigma}_{2,2}^2}{n_1 (n_1 - 1)} + \big(2 + 2 \big) \frac{\hat{\sigma}_{2,2}^2}{n_2 (n_2 - 1)} %
	\\ &\leq \frac{4  \hat{\sigma}_{1,0}^2}{n_1} + \frac{4  \hat{\sigma}_{0,1}^2}{n_2} + \frac{(17 n_1 n_2 + 4 n_2^2+ 4 n_1^2) \hat{\sigma}_{2,2}^2}{n_1 (n_1 - 1) n_2 (n_2 - 1)}.
\end{split}
\end{talign}
In the second inequality, we used that $\hat{\sigma}_{0,0}^2 = 0$ and that, by the law of total variance, $\hat{\sigma}_{i,j}^2 \leq \hat{\sigma}_{2,2}^2$ for all $i,j \in [2]$. Equation \cref{Eq: Variance upper bound} then follows from the fact that $\hat{\sigma}_{1,0}^2\leq \psi_{Y,1} = $, $\hat{\sigma}_{0,1}^2\leq \psi_{Z,1}$, and $\hat{\sigma}_{2,2}^2 \leq \psi_{YZ,2}$ by their definitions.
\subsection{\pcref{lem:exp_variance_2s}} \label{subsec:exp_variance_2s_proofs}
We begin by defining the convenient shorthand 
\begin{talign}
\mathfrak{i} \defeq (i_1,i_2,j_1,j_2,i_1',i_2',j_1',j_2')
    \qtext{and}
\mathfrak{s}(\mathfrak{i}) \defeq |\{i_1,i_2\} \cap \{i_1',i_2'\}| + |\{j_1,j_2\} \cap \{j_1',j_2'\}|.
\end{talign} 
\subsubsection{Proof of \cref{eq:kim_thm_4.1}}
Following the proof of \cite[Thm.~4.1]{kim2022minimax}, we define the index sets
\begin{talign}
\begin{split}
\mathsf{I}_{\text{total}} &\defeq \{ \mathfrak{i} \in \mathbb{N}^8_{+}: (i_1,i_2) \in \mathbf{i}_2^{n_1}, (j_1,j_2) \in \mathbf{i}_2^{n_2}, (i_1',i_2') \in \mathbf{i}_2^{n_1}, (j_1',j_2') \in \mathbf{i}_2^{n_2}\}, \\
\mathsf{I}_{1} &\defeq \{ \mathfrak{i} \in \mathsf{I}_{\text{total}} : \mathfrak{s}(\mathfrak{i})  \leq 1\}, \qtext{and}
\mathsf{I}_{1}^c = \{ \mathfrak{i} \in \mathsf{I}_{\text{total}} : \mathfrak{s}(\mathfrak{i}) >  1\}.
\end{split}
\end{talign}
By the argument of \cite[Thm.~4.1]{kim2022minimax}, we have 
\begin{talign} \label{eq:var_permutations_2s_precise}
    \mE [ \Var (U_{n_1,n_2}^{\pi,n} | \mbb X ) ] \leq \tilde{\psi}_{YZ,2} \frac{|\mathsf{I}_{1}^c|}{(n_1)^2_{(2)} (n_2)^2_{(2)}},
\end{talign}
where $\tilde{\psi}_{YZ,2}$ is the maximum over $\mathfrak{i} \in \mathsf{I}_{1}^c$ of 
\begin{talign} \label{eq:psi_prime_upper}
     &\big|\mE \big[\E[\hbarho(X_{\pi_{i_1}},X_{\pi_{i_2}};X_{\pi_{n_1+j_1}},X_{\pi_{n_1+j_2}}) \hbarho(X_{\pi_{i'_1}},X_{\pi_{i'_2}};X_{\pi_{n_1+j'_1}},X_{\pi_{n_1+j'_2}}) | \mbb{X} ] \big] \big|,
\end{talign}
and $\hbarho$ is the symmetrized version of $\hho$ defined in \cref{Eq: symmetrized kernel}. 
Moreover,
\begin{talign}
|\mathsf{I}_{\text{total}}| &= n_1^2(n_1-1)^2 n_2^2(n_2-1)^2, \\
    |\{\mathfrak{i} \in \mathsf{I}_{\text{total}} \, | \, \mathfrak{s}(\mathfrak{i}) = 0 \}| &=
    n_1(n_1-1)(n_1-2)(n_1-3)n_2(n_2-1)(n_2-2)(n_2-3), \qtext{and}\\
    |\{\mathfrak{i} \in \mathsf{I}_{\text{total}} \, | \, \mathfrak{s}(\mathfrak{i}) = 1 \}| &= 4 n_1(n_1-1)(n_1-2)n_2(n_2-1)((n_2-2)(n_2-3)+ (n_1-2)(n_1-3)),
\end{talign}
and therefore
\begin{talign}
|\mathsf{I}_{1}^c| &= |\mathsf{I}_{\text{total}}| - |\{\mathfrak{i} \in \mathsf{I}_{\text{total}} \, | \, \mathfrak{s}(\mathfrak{i}) = 0 \}| - |\{\mathfrak{i} \in \mathsf{I}_{\text{total}} \, | \, \mathfrak{s}(\mathfrak{i}) = 1 \}| \\
    &= 
n_1(n_1-1) n_2(n_2-1) (2n_1^2 + 2n_2^2 + 16 n_1 n_2 - 34 n_1 - 34 n_2 + 60).
\label{eq:I1c}
\end{talign}
Since $\psi_{YZ,2}$ upper-bounds $\tilde{\psi}_{YZ,2}$ and $n\geq 4$, the following estimate, provided by \cref{eq:var_permutations_2s_precise,eq:I1c}, completes our proof:
\begin{talign}
    \mE [ \Var (U_{n_1,n_2}^{\pi,n} | \mbb{X} ) ] &\leq \frac{2\tilde{\psi}_{YZ,2} (n_1^2 + n_2^2 + 8 n_1 n_2 - 17 (n_1 + n_2) + 30)}{n_1 (n_1 - 1) n_2 (n_2 - 1)} \leq \frac{\tilde{\psi}_{YZ,2}(2n_1^2 + 2n_2^2 + 16 n_1 n_2)}{n_1 (n_1 - 1) n_2 (n_2 - 1)}.
\end{talign}
\subsubsection{Proof of \cref{eq:kim_thm_4.1_cheap}}
In this case, we can write
\begin{talign}
\begin{split} \label{eq:exp_variance_kim_cheap}
&\mE [ \Var (U_{n_1,n_2}^{\pi,s} | \mbb X) ] 
    = 
\mE [ \E [(U_{n_1,n_2}^{\pi,s})^2 | \mbb X] ] - \E[ (\E [U_{n_1,n_2}^{\pi,s} | \mbb X])^2 ] %
    \\ 
    &= 
\frac{1}{n_1^2(n_1 - 1)^2 n_2^2(n_2 - 1)^2} \sum_{\mathfrak{i} \in \mathsf{I}_{\text{total}}} \\
&\quad\ \mathbb{E}\big[\E[\hbarho(X_{\tilde{\pi}_{i_1}},X_{\tilde{\pi}_{i_2}};X_{\tilde{\pi}_{n_1+j_1}},X_{\tilde{\pi}_{n_1+j_2}})  \hbarho(X_{\tilde{\pi}_{i'_1}},X_{\tilde{\pi}_{i'_2}};X_{\tilde{\pi}_{n_1+j'_1}},X_{\tilde{\pi}_{n_1+j'_2}}) | \mbb{X} ] \\ 
    &\quad\ - 
\E [ \hbarho(X_{\tilde{\pi}_{i_1}},X_{\tilde{\pi}_{i_2}};X_{\tilde{\pi}_{n_1+j_1}},X_{\tilde{\pi}_{n_1+j_2}}) | \mbb{X} ]  \E [ \hbarho(X_{\tilde{\pi}_{i'_1}},X_{\tilde{\pi}_{i'_2}};X_{\tilde{\pi}_{n_1+j'_1}},X_{\tilde{\pi}_{n_1+j'_2}}) | \mbb{X} ] \big],
\end{split}
\end{talign}
where $\tilde{\pi}_{i} \defeq 
m \pi_{\lfloor i / m \rfloor} + i - m \lfloor i / m \rfloor$. 
We now express $\mathsf{I}_{\text{total}}$ as the disjoint union of four sets:
\begin{itemize}
    \item $\mathsf{I}_{\mathrm{diff}}$, containing $s_1^2 (s_1 - 1)^2 s_2^2 (s_2 - 1)^2 m^8 = s_1^2 (s_1^2 - 2s_1 + 1) s_2^2 (s_2^2 - 2s_2 + 1) m^8$ indices for which the bins of each component are pairwise different.
    \item $\mathsf{I}_{a}$, containing the %
    indices for which one or more of the pairs among $(i_1,i_2)$, $(i'_1,i'_2)$, $(n_1 + j_1,n_1 + j'_1)$, $(n_1 + j_2, n_1 + j'_2)$ share the same bin, but $\{\lfloor i_1 / m \rfloor, \lfloor i_2 / m \rfloor \} \cap \{\lfloor i'_1 / m \rfloor, \lfloor i'_2 / m \rfloor \} = \emptyset$, $\{ \lfloor (n_1 + j_1) / m \rfloor, \lfloor (n_1 + j_2) / m \rfloor \} \cap \{ \lfloor (n_1 + j'_1) / m \rfloor,$ and $\lfloor (n_1 + j'_2) / m \rfloor \} = \emptyset$. 
    \item $\mathsf{I}_{b}$, containing the indices for which exactly one pair among $(i_1,i_2)$, $(i'_1,i'_2)$, $(n_1 + j_1,n_1 + j'_1)$, $(n_1 + j_2, n_1 + j'_2)$ shares the same bin, and $|\{\lfloor i_1 / m \rfloor, \lfloor i_2 / m \rfloor \} \cap \{\lfloor i'_1 / m \rfloor, \lfloor i'_2 / m \rfloor \}| + |\{ \lfloor (n_1 + j_1) / m \rfloor, \lfloor (n_1 + j_2) / m \rfloor \} \cap \{ \lfloor (n_1 + j'_1) / m \rfloor, \lfloor (n_1 + j'_2) / m \rfloor \}| = 1$. 
    \item $\mathsf{I}_{\mathrm{rest}}$, containing the rest of indices. For indices in $\mathsf{I}_{\mathrm{rest}}$, there has to be at least one pair among $(i_1,i_2)$, $(i'_1,i'_2)$, $(n_1 + j_1,n_1 + j'_1)$, $(n_1 + j_2, n_1 + j'_2)$ that shares the same bin, and there have to be at least two pairs formed by an index in $\{i_1,i_2,n_1 + j_1,n_1 + j_2\}$ and an index in $\{ i'_1,i'_2,n_1 + j'_1,n_1 + j'_2\}$ such that both indices belong to the same bin. From this argument, we can deduce that
    \begin{talign}
        |\mathsf{I}_{\mathrm{rest}}| &\leq (4 s_1 s_2^4 + 4 s_1^2 s_2^3 + 16 s_1^2 s_2^3 + 16 s_1^3 s_2^2 + 4 s_1^4 s_2 + 4 s_1^3 s_2^2) m^8 \\ &= (4 s_1 s_2^4 + 20 s_1^2 s_2^3 + 20 s_1^3 s_2^2 + 4 s_1^4 s_2) m^8.
    \end{talign}
\end{itemize}

\paragraph{Contribution of $\mathsf{I}_{\mathrm{diff}}$ terms} To deal with the contribution from $\mathsf{I}_{\mathrm{diff}}$ terms, we mirror the argument of \citep[Thm. 4.1]{kim2022minimax}. %
In %
\cref{subsec:additional_homogeneity} we have seen that cheap permutation can be treated like standard permutation of dataset bins with the substitutions $n \gets s$, $n_1 \gets s_1$, $n_2 \gets s_2$, $\hho \gets \Hho$, $g \gets G_m$. To this end, we define the index sets
\begin{talign}
\mathsf{K}_{\text{total}} &\defeq \{ \mathfrak{i} \in \mathbb{N}^8_{+}: (i_1,i_2) \in \mathbf{i}_2^{s_1}, (j_1,j_2) \in \mathbf{i}_2^{s_2}, (i_1',i_2') \in \mathbf{i}_2^{s_1}, (j_1',j_2') \in \mathbf{i}_2^{s_2}\}, \\
\mathsf{K}_{1} &\defeq \{ \mathfrak{i} \in \mathsf{K}_{\text{total}} : \mathfrak{s}(\mathfrak{i})  \leq 1\}, \ \ 
\mathsf{K}_{1}^c = \{ \mathfrak{i} \in \mathsf{K}_{\text{total}} : \mathfrak{s}(\mathfrak{i}) >  1\}, 
    \sstext{and}
\mathsf{K}_{2} \defeq \{ \mathfrak{i} \in \mathsf{K}_{\text{total}} : \mathfrak{s}(\mathfrak{i}) = 2 \}.
\end{talign}
We compute the size of $\mathsf{K}_{2}$:
\begin{talign}
|\mathsf{K}_{2}| &= 2 s_1(s_1-1) s_2(s_2-1)(s_2-2)(s_2-3) + 2 s_1(s_1-1)(s_1-2)(s_1-3) s_2(s_2-1) \\ &\qquad + 4 s_1(s_1-1)(s_1-2) \cdot 4 s_2(s_2-1)(s_2-2) %
\\ &= s_1(s_1-1) s_2(s_2-1) ( 2 s_2^2 + 2 s_1^2 + 16 s_1 s_2 - 42 s_1 - 42 s_2 + 88 ).
\end{talign}
Hence, we obtain that
\begin{talign}
|\mathsf{K}_{1}^c \setminus \mathsf{K}_{2}| &= |\mathsf{K}_{1}^c| - |\mathsf{K}_{2}| = s_1(s_1-1) s_2(s_2-1) ( ( - 34 s_1 - 34 s_2 + 60 ) - ( - 42 s_1 - 42 s_2 + 88 ) ) \\ &= s_1(s_1-1) s_2(s_2-1) (8 s_1 + 8 s_2 - 28).
\end{talign}
Our goal is now to upper-bound the contribution of $\mathsf{I}_{\text{diff}}$ to \cref{eq:exp_variance_kim_cheap}, 
\begin{talign}
\begin{split} \label{eq:ts_kim_bound}
&%
\sum_{\mathfrak{i} \in \mathsf{I}_{\text{diff}}} \mE[\E[\hbarho(X_{\tilde{\pi}_{i_1}},X_{\tilde{\pi}_{i_2}};X_{\tilde{\pi}_{n_1+j_1}},X_{\tilde{\pi}_{n_1+j_2}}) \\ &\qquad\qquad\qquad\qquad\qquad\qquad\times \hbarho(X_{\tilde{\pi}_{i'_1}},X_{\tilde{\pi}_{i'_2}};X_{\tilde{\pi}_{n_1+j'_1}},X_{\tilde{\pi}_{n_1+j'_2}}) | \mbb{X} ]] \\
&=%
m^8 \sum_{\mathfrak{i} \in \mathsf{K}_{\mathrm{total}} } \mE\big[\E\big[\Hho(\mathbb{X}_n^{(\pi_{i_1})}, \mathbb{X}_n^{(\pi_{i_2})}; \mathbb{X}_n^{(\pi_{s_1+j_1})}, \mathbb{X}_n^{(\pi_{s_1+j_2})}) \\ &\qquad\qquad\qquad\qquad\qquad\qquad \times \Hho(\mathbb{X}_n^{(\pi_{i'_1})}, \mathbb{X}_n^{(\pi_{i'_2})}; \mathbb{X}_n^{(\pi_{s_1+j'_1})}, X_n^{(\pi_{s_1+j'_2})}) \big| \mbb{X}\big] \big]
\\ &\leq %
\sum_{\mathfrak{i} \in \mathsf{K}_{\mathrm{total}} } %
\mE \big[\E \big[ \big(\sum_{k_1,k_2,\ell_1,\ell_2 = 1}^{m} \hbarho(X_{k_1}^{(\pi_{i_1})}, X_{k_2}^{(\pi_{i_2})}; X_{\ell_1}^{(\pi_{s_1+j_1})}, X_{\ell_2}^{(\pi_{s_1+j_2})}) \big) \\ &\qquad\qquad\qquad\qquad \times \big(\sum_{k'_1,k'_2,\ell'_1,\ell'_2 = 1}^{m} \hbarho(X_{k'_1}^{(\pi_{i'_1})}, X_{k'_2}^{(\pi_{i'_2})}; X_{\ell'_1}^{(\pi_{s_1+j'_1})}, X_{\ell'_2}^{(\pi_{s_1+j'_2})}) \big) \big| \mbb{X}\big] \big] \\ &= %
\sum_{\mathfrak{i} \in \mathsf{K}_{\mathrm{total}} } %
\sum_{\mathfrak{K} \in \mathsf{J}_{\mathrm{total}}} \mE \big[\E \big[ \hbarho(X_{k_1}^{(\pi_{i_1})}, X_{k_2}^{(\pi_{i_2})}; X_{\ell_1}^{(\pi_{s_1+j_1})}, X_{\ell_2}^{(\pi_{s_1+j_2})}) \\ &\qquad\qquad\qquad\qquad\qquad\qquad\qquad \times \hbarho(X_{k'_1}^{(\pi_{i'_1})}, X_{k'_2}^{(\pi_{i'_2})}; X_{\ell'_1}^{(\pi_{s_1+j'_1})}, X_{\ell'_2}^{(\pi_{s_1+j'_2})}) \big| \mbb{X}\big] \big],
\end{split}
\end{talign}
where we define $\mathfrak{K} \defeq (k_1,k_2,\ell_1,\ell_2,k'_1,k'_2,\ell'_1,\ell'_2)$, 
\begin{talign} \label{eq:J_A_defs}
\mathsf{J}_{\text{total}} &\defeq \{1, \dots, m\}^8,\quad
\mathsf{J}_{1} \defeq \{ \mathfrak{K} \in \mathsf{J}_{\text{total}} : \mathfrak{s}(\mathfrak{K}) \leq 1\}, \qtext{and}
 \mathsf{J}_{1}^c = \{ \mathfrak{K} \in \mathsf{J}_{\text{total}} : \mathfrak{s}(\mathfrak{K}) >  1\}.
\end{talign}
Each non-zero summand in the final expression of  \cref{eq:ts_kim_bound} has indices satisfying
\begin{enumerate}
\item $\mathfrak{i} \in \mathsf{K}_2$, with, for the indices in $\mathfrak{i}$ that are equal, the corresponding indices in $\mathfrak{K}$ are also equal (so that, for each $\mathfrak{i} \in \mathsf{K}_2$, there are at most $m^6$ terms with non-zero contribution), or 
\item $\mathfrak{i} \in \mathsf{K}_{1}^c \setminus \mathsf{K}_{2}$, and $\mathfrak{K}\in\mathsf{J}_{1}^c$ (as otherwise the terms would have zero contribution).
\end{enumerate}
We now bound the size of $\mathsf{J}_{1}^c$:
\begin{talign}
|\mathsf{J}_{1}| &= 
|\{\mathfrak{K} \in \mathsf{J}_{\text{total}} \, | \, \mathfrak{s}(\mathfrak{K}) = 0 \}| + |\{\mathfrak{K} \in \mathsf{J}_{\text{total}} \, | \, \mathfrak{s}(\mathfrak{K}) = 1 \}| 
\end{talign}
\begin{talign}
|\{\mathfrak{K} \in \mathsf{J}_{\text{total}} \, | \, \mathfrak{s}(\mathfrak{K}) = 0 \}| &\geq m (m - 1) (m - 2) (m - 3) m (m - 1) (m - 2) (m - 3) \\ &\quad + 4 m (m - 1) (m - 2) m (m - 1) (m - 2) (m - 3) 
\\ &= m (m - 1) (m - 2) m (m - 1) (m - 2) (m - 3) \big( (m - 3) + 4 \big)
\\
|\{\mathfrak{K} \in \mathsf{J}_{\text{total}} \, | \, \mathfrak{s}(\mathfrak{K}) = 1 \}| &\geq 8 m (m - 1) (m - 2) m (m - 1) (m - 2) (m - 3) %
\end{talign}
\begin{talign}
\implies |\mathsf{J}_{1}^c| &= |\mathsf{J}_{\text{total}}| - |\{\mathfrak{K} \in \mathsf{J}_{\text{total}} \, | \, \mathfrak{s}(\mathfrak{K}) = 0 \}| - |\{\mathfrak{K} \in \mathsf{J}_{\text{total}} \, | \, \mathfrak{s}(\mathfrak{K}) = 1 \}| \\ &\leq m^8 - m (m - 1) (m - 2) m (m - 1) (m - 2) (m - 3) \big( (m - 3) + 4 \big) \\ &\quad - 8 m (m - 1) (m - 2) m (m - 1) (m - 2) (m - 3) \\ &= 50 m^6 - 228 m^5 + 419 m^4 - 348 m + 108 \leq 50 m^6.
\end{talign}
Since every non-zero summand is bounded by $\tilde{\psi}_{YZ,2}$, we can bound \cref{eq:ts_kim_bound} by
\begin{talign} \label{eq:I_diff_final_bound}
&|\mathsf{K}_{2}| \cdot m^6 \tilde{\psi}_{YZ,2}
+ 
|\mathsf{K}_{1}^c \setminus \mathsf{K}_{2}| \cdot |\mathsf{J}_{1}^c| \tilde{\psi}_{YZ,2}
\\ &\leq 
s_1(s_1-1) s_2(s_2-1) (( 2 s_2^2 + 2 s_1^2 + 16 s_1 s_2) m^6
    + 
(8 s_1 + 8 s_2) \cdot 50 m^6) \tilde{\psi}_{YZ,2} 
\\ &= %
n_1(n_1-m) n_2(n_2-m) ( 2 n_2^2 + 2 n_1^2 + 16 n_1 n_2 + 400 (n^2_1 + n^2_2)/s) \tilde{\psi}_{YZ,2},
\end{talign}
which, in turn, bounds the contribution of $\mathsf{I}_{\mathrm{diff}}$ to  \cref{eq:exp_variance_kim_cheap}.

\paragraph{Contribution of $\mathsf{I}_{a}$ terms} Without loss of generality, assume that $i_1$ and $i_2$ belong to the same bin but that all other indices belong to pairwise different bins. Then,
\begin{talign}
    &\mE\big[\E[\hbarho(X_{\tilde{\pi}_{i_1}},X_{\tilde{\pi}_{i_2}};X_{\tilde{\pi}_{n_1+j_1}},X_{\tilde{\pi}_{n_1+j_2}}) \hbarho(X_{\tilde{\pi}_{i'_1}},X_{\tilde{\pi}_{i'_2}};X_{\tilde{\pi}_{n_1+j'_1}},X_{\tilde{\pi}_{n_1+j'_2}}) | \mbb{X} ] \\ &\qquad - \E [ \hbarho(X_{\tilde{\pi}_{i_1}},X_{\tilde{\pi}_{i_2}};X_{\tilde{\pi}_{n_1+j_1}},X_{\tilde{\pi}_{n_1+j_2}}) | \mbb{X} ] \E [ \hbarho(X_{\tilde{\pi}_{i'_1}},X_{\tilde{\pi}_{i'_2}};X_{\tilde{\pi}_{n_1+j'_1}},X_{\tilde{\pi}_{n_1+j'_2}}) | \mbb{X} ] \big] \\ &=
    \mE\big[\E [ \hbarho(X_{\tilde{\pi}_{i_1}},X_{\tilde{\pi}_{i_2}};X_{\tilde{\pi}_{n_1+j_1}},X_{\tilde{\pi}_{n_1+j_2}}) | \mbb{X} ] \E [ \hbarho(X_{\tilde{\pi}_{i'_1}},X_{\tilde{\pi}_{i'_2}};X_{\tilde{\pi}_{n_1+j'_1}},X_{\tilde{\pi}_{n_1+j'_2}}) | \mbb{X} ] \\ &\quad - \E [ \hbarho(X_{\tilde{\pi}_{i_1}},X_{\tilde{\pi}_{i_2}};X_{\tilde{\pi}_{n_1+j_1}},X_{\tilde{\pi}_{n_1+j_2}}) | \mbb{X} ] \E [ \hbarho(X_{\tilde{\pi}_{i'_1}},X_{\tilde{\pi}_{i'_2}};X_{\tilde{\pi}_{n_1+j'_1}},X_{\tilde{\pi}_{n_1+j'_2}}) | \mbb{X} ] \big] = 0.
\end{talign}
\paragraph{Contribution of $\mathsf{I}_{b}$ terms} Without loss of generality, assume that $i_1$ and $i_2$ belong to the same bin and that $n_1 + j_1$ and $n_1 + j'_2$ also belong to the same bin. 
Since exchanging the indices $i'_1$ and $n_1 + j'_1$ does not change the distribution, and $\hbarho(x,y; x',y') = -\hbarho(x', y;  x, y')$,  %
\begin{talign}
    &\mE\big[\E[\hbarho(X_{\tilde{\pi}_{i_1}},X_{\tilde{\pi}_{i_2}};X_{\tilde{\pi}_{n_1+j_1}},X_{\tilde{\pi}_{n_1+j_2}}) \hbarho(X_{\tilde{\pi}_{i'_1}},X_{\tilde{\pi}_{i'_2}};X_{\tilde{\pi}_{n_1+j'_1}},X_{\tilde{\pi}_{n_1+j'_2}}) | \mbb{X} ] \big] \\ &= \mE \big[\E[\hbarho(X_{\tilde{\pi}_{i_1}},X_{\tilde{\pi}_{i_2}};X_{\tilde{\pi}_{n_1+j_1}},X_{\tilde{\pi}_{n_1+j_2}}) \hbarho(X_{\tilde{\pi}_{n_1+j'_1}},X_{\tilde{\pi}_{i'_2}};X_{\tilde{\pi}_{i'_1}},X_{\tilde{\pi}_{n_1+j'_2}}) | \mbb{X} ] \big] \\ &= - \mE\big[\E [ \hbarho(X_{\tilde{\pi}_{i_1}},X_{\tilde{\pi}_{i_2}};X_{\tilde{\pi}_{n_1+j_1}},X_{\tilde{\pi}_{n_1+j_2}}) \hbarho(X_{\tilde{\pi}_{n_1+j'_1}},X_{\tilde{\pi}_{i'_2}};X_{\tilde{\pi}_{i'_1}},X_{\tilde{\pi}_{n_1+j'_2}}) | \mbb{X} ] \big] = 0.
\end{talign}
The same argument works to show that 
\begin{talign}
\E [ \hbarho(X_{\tilde{\pi}_{i_1}},X_{\tilde{\pi}_{i_2}};X_{\tilde{\pi}_{n_1+j_1}},X_{\tilde{\pi}_{n_1+j_2}}) | \mbb{X} ] \E [ \hbarho(X_{\tilde{\pi}_{i'_1}},X_{\tilde{\pi}_{i'_2}};X_{\tilde{\pi}_{n_1+j'_1}},X_{\tilde{\pi}_{n_1+j'_2}}) | \mbb{X} ] = 0.
\end{talign}
\paragraph{Contribution of $\mathsf{I}_{\mathrm{rest}}$ terms} We consider further subcases.

\paragraph{$\mathsf{I}_{\mathrm{rest}}$ terms, subcase (i)}
For indices in $\mathsf{I}_{\mathrm{rest}}$ such that $i_1$, $i_2$, $n_1+j_1$, $n_1+j_2$, $i'_1$, $i'_2$, $n_1+j'_1$, $n_1+j'_2$ are all pairwise different, 
\begin{talign}
&\mE\big[\E[\hbarho(X_{\tilde{\pi}_{i_1}},X_{\tilde{\pi}_{i_2}};X_{\tilde{\pi}_{n_1+j_1}},X_{\tilde{\pi}_{n_1+j_2}}) \hbarho(X_{\tilde{\pi}_{i'_1}},X_{\tilde{\pi}_{i'_2}};X_{\tilde{\pi}_{n_1+j'_1}},X_{\tilde{\pi}_{n_1+j'_2}}) | \mbb{X} ] \big] \\ &= \E\big[\mE[\hbarho(X_{\tilde{\pi}_{i_1}},X_{\tilde{\pi}_{i_2}};X_{\tilde{\pi}_{n_1+j_1}},X_{\tilde{\pi}_{n_1+j_2}}) \hbarho(X_{\tilde{\pi}_{i'_1}},X_{\tilde{\pi}_{i'_2}};X_{\tilde{\pi}_{n_1+j'_1}},X_{\tilde{\pi}_{n_1+j'_2}}) | \pi ] \big] \\ &= \E\big[\mE[\hbarho(X_{\tilde{\pi}_{i_1}},X_{\tilde{\pi}_{i_2}};X_{\tilde{\pi}_{n_1+j_1}},X_{\tilde{\pi}_{n_1+j_2}}) | \pi ] \mE[ \hbarho(X_{\tilde{\pi}_{i'_1}},X_{\tilde{\pi}_{i'_2}};X_{\tilde{\pi}_{n_1+j'_1}},X_{\tilde{\pi}_{n_1+j'_2}}) | \pi ] \big].
\end{talign}
Here, $\mE[\hbarho(X_{\tilde{\pi}_{i_1}},X_{\tilde{\pi}_{i_2}};X_{\tilde{\pi}_{n_1+j_1}},X_{\tilde{\pi}_{n_1+j_2}}) | \pi ]$ and $\mE[ \hbarho(X_{\tilde{\pi}_{i'_1}},X_{\tilde{\pi}_{i'_2}};X_{\tilde{\pi}_{n_1+j'_1}},X_{\tilde{\pi}_{n_1+j'_2}}) | \pi ]$ play the same role, so we will focus on the latter. We consider the following cases:
\begin{itemize}
    \item $\tilde{\pi}_{i_1} \leq n_1$, $\tilde{\pi}_{i_2} \leq n_1$, $\tilde{\pi}_{n_1+j_1} > n_1$, $\tilde{\pi}_{n_1+j_2} > n_1$ or analogous configurations: Then,
    \begin{talign}
        \mE[\hbarho(X_{\tilde{\pi}_{i_1}},X_{\tilde{\pi}_{i_2}};X_{\tilde{\pi}_{n_1+j_1}},X_{\tilde{\pi}_{n_1+j_2}}) | \pi ] = \mE[\hbarho(Y_1,Y_2;Z_1,Z_2) ] = \mE[U_{n_1,n_2}],
    \end{talign}
    \item $\tilde{\pi}_{i_1} \leq n_1$, $\tilde{\pi}_{i_2} \leq n_1$, $\tilde{\pi}_{n_1+j_1} \leq n_1$, $\tilde{\pi}_{n_1+j_2} > n_1$ or analogous configurations: Then, 
    \begin{talign}
        &\mE[\hbarho(X_{\tilde{\pi}_{i_1}},X_{\tilde{\pi}_{i_2}};X_{\tilde{\pi}_{n_1+j_1}},X_{\tilde{\pi}_{n_1+j_2}}) | \pi ] = \mE[\hbarho(Y_1,Y_2;Y_3,Z_1) ] \\ &= \mathbb{E}[\hbarho(Y_3,Y_2;Y_1,Z_1) ] = -\mathbb{E}[\hbarho(Y_1,Y_2;Y_3,Z_1) ] = 0.
    \end{talign}
    \item $\tilde{\pi}_{i_1} \leq n_1$, $\tilde{\pi}_{i_2} \leq n_1$, $\tilde{\pi}_{n_1+j_1} \leq n_1$, $\tilde{\pi}_{n_1+j_2} \leq n_1$ or analogous configurations: Then, 
    \begin{talign}
        &\mE[\hbarho(X_{\tilde{\pi}_{i_1}},X_{\tilde{\pi}_{i_2}};X_{\tilde{\pi}_{n_1+j_1}},X_{\tilde{\pi}_{n_1+j_2}}) | \pi ] = \mE[\hbarho(Y_1,Y_2;Y_3,Y_4) ] \\ &= \mathbb{E} [\hbarho(Y_3,Y_2;Y_1,Y_4) ] = -\mathbb{E}[\hbarho(Y_1,Y_2;Y_3,Y_4) ] = 0.
    \end{talign}
\end{itemize}
Thus, for all terms in subcase (i), we have that
\begin{talign}
    &\mE\big[\E[\hbarho(X_{\tilde{\pi}_{i_1}},X_{\tilde{\pi}_{i_2}};X_{\tilde{\pi}_{n_1+j_1}},X_{\tilde{\pi}_{n_1+j_2}}) \hbarho(X_{\tilde{\pi}_{i'_1}},X_{\tilde{\pi}_{i'_2}};X_{\tilde{\pi}_{n_1+j'_1}},X_{\tilde{\pi}_{n_1+j'_2}}) | \mbb{X} ] \big] \leq \mE[U_{n_1,n_2}]^2.
\end{talign}
Moreover, the fraction of $\mathsf{I}_{\mathrm{rest}}$ terms in subcase (i) is $m^2(m-1)^2(m-2)^2(m-3)^2 / m^8$.
\paragraph{$\mathsf{I}_{\mathrm{rest}}$ terms, subcase (ii)}
Now we consider indices in $\mathsf{I}_{\mathrm{rest}}$ such that there is exactly one pair among $i_1$, $i_2$, $n_1+j_1$, $n_1+j_2$, $i'_1$, $i'_2$, $n_1+j'_1$, $n_1+j'_2$ that are the same, and the other indices are pairwise different. Without loss of generality, we can assume that $i_1 = i'_1$ (recall that $i_1 \neq i_2$ always). Then,
\begin{talign}
&\mE\big[\E[\hbarho(X_{\tilde{\pi}_{i_1}},X_{\tilde{\pi}_{i_2}};X_{\tilde{\pi}_{n_1+j_1}},X_{\tilde{\pi}_{n_1+j_2}}) \hbarho(X_{\tilde{\pi}_{i_1}},X_{\tilde{\pi}_{i'_2}};X_{\tilde{\pi}_{n_1+j'_1}},X_{\tilde{\pi}_{n_1+j'_2}}) | \mbb{X} ] \big] \\ &= \E\big[\mE[\hbarho(X_{\tilde{\pi}_{i_1}},X_{\tilde{\pi}_{i_2}};X_{\tilde{\pi}_{n_1+j_1}},X_{\tilde{\pi}_{n_1+j_2}}) \hbarho(X_{\tilde{\pi}_{i_1}},X_{\tilde{\pi}_{i'_2}};X_{\tilde{\pi}_{n_1+j'_1}},X_{\tilde{\pi}_{n_1+j'_2}}) | \pi ] \big] \\ &= \E\big[\mE[ \mE[ \hbarho(X_{\tilde{\pi}_{i_1}},X_{\tilde{\pi}_{i_2}};X_{\tilde{\pi}_{n_1+j_1}},X_{\tilde{\pi}_{n_1+j_2}}) | X_{\tilde{\pi}_{i_1}}, \pi ] \\ &\qquad\qquad\qquad \times \mE[ \hbarho(X_{\tilde{\pi}_{i_1}},X_{\tilde{\pi}_{i'_2}};X_{\tilde{\pi}_{n_1+j'_1}},X_{\tilde{\pi}_{n_1+j'_2}}) | X_{\tilde{\pi}_{i_1}}, \pi ] | \pi ] \big].
\end{talign}
Here, $\mE[ \hbarho(X_{\tilde{\pi}_{i_1}},X_{\tilde{\pi}_{i_2}};X_{\tilde{\pi}_{n_1+j_1}},X_{\tilde{\pi}_{n_1+j_2}}) | X_{\tilde{\pi}_{i_1}}, \pi ]$ and $\mE[ \hbarho(X_{\tilde{\pi}_{i_1}},X_{\tilde{\pi}_{i'_2}};X_{\tilde{\pi}_{n_1+j'_1}},X_{\tilde{\pi}_{n_1+j'_2}}) | X_{\tilde{\pi}_{i_1}}, \pi ]$ also play the same role, so we will focus on the latter. We consider the following cases with $(Y,Y',Y'',Z,Z',Z'')\sim\P\times\P\times\P\times\Q\times\Q\times\Q$ given $(X_{\tilde{\pi}_{i_1}},\pi)$:
\begin{itemize}
    \item $\tilde{\pi}_{i_2} \leq n_1$, $\tilde{\pi}_{n_1+j_1} > n_1$, $\tilde{\pi}_{n_1+j_2} > n_1$ or analogous configurations: Then,
    \begin{talign}
        &\mE[ \hbarho(X_{\tilde{\pi}_{i_1}},X_{\tilde{\pi}_{i_2}};X_{\tilde{\pi}_{n_1+j_1}},X_{\tilde{\pi}_{n_1+j_2}}) | X_{\tilde{\pi}_{i_1}}, \pi ] = \mE[\hbarho(X_{\tilde{\pi}_{i_1}},Y;Z,Z') | X_{\tilde{\pi}_{i_1}}, \pi ].
    \end{talign}
    \item $\tilde{\pi}_{i_2} \leq n_1$, $\tilde{\pi}_{n_1+j_1} \leq n_1$, $\tilde{\pi}_{n_1+j_2} > n_1$ or analogous configurations: Then,
    \begin{talign}
        &\mE[ \hbarho(X_{\tilde{\pi}_{i_1}},X_{\tilde{\pi}_{i_2}};X_{\tilde{\pi}_{n_1+j_1}},X_{\tilde{\pi}_{n_1+j_2}}) | X_{\tilde{\pi}_{i_1}}, \pi ] = \mE[\hbarho(X_{\tilde{\pi}_{i_1}},Y;Y',Z') | X_{\tilde{\pi}_{i_1}}, \pi ].
    \end{talign}
    \item $\tilde{\pi}_{i_2} \leq n_1$, $\tilde{\pi}_{n_1+j_1} > n_1$, $\tilde{\pi}_{n_1+j_2} \leq n_1$ or analogous configurations: Then,
    \begin{talign}
        &\mE[ \hbarho(X_{\tilde{\pi}_{i_1}},X_{\tilde{\pi}_{i_2}};X_{\tilde{\pi}_{n_1+j_1}},X_{\tilde{\pi}_{n_1+j_2}}) | X_{\tilde{\pi}_{i_1}}, \pi ] = \mE[\hbarho(X_{\tilde{\pi}_{i_1}},Y;Z',Y') | X_{\tilde{\pi}_{i_1}}, \pi ] \\ &= \mE[\hbarho(X_{\tilde{\pi}_{i_1}},Y';Z',Y) | X_{\tilde{\pi}_{i_1}}, \pi ] = - \mathbb{E} [\hbarho(X_{\tilde{\pi}_{i_1}},Y;Z',Y') | X_{\tilde{\pi}_{i_1}}, \pi ] = 0.
    \end{talign}
    \item $\tilde{\pi}_{i_2} \leq n_1$, $\tilde{\pi}_{n_1+j_1} \leq n_1$, $\tilde{\pi}_{n_1+j_2} \leq n_1$ or analogous configurations: Then,
    \begin{talign}
        &\mE [\hbarho(X_{\tilde{\pi}_{i_1}},X_{\tilde{\pi}_{i_2}};X_{\tilde{\pi}_{n_1+j_1}},X_{\tilde{\pi}_{n_1+j_2}}) | X_{\tilde{\pi}_{i_1}}, \pi ] = \mE[\hbarho(X_{\tilde{\pi}_{i_1}},Y;Y',Y'') | X_{\tilde{\pi}_{i_1}}, \pi ] \\ &= \mE[\hbarho(X_{\tilde{\pi}_{i_1}},Y'';Y',Y) | X_{\tilde{\pi}_{i_1}}, \pi ] = -\mE[\hbarho(X_{\tilde{\pi}_{i_1}},Y;Y',Y'') | X_{\tilde{\pi}_{i_1}}, \pi ] = 0.
    \end{talign}
\end{itemize}
Thus, we obtain that 
\begin{talign}
&\mE\big[\E[\hbarho(X_{\tilde{\pi}_{i_1}},X_{\tilde{\pi}_{i_2}};X_{\tilde{\pi}_{n_1+j_1}},X_{\tilde{\pi}_{n_1+j_2}}) \hbarho(X_{\tilde{\pi}_{i_1}},X_{\tilde{\pi}_{i'_2}};X_{\tilde{\pi}_{n_1+j'_1}},X_{\tilde{\pi}_{n_1+j'_2}}) | \mbb{X} ] \big] \\ 
    &\leq 
\max\big\{ \E\big[\mE[ \E[\hbarho(X_{\tilde{\pi}_{i_1}},Y;Z,Z') | X_{\tilde{\pi}_{i_1}}, \pi ]^2  ] \big],  \E\big[\mE [ \E[\hbarho(X_{\tilde{\pi}_{i_1}},Y;Y',Z') | X_{\tilde{\pi}_{i_1}}, \pi ]^2  ] \big] \big\} \\ 
    &=  %
\max\big\{ \frac{n_1}{n} \E\big[ \E[\hbarho(Y',Y;Z,Z') | Y' ]^2 \big] + \frac{n_2}{n} \E\big[ \E[\hbarho(Z'',Y;Z,Z') | Z'' ]^2  \big], \\ &\qquad\quad \frac{n_1}{n} \E\big[ \E[\hbarho(Y'',Y;Y',Z') | Y'' ]^2 \big]  + \frac{n_2}{n} \E\big[ \E[\hbarho(Z'',Y;Y',Z) | Z' ]^2  \big]\big\}.
\end{talign}
Now, notice that 
\begin{talign}
\mE\big[ \E[\hbarho(Y',Y;Z,Z') | Y' ]^2 \big]
    &= 
\Var\big( \E[\hbarho(Y',Y;Z,Z') | Y' ] \big) + \E[\hbarho(Y',Y;Z,Z') ]^2
    \\ 
    &= 
\Var\big( \E[\hbarho(Y',Y;Z,Z') | Y' ] \big) + \mathbb{E}[U_{n_1,n_2}]^2, \\
\E\big[\E[\hbarho(Z'',Y;Z,Z') | Z'' ]^2  \big] 
    &= \Var\big( \E[\hbarho(Z'',Y;Z,Z') | Z'' ]  \big) + \E[\hbarho(Z'',Y;Z,Z') ]^2
    \\ &= \Var\big( \E[\hbarho(Z'',Y;Z,Z') | Z'' ]  \big),\\
\E\big[ \E[\hbarho(Y'',Y;Y',Z') | Y'' ]^2 \big] 
    &= 
\Var\big( \E[\hbarho(Y'',Y;Y',Z') | Y'' ] \big), \qtext{and} \\
\E\big[ \E[\hbarho(Z'',Y;Y',Z') | Z'' ]^2  \big] 
    &= 
\Var\big( \E[\hbarho(Y',Y;Z'',Z') | Z'' ]  \big) + \mathbb{E}[U_{n_1,n_2}]^2.
\end{talign}
Thus, for $\mathsf{I}_{\mathrm{rest}}$ terms in subcase (ii), 
\begin{talign}
    &\mE\big[\E[\hbarho(X_{\tilde{\pi}_{i_1}},X_{\tilde{\pi}_{i_2}};X_{\tilde{\pi}_{n_1+j_1}},X_{\tilde{\pi}_{n_1+j_2}}) \hbarho(X_{\tilde{\pi}_{i_1}},X_{\tilde{\pi}_{i'_2}};X_{\tilde{\pi}_{n_1+j'_1}},X_{\tilde{\pi}_{n_1+j'_2}}) | \mbb{X} ] \big] \\ &\leq \max\{ \psi_{Y,1}, \psi_{Z,1} \} + \mathbb{E}[U_{n_1,n_2}]^2.
\end{talign}
Moreover, the fraction of $\mathsf{I}_{\mathrm{rest}}$ terms in this subcase is upper-bounded by $12 m^2(m-1)^2(m-2)^2(m-3) / m^8$.
\paragraph{$\mathsf{I}_{\mathrm{rest}}$ terms, subcase (iii)} For the remaining terms in $\mathsf{I}_{\mathrm{rest}}$ we can write
\begin{talign}
    \mE\big[\E[\hbarho(X_{\tilde{\pi}_{i_1}},X_{\tilde{\pi}_{i_2}};X_{\tilde{\pi}_{n_1+j_1}},X_{\tilde{\pi}_{n_1+j_2}}) \hbarho(X_{\tilde{\pi}_{i_1}},X_{\tilde{\pi}_{i'_2}};X_{\tilde{\pi}_{n_1+j'_1}},X_{\tilde{\pi}_{n_1+j'_2}}) | \mbb{X} ] \big] \leq \tilde{\psi}_{YZ,2}.
\end{talign}
Since there need to be at least two pairs of indices that share the same value, we have that the fraction of elements of $\mathsf{I}_{\mathrm{rest}}$ that fall under subcase (iii) is upper-bounded by
\begin{talign}
    \frac{1}{m^8} \big( {4 \choose 3} m^6 + {4 \choose 3} m^6 + {4 \choose 2} \cdot {4 \choose 2} m^6 \big) = \frac{44}{m^2}.
\end{talign}

\paragraph{Total contribution of $\mathsf{I}_{\mathrm{rest}}$ terms} %
We end up with a bound of the form
\begin{talign} 
\begin{split} \label{eq:I_rest_final_bound}
    &\sum_{\mathfrak{i} \in \mathsf{I}_{\text{diff}}} \mE[\E[\hbarho(X_{\tilde{\pi}_{i_1}},X_{\tilde{\pi}_{i_2}};X_{\tilde{\pi}_{n_1+j_1}},X_{\tilde{\pi}_{n_1+j_2}})  \hbarho(X_{\tilde{\pi}_{i'_1}},X_{\tilde{\pi}_{i'_2}};X_{\tilde{\pi}_{n_1+j'_1}},X_{\tilde{\pi}_{n_1+j'_2}}) | \mbb{X} ]] \\ 
    &\leq |\mathsf{I}_{\mathrm{rest}}| \big( \frac{(m-1)^2(m-2)^2(m-3)^2}{m^6} \mathbb{E}[U_{n_1,n_2}]^2 \\ &\qquad\quad + \frac{12(m-1)^2(m-2)^2(m-3)}{m^6}  \big( \max\{ \psi_{Y,1}, \psi_{Z,1} \} + \mathbb{E}[U_{n_1,n_2}]^2 \big) 
    + \frac{44}{m^2} \tilde{\psi}_{YZ,2} \big) \\ &\leq  (4 s_1 s_2^4 + 20 s_1^2 s_2^3 + 20 s_1^3 s_2^2 + 4 s_1^4 s_2) m^8   \big( \mathbb{E}[U_{n_1,n_2}]^2 + \frac{12}{m} \max\{ \psi_{Y,1}, \psi_{Z,1} \} + \frac{44}{m^2} \tilde{\psi}_{YZ,2} \big).
\end{split}
\end{talign}
Combining our $\mathsf{I}_{\mathrm{diff}}$, $\mathsf{I}_{a},$ $\mathsf{I}_{b},$ and $\mathsf{I}_{\mathrm{rest}}$ estimates, we obtain the advertised variance bound, 
\begin{talign}
&\mE[ \Var (U_{n_1,n_2}^{\pi,s} | \mbb X) ]\, n_1^2(n_1 - 1)^2 n_2^2(n_2 - 1)^2\\ 
    &\leq 
\big( n_1(n_1-m) n_2(n_2-m) ( 2 n_2^2 + 2 n_1^2 + 16 n_1 n_2 + 400 (n^2_1 + n^2_2)/s) \tilde{\psi}_{YZ,2} \\ 
    &+ 
(4 n_1 n_2^4 + 20 n_1^2 n_2^3 + 20 n_1^3 n_2^2 + 4 n_1^4 n_2) m^3 \big( \mathbb{E}[U_{n_1,n_2}]^2 + \frac{12}{m} \max\{ \psi_{Y,1}, \psi_{Z,1} \} + \frac{44}{m^2} \tilde{\psi}_{YZ,2} \big) \big)
    \\ 
    &\leq 
\big( (n_1-m) (n_2-m) ( 2 n_2^2 + 2 n_1^2 + 16 n_1 n_2 + 400 (n^2_1 + n^2_2)/s) \tilde{\psi}_{YZ,2} \\ 
    &+ 
(4 n_2^3 + 20 n_1 n_2^2 + 20 n_1^2 n_2 + 4 n_1^3) m^3 \big( \mathbb{E}[U_{n_1,n_2}]^2 + \frac{12}{m} \max\{ \psi_{Y,1}, \psi_{Z,1} \} + \frac{44}{m^2} \tilde{\psi}_{YZ,2} \big) \big)
    \\ 
    &\leq
\big( (n_1-m) (n_2-m) ( 2 n_2^2 + 2 n_1^2 + 16 n_1 n_2 + 400 (n^2_1 + n^2_2)/s) \tilde{\psi}_{YZ,2} \\ &+ 8 n^3 m^3 \big( \mathbb{E}[U_{n_1,n_2}]^2 + \frac{12}{m} \max\{ \psi_{Y,1}, \psi_{Z,1} \} + \frac{44}{m^2} \tilde{\psi}_{YZ,2} \big) \big).
\end{talign}
\section{\pcref{cor:ts_separation_rate}} \label{sec:proof_cor_ts_separation_rate}
We begin with lemma that relates the parameters of \cref{cor:ts_separation_rate} to those of \cref{thm: Full Two Sample Tests}.
\begin{lemma}[Positive-definite  bounds on $\psi_{Y,1}$, $\psi_{Z,1}$] \label{lem:bounds_sigma}
Under the premises of \cref{cor:ts_separation_rate}, 
\begin{talign}
    \psi_{Y,1} \leq \mmd^2(\P,\Q) \xi_{\Q}
    \qtext{and}
    \psi_{Z,1} \leq \mmd^2(\P,\Q) \xi_{\P}.
\end{talign}
\end{lemma}
\begin{proof}
We prove one case as the others are identical.
By Cauchy-Schwarz, 
\begin{talign}
\begin{split}
&|\mE[\hbarho(Y_{1},Y_{2}; Z_{1},Z_{2})|Z_1]| = |\mE[\hho(Y_{1},Y_{2}; Z_{1},Z_{2})|Z_1]| \\ 
    &= 
|\mE[\kernel(Y_{1},Y_{2}) + \kernel(Z_{1},Z_{2}) - \kernel(Y_{1},Z_{2}) - \kernel(Y_{2},Z_{1})|Z_1]| \\ 
    &= 
\big|\int \kernel(x,x') \, d(\P-\Q)(x) \, d(\P-\delta_{Z_1})(x') \big|
    \\ &= \big|\kinner{ \int \kernel(x,\cdot) \, d(\P-\Q)(x)} {\int \kernel(x,\cdot) \, d(\P-\delta_{Z_1})(x)} \big|
    \\ &\leq \knorm{ \int \kernel(x,\cdot) \, d(\P-\Q)(x)} \knorm{ \int \kernel(x,\cdot) \, d(\P-\delta_{Z_1})(x)} \\ &= \mmd(\P,\Q) \mmd(\delta_{Z_1},\P).
\end{split}
\end{talign}
Since $\E[\mE[\hbarho(Y_{1},Y_{2}; Z_{1},Z_{2})|Z_1]] = 0$, squaring and taking expectations yields the claim.
\end{proof}

We first show that  \cref{eq:psd_ts_full} implies  the sufficient power condition \cref{eq:lb_full_2s} of \cref{thm: Full Two Sample Tests}. 
Since
 \begin{talign}
 \label{eq:psi_mmd_inequality1}
    &\sqrt{\frac{3-\beta}{\beta}\big(\frac{4  \psi_{Y,1}}{n_1} + \frac{4  \psi_{Z,1}}{n_2} \big)} 
    \leq \sqrt{\frac{3-\beta}{\beta}\big(\frac{4  \mmd^2(\P,\Q) \xi_\Q}{n_1} + \frac{4  \mmd^2(\P,\Q) \xi_\P}{n_2} \big)} 
 \end{talign}
 by \cref{lem:bounds_sigma} and $\mE[U_{n_1,n_2}] = \mmd^2(\P,\Q)$ 
 by \citep[Lemma~6]{gretton2012akernel}, the following inequality is a sufficient condition for \cref{eq:lb_full_2s} to hold:
\begin{talign} \label{eq:quadratic_ineq_2s}
    &\qquad a x^2 - bx - c \geq 0 
    \qtext{for} x = \mmd(\P, \Q), \\
    &a = 1, \ \  b = 
    \sqrt{\frac{3-\beta}{\beta}\big(\frac{4 \xi_\Q}{n_1} + \frac{4  \xi_\P}{n_2} \big)}, \sstext{and}
    c =
    \sqrt{\frac{1-\astar}{\astar}\frac{(36 n_1^2 + 36 n_2^2 + 198 n_1 n_2) \psi_{YZ,2}}{\beta n_1 (n_1 - 1) n_2 (n_2 - 1)}}.
\end{talign}
By the quadratic formula, this holds whenever 
$x \geq b + \sqrt{c}$,
yielding the first claim.

We next show that \cref{eq:psd_ts_cheap} implies the sufficient power condition  \cref{eq:lb_cheap_2s} of \cref{thm: Full Two Sample Tests}.  %
Let $f=\sqrt{s^3\frac{\beta\astar}{24(1-\astar)}\nratio}$.
Using the equality $\mE[U_{n_1,n_2}] = \mmd^2(\P,\Q)$ and the inequalities 
\begin{talign}
\sqrt{\frac{201s^2}{4n^2} \psi_{YZ,2} 
    + 
{\frac{12 s}{n} \max( \psi_{Y,1}, \psi_{Z,1})}{}}
    \leq
\sqrt{\frac{201s^2}{4n^2} \psi_{YZ,2}} 
    + 
\sqrt{{\frac{12 s}{n} \max( \xi_{\Q}, \xi_{\P})}{}}\mmd(\P,\Q)
\end{talign}
and \cref{eq:psi_mmd_inequality1}
justified by \cref{lem:bounds_sigma} and the triangle inequality, we find that \cref{eq:lb_cheap_2s} holds whenever
\begin{talign}
    a&x^2-bx-c\geq 0
    \qtext{for}
    x = \mmd(\P, \Q), \quad a = 1-1/f, \\
    b &= 
    \sqrt{\frac{3-\beta}{\beta}\big(\frac{4  \xi_\Q}{n_1} + \frac{4 \xi_\P}{n_2} \big)} + \frac{1}{f}\sqrt{\frac{12 s}{n} \max( \xi_\Q, \xi_\P)},
    \qtext{and}\\
    c 
    &=
\sqrt{\frac{1-\astar}{\astar}\frac{(36 n_1^2 + 36 n_2^2 + 198 n_1 n_2) \psi_{YZ,2}}{\beta n_1 (n_1 - 1) n_2 (n_2 - 1)}} 
    + 
\frac{1}{f}\sqrt{\frac{201s^2}{4n^2} \psi_{YZ,2}}.
\end{talign}
Since $a\leq 1$, this holds when $x \geq \frac{b}{a} + \frac{\sqrt{c}}{a}\geq \frac{b}{a} + \sqrt{\frac{c}{a}}$, yielding the second claim.

\section{\pcref{thm:final_2s}}
\label{sec:proof_cor:final_2s}
We will establish the claim using the quantile comparison method (\cref{Lemma: Generic Two Moments Method}) by identifying complementary quantile bounds (CQBs) $\Phi_{n_1,n_2}, \Psi_{\X,s},$ and $\Psi_{n_1,n_2,s}$ satisfying the tail bounds \cref{eq:Phi_P_n,,eq:Psi_X_n,,eq:Psi_P_n} for the test statistic $U_{n_1,n_2}$, auxiliary sequence  $(U_{n_1,n_2}^{\pi_b,s})_{b=1}^\numperm$, and population parameter $\tconst = \mathbb{E}[U_{n_1,n_2}] = \mmd^2(\P,\Q)$.
Throughout, we will call a random variable \emph{sub-Gaussian} if it has finite sub-Gaussian norm
\begin{talign} \label{eq:subG_norm}
\|X\|_{\psi_{2}} \defeq \sup_{p \geq 1} \frac{(\mathbb{E}[|X|^{p}])^{1/p}}{\sqrt{p}}.
\end{talign}
\subsection{Bounding the fluctuations of the test statistic}
We begin by identifying a CQB $\Phi_{n_1,n_2}$ satisfying the test statistic tail bound \cref{eq:Phi_P_n}. 
We first introduce another way to express the U-statistic $U_{n_1,n_2}$ and its corresponding V-statistic $V_{n_1,n_2}$. If we let $\delta_i = 1$ for $i\in [n_1]$ and $\delta_i = -1$ for $i \in n_1 + [n_2]$ and $K_{ij} = g(X_i,X_j)$, we can write
\begin{talign}
\label{eq:V_n_1_n_2_K}
V_{n_1,n_2} 
    &\defeq
\mmd^2(\PY,\PZ)
    =
\sum_{i,j=1}^{n} \big( \frac{\mathrm{I}[\delta_i = 1, \delta_j = 1]}{n_1^2} + \frac{\mathrm{I}[\delta_i = -1, \delta_j = -1]}{n_2^2} - \frac{\mathrm{I}[\delta_i \delta_j = -1]}{n_1 n_2} \big) K_{ij}, \\
U_{n_1,n_2} 
    &= 
\sum_{i,j=1}^{n} \big( \frac{\mathrm{I}[\delta_i = 1, \delta_j = 1, i \neq j]}{n_1(n_1-1)} + \frac{\mathrm{I}[\delta_i = -1, \delta_j = -1, i \neq j]}{n_2(n_2-1)} - \frac{\mathrm{I}[\delta_i \delta_j = -1]}{n_1 n_2} \big) K_{ij}.
\label{eq:U_n_1_n_2_K}
\end{talign}
Hence,
\begin{talign}
    V_{n_1,n_2} \! - \! U_{n_1,n_2} 
    &= \sum_{i,j=1}^{n} \big( \frac{\mathrm{I}[\delta_i = 1, \delta_j = 1]}{n_1^2} \! + \!  \frac{\mathrm{I}[\delta_i = -1, \delta_j = -1]}{n_2^2} \! - \! \frac{\mathrm{I}[\delta_i = 1, \delta_j = 1, i \neq j]}{n_1(n_1-1)} \! - \! \frac{\mathrm{I}[\delta_i = -1, \delta_j = -1, i \neq j]}{n_2(n_2-1)} \big) K_{ij} \\ &= \sum_{i=1}^{n} (\frac{\mathrm{I}[\delta_i = 1]}{n_1^2} \! + \! \frac{\mathrm{I}[\delta_i = -1]}{n_2^2}) K_{ii} \! - \! \sum_{i,j=1,i\neq j}^{n} \big( \frac{\mathrm{I}[\delta_i = 1, \delta_j = 1]}{n_1^2 (n_1 - 1)} \! + \! \frac{\mathrm{I}[\delta_i = -1, \delta_j = -1]}{n_2^2 (n_2 -1)} \big) K_{ij}
    \label{eq:V_U_diff}
\end{talign}
where we used that $\frac{1}{n_1^2} - \frac{1}{n_1 (n_1 - 1)} %
= -\frac{1}{n_1^2 (n_1 - 1)}$.

Our proof makes use of two lemmas.  The first, essentially due to \citet{maurer2021concentration}, shows that functions of independent variables concentrate around their mean when coordinatewise function differences are sub-Gaussian.

\begin{lemma}[Sub-Gaussian differences inequality {\citep[Sec.~5]{maurer2021concentration}}] \label{thm:mcdiarmid_subG}
Suppose an integrable function $f$ of a sequence of independent datapoints $\X=(X_i)_{i=1}^n$ satisfies %
\begin{talign}
|f(\X)-\E[f(\X)\mid \X_{-k}]|
    \leq
f_{k}(X_k)
\qtext{almost surely}
\end{talign}
for each $k\in[n]$ and measurable functions $(f_k)_{k=1}^n$. 
If each $f_k(X_k)$ is sub-Gaussian, then
\begin{talign}
f(\X) - \mathbb{E}[f(\X)] 
    <
\sqrt{64 e\log(1/\delta) \sum _{k\in [n]}\|f_{k}(X_k)\|_{\psi _{2}}^{2}}
\qtext{with probability at least}
1-\delta.
\end{talign}
\end{lemma}
\begin{proof}
The proof of Thm.~3 in Sec.~5 of \citep{maurer2021concentration} shows that
\begin{talign}
    \label{eq:maurer-inequality}
    \mathrm{Pr}\left(f(\X) - \mathbb{E}[f(\X)] \geq \varepsilon \right) \leq \exp ( 16 e \beta^2 \sum _{k\in [n]}\|f_{k}(X_k)\|_{\psi _{2}}^{2}  - \beta \varepsilon )
    \sstext{for all}
    \varepsilon,\beta > 0.
\end{talign}
Minimizing over $\beta > 0$ yields the result.\footnote{The statement of \citep[Thm.~3]{maurer2021concentration} contains the apparent typo $32$ in place of $64$, presumably due to incorrect minimization of \cref{eq:maurer-inequality}.}
\end{proof}
The second lemma, proved in \cref{proof:MMD_concentration}, uses \cref{thm:mcdiarmid_subG} to 
derive suitable CQBs for $U_{n_1,n_2}$ and $V_{n_1,n_2}$. 

\begin{lemma}[Sub-Gaussian homogeneity quantile bounds]\label{lem:MMD_concentration}
Under the assumptions of \cref{thm:final_2s}, 
for $V_{n_1,n_2}\defeq\mmd^2(\PY, \PZ)$ and each $\beta\in(0,1)$, 
\begin{talign}
\Pr \big(|\mmd^2&(\P, \Q) 
        - V_{n_1,n_2}| 
    \geq
\Lambda_{V_{n_1,n_2}}^2(\beta) + 2 \mmd(\P, \Q) \Lambda_{V_{n_1,n_2}}(\beta) \big) \leq \beta
\quad\text{and} \\
\Pr \big(|\mmd^2&(\P, \Q) 
        - U_{n_1,n_2}| \geq
\Phi_{n_1,n_2}(\beta) %
\big) 
\leq \beta
\quad\text{where} \\
\label{eq:lambda_Vn_def}
\Lambda_{V_{n_1,n_2}}(\beta) 
    &\defeq \sqrt{\frac{\mom[2,\P]}{n_1}} + \sqrt{\frac{\mom[2,\Q]}{n_2}}    + 
\sqrt{64 e\log(\frac{1}{\beta}) \big( \frac{\subg^2 }{n_1} + \frac{\subg[\Q]^2 }{n_2} \big)} \\
\label{eq:lambda_n_def}
\Lambda_{U_{n_1,n_2}}(\beta) 
    &\defeq
    \sqrt{ \frac{2\,\mom[2,\P]}{n_1 - 1}} + \sqrt{\frac{2\,\mom[2,\Q]}{n_2 - 1} } + \sqrt{128 e\log(\frac{2}{\beta}) (\frac{(n_1 + 1)\subg^2}{n_1^2}+\frac{(n_2 + 1)\subg[\Q]^2}{n_2^2})},
\quad\text{and}\\
\label{eq:Phi_n1_n2}
\Phi_{n_1,n_2}(\beta) 
&\defeq \Lambda_{U_{n_1,n_2}}^2(\beta) + 2 \mmd(\P, \Q) \Lambda_{U_{n_1,n_2}}(\beta).
\end{talign}
\end{lemma}

\subsection{Bounding the fluctuations of the threshold}
We next identify CQBs $\Psi_{\X,s}$ and $\Psi_{n_1,n_2,s}$ that satisfy the properties
\begin{talign}
    &\Pr \left( U^{\pi,s}_{n_1,n_2} %
        \geq \Psi_{\X,s}(\alpha) \mid \X \right) \leq \alpha, \qquad \forall \alpha \in (0,1),
        \qtext{and} \label{eq:Psi_X_n_ts} \\
        &\Pr ( %
        \Psi_{\X,s}(\alpha) \geq \Psi_{n_1,n_2,s}(\alpha,\beta)) \leq \beta, \qquad \forall \alpha, \beta \in (0,1) \label{eq:Psi_P_n_ts}
\end{talign}
for the cheaply permuted U-statistic $U^{\pi,s}_{n_1,n_2}$ \cref{eq: Two Sample Cheap U_statistic}.
Our derivation makes use of three lemmas.  
The first is tail bound for random quadratic forms known as the Hanson-Wright inequality.

\begin{lemma}[Hanson-Wright inequality {\citep[Thm.~3.1]{dadush2018fast}}] \label{lem:hanson_wright}
Let $X \in \R^n$ be a random vector with %
independent components satisfying $\mathbb{E}[e^{\lambda X_i}] \leq e^{\lambda^2 \nu^2/2}$
for some $\nu\geq0$, all $\lambda \in \R$, and all $i\in[n]$.
Let $A \in \mathbb{R}^{n \times n}$ be a symmetric matrix such that $a_{ii} = 0$ for $i \in [n]$. Then, there exists a universal constant $\mathfrak{c} > 3/20$ such that for any $t > 0$, 
\begin{talign} \label{eq:hanson_2}
    \mathrm{Pr}(X^{\top} A X > t) \leq \exp(- \mathfrak{c} \min\{ \frac{t^2}{\nu^4 \fronorm{A}^2}, \frac{t}{\nu^2 \opnorm{A}} \}).
\end{talign}
Equivalently, for any $\delta\in(0,1)$, with probability at least $1-\delta$, we have 
\begin{talign} \label{eq:hanson_inverted}
    X^{\top} A X \leq \max \big\{ \nu^2 \fronorm{A} \sqrt{\frac{ \log(1/\delta)}{\mathfrak{c}}}, \frac{ \log(1/\delta) \nu^2 \opnorm{A}}{\mathfrak{c}} \big\}.
\end{talign}
Moreover, when $\delta \in(0,0.86)$, we have 
\begin{talign} \label{eq:hanson_inverted_2}
X^{\top} A X 
    \leq 
\frac{\log(1/\delta) \nu^2 \fronorm{A}}{\mathfrak{c}}
    \leq
\frac{20 \log(1/\delta) \nu^2 \fronorm{A}}{3}.
\end{talign}
\end{lemma}
\begin{proof}
Equation \cref{eq:hanson_2} is the content of  \citep[Thm.~3.1]{dadush2018fast}, and equivalence with \cref{eq:hanson_inverted} follows from selecting  $t = \max \big\{ \sqrt{a \log(1/\delta)}, b \log(1/\delta) \big\}$ in \cref{eq:hanson_2} or  $\delta = \exp(- \min\{ \frac{\mathfrak{c} t^2}{\nu^4 \fronorm{A}^2}, \frac{\mathfrak{c} t}{\nu^2 \opnorm{A}} \})$ in \cref{eq:hanson_inverted}. 
Now, suppose that $\delta \in (0,0.86)$.  Since $0.86 \leq \exp(-3/20)\leq \exp(-\mathfrak{c})$, we have $\sqrt{\frac{ \log(1/\delta)}{\mathfrak{c}}} \leq \frac{\log(1/\delta)}{\mathfrak{c}}$.
Since $\opnorm{A} \leq \fronorm{A}$, the claim \cref{eq:hanson_inverted_2} follows.
\end{proof}

Our second lemma, proved in \cref{sec:proof_of_lem:concentration_simultaneous}, provides high-probability bounds on the sums of MMD powers.
\begin{lemma}[High-probability bounds on sums of MMD powers] 
\label{lem:concentration_simultaneous}
Under the assumptions of \cref{thm:final_2s}, if $p\in\{1,2,4\}$, then
    \begin{talign}
        Z_p(\X) &\defeq \big(\sum_{i=1}^{s_1} \mmd^p(\delta_{\Y_i},\P) + \sum_{i=1}^{s_2} \mmd^p(\delta_{\Z_i},\Q)\big)^{1/p} \\
        &< \mathcal{A}_p(s_1,s_2,m) + \mathcal{B}(s_1,s_2,m,\delta), \label{eq:bound_Z_p} %
    \end{talign}
    with probability at least $1-\delta$ for %
    \begin{talign}
        \mathcal{B}(s_1,s_2,m,\delta) &\defeq 
        \sqrt{64 e \log(1/\delta) \big( \frac{n_1 \subg^2}{m^2} + \frac{n_2 \subg[\Q]^2}{m^2} \big)}, \\
        \mathcal{A}_1(s_1,s_2,m) &\defeq s_1 \sqrt{\frac{\mathbb{E}_{Y \sim \P}[\mmd^2(\delta_{Y},\P)]}{m}}
        + s_2 \sqrt{\frac{\mathbb{E}_{Z \sim \Q}[\mmd^2(\delta_{Z},\Q)]}{m}}, \\
        \mathcal{A}_2(s_1,s_2,m) &\defeq \sqrt{ 
        \frac{s_1 \mathbb{E}_{Y \sim \P}[\mmd^2(\delta_{Y},\P)]}{m}
        + \frac{s_2 \mathbb{E}_{Z \sim \Q}[\mmd^2(\delta_{Z},\Q)]}{m} }, 
        \qtext{and} \\
        \mathcal{A}_4(s_1,s_2,m) &\defeq \big( \big( \frac{3}{m^2} + \frac{1}{m^3} \big) \big( s_1 \mathbb{E}_{Y \sim \P} [ \mmd^4(\delta_{Y},\P) ] + s_2 
        \mathbb{E}_{Z \sim \Q} [ \mmd^4(\delta_{Z},\Q) ] \big) \big)^{1/4}.
    \end{talign}
\end{lemma}

Our third lemma, \cref{Theorem: homogeneity Concentration Cheap}, proved in \cref{sec:proof_of_Theorem: homogeneity Concentration Cheap}, uses \cref{lem:hanson_wright,lem:concentration_simultaneous} to 
verify the CQB properties \cref{eq:Psi_X_n_ts,eq:Psi_P_n_ts} for a particular pairing of $(\Psi_{\X,s},\Psi_{n_1,n_2,s})$.
\begin{lemma}[Sub-Gaussian quantile bound for permutation threshold] %
\label{Theorem: homogeneity Concentration Cheap}
Instantiate the assumptions and notation of \cref{lem:concentration_simultaneous}.
For all $\alpha, \beta \in (0,1)$, %
    \begin{talign}
    &\Psi_{n_1,n_2,s}(\alpha,\beta) \defeq 
    \frac{20 \log(2/\alpha)}{3 s_1^{3/2}} \big( \sqrt{12} \mmd(\P,\Q) (\mathcal{A}_2(s_1,s_2,m) + \mathcal{B}(s_1,s_2,m,\beta/3)) \\ &\qquad\qquad + \sqrt{24} (\mathcal{A}_4(s_1,s_2,m) + \mathcal{B}(s_1,s_2,m,\beta/3))^2 \big)  \\ &\qquad + \frac{2}{s_1^2} \big( 
    (\mathcal{A}_2(s_1,s_2,m) + \mathcal{B}(s_1,s_2,m,\beta/3))^2 \\ &\qquad\qquad 
    + \mmd(\P,\Q) 
    (\mathcal{A}_1(s_1,s_2,m) + \mathcal{B}(s_1,s_2,m,\beta/3)) \big)
    \\ &\qquad 
    + \frac{\mmd^2(\P,\Q) (2 \log(4/\alpha) + 1)}{s_1}.
    \end{talign}
 satisfies $ \Pr (\Psi_{\X,s}(\alpha) \geq \Psi_{n_1,n_2,s}(\alpha,\beta)) \leq \beta$ for $\Psi_{\X,s}(\alpha)$ defined in equation \cref{eq:psi_x_s_def} and satisfying $\Pr \left( U^{\pi,s}_{n_1,n_2} \geq \Psi_{\X,s}(\alpha) \mid \X \right) \leq \alpha$.
 Here, $m \defeq \frac{n_1+n_2}{s}$, $s_1 \defeq \frac{sn_1}{n_1 + n_2}$, and $s_2 \defeq s-s_1$.
\end{lemma}

In the sequel, we will make use of a simplified version of $\Psi_{n_1,n_2,s}$ that, as an immediate corollary to \cref{Theorem: homogeneity Concentration Cheap}, also satisfies the CQB property \cref{eq:Psi_P_n_ts}.
\begin{corollary}[Simplified Sub-Gaussian quantile bound for permutation threshold] %
\label{Corollary: homogeneity Concentration Cheap}
Instantiate the assumptions and notation of \cref{thm:final_2s}. 
For all $\alpha, \beta \in (0,1)$,   
    \begin{talign}
        \Psi_{n_1,n_2,s}(\alpha,\beta) &\defeq \frac{20 \log(2/\alpha)}{3s_1} \big( \frac{ %
        \sqrt{12} \mmd(\P,\Q) (\tilde{\mathcal{A}}_2 + \tilde{\mathcal{B}}(\frac{\beta}{3}))}{\sqrt{m}} + \frac{ %
        \sqrt{24} (\tilde{\mathcal{A}}_4 + \tilde{\mathcal{B}}(\frac{\beta}{3})^2)}{m} \big)
        \\ &\qquad + \frac{2}{s_1} \big( \frac{2(\tilde{\mathcal{A}}_2^2 + \tilde{\mathcal{B}}(\frac{\beta}{3})^2)}{m} + \frac{\mmd(\P,\Q) (\sqrt{2} \tilde{\mathcal{A}}_2 + \tilde{\mathcal{B}}(\frac{\beta}{3}))}{\sqrt{m s_1}} \big)
        \\ &\qquad 
        + \frac{\mmd^2(\P,\Q) (2 \log(4/\alpha) + 1)}{s_1}.
    \end{talign}
 satisfies $ \Pr (\Psi_{\X,s}(\alpha) \geq \Psi_{n_1,n_2,s}(\alpha,\beta)) \leq \beta$ 
 for $m$, $s_1$, and $\Psi_{\X,s}$ as in \cref{Theorem: homogeneity Concentration Cheap}.
\end{corollary}

\subsection{Putting everything together}

Fix any $\beta\in(0,1)$, and suppose that 
\begin{talign} \label{eq:separation_mmd_subg}
    \mE[U_{n_1,n_2}] &\geq \Psi_{n_1,n_2,s}(\astar,\beta/3) + \Phi_{n_1,n_2}(\beta/3) 
    \end{talign}
for $\Psi_{n_1,n_2,s}$ and $\Phi_{n_1,n_2}$ defined in \cref{Corollary: homogeneity Concentration Cheap,lem:MMD_concentration}. 
By \cref{Corollary: homogeneity Concentration Cheap}, \cref{lem:MMD_concentration}, and \cref{Lemma: Generic Two Moments Method}, a cheap permutation test based on $U_{n_1,n_2}$ has power at least $1-\beta$. 
We now show that the assumed condition \cref{eq:Power of cheap homogeneity sub-Gaussian PD} is sufficient to imply \cref{eq:separation_mmd_subg}.
Since $\E[U_{n_1,n_2}]=\mmd^2(\P,\Q)$, the condition \cref{eq:separation_mmd_subg} is equivalent to the inequality $a x^2 - bx - c \geq 0$ with $x = \mmd(\P,\Q)$,
    \begin{talign}
        a &= 1 - \frac{2 \log(4/\astar) + 1}{s_1},
        \\ b &= 
        \frac{40\sqrt{\log(2/\astar)} (\tilde{\mathcal{A}}_2 + \tilde{\mathcal{B}}(\frac{\beta}{3}))}{\sqrt{3 s_1 n_1}} \! + \! \frac{2 (\sqrt{2} \tilde{\mathcal{A}}_2 + \tilde{\mathcal{B}}(\frac{\beta}{3}))}{s_1 \sqrt{n_1}} \! + \! \frac{2\tilde{\Lambda}_{\P}(\beta)}{\sqrt{n_1-1}} 
        \! + \! \frac{2\tilde{\Lambda}_{\Q}(\beta)}{\sqrt{n_2-1}}, 
        \qtext{and}
        \\
        c &=
        \frac{40\sqrt{2 \log(2/\astar)/3} (\tilde{\mathcal{A}}_4 + \tilde{\mathcal{B}}(\frac{\beta}{3})^2) + 4 (\tilde{\mathcal{A}}_2^2 + \tilde{\mathcal{B}}(\frac{\beta}{3})^2)}{n_1} + \bigg( \frac{\tilde{\Lambda}_{\P}(\beta)}{\sqrt{n_1-1}} 
        + \frac{\tilde{\Lambda}_{\Q}(\beta)}{\sqrt{n_2-1}} 
        \bigg)^2.
    \end{talign}
By the triangle inequality and the fact  $\sqrt{a}\geq a$, 
it therefore suffices to have 
\begin{talign}
x 
    \geq 
\frac{1}{a} (b + \sqrt{c})
    \geq
\frac{b}{a} + \sqrt{\frac{c}{a}}
    \geq
\frac{b + \sqrt{b^2 + 4 a c}}{2a}.
\end{talign}
The advertised result now follows from the observation that
    \begin{talign}
        \sqrt{c} \leq \sqrt{\frac{40\sqrt{2 \log(2/\astar)/3} (\tilde{\mathcal{A}}_4 + \tilde{\mathcal{B}}(\frac{\beta}{3})^2) + 4 (\tilde{\mathcal{A}}_2^2 + \tilde{\mathcal{B}}(\frac{\beta}{3})^2)}{n_1}} + \frac{\tilde{\Lambda}_{\P}(\beta)}{\sqrt{n_1-1}} 
        + \frac{\tilde{\Lambda}_{\Q}(\beta)}{\sqrt{n_2-1}}. 
    \end{talign}
\subsection{\pcref{lem:MMD_concentration}}\label{proof:MMD_concentration}
    Fix any $\beta\in(0,1)$.
    Let $\X$ be the concatenation of $\Y$ and $\Z$, and let $\tilde{Y}$ and $\tilde{Z}$ be independent draws from $\P$ and $\Q$ respectively. 
    By the triangle inequality, we have the following almost sure inequalities for each $k\in[n]$:
\begin{talign}
|\mmd(\P, \Q) - \mmd(\PY, \PZ)| 
    &\leq
\mmd(\P, \PY) + \mmd(\Q, \PZ)
    \defeq 
\Delta(\X)  \\ %
|\Delta(\X)
    -
\E[\Delta(\X)\mid \X_{-k}]|
    &\leq
\begin{cases}
\frac{1}{n_1}\E[\mmd(\delta_{Y_k},\delta_{\tilde{Y}})\mid Y_k] &
    \text{if\ } k \leq n_1 \\
\frac{1}{n_2}\E[\mmd(\delta_{Z_{k-n_1}},\delta_{\tilde{Z}})\mid Z_{k-n_1}] &
    \text{otherwise}.
\end{cases}
\end{talign}
Furthermore, by Jensen's inequality, we have
\begin{talign}
\E[\Delta(\X)]
    &\leq
\sqrt{\E[\mmd(\P, \PY)^2]} + \sqrt{\E[\mmd(\Q, \PZ)^2]} \\
    &= 
\sqrt{\E[((\PY-\P)\times(\PY-\P))\kernel]} + \sqrt{\E[((\PZ-\Q)\times(\PZ-\Q))\kernel]} \\ 
    &= 
\sqrt{\E[ \frac{1}{n_1^2} \sum_{i=1}^{n_1} ((\delta_{Y_i}-\P)\times (\delta_{Y_i}-\P))\kernel]}
    +
\sqrt{\E[\frac{1}{n_2^2} \sum_{i=1}^{n_2} ((\delta_{Z_i}-\Q)\times (\delta_{Z_i}-\Q))\kernel
]} \\
    &= 
\sqrt{\frac{\mom[2,\P]}{n_1}}+\sqrt{\frac{\mom[2,\Q]}{n_2}}.
\end{talign}
Hence, by \cref{thm:mcdiarmid_subG}, with probability at least $1-\beta$,
\begin{talign}
|\mmd(\P, \Q) - \mmd(\PY, \PZ)| 
    <
\sqrt{\frac{\mom[2,\P]}{n_1}}+\sqrt{\frac{\mom[2,\Q]}{n_2}}
    +
\sqrt{64 e\log(\frac{1}{\beta}) (\frac{\subg^2}{n_1}+\frac{\subg[\Q]^2}{n_2})}
    =
\Lambda_{V_{n_1,n_2}}(\beta).
\end{talign}
Since $|a^2-b^2|=|a-b| |b+a| \leq |a-b| (|b-a|+2|a|)$ for all real $a,b$, we further have 
\begin{talign}
    |
    \mmd^2(\P, \Q) - \mmd^2(\PY, \PZ)| 
    <
    \Lambda_{V_{n_1,n_2}}^2(\beta) + 2 \mmd(\P, \Q) \Lambda_{V_{n_1,n_2}}(\beta)
    \label{eq:mmd_sqd_minus_V}
\end{talign}
with probability at least $1-\beta$, as advertised.

    To establish the $U_{n_1,n_2}$ bound, we begin by writing
    \begin{talign} \label{eq:decomposition_U_bound}
        &\mmd^2(\P, \Q) - U_{n_1,n_2} 
        = \mmd^2(\P, \Q) - V_{n_1,n_2} + V_{n_1,n_2}  - U_{n_1,n_2}. \!\!
    \end{talign}
    To bound the term  $V_{n_1,n_2} - U_{n_1,n_2}$, 
    observe that, by \cref{eq:V_U_diff},
    \begin{talign}
        V_{n_1,n_2} - U_{n_1,n_2} &= \frac{1}{n_1^2}\sum_{i=1}^{n_1} \mmd^2(\delta_{Y_i},\delta_{\Y_{-i}}) + \frac{1}{n_2^2}\sum_{i=1}^{n_2} \mmd^2(\delta_{Z_i},\delta_{\Z_{-i}})  \geq 0.
    \end{talign}
    Hence, we can write
    \begin{talign}
    \tilde{\Delta}(\X)
    		=
     \sqrt{V_{n_1,n_2} - U_{n_1,n_2}} = \big\| \big( \frac{1}{n_1}(\mmd(\delta_{Y_i},\delta_{\Y_{-i}}))_{i=1}^{n_1}, \frac{1}{n_2}(\mmd(\delta_{Z_i},\delta_{\Z_{-i}}))_{i=1}^{n_2} \big) \big\|_{2}.
    \end{talign}
By repeated use of the triangle inequality, we have, for each $k \in [n_1]$,   
    \begin{talign} \label{eq:Delta_k_bound}
        &\big|\tilde{\Delta}(\X) - \mathbb{E}[\tilde{\Delta}(\X)|\Y_{-k},\Z] \big|
        \leq \frac{1}{n_1} \E \big[ \| \big( \big( \frac{1}{n_1} \mmd(\delta_{Y_k},\delta_{\tilde{Y}}) \big)_{i=1, i \neq k}^{n_1}, \mmd(\delta_{Y_k},\delta_{\tilde{Y}}) \big) \|_2 | Y_k \big] \\ 
        &= \frac{1}{n_1} \sqrt{1+\frac{n_1-1}{n_1^2}}\E \big[ \mmd(\delta_{Y_k},\delta_{\tilde{Y}}) | Y_k \big] %
        \leq %
       \frac{1}{n_1} \sqrt{1+\frac{1}{n_1}} \E \big[ \mmd(\delta_{Y_k},\delta_{\tilde{Y}}) | Y_k \big]
    \end{talign}
    almost surely.
     Identical reasoning yields the following almost sure bound for each $k \in [n_2]$: 
    \begin{talign}
    \big|\tilde{\Delta}(\X) - \mathbb{E}[\tilde{\Delta}(\X)|\Y,\Z_{-k}] \big| \leq \frac{1}{n_2} \sqrt{1+\frac{1}{n_2}} \E \big[ \mmd(\delta_{Z_k},\delta_{\tilde{Z}}) | Z_k \big].
    \end{talign}
    
    Next, by the Cauchy-Schwarz inequality, we find that
    \begin{talign}
        &\mathbb{E}[\tilde{\Delta}(\X)]^2 \leq \mathbb{E}[\tilde{\Delta}(\X)^2] = \E\big[\frac{1}{n_1^2}\sum_{i=1}^{n_1} \mmd^2(\delta_{Y_i},\delta_{\Y_{-i}}) + \frac{1}{n_2^2}\sum_{i=1}^{n_2} \mmd^2(\delta_{Z_i},\delta_{\Z_{-i}}) \big],
    \end{talign}
	where
    \begin{talign}
        \E &\big[ \mmd^2(\delta_{Y_i},\delta_{\Y_{-i}}) \big]
        = \E\big[\iint k(x,y) \, \mathrm{d}(\delta_{Y_i} - \P + \P - \delta_{\Y_{-i}})(x) \, \mathrm{d}(\delta_{Y_i} - \P + \P - \delta_{\Y_{-i}})(y)\big]
        \\ &= \E \big[(1+\frac{1}{n_1-1}) \iint k(x,y) \, \mathrm{d}(\delta_{Y_i} - \P)(x) \, \mathrm{d}(\delta_{Y_i} - \P)(y)\big]
       = \big(1+\frac{1}{n_1-1}\big) \mom[2,\P].
    \end{talign}
    Analogously, $\E \big[ \mmd^2(\delta_{Z_i},\delta_{\Z_{-i}}) \big] = \big(1+\frac{1}{n_2-1}\big) \mom[2,\Q] $.
    Thus, we conclude that
    \begin{talign}
        \mathbb{E}[\tilde{\Delta}(\X)] \leq \sqrt{ \frac{\mom[2,\P]}{n_1 - 1} + \frac{\mom[2,\Q]}{n_2 - 1} }.
    \end{talign}
    Hence, applying \cref{thm:mcdiarmid_subG}, we obtain that with probability at least $1-\beta/2$, 
    \begin{talign}
    \sqrt{V_{n_1,n_2} - U_{n_1,n_2}} 
        	<  
	\sqrt{ \frac{\mom[2,\P]}{n_1 - 1} + \frac{\mom[2,\Q]}{n_2 - 1} }
		+
    \sqrt{64 e\log(\frac{2}{\beta}) (\frac{(n_1 + 1)\subg^2}{n_1^2}+\frac{(n_2 + 1)\subg[\Q]^2}{n_2^2})}.
    \end{talign}
    This estimate together with the union bound and the relations \cref{eq:mmd_sqd_minus_V,eq:decomposition_U_bound} yield the claim.

\subsection{\pcref{lem:concentration_simultaneous}}
\label{sec:proof_of_lem:concentration_simultaneous}
Define 
    \begin{talign}
        \phi(y) \defeq \mathbb{E}_{Y' \sim \P}[\mmd(\delta_{y}, \delta_{Y'})] \qtext{and} \phi'(z) \defeq \mathbb{E}_{Z' \sim \Q}[\mmd(\delta_{z}, \delta_{Z'})],
    \end{talign}
fix any $p\in\{1,2,4\}$, 
and note that we may write %
    \begin{talign} 
    \label{eq:Z_2_char}
        &Z_p(\X) =
        \big\| \big( (\mmd(\delta_{\Y_i},\P))_{i=1}^{s_1}, (\mmd(\delta_{\Z_i},\Q))_{i=1}^{s_2} \big) \big\|_{p}, 
        \qtext{and} \\
&|Z_p(\X)-\E[Z_p(\X)\mid \X_{-k}]|
    \leq
\begin{cases}
\frac{1}{m}\phi(Y_k) &
    \text{if\ } k \leq n_1 \\
\frac{1}{m}\phi'(Z_{k-n_1})&
    \text{otherwise}
\end{cases}
\end{talign}
almost surely by the triangle inequality. 
    Applying \cref{thm:mcdiarmid_subG} 
    we therefore have
    \begin{talign}
    Z_p(\X) - \mathbb{E}[Z_p(\X)]
     <
\mathcal{B}(s_1,s_2,m,\delta)
\qtext{with probability at least}
1-\delta.
    \end{talign}
    Hence, it remains to upper-bound $\mathbb{E}[Z_p(\X)]$. 
    Applying Hölder's inequality, we find that
    \begin{talign} 
    \label{eq:exp_Z_1}
        \mathbb{E}[Z_1(\X)] &= \sum_{i=1}^{s_1} \mathbb{E}[\mmd(\delta_{\Y_i},\P)] + \sum_{i=1}^{s_2} \mathbb{E}[\mmd(\delta_{\Z_i},\Q)]
        \\ &= \sum_{i=1}^{s_1} \sqrt{\mathbb{E}[\mmd^2(\delta_{\Y_i},\P)]} + \sum_{i=1}^{s_2} \sqrt{\mathbb{E}[\mmd^2(\delta_{\Z_i},\Q)]}
        =
         \mathcal{A}_1(s_1,s_2,m), \\
    \label{eq:exp_Z_2}
        \mathbb{E}[Z_2(\X)]^2 &\leq \mathbb{E}[Z_2(\X)^2]  = \sum_{i=1}^{s_1} \mathbb{E}[\mmd^2(\delta_{\Y_i},\P)] + \sum_{i=1}^{s_2} \mathbb{E}[\mmd^2(\delta_{\Z_i},\Q)] \\ 
        &=  \mathcal{A}_2(s_1,s_2,m), \qtext{and}\\
        \mathbb{E}[Z_4(\X)]^4 &\leq \mathbb{E}[Z_4(\X)^4] = \sum_{i=1}^{s_1} \mathbb{E}[\mmd^4(\delta_{\Y_i},\P)] + \sum_{i=1}^{s_2} \mathbb{E}[\mmd^4(\delta_{\Z_i},\Q)] %
    \end{talign}
    Now consider the shorthand $(\delta_{\Y_i} - \P)(\delta_{\Y_i} - \P)\kernel \defeq \iint k(x,y) \, d(\delta_{\Y_i} - \P)(x) \, d(\delta_{\Y_i} - \P)(y)$. 
    To conclude our fourth moment bound, we notice that %
    \begin{talign}
        &\mathbb{E}[\mmd^4(\delta_{\Y_i},\P)] \\ &= \E\big[ \big((\delta_{\Y_i} - \P)(\delta_{\Y_i} - \P)\kernel \big)^2 \big] \\ &= \E \big[ \frac{1}{m^4}\sum_{k,l,k',l'=1}^{m} \big( (\delta_{Y_{mi + k}} - \P)(\delta_{Y_{mi + l}} - \P)\kernel \big) \big( (\delta_{Y_{mi + k'}} - \P)(\delta_{Y_{mi + l'}} - \P)\kernel \big)
        \big] \\ 
        &= \E \big[ \frac{1}{m^4} \big(\sum_{k=1}^{m} (\delta_{Y_{mi + k}} - \P)(\delta_{Y_{mi + k}} - \P)\kernel \big) \big( \sum_{l=1}^{m} (\delta_{Y_{mi + l}} - \P)(\delta_{Y_{mi + l}} - \P)\kernel \big) \\ &\quad + \frac{2}{m^4} \sum_{k,l=1}^{m} \big( (\delta_{Y_{mi + k}} - \P)(\delta_{Y_{mi + l}} - \P)\kernel \big)^2 + \frac{1}{m^4} \sum_{k=1}^{m} \big( (\delta_{Y_{mi + k}} - \P)(\delta_{Y_{mi + k}} - \P)\kernel \big)^2 \big] 
        \\ &\leq  \mathbb{E}_{Y \sim \P} \big[ \big( \frac{1}{m^2} + \frac{2}{m^2} + \frac{1}{m^3} \big) \mmd^4(\delta_{Y},\P) \big] = \mathbb{E}_{Y \sim \P} \big[ \big( \frac{3}{m^2} + \frac{1}{m^3} \big) \mmd^4(\delta_{Y},\P) \big]
    \end{talign}
    and the analogous result for %
    $\mathbb{E}[\mmd^4(\delta_{\Z_i},\P)]$.
    This completes the proof.

\subsection{\pcref{Theorem: homogeneity Concentration Cheap}}
\label{sec:proof_of_Theorem: homogeneity Concentration Cheap}
Fix any $\alpha,\beta\in(0,1)$. 
 We begin by identifying  $\Psi_{\X,s}(\alpha)$ satisfying
 \begin{talign}
 \Pr \left( V^{\pi,s}_{n_1,n_2} 
        \geq \Psi_{\X,s}(\alpha) \mid \X \right) \leq \alpha %
 \end{talign}
 where 
 $V_{n_1,n_2}^{\pi,s}$ is the cheaply permuted V-statistic  \cref{eq: Two Sample Cheap V_statistic}. 
 This in turn will imply \cref{eq:Psi_X_n_ts} as
 $V_{n_1,n_2}^{\pi,s} \geq U_{n_1,n_2}^{\pi,s}$ 
 by the same argument \cref{eq:V_U_diff} used to show that $V_{n_1,n_2} \geq U_{n_1,n_2}$. 

    We follow a structure similar to that in the proof of \citep[Thm.~6.1]{kim2022minimax}.
    Our first step is to identify a second random variable that, conditioned on $\X$, has the same distribution as $V_{n_1,n_2}^{\pi,s}$.
  Let $\ell = (\ell_1,\dots,\ell_{s_1})$ be an $s_1$ tuple drawn without replacement from $[s_2]$ %
  generated independently of $(\X,\pi)$ and $\zeta=(\zeta_k)_{k=1}^{s_1}$ be independent random variables that take values $-1$ and $1$ with equal probability, independently of $(\X,\pi,\ell)$. 
  For $k_1,k_2\in[s_1]$, define the random measures 
    \begin{talign}
        \mu_{k} = 
        \begin{cases}
            \P &\text{if } \pi_{k} \leq s_1 \\
            \Q &\text{otherwise}
        \end{cases} \qtext{and}
        \mu'_{k'} = 
        \begin{cases}
            \P &\text{if } \pi_{s_1 + \ell_{k'}} \leq s_1 \\
            \Q &\text{otherwise}.
        \end{cases} 
    \end{talign}
    By construction, $(\mu_{k_1}$, $\mu_{k_2}$, $\mu'_{k_1}$, $\mu'_{k_2})$ are population distributions of the  respective sample points
    $(\mathbb{X}_{\pi_{k_1}}$, $\mathbb{X}_{\pi_{k_2}}$, $\mathbb{X}_{\pi_{s_1+\ell_{k_1}}}$, $\mathbb{X}_{\pi_{s_1+\ell_{k_2}}})$. 
    Next, let us define 
    \begin{talign}
    \widetilde{V}_{n_1,n_2}^{\pi,\ell,\zeta,s} \defeq \frac{1}{s_1^2} \sum_{k_1,k_2 =1}^{s_1} \zeta_{k_1} \zeta_{k_2} \Hho (\mathbb{X}_{\pi_{k_1}},\mathbb{X}_{\pi_{k_2}}; \mathbb{X}_{\pi_{s_1+\ell_{k_1}}},\mathbb{X}_{\pi_{s_1+\ell_{k_2}}}),
    \end{talign}
    where $\Hho$ is defined in \cref{eq:H_ts_def}.
    By \citep[Sec.~6.1]{kim2022minimax}, conditional on $\X$,  the distribution of $\E[\widetilde{V}_{n_1,n_2}^{\pi,\ell,\zeta,s}\mid \X,\pi,\zeta]$ matches that of $V_{n_1,n_2}^{\pi,s}$.

    Our goal is now to bound the tail of $\E[\widetilde{V}_{n_1,n_2}^{\pi,\ell,\zeta,s}\mid \X,\pi,\zeta]$.
    To this end, we define
    \begin{talign}
    &a_{k_1,k_2}(\pi,\ell) =  %
    \frac{1}{s_1^2}
    (\Hho (\mathbb{X}_{\pi_{k_1}},\mathbb{X}_{\pi_{k_2}}; \mathbb{X}_{\pi_{s_1+\ell_{k_1}}},\mathbb{X}_{\pi_{s_1+\ell_{k_2}}}) - \langle \mu_{k_1} \kernel - %
    \mu'_{k_1} \kernel, %
    \mu_{k_2} \kernel - \mu'_{k_2} \kernel \rangle_{\kernel}), 
    \end{talign}
    for $1 \leq k_1, k_2 \leq s_1$ and $\mathbf{A}_{\pi,\ell}=(a_{k_1,k_2}(\pi,\ell))_{k_1,k_2=1}^{s_1}$.  We remark that 
    \begin{talign} \label{eq:tildeU_def}
        &\widetilde{V}_{n_1,n_2}^{\pi,\ell,\zeta,s} = \zeta^{\top} \mathbf{A}_{\pi,\ell} \zeta + c_{\zeta,\pi,\ell}, \\
        &\text{where } c_{\zeta,\pi,\ell} \defeq \frac{1}{s_1^2} \sum_{k_1,k_2 = 1}^{s_1} \zeta_{k_1} \zeta_{k_2} \langle \mu_{k_1} \kernel - \mu'_{k_1} \kernel, \mu_{k_2} \kernel - \mu'_{k_2} \kernel \rangle_{\kernel},
    \end{talign}
    which means that $\E[\widetilde{V}_{n_1,n_2}^{\pi,\ell,\zeta,s}\mid \X,\pi,\zeta] = \zeta^{\top} \E[\mathbf{A}_{\pi,\ell}\mid \X,\pi] \zeta + \E[c_{\zeta,\pi,\ell}\mid \zeta, \pi]$. 
    
    Observe that %
    $\langle \mu_{k_1} \kernel - \mu'_{k_1} \kernel, \mu_{k_2} \kernel - \mu'_{k_2} \kernel \rangle_{\kernel}$ can only take on the values $0$ and $\pm \mmd^2(\P,\Q)$. To make this more precise, we can partition $[s_1]$ into the following subsets: $\mathcal{A} = \{ k \in [s_1] \mid \mu_{k} = \P, \mu'_{k} = \P \}$, $\mathcal{B} = \{ k \in [s_1] \mid \mu_{k} = \P, \mu'_{k} = \Q \}$, $\mathcal{C} = \{ k \in [s_1] \mid \mu_{k} = \Q, \mu'_{k} = \P \}$, $\mathcal{D} = \{ k \in [s_1] \mid \mu_{k} = \Q, \mu'_{k} = \Q \}$. 
    We will also introduce the partition $\mathcal{G} \sqcup \mathcal{F}$,\footnote{$\sqcup$ denotes a disjoint union.} where $\mathcal{G} = \{ k \in [s_1] \mid \mu_{k} = \P\}$ and $\mathcal{F} = \{ k \in [s_1] \mid \mu_{k} = \Q\}$. Observe that $\mathcal{G} = \mathcal{A} \sqcup \mathcal{B}$ and $\mathcal{F} = \mathcal{C} \sqcup \mathcal{D}$ and that the membership of $k$ in $\mathcal{G}$ or $\mathcal{F}$ is determined by $\pi$, while its membership in  $\mathcal{A}$, $\mathcal{B}$, $\mathcal{C}$, or $\mathcal{D}$ depends on $\ell$ as well. 

    In terms of indicator function notation,
    we have
    \begin{talign}
    \kinner{\mu_{k_1} \kernel - \mu'_{k_1} \kernel}{ \mu_{k_2} \kernel - \mu'_{k_2} \kernel }
        &= 
    \mmd^2(\P,\Q)(\mbf{1}_{(k_1, k_2) \in \mathcal{B} \times \mathcal{B} \cup \mathcal{C} \times \mathcal{C}}
        -
    \mbf{1}_{(k_1, k_2) \in \mathcal{B} \times \mathcal{C} \cup \mathcal{C} \times \mathcal{B}}),
    \end{talign}
    which means that
    \begin{talign}
        \E[c_{\zeta,\pi,\ell}\mid \zeta, \pi] \! = \! \frac{1}{s_1^2} \sum_{k_1,k_2 = 1}^{s_1} \zeta_{k_1} \zeta_{k_2} \E \big[ \mathbf{1}_{(k_1,k_2) \in \mathcal{B} \times \mathcal{B} \cup \mathcal{C} \times \mathcal{C}} \! - \! \mathbf{1}_{(k_1,k_2) \in \mathcal{B} \times \mathcal{C} \cup \mathcal{C} \times \mathcal{B}} \mid \pi\big] \mmd^2(\P,\Q).
    \end{talign}
    If we let
    $b_{k_1,k_2}(\pi) \defeq \E \big[ \mathbf{1}_{(k_1,k_2) \in \mathcal{B} \times \mathcal{B} \cup \mathcal{C} \times \mathcal{C}} - \mathbf{1}_{(k_1,k_2) \in \mathcal{B} \times \mathcal{C} \cup \mathcal{C} \times \mathcal{B}} \mid \pi\big]$, 
    we have that $b_{k_1,k_2}(\pi)$ takes the following values depending on the pair $(k_1,k_2)$:
    \begin{talign}
        b_{k_1,k_2}(\pi) = 
        \begin{cases}
            \E \big[ \mathbf{1}_{(k_1,k_2) \in \mathcal{B} \times \mathcal{B}} \mid \pi, k_1, k_2 \in \mathcal{G}, k_1 \neq k_2 \big] \defeq b_1 &\text{if } k_1, k_2 \in \mathcal{G}, k_1 \neq k_2 \\
            \E \big[ \mathbf{1}_{k_1 \in \mathcal{B}} \mid \pi, k_1 \in \mathcal{G} \big] \defeq b_2 &\text{if } k_1, k_2 \in \mathcal{G}, k_1 = k_2 \\
            \E \big[ \mathbf{1}_{(k_1,k_2) \in \mathcal{C} \times \mathcal{C}} \mid \pi, k_1, k_2 \in \mathcal{F}, k_1 \neq k_2 \big] \defeq b_3 &\text{if } k_1, k_2 \in \mathcal{F}, k_1 \neq k_2 \\
            \E \big[ \mathbf{1}_{k_1 \in \mathcal{C}} \mid \pi, k_1 \in \mathcal{F} \big] \defeq b_4 &\text{if } k_1, k_2 \in \mathcal{F}, k_1 = k_2 \\
            -\E \big[ \mathbf{1}_{(k_1,k_2) \in \mathcal{B} \times \mathcal{C}} \mid \pi, k_1 \in \mathcal{G}, k_2 \in \mathcal{F} \big] \defeq - b_5 &\text{if } k_1 \in \mathcal{G}, k_2 \in \mathcal{F} \\
            -\E \big[ \mathbf{1}_{(k_1,k_2) \in \mathcal{B} \times \mathcal{B}} \mid \pi, k_1 \in \mathcal{F}, k_2 \in \mathcal{G} \big] \defeq - b_6 &\text{if } k_1 \in \mathcal{F}, k_2 \in \mathcal{G}.
        \end{cases}
    \end{talign}
    Given $\pi$, let $q_{\mathcal{G}}$ be the number of indices $k$ such that $1\leq k \leq s_2$ 
    and $\pi_{s_1 + k} \leq s_1$, and analogously, let $q_{\mathcal{F}}$ be the number of indices $k$ such that $1\leq k \leq s_2$ and $\pi_{s_1 + k_2} > s_1$, which means that $q_{\mathcal{G}} + q_{\mathcal{F}} = s_2$.
    Observe that $b_1 = \E \big[ \mathbf{1}_{(k_1,k_2) \in \mathcal{B} \times \mathcal{B}} \mid \pi, k_1, k_2 \in \mathcal{G}, k_1 \neq k_2 \big] = \frac{q_{\mathcal{G}} (q_{\mathcal{G}} - 1)}{s_2 (s_2 - 1)} = \frac{q_{\mathcal{G}}^2}{s_2^2} - \frac{q_{\mathcal{G}} (s_2 - q_{\mathcal{G}})}{s^2_2 (s_2 - 1)} = \frac{q_{\mathcal{G}}^2}{s_2^2} - \frac{q_{\mathcal{G}} q_{\mathcal{F}}}{s^2_2 (s_2 - 1)}$. This is because the index $k_1$ belongs to $\mathcal{B}$ with probability $\frac{q_{\mathcal{G}}}{s_2}$, and then conditioned on this fact, the index $k_2$ belongs to $\mathcal{B}$ with probability $\frac{q_{\mathcal{G}}-1}{s_2-1}$.
    Similarly, $b_5 = \E \big[ \mathbf{1}_{(k_1,k_2) \in \mathcal{B} \times \mathcal{C}} \mid \pi, k_1 \in \mathcal{G}, k_2 \in \mathcal{F} \big] = \frac{q_{\mathcal{G}} q_{\mathcal{F}}}{s_2(s_2-1)} = \frac{q_{\mathcal{G}} q_{\mathcal{F}}}{s_2^2} + \frac{q_{\mathcal{G}} q_{\mathcal{F}}}{s_2^2 (s_2 - 1)}  
    $. Through analogous arguments we find that
    \begin{talign}
        b_3 = \frac{q_{\mathcal{F}}^2}{s_2^2} - \frac{q_{\mathcal{F}} q_{\mathcal{G}}}{s^2_2 (s_2 - 1)} 
        \qtext{and} 
        b_6 = \frac{q_{\mathcal{G}} q_{\mathcal{F}}}{s_2^2} + \frac{q_{\mathcal{G}} q_{\mathcal{F}}}{s_2^2 (s_2 - 1)}.
    \end{talign}
Therefore,
\begin{talign}
        &\sum_{k_1,k_2=1}^{s_1} \zeta_{k_1} \zeta_{k_2} b_{k_1,k_2}(\pi) \\ &= \big( \sum_{k \in \mathcal{G}} \zeta_{k} \big)^2 b_1 + \big( \sum_{k \in \mathcal{G}} 1 \big) \big( b_2 - b_1 \big)
        + \big( \sum_{k \in \mathcal{F}} \zeta_{k} \big)^2 b_3 + \big( \sum_{k_1 \in \mathcal{F}} 1 \big) \big( b_4 - b_3 \big) \\ &\qquad - \big( \sum_{k \in \mathcal{G}} \zeta_{k} \big) \big( \sum_{k \in \mathcal{F}} \zeta_{k} \big) b_5 - \big( \sum_{k \in \mathcal{G}} \zeta_{k} \big) \big( \sum_{k \in \mathcal{F}} \zeta_{k} \big) b_6 
        \\ &= \big( \sum_{k \in \mathcal{G}} \zeta_{k} \big)^2 \big( \frac{q_{\mathcal{G}}^2}{s_2^2} - \frac{q_{\mathcal{G}}q_{\mathcal{F}}}{s^2_2 (s_2 - 1)} \big) + \big( \sum_{k \in \mathcal{G}} 1 \big) \big( b_2 - b_1 \big)
        \\ &\qquad + \big( \sum_{k \in \mathcal{F}} \zeta_{k} \big)^2 \big( \frac{q_{\mathcal{F}}^2}{s_2^2} - \frac{q_{\mathcal{F}} q_{\mathcal{G}}}{s^2_2 (s_2 - 1)} \big) + \big( \sum_{k_1 \in \mathcal{F}} 1 \big) \big( b_4 - b_3 \big) \\ &\qquad - \big( \sum_{k \in \mathcal{G}} \zeta_{k} \big) \big( \sum_{k \in \mathcal{F}} \zeta_{k} \big) \big( \frac{q_{\mathcal{G}} q_{\mathcal{F}}}{s_2^2} + \frac{q_{\mathcal{G}} q_{\mathcal{F}}}{s_2^2 (s_2 - 1)} \big) - \big( \sum_{k \in \mathcal{G}} \zeta_{k} \big) \big( \sum_{k \in \mathcal{F}} \zeta_{k} \big) \big( \frac{q_{\mathcal{G}} q_{\mathcal{F}}}{s_2^2} + \frac{q_{\mathcal{G}} q_{\mathcal{F}}}{s_2^2 (s_2 - 1)} \big) 
        \\ &= \big( \big( \sum_{k \in \mathcal{G}} \zeta_{k} \big) \frac{q_{\mathcal{G}}}{s_2}
        - \big( \sum_{k \in \mathcal{F}} \zeta_{k} \big) \frac{q_{\mathcal{F}}}{s_2}
        \big)^2
        - \frac{q_{\mathcal{G}} q_{\mathcal{F}}}{s_2^2 (s_2 - 1)} \big( \sum_{k \in \mathcal{G}} \zeta_{k} + \sum_{k \in \mathcal{F}} \zeta_{k}\big)^2 
        \\ &\qquad + |\mathcal{G}| \big( b_2 - b_1 \big) + |\mathcal{F}| \big( b_4 - b_3 \big)
        \\ &\leq \big( \big( \sum_{k \in \mathcal{G}} \zeta_{k} \big) \E \big[ \mathbf{1}_{k \in \mathcal{B}} \mid \pi, k \in \mathcal{G}\big] - \big( \sum_{k \in \mathcal{F}} \zeta_{k} \big) \E \big[ \mathbf{1}_{k \in \mathcal{C}} \mid \pi, k \in \mathcal{F}\big] \big)^2 + s_1.
    \end{talign}
    Let $c_1 = \E \big[ \mathbf{1}_{k \in \mathcal{B}} \mid \pi, k \in \mathcal{G}\big]$, $c_2 = \E \big[ \mathbf{1}_{k \in \mathcal{C}} \mid \pi, k \in \mathcal{F}\big]$. 
    By Hoeffding's inequality \citep[Thm.~2]{hoeffding1963probability} and the union bound 
    \begin{talign}
        \Pr \big( |c_1 \sum_{k \in \mathcal{G}} \zeta_{k} + c_2 \sum_{k \in \mathcal{F}} \zeta_{k}| \geq t \mid \pi\big) \leq 2\exp \big( - \frac{t^2}{2 (c_1^2|\mathcal{G}|+c_2^2|\mathcal{F}|)} \big)
        \leq
        2\exp \big( - \frac{t^2}{2 s_1} \big).
    \end{talign}
    Therefore, with probability at least $1-\delta''$, 
    \begin{talign}
        \E[c_{\zeta,\pi,\ell}\mid \zeta, \pi] \! &= \! \frac{\mmd^2(\P,\Q)}{s_1^2} \sum_{k_1,k_2 = 1}^{s_1} \zeta_{k_1} \zeta_{k_2} b_{k_1,k_2}(\pi)
        \\ &\leq \frac{\mmd^2(\P,\Q)}{s_1^2} \big( \big(c_1 \sum_{k \in \mathcal{G}} \zeta_{k} + c_2 \sum_{k \in \mathcal{F}} \zeta_{k} \big)^2 + s_1 \big)
        \\ &< \frac{\mmd^2(\P,\Q)}{s_1^2} \big(
        2 \log(2/\delta'') s_1 + s_1 \big) = \frac{\mmd^2(\P,\Q) (2 \log(2/\delta'') + 1)}{s_1}.
    \end{talign}

    It remains to bound the tail of 
    \begin{talign}
    \zeta^{\top} \E[\mathbf{A}_{\pi,\ell}\mid \X,\pi] \zeta
        =
    \zeta^{\top} \E[\mathbf{A}^{\mathrm{od}}_{\pi,\ell}\mid \X,\pi] \zeta + \E[\mathrm{Tr}(\mathbf{A}_{\pi,\ell}) \mid \X,\pi], 
    \end{talign} 
    where $\mathbf{A}^{\mathrm{od}}_{\pi,\ell}$ has the off-diagonal components of $\mathbf{A}_{\pi,\ell}$ and zeros in the diagonal. 
    We bound $\zeta^{\top} \E[\mathbf{A}^{\mathrm{od}}_{\pi,\ell}\mid \X,\pi] \zeta$ first.
    Since $\sup_{\lambda\in\R} \E[e^{\lambda {\zeta}_1-\lambda^2 /2}] \leq 1$ \citep[Lem.~2.2]{boucheron2013concentration}, %
     \cref{lem:hanson_wright} implies that 
\begin{talign} \label{eq:hanson_wright_applied_4}
    \zeta^{\top} \E[\mathbf{A}^{\mathrm{od}}_{\pi,\ell} \mid \X, \pi]\zeta 
        \leq 
    \frac{20 \log(1/\delta')  \|\E[\mathbf{A}^{\mathrm{od}}_{\pi,\ell}\mid \X,\pi]\|_F}{3}
        \leq
    \frac{20 \log(1/\delta') \sup_{\pi, \ell} \|\mathbf{A}^{\mathrm{od}}_{\pi,\ell}\|_F}{3}
\end{talign}
with probability at least $1-\delta'$ conditioned on $(\X,\pi)$, provided that $\delta'\in(0,0.86)$.
Setting $\delta'' = \frac{\alpha}{2}$ and $\delta' = \frac{\alpha}{2}$, we find that with probability at least $1-\alpha$ conditioned on $\X$, 
\begin{talign} 
   \E[\widetilde{V}_{n_1,n_2}^{\pi,\ell,\zeta,s} &\mid \X,\pi,\zeta] 
   = \zeta^{\top} \E[\mathbf{A}^{\mathrm{od}}_{\pi,\ell} \mid \X, \pi] \zeta + \E[\mathrm{Tr}(\mathbf{A}_{\pi,\ell})\mid \X, \pi] + \E[c_{\zeta,\pi,\ell} \mid \pi,\zeta] \\ 
   &< \Psi_{\X,s}(\alpha) \! \defeq \! \frac{20 \log(\frac{2}{\alpha}) \sup_{\pi, \ell} \|\mathbf{A}^{\mathrm{od}}_{\pi,\ell}\|_F}{3} \! + \! \sup_{\pi,\ell} \mathrm{Tr}(\mathbf{A}_{\pi,\ell}) + \frac{\mmd^2(\P,\Q) (2 \log(4/\alpha) + 1)}{s_1}.
 \label{eq:psi_x_s_def}
\end{talign}

Now it only remains to find $\Psi_{n_1,n_2,s}$ satisfying $\Pr(\Psi_{\X,s}(\alpha) \geq \Psi_{n_1,n_2,s}(\alpha,\beta)) \leq \beta$. %
To achieve this we will develop quantile bounds for $\sup_{\pi, \ell} \|\mathbf{A}^{\mathrm{od}}_{\pi,\ell}\|_F$ and $\sup_{\pi,\ell} \mathrm{Tr}(\mathbf{A}_{\pi,\ell})$.
We begin with $\sup_{\pi, \ell} \|\mathbf{A}^{\mathrm{od}}_{\pi,\ell}\|_F$. 
Since 
\begin{talign}
\Hho(\mathbb{X}_{\pi_{k_1}}, \mathbb{X}_{\pi_{k_2}}; \mathbb{X}_{\pi_{s_1+\ell_{k_1}}}, \mathbb{X}_{\pi_{s_1+\ell_{k_2}}})
    =
\langle \delta_{\mathbb{X}_{\pi_{k_1}}} \kernel - \delta_{\mathbb{X}_{\pi_{s_1+\ell_{k_1}}}} \kernel,  \delta_{\mathbb{X}_{\pi_{k_2}}} \kernel - \delta_{\mathbb{X}_{\pi_{s_1+\ell_{k_2}}}} \kernel \rangle_{\kernel} 
\end{talign}
we have, by the triangle inequality and Cauchy-Schwarz, 
\begin{talign}
&s_1^2 |a_{k_1,k_2}(\pi,\ell)| 
    = 
|\Hho(\mathbb{X}_{\pi_{k_1}}, \mathbb{X}_{\pi_{k_2}}; \mathbb{X}_{\pi_{s_1+\ell_{k_1}}}, \mathbb{X}_{\pi_{s_1+\ell_{k_2}}}) - \langle \mu_{k_1} \kernel - \mu'_{k_1} \kernel, \mu_{k_2} \kernel - \mu'_{k_2} \kernel \rangle_{\kernel}| \\ 
    &= 
|\langle \delta_{\mathbb{X}_{\pi_{k_1}}} \kernel - \delta_{\mathbb{X}_{\pi_{s_1+\ell_{k_1}}}} \kernel,  \delta_{\mathbb{X}_{\pi_{k_2}}} \kernel - \delta_{\mathbb{X}_{\pi_{s_1+\ell_{k_2}}}} \kernel \rangle_{\kernel} - \langle \mu_{k_1} \kernel - \mu'_{k_1} \kernel, \mu_{k_2} \kernel - \mu'_{k_2} \kernel \rangle_{\kernel} | \\ &\leq |\langle \delta_{\mathbb{X}_{\pi_{k_1}}} \kernel - \delta_{\mathbb{X}_{\pi_{s_1+\ell_{k_1}}}} \kernel - (\mu_{k_1} \kernel - \mu'_{k_1} \kernel), \delta_{\mathbb{X}_{\pi_{k_2}}} \kernel - \delta_{\mathbb{X}_{\pi_{s_1+\ell_{k_2}}}} \kernel \rangle_{\kernel}| \\ &\quad+ |\langle \mu_{k_1} \kernel - \mu'_{k_1} \kernel,  \delta_{\mathbb{X}_{\pi_{k_2}}} \kernel - \delta_{\mathbb{X}_{\pi_{s_1+\ell_{k_2}}}} \kernel - (\mu_{k_2} \kernel - \mu'_{k_2} \kernel) \rangle_{\kernel}| \\ &\leq \| \delta_{\mathbb{X}_{\pi_{k_1}}} \kernel - \delta_{\mathbb{X}_{\pi_{s_1+\ell_{k_1}}}} \kernel - (\mu_{k_1} \kernel - \mu'_{k_1} \kernel) \|_{\kernel} \| \delta_{\mathbb{X}_{\pi_{k_2}}} \kernel - \delta_{\mathbb{X}_{\pi_{s_1+\ell_{k_2}}}} \kernel \|_{\kernel} \\ &\quad + \| \mu_{k_1} \kernel - \mu'_{k_1} \kernel \|_{\kernel}  \| \delta_{\mathbb{X}_{\pi_{k_2}}} \kernel - \delta_{\mathbb{X}_{\pi_{s_1+\ell_{k_2}}}} \kernel - (\mu_{k_2} \kernel - \mu'_{k_2} \kernel) \|_{\kernel} \\ &\leq (\mmd(\delta_{\mathbb{X}_{\pi_{k_1}}},\mu_{k_1}) \! + \! \mmd(\delta_{\mathbb{X}_{\pi_{s_1 \! + \! \ell_{k_1}}}},\mu'_{k_1})) \\ &\qquad \times (\mmd(\delta_{\mathbb{X}_{\pi_{k_2}}},\mu_{k_2}) \! + \! \mmd(\delta_{\mathbb{X}_{\pi_{s_1 \! + \! \ell_{k_2}}}},\mu'_{k_2}) \! + \! \mmd(\mu_{k_2},\mu'_{k_2})) \\ &\quad+ \mmd(\mu_{k_1},\mu'_{k_1}) (\mmd(\delta_{\mathbb{X}_{\pi_{k_2}}},\mu'_{k_2}) \! + \! \mmd(\delta_{\mathbb{X}_{\pi_{s_1 \! + \! \ell_{k_2}}}},\mu'_{k_2})).
\end{talign}
Hence, by Jensen's inequality and the arithmetic-geometric mean inequality, 
\begin{talign}
\begin{split}
    &\sum_{k_1,k_2=1}^{s_1} s_1^4 |a_{k_1,k_2}(\pi,\ell)|^2 \\ &\leq \sum_{k_1,k_2=1}^{s_1} \bigg( (\mmd(\delta_{\mathbb{X}_{\pi_{k_1}}},\mu_{k_1}) \! + \! \mmd(\delta_{\mathbb{X}_{\pi_{s_1 \! + \! \ell_{k_1}}}},\mu'_{k_1})) \\ &\qquad\qquad\qquad \times (\mmd(\delta_{\mathbb{X}_{\pi_{k_2}}},\mu_{k_2}) \! + \! \mmd(\delta_{\mathbb{X}_{\pi_{s_1 \! + \! \ell_{k_2}}}},\mu'_{k_2}) \! + \! \mmd(\mu_{k_2},\mu'_{k_2})) \\ &\qquad\qquad\qquad \! + \! \mmd(\mu_{k_1},\mu'_{k_1}) (\mmd(\delta_{\mathbb{X}_{\pi_{k_2}}},\mu'_{k_2}) \! + \! \mmd(\delta_{\mathbb{X}_{\pi_{s_1 \! + \! \ell_{k_2}}}},\mu'_{k_2})) \bigg)^2 \\ &\leq \sum_{k_1,k_2=1}^{s_1} 3 (\mmd(\delta_{\mathbb{X}_{\pi_{k_1}}},\mu_{k_1}) \! + \! \mmd(\delta_{\mathbb{X}_{\pi_{s_1 \! + \! \ell_{k_1}}}},\mu'_{k_1}))^2
    \\ &\qquad\qquad\quad \times 
    (\mmd(\delta_{\mathbb{X}_{\pi_{k_2}}},\mu_{k_2}) \! + \! \mmd(\delta_{\mathbb{X}_{\pi_{s_1 \! + \! \ell_{k_2}}}},\mu'_{k_2}))^2 \\ &\qquad\qquad \! + \! 3 \mmd^2(\mu_{k_2},\mu'_{k_2}) (\mmd(\delta_{\mathbb{X}_{\pi_{k_1}}},\mu_{k_1}) \! + \! \mmd(\delta_{\mathbb{X}_{\pi_{s_1 \! + \! \ell_{k_1}}}},\mu'_{k_1}))^2 \\ &\qquad\qquad \! + \! 3 \mmd^2(\mu_{k_1},\mu'_{k_1}) (\mmd(\delta_{\mathbb{X}_{\pi_{k_2}}},\mu'_{k_2}) \! + \! \mmd(\delta_{\mathbb{X}_{\pi_{s_1 \! + \! \ell_{k_2}}}},\mu'_{k_2}))^2 \\ &\leq \sum_{k_1,k_2=1}^{s_1} \frac{3}{2} \big((\mmd(\delta_{\mathbb{X}_{\pi_{k_1}}},\mu_{k_1}) \! + \! \mmd(\delta_{\mathbb{X}_{\pi_{s_1 \! + \! \ell_{k_1}}}},\mu'_{k_1}))^4 \\ &\qquad\qquad\qquad + \! (\mmd(\delta_{\mathbb{X}_{\pi_{k_2}}},\mu_{k_2}) \! + \! \mmd(\delta_{\mathbb{X}_{\pi_{s_1 \! + \! \ell_{k_2}}}},\mu'_{k_2}))^4 \big) \\ &\qquad\qquad \! + \! 6 \mmd^2(\mu_{k_2},\mu'_{k_2}) (\mmd^2(\delta_{\mathbb{X}_{\pi_{k_1}}},\mu_{k_1}) \! + \! \mmd^2(\delta_{\mathbb{X}_{\pi_{s_1 \! + \! \ell_{k_1}}}},\mu'_{k_1})) \\ &\qquad\qquad \! + \! 6 \mmd^2(\mu_{k_1},\mu'_{k_1}) (\mmd^2(\delta_{\mathbb{X}_{\pi_{k_2}}},\mu'_{k_2}) \! + \! \mmd^2(\delta_{\mathbb{X}_{\pi_{s_1 \! + \! \ell_{k_2}}}},\mu'_{k_2}))
    \\ &\leq \sum_{k_1,k_2=1}^{s_1} 12 \big(\mmd^4(\delta_{\mathbb{X}_{\pi_{k_1}}},\mu_{k_1}) \! + \! \mmd^4(\delta_{\mathbb{X}_{\pi_{s_1 \! + \! \ell_{k_1}}}},\mu'_{k_1}) \\ &\qquad\qquad\qquad + \! \mmd^4(\delta_{\mathbb{X}_{\pi_{k_2}}},\mu_{k_2}) \! + \! \mmd^4(\delta_{\mathbb{X}_{\pi_{s_1 \! + \! \ell_{k_2}}}},\mu'_{k_2}) \big) \\ &\qquad \! + \! 6 \mmd^2(\P,\Q) (\mmd^2(\delta_{\mathbb{X}_{\pi_{k_1}}},\mu_{k_1}) \! + \! \mmd^2(\delta_{\mathbb{X}_{\pi_{s_1 \! + \! \ell_{k_1}}}},\mu'_{k_1}) \\ &\qquad\qquad\qquad\qquad\qquad \! + \! \mmd^2(\delta_{\mathbb{X}_{\pi_{k_2}}},\mu'_{k_2}) \! + \! \mmd^2(\delta_{\mathbb{X}_{\pi_{s_1 \! + \! \ell_{k_2}}}},\mu'_{k_2})) \\ &\leq \sum_{i=1}^{s_1} \big(
    24 s_1 \mmd^4(\delta_{\mathbb{Y}_{i}},\P) \! + \!
    12 s_1 \mmd^2(\P,\Q) \mmd^2(\delta_{\mathbb{Y}_{i}},\P) \big) \\ &\qquad \! + \! \sum_{i=1}^{s_2} \big(
    24 s_1 \mmd^4(\delta_{\mathbb{Z}_{i}},\Q) \! + \! 
    12 s_1 \mmd^2(\P,\Q) \mmd^2(\delta_{\mathbb{Z}_{i}},\Q) \big)
\end{split}
\end{talign}

Let $\event$ be the event that the inequality of \cref{lem:concentration_simultaneous} with $\delta = \frac{\beta}{3}$ holds simultaneously for all $p\in\{1,2,4\}$. 
By the union bound, this holds with probability at least $1-\beta$. 
Moreover, on $\event$,
\begin{talign}
    \sum_{k_1,k_2=1}^{s_1} s_1^4 |a_{k_1,k_2}(\pi,\ell)|^2 &< 24 s_1 
    (\mathcal{A}_4(s_1,s_2,m) + \mathcal{B}(s_1,s_2,m,\beta/3))^4 \\ &+ 12 s_1 \mmd^2(\P,\Q) 
    (\mathcal{A}_2(s_1,s_2,m) + \mathcal{B}(s_1,s_2,m,\beta/3))^2
\end{talign}
Therefore, on $\event$,
\begin{talign} \label{eq:sup_frobenius}
    &\sup_{\pi,\ell} \|\mathbf{A}_{\pi,\ell}^{\mathrm{od}} \|_F \\ &<
    \sqrt{\frac{24 s_1 (\mathcal{A}_4(s_1,s_2,m) + \mathcal{B}(s_1,s_2,m,\frac{\beta}{3}))^4 + 12 s_1 \mmd^2(\P,\Q) (\mathcal{A}_2(s_1,s_2,m) + \mathcal{B}(s_1,s_2,m,\frac{\beta}{3}))^2}{s_1^4}} \\ &\leq \frac{\sqrt{24} (\mathcal{A}_4(s_1,s_2,m) + \mathcal{B}(s_1,s_2,m,\frac{\beta}{3}))^2}{s_1^{3/2}} + \frac{\sqrt{12} \mmd(\P,\Q) (\mathcal{A}_2(s_1,s_2,m) + \mathcal{B}(s_1,s_2,m,\frac{\beta}{3}))}{s_1^{3/2}}.
\end{talign}
Next, we bound $\sup_{\pi,\ell} \mathrm{Tr}(\mathbf{A}_{\pi,\ell})$ using Jensen's inequality:
\begin{talign}
    &\mathrm{Tr}(\mathbf{A}_{\pi,\ell}) = \sum_{k_1=1}^{s_1} a_{k_1,k_1}(\pi,\ell) \\ 
    &\leq\frac{1}{s_1^2} \sum_{k_1=1}^{s_1} (\mmd(\delta_{\mathbb{X}_{\pi_{k_1}}},\mu_{k_1}) \! + \! \mmd(\delta_{\mathbb{X}_{\pi_{s_1 \! + \! \ell_{k_1}}}},\mu'_{k_1}))^2 \\ &\qquad\quad+ 2 \mmd(\mu_{k_1},\mu'_{k_1}) (\mmd(\delta_{\mathbb{X}_{\pi_{k_1}}},\mu_{k_1}) \! + \! \mmd(\delta_{\mathbb{X}_{\pi_{s_1 \! + \! \ell_{k_1}}}},\mu'_{k_1}))
    \\ &\leq \frac{1}{s_1^2} \sum_{k_1=1}^{s_1} 2 (\mmd^2(\delta_{\mathbb{X}_{\pi_{k_1}}},\mu_{k_1}) \! + \! \mmd^2(\delta_{\mathbb{X}_{\pi_{s_1 \! + \! \ell_{k_1}}}},\mu'_{k_1})) \\ &\qquad\quad+ 2 \mmd(\P,\Q) (\mmd(\delta_{\mathbb{X}_{\pi_{k_1}}},\mu_{k_1}) \! + \! \mmd(\delta_{\mathbb{X}_{\pi_{s_1 \! + \! \ell_{k_1}}}},\mu'_{k_1})) \\ &\leq \frac{2}{s_1^2} \big( \sum_{i=1}^{s_1} \mmd^2(\delta_{\mathbb{Y}_{i}},\P) + \sum_{i=1}^{s_2} \mmd^2(\delta_{\mathbb{Z}_{i}},\Q) \\ &\qquad\quad + \mmd(\P,\Q) \big( \sum_{i=1}^{s_1} \mmd(\delta_{\mathbb{Y}_{i}},\P) + \sum_{i=1}^{s_2} \mmd(\delta_{\mathbb{Z}_{i}},\Q) \big) \big)
\end{talign}
On $\event$ we therefore have  
\begin{talign} 
\begin{split} \label{eq:trace_bound}
    \sup_{\pi,\ell} \mathrm{Tr}(\mathbf{A}_{\pi,\ell}) &< \frac{2}{s_1^2} \big( 
    (\mathcal{A}_2(s_1,s_2,m) + \mathcal{B}(s_1,s_2,m,\beta/3))^2 \\ &\qquad + \mmd(\P,\Q) 
    (\mathcal{A}_1(s_1,s_2,m) + \mathcal{B}(s_1,s_2,m,\beta/3)) \big).
\end{split}
\end{talign}
From 
\cref{eq:sup_frobenius,,eq:hanson_wright_applied_4,,eq:trace_bound}, we therefore obtain that with probability at least 
$1-\beta$,
\begin{talign}
    \Psi_{\X,s}(\alpha)
    &<
    \Psi_{n_1,n_2,s}(\alpha,\beta).
\end{talign}

\section{\pcref{thm: Full Independence Tests}} \label{subsec:proof_independence_finite_variance}

We will establish both claims using the refined two moments method (\cref{Lemma: Refined Two Moments Method}) with the auxiliary sequence  $(V_{n}^{\pi_b,s})_{b=1}^\numperm$ of cheaply-permuted independence V-statistics \cref{eq:cheaply-permuted-independence-V}. 
By \cref{Lemma: Refined Two Moments Method} with option  \cref{eq:refined_Psi_n_nonzero}, it suffices to bound $\Var(V_n)$, $\E[\Var(V_n^{\pi,s}\mid \X)]$, $\Var(\E[V_n^{\pi,s}\mid \X])$, and $\E[V_n^{\pi,s}]$.
These quantities are bounded in \cref{lem:variance_indep_V,,lem:exp_variance_ind,,lem:exp_exp_variance_exp_ind} with proofs in \cref{subsec:variance_indep_V_proof,,subsec:exp_variance_ind_proof,,subsec:exp_exp_variance_exp_ind_proof}.

\begin{lemma}[Variance of independence V-statistics] \label{lem:variance_indep_V}
Under the assumptions of \cref{thm: Full Independence Tests},
\begin{talign} \label{Eq: Variance upper bound II V}
\Var(V_n) \leq \frac{16 \psi_{1}'}{n} + \frac{144 \psi_{2}'}{n^2}.
\end{talign}
\end{lemma}

\begin{lemma}[Conditional variance of permuted independence V-statistics] \label{lem:exp_variance_ind}
Under the assumptions of \cref{thm: Full Independence Tests},
\begin{talign} 
\label{eq:exp_var_full_ind_bound}
&\max(\Var (\mathbb{E}[V_n^{\pi,n} \mid  \mathbb{X}] ), \mE [ \Var (V_n^{\pi,n} \mid  \mbb{X}) ]) 
    \leq
\Var(V_n^{\pi,n})
    \leq 
\frac{32 \tilde{\psi}_1'}{n} + \frac{960 \psi_2'}{n^2} 
    \qtext{and} 
    \\
\begin{split}
&\max(\Var (\mathbb{E}[V_n^{\pi,s} \mid  \mathbb{X}] ), 
    \mE[\Var (V_n^{\pi,s} \mid  \mbb{X}) ])
    \leq
\Var(V_n^{\pi,s}) \\
    &\qquad\leq %
    \tilde{\psi}_1' \big( \frac{32}{n} + \frac{2304}{n s} + \frac{460800}{n s^2} \big) + \psi_2' \big( \frac{960}{n^2} + \frac{4147200}{n^2 s} \big).
\label{eq:exp_var_cheap_ind_bound}
\end{split}
\end{talign}
\end{lemma}

\begin{lemma}[Mean of permuted independence V-statistics] \label{lem:exp_exp_variance_exp_ind}
Under the assumptions of \cref{thm: Full Independence Tests},
\begin{talign}
\label{eq:Vnpin-mean}
\mE [V_n^{\pi,n}] &\leq 
\frac{1}{4} \mathcal{U}
+ \frac
{10n^2 - 44 n + 48}
{n^3} \tilde{\xi} 
    \qtext{and} 
    \\
\label{eq:Vnpis-mean}
    \mE [V_n^{\pi,s}]  &\leq 
    \frac{1}{4} \big( 1 + 
    \frac{s^2 + s - 4}{s^3}
    \big) \mathcal{U} + \frac
    {10n^2 - 44 n + 48}
    {n^3} \tilde{\xi}.
\end{talign}
\end{lemma}
Now fix any $\alpha,\beta\in(0,1)$.
By \cref{lem:exp_variance_ind},
\begin{talign}
&\sqrt{(\frac2{\alpha\beta}\!-\frac2\beta) \mE [ \Var (V_n^{\pi,s} \mid  \mbb{X}) ] }  
    + \!\sqrt{\!(\frac{2}{\beta} -1) \Var (\mathbb{E}[V_n^{\pi,s} \mid  \mathbb{X}] )} \\ &\leq \big( \sqrt{\frac2{\alpha\beta}\!-\!\frac2\beta} + \sqrt{\frac{2}{\beta} -1} \big) \sqrt{  \Var (V_n^{\pi,s}) } \leq \sqrt{ \big( \frac4{\alpha\beta}\!-\!2 \big) \Var (V_n^{\pi,s})},
\end{talign}
with the same result holding when $V_n^{\pi,n}$ is substituted for $V_n^{\pi,s}$. 
Hence, by 
\cref{Lemma: Refined Two Moments Method,,lem:variance_indep_V,,lem:exp_variance_ind,,lem:exp_exp_variance_exp_ind} the standard and cheap tests enjoy power at least $1-\beta$ whenever 
\begin{talign} \label{eq:ind_full_suff_app}
    &\mathbb{E}[V_{n}] \! - \! \frac{1}{4} \mathcal{U} \geq 
    \sqrt{(\frac{3}{\beta} - 1) \big(\frac{16 \psi_1'}{n} + \frac{144 \psi_2'}{n^2} \big)}
    + \sqrt{\big( \frac{12}{\astar \beta} - 2\big)
    \big(
    \frac{32 \tilde{\psi}_1'}{n} + \frac{960 \psi_2'}{n^2} \big)}
    + \frac{10}{n} \tilde{\xi}
    \stext{and}\\
    &\mathbb{E}[V_{n}] \! - \! \frac{1}{4} \big( 1 + 
    \frac{s^2 + s - 4}{s^3}
    \big)\mathcal{U} \geq 
    \sqrt{(\frac{3}{\beta} - 1) \big(\frac{16 \psi_1'}{n} + \frac{144 \psi_2'}{n^2} \big)} 
    \\ &\qquad\ + \! \sqrt{\big( \frac{12}{\astar \beta} - 2\big)
    \big(\tilde{\psi}_1' \big( 
    \frac{32}{n} + \frac{2304}{n s} + \frac{460800}{n s^2}
    \big) + \psi_2' \big( 
    \frac{960}{n^2} + \frac{4147200}{n^2 s}
    \big) \big)} \! + \! \frac{10}{n} \tilde{\xi}
    \label{eq:ind_cheap_suff_app}
\end{talign}
respectively.
To recover the looser sufficient condition \cref{eq:ind_full_suff} for standard testing, we note that, by \cref{lem:V_n_U_n_generic}, $\E[V_n] \geq \mathcal{U} - 
\frac{24 \sqrt{\psi_2'} + 6 \tilde{\xi}}{n}$.
To recover the looser sufficient condition \cref{eq:ind_cheap_suff} for cheap testing, we again invoke $\E[V_n] \geq \mathcal{U} - 
\frac{24 \sqrt{\psi_2'} + 6 \tilde{\xi}}{n}$ and additionally note that 
\begin{talign}
\begin{split}
    & \sqrt{(\frac{3}{\beta} - 1) \big(\frac{16 \psi_1'}{n} + \frac{144 \psi_2'}{n^2} \big)} 
    \! + \! \sqrt{\big( \frac{12}{\astar \beta} - 2\big)
    \big(\tilde{\psi}_1' \big( 
    \frac{32}{n} + \frac{2304}{n s} + \frac{460800}{n s^2}
    \big)+ \psi_2' \big(
    \frac{960}{n^2} + \frac{4147200}{n^2 s}
    \big) \big)} \\ &\qquad\quad + \! \frac{10}{n} \tilde{\xi} \! + \! 
    \frac{24 \sqrt{\psi_2'} + 6 \tilde{\xi}}{n}
    \\ & \leq \frac{3}{4}\gamma_n + \sqrt{\big( \frac{12}{\astar \beta} - 2\big)
    \big(\tilde{\psi}_1' \big( \frac{2304}{n s} + \frac{460800}{n s^2}
    \big) + \psi_2' \frac{4147200}{n^2 s} \big)}.
\end{split}
\end{talign}
\begin{lemma}[Independence V vs.~U-statistics] \label{lem:V_n_U_n_generic}
Let $U_n$ be the U-statistic associated to $h_{\itag}$, which takes the form $U_n = \frac{1}{(n)_4} \sum_{(i_1,i_2,i_3,i_4) \in \mathbf{i}_4^n} h_{\itag}(X_{i_1}, X_{i_2}, X_{i_3}, X_{i_4})$, and satisfies $\mE[U_n] = \mE[h_{\itag}(X_{1}, X_{2}, X_{3}, X_{4})] = \mathcal{U}$.
Under the assumptions of \cref{thm: Full Independence Tests},  
    \begin{talign}
        |\mathbb{E}[ V_n - U_n ]| \leq \frac{24 \sqrt{\psi_2'} + 6 \tilde{\xi}}{n}.
    \end{talign}
\end{lemma}
\begin{proof}
    By definition, 
    \begin{talign}
    \begin{split}
        V_n - U_n &= \frac{1}{n^4} \sum_{(i_1,i_2,i_3,i_4) \in [n]^4} h_{\itag}(X_{i_1}, X_{i_2}, X_{i_3}, X_{i_4}) \\ &\quad - \frac{1}{(n)_4} \sum_{(i_1,i_2,i_3,i_4) \in \mathbf{i}_4^n} h_{\itag}(X_{i_1}, X_{i_2}, X_{i_3}, X_{i_4}) \\ &= \big( \frac{1}{n^4} - \frac{1}{(n)_4} \big) \sum_{(i_1,i_2,i_3,i_4) \in \mathbf{i}_4^n} h_{\itag}(X_{i_1}, X_{i_2}, X_{i_3}, X_{i_4}) \\ &\quad + \frac{1}{n^4} \sum_{(i_1,i_2,i_3,i_4) \in [n]^4 \setminus \mathbf{i}_4^n} h_{\itag}(X_{i_1}, X_{i_2}, X_{i_3}, X_{i_4}). 
    \end{split}
    \end{talign}
    Thus, we have that 
     \begin{talign} 
     \begin{split} \label{eq:abs_E_V_U}
        |\mathbb{E}[ V_n - U_n ]| &\leq \big| \frac{1}{n^4} - \frac{1}{(n)_4} \big| \sum_{(i_1,i_2,i_3,i_4) \in \mathbf{i}_4^n} \big|\mathbb{E}\big[ h_{\itag}(X_{i_1}, X_{i_2}, X_{i_3}, X_{i_4}) \big] \big| \\ &\quad + \frac{1}{n^4} \sum_{(i_1,i_2,i_3,i_4) \in [n]^4 \setminus \mathbf{i}_4^n} \big| \mathbb{E}\big[ h_{\itag}(X_{i_1}, X_{i_2}, X_{i_3}, X_{i_4}) \big] \big| .
    \end{split}
    \end{talign}
    We bound the first term on the right-hand side 
    using Jensen's inequality and the $\psi_2'$ definition:
    \begin{talign}
        &\big| \frac{1}{n^4} - \frac{1}{(n)_4} \big| \sum_{(i_1,i_2,i_3,i_4) \in \mathbf{i}_4^n} \big|\mathbb{E}\big[ h_{\itag}(X_{i_1}, X_{i_2}, X_{i_3}, X_{i_4}) \big] \big| \\ &\leq \big| \frac{1}{n^4} - \frac{1}{(n)_4} \big| \sum_{(i_1,i_2,i_3,i_4) \in \mathbf{i}_4^n} \big(\mathbb{E}\big[ h_{\itag}(X_{i_1}, X_{i_2}, X_{i_3}, X_{i_4})^2 \big] \big)^{1/2} \\ &\leq \big| \frac{1}{n^4} - \frac{1}{(n)_4} \big|  (n)_4 \big( 16 \psi_2' \big)^{1/2}.
    \end{talign}
    We bound the second term using the definition of $\tilde{\xi}$:
    \begin{talign}
        \frac{1}{n^4} \sum_{(i_1,i_2,i_3,i_4) \in [n]^4 \setminus \mathbf{i}_4^n} \big| \mathbb{E}\big[ h_{\itag}(X_{i_1}, X_{i_2}, X_{i_3}, X_{i_4}) \big] \big| \leq \big| \frac{1}{n^4} - \frac{1}{(n)_4} \big|  (n)_4 \tilde{\xi}.
    \end{talign}
    The advertised conclusion follows as %
    $1 - \frac{(n)_4}{n^4} = 1 - \frac{(n-1)(n-2)(n-3)}{n^3} = \frac{6n^2 - 11n + 6}{n^3} \leq \frac{6}{n}$. 
\end{proof}
\subsection{\pcref{lem:variance_indep_V}} \label{subsec:variance_indep_V_proof}
Define the index sets
\begin{talign}
\mathsf{K}_{\text{total}} &\defeq [n]^8, \\
\mathsf{K}_{1} &\defeq \{ (i_1,i_2,i_3,i_4,i_1',i_2',i_3',i_4') \in \mathsf{K}_{\text{total}} : |\{i_1,i_2,i_3,i_4\}\cap \{i_1',i_2',i_3',i_4'\} | = 0 \}, \\
\mathsf{K}_{2} &\defeq \{ (i_1,i_2,i_3,i_4,i_1',i_2',i_3',i_4') \in \mathsf{K}_{\text{total}} : |\{i_1,i_2,i_3,i_4\}\cap \{i_1',i_2',i_3',i_4'\} | = 1 \}, \stext{and} \\
(\mathsf{K}_{1} \cup \mathsf{K}_{2})^c &\defeq \{ (i_1,i_2,i_3,i_4,i_1',i_2',i_3',i_4') \in \mathsf{K}_{\text{total}} : |\{i_1,i_2,i_3,i_4\}\cap \{i_1',i_2',i_3',i_4'\} | > 1 \}.
\end{talign}
Then,
\begin{talign}
\begin{split} \label{eq:variance_decomp_ind}
    \Var(V_n) &= \mE[V_n^2] - \mE[V_n]^2 \\ &= \frac{1}{n^8}  \sum_{(i_1,\dots,i_4,i_1',\dots,i_4') \in \mathsf{K}_{\text{total}}} \big( \mE[\hbarin(X_{i_1},X_{i_2},X_{i_3},X_{i_4}) \hbarin(X_{i_1'},X_{i_2'},X_{i_3'},X_{i_4'})] \\ &\qquad\qquad - \mE[\hbarin(X_{i_1},X_{i_2},X_{i_3},X_{i_4})] \mE[\hbarin(X_{i_1'},X_{i_2'},X_{i_3'},X_{i_4'})] \big) \\ &= (I) + (II) + (III),
\end{split}
\end{talign}
 where $(I)$ is the summation over $(i_1,\ldots,i_4') \in \mathsf{K}_{1}$, $(II)$ is the summation over $(i_1,\ldots,i_4') \in \mathsf{K}_{2}$, and $(III)$ is the summation over $(i_1,\ldots,i_4') \in (\mathsf{K}_{1} \cup \mathsf{K}_{2})^c$. Note that
 \begin{talign}
    |\mathsf{K}_{2}| \leq n^4 \cdot 4 \cdot 4 \cdot n^3 = 16 n^7 \qtext{and}
    |(\mathsf{K}_{1} \cup \mathsf{K}_{2})^c| \leq n^4 \cdot 4 \cdot 3 \cdot 4 \cdot 3 \cdot n^2 = 144 n^6.
 \end{talign}
 We also remark that for $(i_1,\ldots,i_4') \in \mathsf{K}_{1}$, $\mE[\hin(X_{i_1},X_{i_2},X_{i_3},X_{i_4}) \hin(X_{i_1'},X_{i_2'},X_{i_3'},X_{i_4'})] = \mE[\hin(X_{i_1},X_{i_2},X_{i_3},X_{i_4})] \mE[\hin(X_{i_1'},X_{i_2'},X_{i_3'},X_{i_4'})]$,
 which implies that $(I) = 0$. 
 
 By the definition \cref{Eq: definition of psi prime functions} of $\psi_{1}'$, we have that $(II) \leq \frac{1}{n^8}|\mathsf{K}_{2}| \psi_{1}' \leq \frac{16 \psi_{1}'}{n}$.
 By the nonnegativity of $\hbarin$, Cauchy-Schwarz, Jensen's inequality, and the definition \cref{Eq: definition of psi prime functions} of $\psi_{2}'$, we have that each summand of $(III)$ is upper-bounded by
 \begin{talign}
 \begin{split}
    &\mE[\hbarin(X_{i_1},X_{i_2},X_{i_3},X_{i_4})\hbarin(X_{i_1'},X_{i_2'},X_{i_3'},X_{i_4'})] 
    \\ &\leq \big( \mE[\hbarin(X_{i_1},X_{i_2},X_{i_3},X_{i_4})^2 ] \big)^{1/2} \big( \mE[\hbarin(X_{i_1'},X_{i_2'},X_{i_3'},X_{i_4'})^2] \big)^{1/2}
    \\ &\leq \big( \mE[\hin(X_{i_1},X_{i_2},X_{i_3},X_{i_4})^2 ] \big)^{1/2} \big( \mE[\hin(X_{i_1'},X_{i_2'},X_{i_3'},X_{i_4'})^2] \big)^{1/2} 
    \leq \psi_{2}'.
 \end{split}
 \end{talign}
 Hence, $(III) \leq \frac{1}{n^8}|(\mathsf{K}_{1} \cup \mathsf{K}_{2})^c| \psi_{2}' \leq \frac{144 \psi_{2}'}{n^2}$, which concludes the proof.

\subsection{\pcref{lem:exp_variance_ind}} \label{subsec:exp_variance_ind_proof}
Throughout we use the shorthand
$\mathfrak{i} \defeq (i_1,i_2,j_1,j_2,i_1',i_2',j_1',j_2')$.
\subsubsection{Proof of \cref{eq:exp_var_full_ind_bound}}\label{proof:exp_var_full_ind_bound}
Our approach is similar to that employed in \cite[Thm.~5.1]{kim2022minimax} to study  independence U-statistics.
We begin by defining the convenient shorthand 
\begin{talign}
\mathfrak{s}(\mathfrak{i}) &\defeq |\{i_1,i_2,i_3,i_4,n\!+\!1\!-\!i_1,n\!+\!1\!-\!i_2,n\!+\!1\!-\!i_3,n\!+\!1\!-\!i_4\}\cap \{i_1',i_2',i_3',i_4'\} |,\\
    \mathsf{K}_{\text{total}} &\defeq [n]^8, \quad
    \mathsf{K}_{1} \defeq \{ \mathfrak{i} \in \mathsf{K}_{\text{total}} :  \mathfrak{s}(\mathfrak{i}) = 0 \},  \quad
    \mathsf{K}_{2} \defeq \{ \mathfrak{i} \in \mathsf{K}_{\text{total}} : \mathfrak{s}(\mathfrak{i}) = 1 \}, \qtext{and} \\
    (\mathsf{K}_{1} \cup \mathsf{K}_{2})^c &\defeq \{ \mathfrak{i} \in \mathsf{K}_{\text{total}} : \mathfrak{s}(\mathfrak{i}) > 1 \},
    \end{talign}
     and writing 
    \begin{talign}
    	\hin(x_1,x_2,x_3,x_4) &= \hinY(y_1,y_2,y_3,y_4) \hinZ(z_1,z_2,z_3,z_4), \qtext{for} \\
        \hinY(y_1,y_2,y_3,y_4) &\defeq g_Y(y_1,y_2) + g_Y(y_3,y_4) - g_Y(y_1,y_3) - g_Y(y_2,y_4) \qtext{and} \\
        \hinZ(z_1,z_2,z_3,z_4) &\defeq g_Z(z_1,z_2) + g_Z(z_3,z_4) - g_Z(z_1,z_3) - g_Z(z_2,z_4).
    \end{talign}
    By Cauchy-Schwarz, 
    \begin{talign} 
    \begin{split} \label{eq:indep_V_stat_variance}
        &\mE [\Var (V_n^{\pi,n} | \mbb{X}) ] = \mE [\mE [(V_n^{\pi,n})^2 | \mbb{X}] - (\mE[V_n^{\pi,n} | \mbb{X}])^2 ] \leq \mE [(V_n^{\pi,n})^2 ] - (\mE[V_n^{\pi,n}])^2 
        \\ &~=~  \frac{1}{n^8}  \sum_{\mathfrak{i} \in \mathsf{K}_{\text{total}}} \mE \big[ \mathfrak{H}_Y\mathfrak{H}'_Y \hinZ(Z_{\pi_{i_1}},Z_{\pi_{i_2}},Z_{\pi_{i_3}},Z_{\pi_{i_4}})\hinZ(Z_{\pi_{i_1'}},Z_{\pi_{i_2'}},Z_{\pi_{i_3'}},Z_{\pi_{i_4'}}) \big] \\ &\qquad\qquad - \mE \big[ \mathfrak{H}_Y  \hinZ(Z_{\pi_{i_1}},Z_{\pi_{i_2}},Z_{\pi_{i_3}},Z_{\pi_{i_4}})  \big]  \mE \big[ \mathfrak{H}'_Y \hinZ(Z_{\pi_{i_1'}},Z_{\pi_{i_2'}},Z_{\pi_{i_3'}},Z_{\pi_{i_4'}}) \big] \\
    	&~= ~ (I) + (II) + (III),
    \end{split}
    \end{talign}
    where $\mathfrak{H}_Y \defeq \hinY(Y_{i_1},Y_{i_2},Y_{i_3},Y_{i_4})$, $\mathfrak{H}'_Y \defeq \hinY(Y_{i'_1},Y_{i'_2},Y_{i'_3},Y_{i'_4})$, and $(I),(II),(III)$ are the summations over $\mathfrak{i} \in \mathsf{K}_{1}$,  $\mathfrak{i} \in \mathsf{K}_{2}$, and  $\mathfrak{i} \in (\mathsf{K}_{1} \cup \mathsf{K}_{2})^c$ respectively.

    Note that for $\mathfrak{i} \in \mathsf{K}_{1}$,
    \begin{talign}
    \begin{split}
        &\mE \big[ \mathfrak{H}_Y \mathfrak{H}_Y'  \hinZ(Z_{\pi_{i_1}},Z_{\pi_{i_2}},Z_{\pi_{i_3}},Z_{\pi_{i_4}})\hinZ(Z_{\pi_{i_1'}},Z_{\pi_{i_2'}},Z_{\pi_{i_3'}},Z_{\pi_{i_4'}}) \big]
        \\ &= \mE \big[ \mathfrak{H}_Y \mathfrak{H}_Y'  \hinZ(Z_{\pi_{i_1}},Z_{\pi_{i_2}},Z_{\pi_{i_3}},Z_{\pi_{i_4}})   \hinZ(Z_{\pi_{i_1'}},Z_{\pi_{i_2'}},Z_{\pi_{i_3'}},Z_{\pi_{i_4'}}) \big]
        \\ &= \mE \big[ \mathfrak{H}_Y  \hinZ(Z_{\pi_{i_1}},Z_{\pi_{i_2}},Z_{\pi_{i_3}},Z_{\pi_{i_4}})  \big] \mE \big[ \mathfrak{H}_Y'  \hinZ(Z_{\pi_{i_1'}},Z_{\pi_{i_2'}},Z_{\pi_{i_3'}},Z_{\pi_{i_4'}}) \big],  %
    \end{split}
    \end{talign}
    so that $(I) = 0$.
    Also, note that 
    \begin{talign}
        |\mathsf{K}_{2}| \leq n^4 \cdot 2 \cdot 4 \cdot 4 \cdot n^3 = 32 n^7 \qtext{and}
        |(\mathsf{K}_{1} \cup \mathsf{K}_{2})^c| \leq n^4 \cdot 2 \cdot 4 \cdot 3 \cdot 2 \cdot 4 \cdot 3 \cdot n^2 = 576 n^6.
    \end{talign}
    
For $i,j \in [4]$, let us define $\bar{\mathsf{K}}_{2,i,j}$ as the set of indices $\mathfrak{i} \in \mathsf{K}_{2}$ with $|\mathfrak{i} | = 7$ and either $i_i = i'_j$ or $i_i = n\!+\!1\!-\!i'_j$. 
We further define $\bar{\mathsf{K}}_{2}$ as the union of these 16 sets. 
Note that $\mathsf{K}_{2} \setminus \tilde{\mathsf{K}}_{2}$ can be characterized as the set of $\mathfrak{i} \in \mathsf{K}_{2}$ such that, for exactly three indices $i, j, k$ of the eight indices, the pairs $\{ i, n\!+\!1\!-\!i\}$, $\{ j, n\!+\!1\!-\!j\}$ and $\{ k, n\!+\!1\!-\!k\}$ are equal. This characterization allows us to write 
\begin{talign}
    |\mathsf{K}_{2} \setminus \bar{\mathsf{K}}_{2}| = 4 \cdot 4 \cdot 6 \cdot 2^2 \cdot n^6 = 384 n^6. %
\end{talign}
Note that if $\mathfrak{i} \in \bar{\mathsf{K}}_{2}$ with $i_1 = i'_1$, then the corresponding summand from \cref{eq:indep_V_stat_variance} equals %
\begin{talign}
    &\mE \big[ \big( \hinY(Y_{1},Y_{2},Y_{3},Y_{4})  \hinZ(Z_{\pi_{1}},Z_{\pi_{2}},Z_{\pi_{3}},Z_{\pi_{4}}) \\ &\qquad\qquad - \E \big[ \hinY(Y_{1},Y_{2},Y_{3},Y_{4})  \hinZ(Z_{\pi_{1}},Z_{\pi_{2}},Z_{\pi_{3}},Z_{\pi_{4}}) \big] \big) \\ &\qquad \times \big( \hinY(Y_{1},Y_{5},Y_{6},Y_{7})  \hinZ(Z_{\pi_{1}},Z_{\pi_{5}},Z_{\pi_{6}},Z_{\pi_{7}}) \\ &\qquad\qquad - \E \big[ \hinY(Y_{1},Y_{5},Y_{6},Y_{7})  \hinZ(Z_{\pi_{1}},Z_{\pi_{5}},Z_{\pi_{6}},Z_{\pi_{7}}) \big] \big) \big]
    \\ &= \E \big[ \Var(  \hinY(Y_{1},Y_{2},Y_{3},Y_{4})  \hinZ(Z_{\pi_{1}},Z_{\pi_{2}},Z_{\pi_{3}},Z_{\pi_{4}}) \mid Y_1, \pi_1, Z_{\pi_1}) %
    \big] = \tilde{\psi}'_1,
\end{talign}
where the last equality holds since the distribution of $\hbarin((Y_1,Z_{\pi_1}),(Y_2,Z_{\pi_2}),(Y_3,Z_{\pi_3}),(Y_4,Z_{\pi_4}))$ conditional on $(Y_1,Z_{\pi_1}, \pi_1)$ is invariant to the order of the four arguments of $\hbarin$ and hence %
\begin{talign}
    &\E[\hin((Y_1,Z_{\pi_1}),(Y_2,Z_{\pi_2}),(Y_3,Z_{\pi_3}),(Y_4,Z_{\pi_4})) \mid (Y_1,Z_{\pi_1}, \pi_1)] \\ &= \E[\hbarin((Y_1,Z_{\pi_1}),(Y_2,Z_{\pi_2}),(Y_3,Z_{\pi_3}),(Y_4,Z_{\pi_4})) \mid (Y_1,Z_{\pi_1}, \pi_1)].
\end{talign}

Similarly, if $\mathfrak{i} \in \bar{\mathsf{K}}_{2}$ with $i_1 = n+1-i'_1$, 
then the corresponding summand of \cref{eq:indep_V_stat_variance} equals%
\begin{talign}
    &\mE \big[ \big( \hinY(Y_{1},Y_{2},Y_{3},Y_{4})  \hinZ(Z_{\pi_{1}},Z_{\pi_{2}},Z_{\pi_{3}},Z_{\pi_{4}}) \\ 
    &\qquad\qquad - \E \big[ \hinY(Y_{1},Y_{2},Y_{3},Y_{4})  \hinZ(Z_{\pi_{1}},Z_{\pi_{2}},Z_{\pi_{3}},Z_{\pi_{4}}) \big] \big) \\ &\qquad \times \big( \hinY(Y_{n},Y_{5},Y_{6},Y_{7})  \hinZ(Z_{\pi_{n}},Z_{\pi_{5}},Z_{\pi_{6}},Z_{\pi_{7}}) \\ &\qquad\qquad - \mE \big[ \hinY(Y_{n},Y_{5},Y_{6},Y_{7})  \hinZ(Z_{\pi_{n}},Z_{\pi_{5}},Z_{\pi_{6}},Z_{\pi_{7}}) \big] \big) \big]
    \\ &\leq \half\E[\Var(\hinY(Y_{1},Y_{2},Y_{3},Y_{4})  \hinZ(Z_{\pi_{1}},Z_{\pi_{2}},Z_{\pi_{3}},Z_{\pi_{4}})\mid Y_1,\pi_1, Z_{\pi_1})] 
    \\ &+ \half\E[\Var(\hinY(Y_{n},Y_{5},Y_{6},Y_{7})  \hinZ(Z_{\pi_{n}},Z_{\pi_{5}},Z_{\pi_{6}},Z_{\pi_{7}})\mid Y_n,\pi_n, Z_{\pi_n})]
    = \tilde{\psi}'_1,
\end{talign}
by Cauchy-Schwarz and the arithmetic-geometric mean inequality.

Moreover, for any $\mathfrak{i} \in (\mathsf{K}_{1} \cup \mathsf{K}_{2})^c \cup (\mathsf{K}_{2} \setminus \bar{\mathsf{K}}_{2})$, the corresponding summand of \cref{eq:indep_V_stat_variance} equals
\begin{talign} \label{eq:psi_prime_upper_in_2}
    &\mE \big[ \big( \hinY(Y_{i_1},Y_{i_2},Y_{i_3},Y_{i_4})  \hinZ(Z_{\pi_{i_1}},Z_{\pi_{i_2}},Z_{\pi_{i_3}},Z_{\pi_{i_4}}) \\ &\qquad\qquad - \mE \big[ \hinY(Y_{i_1},Y_{i_2},Y_{i_3},Y_{i_4})  \hinZ(Z_{\pi_{i_1}},Z_{\pi_{i_2}},Z_{\pi_{i_3}},Z_{\pi_{i_4}}) \big] \big) \\ &\qquad \times \big( \hinY(Y_{i_1'},Y_{i_2'},Y_{i_3'},Y_{i_4'}) \hinZ(Z_{\pi_{i_1'}},Z_{\pi_{i_2'}},Z_{\pi_{i_3'}},Z_{\pi_{i_4'}}) \\ &\qquad\qquad - \mE \big[ \hinY(Y_{i_1'},Y_{i_2'},Y_{i_3'},Y_{i_4'}) \hinZ(Z_{\pi_{i_1'}},Z_{\pi_{i_2'}},Z_{\pi_{i_3'}},Z_{\pi_{i_4'}}) \big] \big) \big] \\
    &\leq \big( \mE \big[ \big( \hinY(Y_{i_1},Y_{i_2},Y_{i_3},Y_{i_4})  \hinZ(Z_{\pi_{i_1}},Z_{\pi_{i_2}},Z_{\pi_{i_3}},Z_{\pi_{i_4}}) \big)^2 \big] \big)^{1/2} \\ &\qquad \times \big( \mE \big[ \big( \hinY(Y_{i_1'},Y_{i_2'},Y_{i_3'},Y_{i_4'}) \hinZ(Z_{\pi_{i_1'}},Z_{\pi_{i_2'}},Z_{\pi_{i_3'}},Z_{\pi_{i_4'}}) %
    \big)^2 \big] \big)^{1/2} \leq \psi_2'
\end{talign}
by Cauchy-Schwarz. 
Thus, $(II)$ is upper-bounded by $\psi_2' \frac{|\mathsf{K}_{2} \setminus \bar{\mathsf{K}}_{2}|}{n^8} + \tilde{\psi}_1' \frac{|\bar{\mathsf{K}}_{2}|}{n^8}$, %
and we find that
\begin{talign}
        \Var(V_n^{\pi,n}) &\leq \psi_2' \frac{|(\mathsf{K}_{1} \cup \mathsf{K}_{2})^c| + |\mathsf{K}_{2} \setminus \bar{\mathsf{K}}_{2}|}{n^8} + \tilde{\psi}_1' \frac{|\bar{\mathsf{K}}_{2}|}{n^8} 
        \leq \frac{960 \psi_2'}{n^2} + \frac{32 \tilde{\psi}_1'}{n}.
    \end{talign}
Finally, by Jensen's inequality, 
\begin{talign}
    \Var(\mathbb{E}[V_n^{\pi,n} | \mathbb{X}]) &= \mE [ (\mathbb{E}[V_n^{\pi,n} | \mathbb{X}])^2 ] - (\mE [ \mathbb{E}[V_n^{\pi,n} | \mathbb{X}] ])^2 \\ &\leq \mE [ \mathbb{E}[(V_n^{\pi,n})^2 | \mathbb{X}] ] - (\mE [ \mathbb{E}[V_n^{\pi,n} | \mathbb{X}] ])^2 
    = \Var(V_n^{\pi,n}).
\end{talign}
\subsubsection{Proof of \cref{eq:exp_var_cheap_ind_bound}}
Our argument is analogous to that of \cref{proof:exp_var_full_ind_bound} with $s$ substituted for $n$ and $\Hin$ substituted for $\hin$.
We begin by defining the convenient shorthand 
\begin{talign}
\mathfrak{s}(\mathfrak{i}) &\defeq |\{i_1,i_2,i_3,i_4,s\!+\!1\!-\!i_1,s\!+\!1\!-\!i_2,s\!+\!1\!-\!i_3,s\!+\!1\!-\!i_4\}\cap \{i_1',i_2',i_3',i_4'\} |,\\
    \mathsf{K}_{\text{total}} &\defeq [s]^8, \quad
    \mathsf{K}_{1} \defeq \{ \mathfrak{i} \in \mathsf{K}_{\text{total}} :  \mathfrak{s}(\mathfrak{i}) = 0 \},  \quad
    \mathsf{K}_{2} \defeq \{ \mathfrak{i} \in \mathsf{K}_{\text{total}} : \mathfrak{s}(\mathfrak{i}) = 1 \}, \qtext{and} \\
    (\mathsf{K}_{1} \cup \mathsf{K}_{2})^c &\defeq \{ \mathfrak{i} \in \mathsf{K}_{\text{total}} : \mathfrak{s}(\mathfrak{i}) > 1 \},
\end{talign}
and define $\bar{K}_2$ as the set of indices $\mathfrak{i} \in \mathsf{K}_{2}$ with $|\mathfrak{i} | = 7$ and either $i_i = i'_j$ or $i_i = n\!+\!1\!-\!i'_j$ for some $i,j \in [4]$. 
As in \cref{proof:exp_var_full_ind_bound} case we have 
\begin{talign}
    &|\mathsf{K}_{2}| \leq 32 s^7, %
    \qquad |(\mathsf{K}_{1} \cup \mathsf{K}_{2})^c| \leq 576 s^6,
    \qquad |\mathsf{K}_{2} \setminus \bar{\mathsf{K}}_{2}| = 384 s^6.
\end{talign}

Next, we define
$\tilde{\mathsf{K}}_{2}$ as the set of $\mathfrak{i} \in \mathsf{K}_{2}$ such that for exactly three indices $i, j, k$ of the eight indices, the pairs $\{ i, s\!+\!1\!-\!i\}$, $\{ j, s\!+\!1\!-\!j\},$ and $\{ k, s\!+\!1\!-\!k\}$ are equal, and the rest of such pairs are all different. We then have 
\begin{talign}
    |\tilde{\mathsf{K}}_{2}| &= 4 \cdot 4 \cdot 6 \cdot 2^2 \cdot s (s-2) (s-4) (s-6) (s-8) (s-10) \\ &= 384 s (s-2) (s-4) (s-6) (s-8) (s-10) \qtext{and} \\ %
    |\mathsf{K}_{2} \setminus (\bar{\mathsf{K}}_{2} \cup \tilde{\mathsf{K}}_{2})| &= |\mathsf{K}_{2} \setminus \bar{\mathsf{K}}_{2}| - |\tilde{\mathsf{K}}_{2}| \\ &\leq 384 (s^6 - s (s-2) (s-4) (s-6) (s-8) (s-10)) \leq 384 \cdot 5 \cdot 6 s^5 = 11520 s^5.
\end{talign}

We next define $\mathsf{K}_{3}$ as the set of $\mathfrak{i} \in (\mathsf{K}_{1} \cup \mathsf{K}_{2})^c$ such that for exactly four indices $i, j, k, r$ of the eight indices, we have that $\{ i, s\!+\!1\!-\!i\} = \{ j, s\!+\!1\!-\!j\}$ and $\{ k, s\!+\!1\!-\!k\} = \{ r, s\!+\!1\!-\!r\}$, and the rest of such pairs are all distinct. Note that 
\begin{talign}
    &|\mathsf{K}_{3}| = 4 \cdot 3 \cdot 2 \cdot 4 \cdot 3 \cdot 2 s (s-2) (s-4) (s-6) (s-8) (s-10) \\ 
    &\quad= 576 s (s-2) (s-4) (s-6) (s-8) (s-10) \qtext{and}\\
    &|(\mathsf{K}_{1} \cup \mathsf{K}_{2} \cup \mathsf{K}_{3})^c| = |(\mathsf{K}_{1} \cup \mathsf{K}_{2})^c| - |\mathsf{K}_{3}| \\ 
    &\quad\leq 576 (s^6 - s (s-2) (s-4) (s-6) (s-8) (s-10)) \leq 576 \cdot 5 \cdot 6 \cdot s^5 = 17280 s^5.
\end{talign}

Finally, for $\mathfrak{K} \defeq (k_1,k_2,k_3,k_4,k'_1,k'_2,k'_3,k'_4)$ and $\mathfrak{s}(\mathfrak{K})\defeq |\{k_1,k_2,k_3,k_4\}\cap \{k_1',k_2',k_3',k_4'\} |$, we define the index sets
\begin{talign}
\mathsf{L}_{\text{total}} &= [m]^8, \quad
\mathsf{L}_{1} = \{ \mathfrak{K} \in \mathsf{L}_{\text{total}} : %
\mathfrak{s}(\mathfrak{K}) = 0 \}, 
\quad
\mathsf{L}_{2} = \{ \mathfrak{K} \in \mathsf{L}_{\text{total}} : \mathfrak{s}(\mathfrak{K})  = 1\}, 
\qtext{and}\\
(\mathsf{L}_{1} \cup \mathsf{L}_{2})^c &= \{ \mathfrak{K} \in \mathsf{L}_{\text{total}} : \mathfrak{s}(\mathfrak{K}) >  1\} 
\end{talign}
and note that 
\begin{talign}
    |\mathsf{L}_{2}| &\leq m^4 \cdot 4 \cdot 4 \cdot m^3 = 16 m^7
    \qtext{and} %
    |(\mathsf{L}_{1} \cup \mathsf{L}_{2})^c| \leq m^4 \cdot 4 \cdot 3 \cdot 4 \cdot 3 \cdot m^2 = 144 m^6.
\end{talign}

With this notation in hand, we use Cauchy-Schwarz to write 
\begin{talign} 
\begin{split} \label{eq:E_Var_indep_cheap}
    &\mE [\Var (V_n^{\pi,s} | \mbb{X}) ] 
    = \mE [ \mE [(V_n^{\pi,s})^2 | \mbb{X}] - (\mE[V_n^{\pi,s} | \mbb{X}])^2 ] 
    \leq \mE [(V_n^{\pi,s})^2] - (\mE[V_n^{\pi,s} ])^2 \\
    &~=~  \frac{1}{n^8}  \sum_{\mathfrak{i} \in \mathsf{K}_{\text{total}}} \sum_{\mathfrak{K} \in \mathsf{L}_{\text{total}}} \bigg( \mE \big[ \hinY(Y_{k_1}^{(i_1)},Y_{k_2}^{(i_2)},Y_{k_3}^{(i_3)},Y_{k_4}^{(i_4)})\hinY(Y_{k'_1}^{(i'_1)},Y_{k'_2}^{(i'_2)},Y_{k'_3}^{(i'_3)},Y_{k'_4}^{(i'_4)}) \\
    &\qquad \times  \hinZ(Z_{k_1}^{(\pi_{i_1})},Z_{k_2}^{(\pi_{i_2})},Z_{k_3}^{(\pi_{i_3})},Z_{k_4}^{(\pi_{i_4})})\hinZ(Z_{k'_1}^{(\pi_{i'_1})},Z_{k'_2}^{(\pi_{i'_2})},Z_{k'_3}^{(\pi_{i'_3})},Z_{k'_4}^{(\pi_{i'_4})})  \big] \\ &\qquad\qquad - \mE \big[ \hinY(Y_{k_1}^{(i_1)},Y_{k_2}^{(i_2)},Y_{k_3}^{(i_3)},Y_{k_4}^{(i_4)})  \hinZ(Z_{k_1}^{(\pi_{i_1})},Z_{k_2}^{(\pi_{i_2})},Z_{k_3}^{(\pi_{i_3})},Z_{k_4}^{(\pi_{i_4})})   \big] \\ &\qquad\qquad\quad \times \mE \big[ \hinY(Y_{k'_1}^{(i'_1)},Y_{k'_2}^{(i'_2)},Y_{k'_3}^{(i'_3)},Y_{k'_4}^{(i'_4)})  \hinZ(Z_{k'_1}^{(\pi_{i'_1})},Z_{k'_2}^{(\pi_{i'_2})},Z_{k'_3}^{(\pi_{i'_3})},Z_{k'_4}^{(\pi_{i'_4})})  \big] \bigg) \\[.5em]
    &~= ~ (I') + (II') + (III'),
\end{split}
\end{talign}
where $(I')$ is the summation over $\mathfrak{i} \in \mathsf{K}_{1}$, $(II')$ is the summation over $\mathfrak{i} \in \mathsf{K}_{2}$, %
and $(III')$ is the summation over $\mathfrak{i} \in (\mathsf{K}_{1} \cup \mathsf{K}_{2})^c$. Note that $(I') = 0$ by the same argument that we used to show that $(I) = 0$ in \eqref{eq:indep_V_stat_variance}. 

To bound $(II')$, we will further divide the corresponding summands into subcases:
\begin{itemize}
    \item  $\mathfrak{i} \in \bar{\mathsf{K}}_{2}$: Suppose without loss of generality that $i_1 = i'_1$. The only non-zero summands in the sum over $\mathfrak{K}$ have $k_1 = k'_1$. There are $m^7$ such summands, and each contributes at most $\tilde{\psi}'_1$ to the sum.
    \item $\mathfrak{i} \in \tilde{\mathsf{K}}_{2}$: 
    Consider the three indices of $\mathfrak{i}$ that satisfy the defining $\tilde{\mathsf{K}}_{2}$ property. In the summation over $\mathfrak{K} \in \mathsf{L}_{\text{total}}$, a summand contributes at most $\tilde{\psi}'_1$ when exactly two of the three $\mathfrak{i}$ indices share the same $k$-index, and there are at most $3m^7$ such summands. The only other non-zero summands contribute at most $\psi_2'$ when all three $\mathfrak{i}$ indices share the same $k$-index, and there are are most $m^6$ such summands.
    \item $\mathfrak{i} \in \mathsf{K}_{2} \setminus (\bar{\mathsf{K}}_{2} \cup \tilde{\mathsf{K}}_{2})$: Considering the summation over $\mathfrak{K} \in \mathsf{L}_{\text{total}}$, there will be a contribution bounded by $\tilde{\psi}'_1$ when $\mathfrak{K} \in \mathsf{L}_{2}$, and there are at most  $16m^7$ such summands, The only other non-zero contributions are bounded by $\psi_2'$ and occur when $\mathfrak{K} \in (\mathsf{L}_{1} \cup \mathsf{L}_{2})^c$, so the number of such summands is bounded by $144m^6$.
\end{itemize}
Thus, we have that $(II')$ is upper-bounded by
\begin{talign}
     &\tilde{\psi}'_1 \frac{|\bar{\mathsf{K}}_{2}| m^7}{n^8} + \tilde{\psi}'_1 \frac{|\tilde{\mathsf{K}}_{2}| \cdot 3 m^7}{n^8} + \tilde{\psi}'_1 \frac{|\mathsf{K}_{2} \setminus (\bar{\mathsf{K}}_{2} \cup \tilde{\mathsf{K}}_{2})| \cdot 16 m^7}{n^8} + \psi_2' \frac{|\tilde{\mathsf{K}}_{2}| \cdot m^6}{n^8} + \psi_2' \frac{|\mathsf{K}_{2} \setminus (\bar{\mathsf{K}}_{2} \cup \tilde{\mathsf{K}}_{2})| \cdot |(\mathsf{L}_{1} \cup \mathsf{L}_{2})^c|}{n^8}
    \\ &\leq \tilde{\psi}_1' \big( \frac{32s^7 m^7}{n^8} + \frac{384 s^6 \cdot 3 m^7}{n^8} + \frac{11520 s^5 \cdot 16 m^7}{n^8} \big) + \psi_2' \big( \frac{384 s^6 m^6}{n^8} + \frac{11520 s^5 \cdot 144 m^6}{n^8} \big).
\end{talign}

Similarly, to bound $(III')$, we divide the corresponding summands into subcases:
\begin{itemize}
    \item $\mathfrak{i} \in \mathsf{K}_{3}$: 
    Consider the two pairs of indices of $\mathfrak{i}$ satisfying the defining property of $\mathsf{K}_{3}$.
    When each pair respectively shares the same $k$-index, the summand contributes at most $\psi'_2$, and there are at most $m^6$ such summands.
    The only other non-zero summands have exactly one of the two pairs sharing a $k$-index; these summands contribute at most  $\tilde{\psi}'_1$, and there are at most $2m^7$ such summands.
    \item $\mathfrak{i} \in (\mathsf{K}_{1} \cup \mathsf{K}_{2} \cup \mathsf{K}_{3})^c$: In the summation over $\mathfrak{K} \in \mathsf{L}_{\text{total}}$, a summand will contribute at most $\tilde{\psi}'_1$ when $\mathfrak{K} \in \mathsf{L}_{2}$, and there are at most $16m^7$ such summands. The only other non-zero summands are bounded by $\psi'_2$ and occur when $\mathfrak{K} \in (\mathsf{L}_{1}\cup\mathsf{L}_{2})^c$, so there are at most $144m^6$ such summands.
\end{itemize}
Thus, $(III')$ is upper-bounded by
\begin{talign} 
    &\tilde{\psi}'_1 \frac{|\mathsf{K}_{3}|\cdot2m^7}{n^8} + \tilde{\psi}'_1 \frac{|(\mathsf{K}_{1} \cup \mathsf{K}_{2} \cup \mathsf{K}_{3})^c|\cdot|\mathsf{L}_{2}|}{n^8}
    + \psi_2' \frac{|\mathsf{K}_{3}| \cdot m^6}{n^8} + \psi_2' \frac{|(\mathsf{K}_{1} \cup \mathsf{K}_{2} \cup \mathsf{K}_{3})^c| \cdot |(\mathsf{L}_{1} \cup \mathsf{L}_{2})^c|}{n^8} \\ &\leq \tilde{\psi}'_1 \big(\frac{576 s^6 \cdot 2 m^7}{n^8} + \frac{17280 s^5 \cdot 16 m^7}{n^8} \big) + \psi_2' \big( \frac{576 s^6 \cdot m^6}{n^8} + \frac{17280 s^5 \cdot 144m^6}{n^8} \big)
\end{talign}
We conclude that 
\begin{talign}
    \Var(V_n^{\pi,s})
    &\leq %
    \tilde{\psi}_1' \big( \frac{32}{n} + \frac{2304}{n s} + \frac{460800}{n s^2} \big) + \psi_2' \big( \frac{960}{n^2} + \frac{4147200}{n^2 s} \big). 
\end{talign}
Finally, by Jensen's inequality, 
\begin{talign}
    \Var(\mathbb{E}[V_n^{\pi,s} | \mathbb{X}]) &= \mE [ (\mathbb{E}[V_n^{\pi,s} | \mathbb{X}])^2 ] - (\mE [ \mathbb{E}[V_n^{\pi,s} | \mathbb{X}] ])^2 \\ &\leq \mE [ \mathbb{E}[(V_n^{\pi,s})^2 | \mathbb{X}] ] - (\mE [ \mathbb{E}[V_n^{\pi,s} | \mathbb{X}] ])^2 
    = \Var(V_n^{\pi,s}).
\end{talign}

\subsection{\pcref{lem:exp_exp_variance_exp_ind}}
\label{subsec:exp_exp_variance_exp_ind_proof}
\subsubsection{Proof of \cref{eq:Vnpin-mean}} Define the sets of indices
\begin{talign}
\begin{split} \label{eq:J_total_1_2}
    \mathsf{J}_{\text{total}} &\defeq [n]^4, \\
    \mathsf{J}_{1} &\defeq \{ (i_1,i_2,i_3,i_4) \in \mathsf{J}_{\text{total}} : \\ &\qquad |\{i_1,i_2,i_3,i_4,\swap(i_1),\swap(i_2),\swap(i_3),\swap(i_4)\} | = 8 \},\\
    \mathsf{J}_{2} &\defeq \{ (i_1,i_2,i_3,i_4) \in \mathsf{J}_{\text{total}} : \ (i_1 = i_4) \vee (i_2 = i_3) \},
\end{split}
\end{talign}
and note that 
$|\mathsf{J}_{1}| = n(n-2)(n-4)(n-6)$, 
$|\mathsf{J}_{2}| = 2n^3$, and 
\begin{talign}
|(\mathsf{J}_{1} \cup \mathsf{J}_{2})^c| = n^4 - 
(n^4 - 12n^3 + 44 n^2 - 48 n)
- 2n^3 = 
10n^3 - 44 n^2 + 48n.
\end{talign} 
Recalling the definitions \eqref{eq:h_in_Y_Z_def_2} of $\hinY$ and $\hinZ$, we have 
\begin{talign}
    \mE [V_n^{\pi,n}] 
    &= \frac{1}{n^4} \sum_{(i_1, \dots, i_4) \in [n]^4} \mE \big[ \hinY(Y_{i_1},Y_{i_2},Y_{i_3},Y_{i_4})  \hinZ(Z_{\pi_{i_1}},Z_{\pi_{i_2}},Z_{\pi_{i_3}},Z_{\pi_{i_4}}) \big] \\ &= \frac{1}{n^4} \sum_{(i_1, \dots, i_4) \in [n]^4}  \mE \big[ \mE \big[ \hinY(Y_{i_1},Y_{i_2},Y_{i_3},Y_{i_4}) \hinZ(Z_{\pi_{i_1}},Z_{\pi_{i_2}},Z_{\pi_{i_3}},Z_{\pi_{i_4}}) \mid \pi \big] \big] \\ &\leq \frac
    {n(n-2)(n-4)(n-6)}
    {n^4} \mE[\mathcal{U}_{\pi}] + \frac
    {10n^3 - 44 n^2 + 48n}
    {n^4} \tilde{\xi} \leq \mE[\mathcal{U}_{\pi}] + 
    \frac{10n^2 - 44 n + 48}{n^3}
    \tilde{\xi} \qtext{for}\\
\mathcal{U}_{\pi} 
    &\defeq \mE \big[\hinY(Y_{1},Y_{2},Y_{3},Y_{4}) \hinZ(Z_{\pi_{1}},Z_{\pi_{2}},Z_{\pi_{3}},Z_{\pi_{4}}) \mid \pi \big]. 
\end{talign}
The penultimate inequality follows from the following facts:
\begin{itemize}
\item For $(i_1, \dots, i_4) \in \mathsf{J}_{1}$, $\mE \big[ \hinY(Y_{i_1},Y_{i_2},Y_{i_3},Y_{i_4}) \hinZ(Z_{\pi_{i_1}},Z_{\pi_{i_2}},Z_{\pi_{i_3}},Z_{\pi_{i_4}}) | \pi \big] = \mathcal{U}_{\pi}$.
\item For $(i_1, \dots, i_4) \in \mathsf{J}_{2}$, $\mE \big[ \hinY(Y_{i_1},Y_{i_2},Y_{i_3},Y_{i_4}) \hinZ(Z_{\pi_{i_1}},Z_{\pi_{i_2}},Z_{\pi_{i_3}},Z_{\pi_{i_4}}) | \pi \big] = 0$.
\item For $(i_1, \dots, i_4) \in (\mathsf{J}_{1} \cup \mathsf{J}_{2})^c$, $\mE \big[ \hinY(Y_{i_1},Y_{i_2},Y_{i_3},Y_{i_4}) \hinZ(Z_{\pi_{i_1}},Z_{\pi_{i_2}},Z_{\pi_{i_3}},Z_{\pi_{i_4}}) | \pi \big] \leq \tilde{\xi}$, where $\tilde{\xi}$ is defined in \eqref{Eq: mean component}.
\end{itemize}

Finally, equation \cref{eq:E_V_pi} of the following lemma %
shows that $\E[\mathcal{U}_{\pi}] = \quarter\E[\mathcal{U}]$ completing the proof.
\begin{lemma}[Range and expectation of $\mathcal{U}_{\pi}$] \label{lem:bar_k_bound}
Instantiate the assumptions of \cref{thm: Full Independence Tests}, fix 
any $i_1, i_2, i_3, i_4 \in [n]$ satisfying 
\begin{talign}
|\{i_1,i_2,i_3,i_4,\swap(i_1), \swap(i_2), \swap(i_3), \swap(i_4)\}| = 8,   
\end{talign}
and let $\pi$ be any permutation on $[n]$ with 
\begin{talign}
(\pi_{i_j},\pi_{\swap(i_j)}) \dist \Unif(\{ (i_j, \swap(i_j)), (\swap(i_j), i_j)\})
\qtext{for each}
j\in[4].
\end{talign}
Then we always have
\begin{talign} 
\begin{split}
    \mathcal{U}_{\pi} 
    &\defeq \mE \big[\hinY(Y_{i_1},Y_{i_2},Y_{i_3},Y_{i_4}) \hinZ(Z_{\pi_{i_1}},Z_{\pi_{i_2}},Z_{\pi_{i_3}},Z_{\pi_{i_4}}) | \pi \big] \\ 
    &\in 
    \mE \big[\hin(X_{i_1},X_{i_2},X_{i_3},X_{i_4}) ]\{0, \frac{1}{4}, \frac{1}{2}, 1\}
\label{eq:V_pi}
\end{split}
\end{talign}
with
\begin{talign}
\label{eq:E_V_pi_5}
    \mE[\mathcal{U}_{\pi}] &\leq \frac{1}{2} \mE \big[\hin(X_{i_1},X_{i_2},X_{i_3},X_{i_4})].    
\end{talign}
If ($\pi_{i_1}$, $\pi_{i_2}$, $\pi_{i_3}$, $\pi_{i_4}$) are independent, then
\begin{talign}
    \mE[\mathcal{U}_{\pi}] &= \frac{1}{4} \mE \big[\hin(X_{i_1},X_{i_2},X_{i_3},X_{i_4})]. %
    \label{eq:E_V_pi}
\end{talign} 
If $\pi_{i_1} = i_1\iff\pi_{i_4} = i_4$ 
with independent $(\pi_{i_1},\pi_{i_2},\pi_{i_3})$ or $\pi_{i_2} = i_2 \iff \pi_{i_3} = i_3$ 
with   independent $(\pi_{i_1},\pi_{i_2},\pi_{i_4})$, then
\begin{talign} \
    \mE[\mathcal{U}_{\pi}] &= \frac{1}{4} \mE \big[\hin(X_{i_1},X_{i_2},X_{i_3},X_{i_4})]. %
    \label{eq:E_V_pi_2}
\end{talign}
If $\pi_{i_1} = i_1\iff\pi_{i_2} = i_2$ with independent $(\pi_{i_1},\pi_{i_3},\pi_{i_4})$; $\pi_{i_1} = i_1\iff\pi_{i_3} = i_3$ with independent $(\pi_{i_1},\pi_{i_2},\pi_{i_4})$; $\pi_{i_2} = i_2\iff\pi_{i_4} = i_4$ with independent $(\pi_{i_1},\pi_{i_2},\pi_{i_3})$; or $\pi_{i_3} = i_3\iff\pi_{i_4} = i_4$ with independent $(\pi_{i_1},\pi_{i_2},\pi_{i_3})$, then
\begin{talign} \
    \mE[\mathcal{U}_{\pi}] &= \frac{3}{8} \mE \big[\hin(X_{i_1},X_{i_2},X_{i_3},X_{i_4})]. %
    \label{eq:E_V_pi_3}
\end{talign}
\end{lemma}
\begin{proof}
We will use the shorthand
\begin{talign}
P_{i,j} \defeq 
\begin{cases}
    \pjnt &\text{if } i = j \\
    \P \times \Q &\text{otherwise},
\end{cases}
\end{talign}
and the notation $(\P^{(i)},\Q^{(i)})$ to denote copies of $(\P,\Q)$ while keeping track of indices to make the derivation easier to follow.
We begin by proving \cref{eq:V_pi}. 
$\mathcal{U}_{\pi}$ may be rewritten as
\begin{talign} 
\begin{split} \label{eq:exp_P_inner_prod}
\mathcal{U}_{\pi} 
&= \iint 
    g_Y(y,y') g_Z(z,z') \, d(P_{i_1,\pi_{i_1}} + P_{i_4,\pi_{i_4}} - \P \times \Q^{(\pi_{i_4})} - \P \times \Q^{(\pi_{i_1})})(y,z) \\ &\qquad\qquad\qquad d(P_{i_2,\pi_{i_2}} + P_{i_3,\pi_{i_3}} - \P \times \Q^{(\pi_{i_3})} - \P \times \Q^{(\pi_{i_2})})(y',z'). %
\end{split}
\end{talign}
Thus, if we define the $\Unif(\{0,1\})$ random variables $\epsilon_{i_k} = \indic{\pi_{i_k} = i_k}$, we can write 
\begin{talign} \label{eq:U_pi_Rademacher}
\mathcal{U}_{\pi} &= \iint 
    g_Y(y,y') g_Z(z,z') \, d( \epsilon_{i_1} (\pjnt  - \P \times \Q) + \epsilon_{i_4}(\pjnt  - \P \times \Q))(y,z) \\ &\qquad\qquad\qquad d(\epsilon_{i_2} (\pjnt  - \P \times \Q) + \epsilon_{i_3}(\pjnt  - \P \times \Q))(y',z')
    \\ &= \big( \epsilon_{i_1} + \epsilon_{i_4} \big) \big( \epsilon_{i_2} + \epsilon_{i_3} \big) \iint 
    g_Y(y,y') g_Z(z,z') \, d(\pjnt - \P \times \Q )(y,z) \\ &\qquad\qquad\qquad d( \pjnt - \P \times \Q)(y',z')
    \\ &= 
\frac{1}{4}\big( \epsilon_{i_1} + \epsilon_{i_4} \big) \big( \epsilon_{i_2} + \epsilon_{i_3} \big) \mE \big[\hin(X_{1},X_{2},X_{3},X_{4})].
\end{talign}
Setting $\epsilon_{i_1}, \epsilon_{i_2}, \epsilon_{i_3}, \epsilon_{i_4}$ to $\{0,1\}$ implies that all the possible values $\mathcal{U}_{\pi}$ can take are $\{0, \frac{1}{4}, \frac{1}{2}, 1 \} \mE \big[\hin(X_{1},X_{2},X_{3},X_{4})]$.

To show \eqref{eq:E_V_pi_5}, we will prove that for any joint distribution over the $\Unif(\{0,1\})$  variables $\epsilon_{i_1}, \epsilon_{i_2}, \epsilon_{i_3}, \epsilon_{i_4}$, we have $\mathbb{E}[\big( \epsilon_{i_1} + \epsilon_{i_4} \big) \big( \epsilon_{i_2} + \epsilon_{i_3} \big)] \leq 2$. Indeed, by Cauchy-Schwarz,
\begin{talign}
    \mathbb{E}\big[\big( \epsilon_{i_1} + \epsilon_{i_4} \big) \big( \epsilon_{i_2} + \epsilon_{i_3} \big) \big] \leq \mathbb{E} \big[\big( \epsilon_{i_1} + \epsilon_{i_4} \big)^2 \big]^{1/2} \mathbb{E} \big[\big( \epsilon_{i_2} + \epsilon_{i_3} \big)^2 \big]^{1/2},
\end{talign}
and $\mathbb{E} \big[\big( \epsilon_{i_1} + \epsilon_{i_4} \big)^2 \big]$ is maximized when $\epsilon_{i_1} = \epsilon_{i_4}$ almost surely, with maximum value $0 \times \frac{1}{2} + 4 \times \frac{1}{2} = 2$. Thus, $\mathbb{E}\big[\big( \epsilon_{i_1} + \epsilon_{i_4} \big) \big( \epsilon_{i_2} + \epsilon_{i_3} \big) \big]$ has maximum value 2. %

To show \cref{eq:E_V_pi}, we compute the probability that $\mathcal{U}_{\pi}$ takes each of its possible values when  $(\pi_1,\dots,\pi_4)$ are independent: 
\begin{talign}
    &\mE[\mathcal{U}_{\pi}] = \frac{1}{2^4} \big( 0 \times {4 \choose 0} + 0 \times (2 + {4 \choose 1}) + %
    \frac{1}{4} \mE \big[\hin(X_{i_1},X_{i_2},X_{i_3},X_{i_4})] \times 4 \\ &\qquad + %
    \frac{1}{2} \mE \big[\hin(X_{i_1},X_{i_2},X_{i_3},X_{i_4})] \times {4 \choose 3} + %
    \mE \big[\hin(X_{i_1},X_{i_2},X_{i_3},X_{i_4})] \times {4 \choose 4} \big) \\ &= \frac{1}{16} %
    \big( \mE \big[\hin(X_{i_1},X_{i_2},X_{i_3},X_{i_4})] + 2 
    \mE \big[\hin(X_{i_1},X_{i_2},X_{i_3},X_{i_4})] + \mE \big[\hin(X_{i_1},X_{i_2},X_{i_3},X_{i_4})] \big)
    \\ &= \frac{1}{4} %
    \mE \big[\hin(X_{i_1},X_{i_2},X_{i_3},X_{i_4})].
\end{talign}

Similarly, when $\pi_{i_1} = i_1 \iff\pi_{i_4} = i_4$ and $(\pi_{i_1},\pi_{i_2},\pi_{i_3})$ are independent,
\begin{talign}
    \mE[\mathcal{U}_{\pi}] &= \frac{1}{2^3} \big( 0 \times 4 + 0 \times 1 + 
    \frac{1}{2} \mE \big[\hin(X_{i_1},X_{i_2},X_{i_3},X_{i_4})] \times 2 + 
    \mE \big[\hin(X_{i_1},X_{i_2},X_{i_3},X_{i_4})] \times 1 \big) \\ &= \frac{1}{4} \mE \big[\hin(X_{i_1},X_{i_2},X_{i_3},X_{i_4})].
\end{talign}

Alternatively, when 
$\pi_{i_1} = i_1\iff\pi_{i_2} = i_2$ and $(\pi_{i_1},\pi_{i_3},\pi_{i_4})$ are independent,
\begin{talign}
    \mE[\mathcal{U}_{\pi}] &= \frac{1}{2^3} \big( 0 \times 3 + 
    \frac{1}{4} \mE \big[\hin(X_{i_1},X_{i_2},X_{i_3},X_{i_4})] \times 2 + 
    \frac{1}{2} \mE \big[\hin(X_{i_1},X_{i_2},X_{i_3},X_{i_4})] \times 2 \\ &\quad + 
    \mE \big[\hin(X_{i_1},X_{i_2},X_{i_3},X_{i_4})] \times 1 \big) = \frac{3}{8} \mE \big[\hin(X_{i_1},X_{i_2},X_{i_3},X_{i_4})].
\end{talign}
\end{proof}

\subsubsection{Proof of \cref{eq:Vnpis-mean}}
In the cheap permutation case, the permutation values $\pi_{i}$, $\pi_{j}$ are not independent when $i$ and $j$ belong to the same bin, which requires a finer analysis. In the following, for $1 \leq i \leq s$, we let $\swap_s(i) \defeq i + \frac{s}{2} - s\indic{i  > \frac{s}{2}}$ and recall that $\swap(i) \defeq i + \frac{n}{2} - n\indic{i  > \frac{n}{2}}$. 
We have that 
\begin{talign}
    &\mE [V_n^{\pi,s} ] \\ &= \frac{1}{n^4} \sum_{(i_1, \dots, i_4) \in [s]^4} \sum_{(j_1,\dots,j_4) \in \pjnt_{k=1}^{4} \bin_{i_k}
    } \mE \big[ \hinY(Y_{j_1},Y_{j_2},Y_{j_3},Y_{j_4})  \hinZ(Z_{\pi_{j_1}},Z_{\pi_{j_2}},Z_{\pi_{j_3}},Z_{\pi_{j_4}}) \big] \\ &= (I) + (II) + (III) + (IV) + (V), 
\end{talign}
where 
\begin{itemize}
    \item $(I)$ includes the terms for which the bin indices 
    $i_1,\dots,i_4,\swap_{s}(i_1),\dots,\swap_{s}(i_4)$
    are distinct. For such terms, $\mE \big[ \hinY(Y_{j_1},Y_{j_2},Y_{j_3},Y_{j_4})  \hinZ(Z_{\pi_{j_1}},Z_{\pi_{j_2}},Z_{\pi_{j_3}},Z_{\pi_{j_4}})  \big] = \frac{1}{4} \mE \big[\hin(X_{1},X_{2},X_{3},X_{4})] = 
    \frac{1}{4} \mathcal{U}$ by equation \eqref{eq:E_V_pi} of \cref{lem:bar_k_bound}. The number of terms in $(I)$ is $s(s-2)(s-4)(s-6)n^4$.
    \item $(II)$ includes the terms 
    for which there are exactly two pairs of repeated bin indices among $i_1,\dots,i_4,\swap_{s}(i_1),\dots,\swap_{s}(i_4)$, and the two pairs are one of the following: $\{(i_1,i_4),(\swap_{s}(i_1),\swap_{s}(i_4))\}$, $\{(i_1,\swap_{s}(i_4)),(\swap_{s}(i_1),i_4)\}$ (so that $\pi_{i_1} = i_1\iff\pi_{i_4} = i_4$), or $\{ (i_2,i_3), (\swap_{s}(i_2),\swap_{s}(i_3)) \}$, $\{ (i_2,\swap_{s}(i_3)), (\swap_{s}(i_2),i_3) \}$, (so that $\pi_{i_2} = i_2\iff\pi_{i_3} = i_3$). For all such terms, equation \eqref{eq:E_V_pi_2} of \cref{lem:bar_k_bound}  implies $\mE \big[ \hinY(Y_{j_1},Y_{j_2},Y_{j_3},Y_{j_4}) \hinZ(Z_{\pi_{j_1}},Z_{\pi_{j_2}},Z_{\pi_{j_3}},Z_{\pi_{j_4}}) \big] = 
    \frac{1}{4} \mathcal{U}$ when $j_1 \neq j_4$ (respectively, $j_2 \neq j_3$), and $\mE \big[ \hinY(Y_{j_1},Y_{j_2},Y_{j_3},Y_{j_4}) \hinZ(Z_{\pi_{j_1}},Z_{\pi_{j_2}},Z_{\pi_{j_3}},Z_{\pi_{j_4}}) \big] = 0$ when $j_1 \neq j_4$ (respectively, $j_2 \neq j_3$), since $\hinY(Y_{j_1},Y_{j_2},Y_{j_3},Y_{j_4}) = 0$. The number of terms in $(II)$ is $4s(s-2)(s-4)n^4$.
    \item $(III)$ includes the terms for which all indices $j_1,\dots,j_4,\swap(j_1),\dots,\swap(j_4)$ are different and 
    for which there are exactly two pairs of repeated bin indices among $i_1,\dots,i_4,\swap_s(i_1),\dots,\swap_s(i_4)$, and these two pairs are one of the following: $\{(i_1,i_2),(\swap_s(i_1),\swap_s(i_2)) \}$, $\{(i_1,\swap_s(i_2)), (\swap_s(i_1),i_2) \}$ (so that $\pi_{i_1} = i_1 \iff \pi_{i_2} = i_2$); $\{ (i_1,i_3), (\swap_s(i_1),\swap_s(i_3)) \}$, $\{ (i_1,\swap_s(i_3)), (\swap_s(i_1),i_3) \}$ (both of which imply that $\pi_{i_1} = i_1 \iff \pi_{i_2} = i_2$);  $\{(i_2,i_4), (\swap_s(i_2),\swap_s(i_4))\}$, $\{(i_2,\swap_s(i_4)),(\swap_s(i_2),i_4)\}$  (both of which imply that $\pi_{i_2} = i_2 \iff \pi_{i_4} = i_4$); or $\{(i_3,i_4),(\swap_s(i_3),\swap_s(i_4))\}$,  $\{(i_3,\swap_s(i_4)),(\swap_s(i_3),i_4)\}$ (both of which imply that $\pi_{i_3} = i_3 \iff \pi_{i_4} = i_4$). By equation \eqref{eq:E_V_pi_2} of \cref{lem:bar_k_bound}, for all such terms we can write $\mE \big[ \hinY(Y_{j_1},Y_{j_2},Y_{j_3},Y_{j_4}) \mE \big[ \hinZ(Z_{\pi_{j_1}},Z_{\pi_{j_2}},Z_{\pi_{j_3}},Z_{\pi_{j_4}}) \big] \big] = 
    \frac{3}{8} \mathcal{U}
    $. The number of terms in $(III)$ is $8s(s-2)(s-4)n^3(n-1) \leq 8s(s-2)(s-4)n^4$.
    \item $(IV)$ includes the terms for which all indices $j_1,\dots,j_4,\swap(j_1),\dots,\swap(j_4)$ are different and which do not belong to either $(I)$, $(II)$ or $(III)$.
    For such terms, we have that $\mE \big[ \hinY(Y_{j_1},Y_{j_2},Y_{j_3},Y_{j_4}) \mE \big[ \hinZ(Z_{\pi_{j_1}},Z_{\pi_{j_2}},Z_{\pi_{j_3}},Z_{\pi_{j_4}}) \big] \big] \leq
    \frac{1}{2} \mathcal{U}$, by equation \eqref{eq:E_V_pi_5}.
    \item $(V)$ includes all terms for which $(j_1,\dots,j_4)$ belongs to $(\mathsf{J}_{1} \cup \mathsf{J}_{2})^c$, where $\mathsf{J}_{1}$ and $\mathsf{J}_{2}$ are defined in \eqref{eq:J_total_1_2}. The total contribution of such terms can be upper-bounded by $\frac{4 n^2 - 11 n + 6}{n^3} \tilde{\xi}$ by the same argument as in the standard permutation case.
\end{itemize}
We conclude that
\begin{talign}
    \mE [V_n^{\pi,s}] &\leq \underbrace{
    \frac
    {s(s-2)(s-4)(s-6)}
    {s^4} \times \frac{1}{4}\mathcal{U}}_{(I)} + \underbrace{\frac
    {4 s (s-2)(s-4)}
    {s^4}  \times \frac{1}{4}\mathcal{U}}_{(II)}
    + \underbrace{\frac
    {8 s (s-2)(s-4)}
    {s^4} \times \frac{3}{8}\mathcal{U}}_{(III)} \\ &\quad + \underbrace{\frac
    {s^4 - s(s-2)(s-4)(s-6) - 12 s (s-2)(s-4)}
    {s^4} \times 
    \frac{1}{2}\mathcal{U}}_{(IV)} 
    + \underbrace{\frac
    {10n^2 - 44 n + 48}
    {n^3} \tilde{\xi}}_{(V)} 
    \\ &= \frac{1}{4} \big( 1 + \frac{s^2 + s - 4}{s^3} \big) \mathcal{U} + \frac
    {10n^2 - 44 n + 48}
    {n^3} \tilde{\xi}.
\end{talign}
\section{\pcref{cor:psd_gs}}
\label{subsec:cor_psd_ind}
To establish the standard testing claim, we will show that \cref{eq:psd_ind_full} implies the sufficient power condition \cref{eq:ind_full_suff_app} in the proof of \cref{thm: Full Independence Tests}. 
Our proof makes use of three lemmas, proved in \cref{proof:tilde_psi_1_prime_bound,,proof:E_V_n_mmd,,proof:psi_1_prime_bound} respectively, that relate parameters of \cref{eq:ind_full_suff_app} to $\mmd(\pjnt,\P \times \Q)$.

\begin{lemma}[{$\E[V_n]$} bound] \label{lem:E_V_n_mmd}
Under the assumptions of \cref{cor:psd_gs},
\begin{talign}
\mE[V_n] 
    \geq 
4 \mmd^2(\pjnt,\P \times \Q)
    =
\mathcal{U}.
\end{talign}
\end{lemma}
\begin{lemma}[$\psi_{1}'$ bound] \label{lem:psi_1_prime_bound}
    Under the assumptions of \cref{cor:psd_gs}, 
    \begin{talign}
        \psi_{1}' \leq %
        4 \mmd^2(\pjnt, \P \times \Q) \xi.
    \end{talign}
    for $\psi'_1$ as defined in equation \cref{Eq: definition of psi prime functions} and $\xi$ as defined in \cref{cor:psd_gs}. 
\end{lemma}
\begin{lemma}[$\tilde{\psi}_{1}'$ bound]\label{lem:tilde_psi_1_prime_bound}
Under the assumptions of \cref{cor:psd_gs}, 
\begin{talign}
    \tilde{\psi}_1' \leq 4 \wildxi \mmd^2(\pjnt, \P \times \Q) + \mmd^4(\pjnt, \P \times \Q),
\end{talign}
for $\tilde{\psi}'_1$ as defined in equation \cref{Eq: definition of psi prime functions} and $\xi$ as defined in \cref{cor:psd_gs}. 
\end{lemma}

Since $\E[V_n] \geq 4 \mmd^2(\pjnt, \P \times \Q) %
= \mathcal{U}$ by \cref{lem:E_V_n_mmd}, 
equation \cref{eq:ind_full_suff_app} holds whenever
\begin{talign}
\begin{split}
    3 \mmd^2(\pjnt, \P \times \Q) &\geq
    \sqrt{(\frac{3}{\beta} - 1) \big(\frac{16 \psi_1'}{n} + \frac{144 \psi_2'}{n^2} \big)}
    + \sqrt{\big( \frac{12}{\astar \beta} - 2\big) \big( \frac{32 \tilde{\psi}_1'}{n} + \frac{960 \psi_2'}{n^2} \big)} + \frac{10}{n} \tilde{\xi}.
\end{split}
\end{talign}
Next, by \cref{lem:psi_1_prime_bound,lem:tilde_psi_1_prime_bound}, we can upper-bound the right-hand side of this expression by 
\begin{talign}
\begin{split}
    &\sqrt{(\frac{3}{\beta} - 1) \big(\frac{64 \mmd^2(\pjnt, \P \times \Q) \xi}{n} + \frac{144 \psi_{2}'}{n^2} \big)}
    \\ &\qquad + \sqrt{\big( \frac{12}{\astar \beta} - 2\big) \big(\frac{128 \wildxi \mmd^2(\pjnt, \P \times \Q) + 32 \mmd^4(\pjnt, \P \times \Q)}{n} + \frac{960 \psi_2'}{n^2} \big)}
    + \frac{10}{n} \tilde{\xi} \\ &\leq \sqrt{ 
    \big( \frac{12}{\astar \beta} - 2\big) \frac{(256 \wildxi + 32 \xi) \mmd^2(\pjnt, \P \times \Q)}{n}
    } + \sqrt{ 
    \big( \frac{12}{\astar \beta} - 2\big) \frac{32 \mmd^4(\pjnt, \P \times \Q)}{n}
    } \\ &\qquad + \sqrt{
    \big( \frac{12}{\astar \beta} - 2\big) \frac{960 \psi_2'}{n^2}
    }
    + \frac{10}{n} \tilde{\xi}, 
\end{split}
\end{talign}
where we used that $\frac{3}{\beta} - 1 \leq \frac{1}{4} \big( \frac{12}{\astar \beta} - 2\big)$.
Hence, for \cref{eq:ind_full_suff_app} to hold, it suffices to have
\begin{talign}
    a x^2 &- bx - c \geq 0, \\
    \text{with } \qquad  x &= \mmd(\pjnt, \P \times \Q), \qquad 
    a = 3 - \sqrt{\big( \frac{12}{\astar \beta} - 2\big) \frac{32}{n}},
    \\ 
    b &= \sqrt{ \big( \frac{12}{\astar \beta} - 2\big) \frac{(256 \wildxi + 32 \xi)}{n}}, 
    \qquad
    c = \sqrt{\big( \frac{12}{\astar \beta} - 2\big) \frac{960 \psi_2'}{n^2}} + \frac{10}{n} \tilde{\xi}.
\end{talign}
Equation \cref{eq:ind_full_suff_app} therefore holds when $x\geq \frac{b}{a} + \sqrt{\frac{c}{a}} \geq \frac{b + \sqrt{b^2 + 4 a c}}{2a}$, concluding the proof of \cref{eq:psd_ind_full}.
To establish the cheap testing claim, we will show that \cref{eq:psd_ind_cheap} implies the sufficient power condition  \cref{eq:ind_cheap_suff_app}
in the proof of \cref{thm: Full Independence Tests}. 
Since $\E[V_n] \geq \mathcal{U}$ by  \cref{lem:E_V_n_mmd}, 
\cref{eq:ind_cheap_suff_app} holds when 
\begin{talign}
    \big( 3 \! - \!
    \frac{s+1}{s^2}
    \big) \mmd^2(\pjnt, &\P \times \Q) \! \geq \!
    \sqrt{(\frac{3}{\beta} - 1) \big(\frac{16 \psi_1'}{n} + \frac{144 \psi_2'}{n^2} \big)}
    \\ &+ \! \sqrt{\big( \frac{12}{\astar \beta} - 2\big)
    \big(\tilde{\psi}_1' \big( 
    \frac{32}{n} + \frac{2304}{n s} + \frac{460800}{n s^2}
    \big) + \psi_2' \big(
    \frac{960}{n^2} + \frac{4147200}{n^2 s}
    \big) \big)} \! + \! \frac{10}{n} \tilde{\xi}. 
\end{talign}
Next, by \cref{lem:psi_1_prime_bound,lem:tilde_psi_1_prime_bound}, we can upper-bound the right-hand side of this expression by 
\begin{talign}
    &\sqrt{(\frac{3}{\beta} - 1) \big(\frac{64 \mmd^2(\pjnt, \P \times \Q) \xi}{n} + \frac{144 \psi_{2}'}{n^2} \big)} 
    + \frac{10}{n} \tilde{\xi} \\ &\quad + 
    \big( \frac{12}{\astar \beta} - 2\big)^{1/2}
    \big(
    (4 \wildxi \mmd^2(\pjnt, \P \times \Q) + \mmd^4(\pjnt, \P \times \Q) ) (\frac{32}{n} + \frac{2304}{n s} + \frac{460800}{n s^2})
    \\ &\qquad\qquad\qquad\qquad + \psi_2' \big( \frac{960}{n^2} + \frac{4147200}{n^2 s} \big) \big)^{1/2}
    \\ &\leq 
    (3-\frac{s+1}{s^2}-a)\mmd^2(\pjnt, \P \times \Q)
    + b \mmd(\pjnt, \P \times \Q) + c 
    \qtext{for}\\
    a
        &\defeq 
    3-\frac{s+1}{s^2}-\sqrt{
    \big( \frac{12}{\astar \beta} - 2\big)\big( \frac{32}{n} + \frac{2304}{n s} + \frac{460800}{n s^2} \big)},    
        \\
    b
        &\defeq
    \sqrt{
    \big( \frac{12}{\astar \beta} - 2\big)
     \big(\wildxi (\frac{256}{n} + \frac{18432}{n s} + \frac{3686400}{n s^2}) + \frac{32\xi}{n} \big)}, 
     \qtext{and}  \\ 
    c
        &\defeq 
    \sqrt{
    \big( \frac{12}{\astar \beta} - 2\big)
    \psi_2' \big( \frac{960}{n^2} + \frac{4147200}{n^2 s} \big)} + \frac{10}{n} \tilde{\xi}.  
\end{talign}
The sufficiency of \cref{eq:psd_ind_cheap} now follows from the sufficiency of the quadratic inequality $ax^2-bx-c\geq0$ with $x = \mmd(\pjnt, \P \times \Q)$, exactly as in the standard case, and the triangle inequality.

\subsection{\pcref{lem:E_V_n_mmd}}\label{proof:E_V_n_mmd}
    Let 
    $(\PX', \PY', \PZ')$ represent an independent copy of $(\PX,\PY,\PZ)$.
    Then, two applications of Cauchy-Schwarz and \cite[App.~B]{schrab2022efficient} imply that
    \begin{talign}
        &\mmd^2(\pjnt,\P \times \Q) = \mE \big[ \iint \kernel(x,x') \, d(\PX - \PY \times \PZ)(x) \, d(\PX' - \PY' \times \PZ')(x') \big] \\ &\leq \mE \big[ \knorm{\int \kernel(x,\cdot) \, d(\PX - \PY \times \PZ)(x) } \knorm{\int \kernel(x,\cdot) \, d(\PX' - \PY' \times \PZ')(x') } \big] 
        \\ &\leq \mE \big[ \knorm{\int \kernel(x,\cdot) \, d(\PX - \PY \times \PZ)(x) }^2 \big] = \frac{1}{4} \mE[V_n].
    \end{talign}
Moreover, we have
\begin{talign} \label{eq:mmd_to_marginals}
\mathcal{U} 
    &= 
\mE[\hin(X_{i_1},X_{i_2},X_{i_3},X_{i_4})] \\
    &=
\mE\big[ \big(g_Y(Y_1,Y_2) + g_Y(Y_3,Y_4)  - g_Y(Y_1,Y_3) - g_Y(Y_2,Y_4) \big) \\
    &\quad \times  \big(g_Z(Z_1,Z_2) + g_Z(Z_3,Z_4) - g_Z(Z_1,Z_3) - g_Z(Z_2,Z_4) \big) \big] \\
    &=
4(\pjnt\times\pjnt)(g_Yg_Z) 
    +
4((\P \times \Q)\times(\P \times \Q))(g_Yg_Z) \\
    &\quad-
4(\pjnt\times(\P\times\Q))(g_Yg_Z) 
    -
4((\P\times\Q)\times\pjnt)(g_Yg_Z) 
    =
4\mmd^2(\pjnt,\P \times \Q). 
\end{talign}

\subsection{\pcref{lem:psi_1_prime_bound}}\label{proof:psi_1_prime_bound}

By Jensen's inequality and
Cauchy Schwarz, 
\begin{talign} 
\begin{split} \label{eq:psi_1_prime_bound}
\psi_{1}' &= \Var(\mE[ \hbarin(X_1,X_2,X_3,X_4)\mid X_1]) \\ &\leq \frac{1}{4!} \sum_{(i_1,i_2,i_3,i_4) \in \mathbf{i}_4^{4}} \Var(\mE[\hin(X_{i_1},X_{i_2},X_{i_3},X_{i_4})\mid X_1]) \\ &= \Var(\mE[\hin(X_1,X_2,X_3,X_4)\mid X_1]) \leq \mE[(\mE[\hin(X_1,X_2,X_3,X_4)\mid X_1])^2] \\ 
&= \E[(2\iint \kernel(x,x') \, d(\delta_{X_1} + \pjnt - \delta_{Y_1} \times \Q - \P \times \delta_{Z_1})(x)  d(\pjnt - \P \times \Q)(x'))^2] \\
&\leq
4\E[\knorm{ \int \kernel(x,\cdot) \, d(\delta_{X_1} + \pjnt - \delta_{Y_1} \times \Q - \P \times \delta_{Z_1})(x)}^2] \knorm{ \int \kernel(x,\cdot) \, d(\pjnt - \P \times \Q)(x) }^2.
\end{split}
\end{talign}
The advertised result now follows from the observation that 
\begin{talign}
    \big\| \int \kernel(x,\cdot) \, d(\pjnt - \P \times \Q)(x) \big\|_{\kernel} =  \mmd(\pjnt,\P \times \Q),
\end{talign}
and that
\begin{talign}
    &\E[\big\| \int \kernel(x,\cdot) \, d(\delta_{X_1} + \pjnt - \delta_{Y_1} \times \Q - \P \times \delta_{Z_1})(x) \big\|_{\kernel}^2] \\ &= \E[\iint \kernel(x,x') \, d(\delta_{X_1} + \pjnt - \delta_{Y_1} \times \Q - \P \times \delta_{Z_1})(x)\, d(\delta_{X_1} + \pjnt - \delta_{Y_1} \times \Q - \P \times \delta_{Z_1})(x')] \\
    &=
\E[\E[\hin(X_1,X_2,X_3,X_4)\mid X_1]] = \xi.
\end{talign}
\subsection{\pcref{lem:tilde_psi_1_prime_bound}}\label{proof:tilde_psi_1_prime_bound}
By the definition \cref{Eq: definition of psi prime functions} of $\tilde{\psi}_1'$, we can write
\begin{talign}
    \tilde{\psi}_1' %
    &= 
    \mE \big[ \hinY(Y_{1},Y_{2},Y_{3},Y_{4})  \hinZ(Z_{\pi_{1}},Z_{\pi_{2}},Z_{\pi_{3}},Z_{\pi_{4}}) \\ &\qquad\qquad \times \hinY(Y_{1},Y_{5},Y_{6},Y_{7})  \hinZ(Z_{\pi_{1}},Z_{\pi_{5}},Z_{\pi_{6}},Z_{\pi_{7}}) \big] \\ &\qquad - \E \big[ \hinY(Y_{1},Y_{2},Y_{3},Y_{4})  \hinZ(Z_{\pi_{1}},Z_{\pi_{2}},Z_{\pi_{3}},Z_{\pi_{4}}) \big] \\ &\qquad\qquad \times \E \big[ \hinY(Y_{1},Y_{5},Y_{6},Y_{7})  \hinZ(Z_{\pi_{1}},Z_{\pi_{5}},Z_{\pi_{6}},Z_{\pi_{7}}) \big].
\end{talign}
The result now follows immediately from the following lemma.
\begin{lemma}[Permuted $\hin$ moments] \label{lem:F4_PD}
    Instantiate the assumptions of \cref{cor:psd_gs}, and suppose $\pi$ is an independent wild bootstrap permutation of $[n]$.  Then 
    \begin{talign} \label{eq:lem_F4_PD_1}
    \E \big[ \hinY(Y_{1},Y_{2},Y_{3},Y_{4})  \hinZ(Z_{\pi_{1}},Z_{\pi_{2}},Z_{\pi_{3}},Z_{\pi_{4}}) \big]
        = \mmd^2(\pjnt, \P \times \Q).
    \end{talign}
    If, in addition, $n \geq 14$, then
    \begin{talign}
    \begin{split} \label{eq:bound_exp_P_pi}
    &\mE \big[\hinY(Y_{1},Y_{2},Y_{3},Y_{4}) \hinY(Y_{1},Y_{5},Y_{6},Y_{7}) \\ &\quad 
    \times \hinZ(Z_{\pi_{1}},Z_{\pi_{2}},Z_{\pi_{3}},Z_{\pi_{4}}) \hinZ(Z_{\pi_{1}},Z_{\pi_{5}},Z_{\pi_{6}},Z_{\pi_{7}}) \big]
    \leq 4 \wildxi \mmd^2(\pjnt, \P \times \Q).
    \end{split}
    \end{talign}
\end{lemma}

\begin{proof}
Since $n\geq 8$, $(\pi_1,\pi_2,\pi_3,\pi_4)$ are independent. Hence, 
equation \eqref{eq:lem_F4_PD_1} follows from equation \eqref{eq:E_V_pi} of \cref{lem:bar_k_bound} and \cref{lem:E_V_n_mmd}.

Now suppose that $n \geq 14$, so that $|[7]\cup\{\swap(i) : i \in[7]\}|=14$, and define the $\Unif(\{0,1\})$ variables $\eps_{i}\defeq\indic{i=\pi_i}$ for $i\in[7]$ and 
$\mu_{i} \defeq \eps_{i}\pjnt + 
(1-\eps_{i})\P\times\Q$.
We have
\begin{talign}
\begin{split} \label{eq:bound_by_tensor_prod_mmd}
&\E[
\hinY(Y_{1},Y_{2},Y_{3},Y_{4}) \hinY(Y_{1},Y_{5},Y_{6},Y_{7}) \\ &\qquad\qquad 
    \times \hinZ(Z_{\pi_{1}},Z_{\pi_{2}},Z_{\pi_{3}},Z_{\pi_{4}}) \hinZ(Z_{\pi_{1}},Z_{\pi_{5}},Z_{\pi_{6}},Z_{\pi_{7}}) \mid \pi,Y_1,Z_{\pi_1}] \\
    &= \iiiint \kernel((y,z),(y',z')) \kernel((y'',z''),(y''',z''')) \\ &\qquad\qquad d(\delta_{(Y_{1},Z_{\pi_{1}})} + \mu_{4} - \delta_{Y_{1}} \times \Q - \P \times \delta_{Z_{\pi_{1}}})(y,z) \\ &\qquad\qquad d(\mu_{2} + \mu_{3} - \P \times \Q - \P \times \Q)(y',z') \\ &\qquad\qquad d(\delta_{(Y_{1},Z_{\pi_{1}})} + \mu_{7} - \delta_{Y_{1}} \times \Q - \P \times \delta_{Z_{\pi_{1}}})(y'',z'') \\ &\qquad\qquad d(\mu_{5} + \mu_{6} - \P \times \Q - \P \times \Q)(y''',z''') 
    \\ &\leq \big\| \iint \kernel((y,z),\cdot) \kernel((y'',z''),\cdot) \, d(\delta_{(Y_{1},Z_{\pi_{1}})} + \mu_{4} - \delta_{Y_{1}} \times \Q - \P \times \delta_{Z_{\pi_{1}}})(y,z) \\ &\qquad\qquad\qquad d(\delta_{(Y_{1},Z_{\pi_{1}})} + \mu_{7} - \delta_{Y_{1}} \times \Q - \P \times \delta_{Z_{\pi_{1}}})(y'',z'') \big\|_{\kernel \otimes \kernel}
    \\ &\qquad \times \big\| \iint \kernel((y',z'),\cdot) \kernel((y''',z'''),\cdot) \, d(\mu_{2} + \mu_{3} - \P \times \Q - \P \times \Q)(y',z') \\ &\qquad\qquad\qquad d(\mu_{5} + \mu_{6} - \P \times \Q - \P \times \Q)(y''',z''') \big\|_{\kernel \otimes \kernel} 
    \\ &= 4\mmd\bigg(\frac{\delta_{(Y_{1},Z_{\pi_{1}})} + \mu_{4}}{2},\frac{\delta_{Y_{1}} \times \Q + \P \times \delta_{Z_{\pi_{1}}}}{2} \bigg) \mmd\bigg(\frac{\delta_{(Y_{1},Z_{\pi_{1}})} + \mu_{7}}{2},\frac{\delta_{Y_{1}} \times \Q + \P \times \delta_{Z_{\pi_{1}}}}{2} \bigg) \\
    &\qquad\quad
\times ( \epsilon_{2} + \epsilon_{3} ) ( \epsilon_{5} + \epsilon_{6} ) \mmd^2(\pjnt, \P \times \Q)
\end{split}
\end{talign}
by Cauchy-Schwarz. Here, $\| \cdot \|_{\kernel \otimes \kernel}$ denotes the RKHS norm of the tensor product RKHS $\rkhs \otimes \rkhs$.
The result now follows by taking expectations and applying Cauchy-Schwarz.
\end{proof}

\section{\pcref{thm:final_ind}
}\label{sec:proof_thm:final_ind}

We will establish the claim using the quantile comparison method (\cref{Lemma: Generic Two Moments Method}) by identifying complementary quantile bounds (CQBs) $\Phi_{n}, \Psi_{\X}, \Psi_{n}$ satisfying the tail bounds \cref{eq:Phi_P_n,,eq:Psi_X_n,,eq:Psi_P_n} for the test statistic $V_{n}$, auxiliary sequence  $(V_{n}^{\pi_b,s})_{b=1}^\numperm$, and population parameter $\tconst = 4 \mmd^2(\Pi, \P \times \Q)$.

\subsection{Bounding the fluctuations of the test statistic}
We begin by identifying a CQB $\Phi_{n}$ satisfying the test statistic tail bound \cref{eq:Phi_P_n}. 
Our proof makes use of two lemmas. The first shows that functions of independent variables concentrate around their mean when coordinatewise function differences are bounded.

\begin{lemma}[Bounded differences inequality {\citep[Thm.~D.8]{mohri2012foundations}}] \label{thm:mcdiarmid_bounded}
Suppose that, for all $x_{1}\in {\mathcal {X}}_{1},\,x_{2}\in {\mathcal {X}}_{2},\,\ldots ,\,x_{n}\in {\mathcal {X}}_{n}$ and $i\in[n]$, a function $f:{\mathcal {X}}_{1}\times {\mathcal {X}}_{2}\times \cdots \times {\mathcal {X}}_{n}\rightarrow \mathbb {R}$  satisfies
\begin{talign}
    \sup_{x_i' \in \mathcal{X}_i} |f(x_1,\dots,x_{i-1}, x_i,x_{i+1},\dots,x_n)-f(x_1,\dots,x_{i-1}, x_i',x_{i+1},\dots,x_n)| \leq c_i.
\end{talign}
If $X_{1},X_{2},\dots ,X_{n}$ are independent with each  
$X_{i}\in {\mathcal {X}}_{i}$ then, 
with probability at least $1-\delta$,
\begin{talign}
    f(X_{1},\ldots ,X_{n}) - \mathbb{E}[f(X_{1},\ldots ,X_{n})] < \sqrt{\half\log(1/\delta)\sum_{i=1}^nc_i^2}.
\end{talign}
\end{lemma}
The second lemma uses \cref{thm:mcdiarmid_bounded} to derive a suitable CQB for $V_n$.
\begin{lemma}[Bounded independence quantile bound]\label{bounded_independence_quantile_bound}
Under the assumptions of \cref{thm:final_ind}, 
    \begin{talign} \label{eq:tail_bound_V_n_2}
    \mathrm{Pr} \big( |4\mmd^2(\pjnt,\P \times \Q) - V_n| \geq \Phi_{n}(\beta)  \big) \leq \beta
    \qtext{for all}
    \beta\in(0,1)
    \end{talign}
    where 
    $\Phi_{n}(\beta) \defeq \Lambda^2(\beta) + 4 \mmd(\pjnt,\P \times \Q) \Lambda(\beta)$
    and
    $\Lambda(\beta) 
    \defeq \frac{6\sqrt{K}+6\sqrt{2K\log(1/\beta)}}{\sqrt{n}}$.
\end{lemma}

\begin{proof}
Let $\tilde{X}=(\tilde{Y},\tilde{Z})$ be an independent draw from $\pjnt$.
By 
\citep[App.~B]{schrab2022efficient},
\begin{talign} \label{Eq: V-statistic for independence testing 2}
V_n &=  \frac{1}{n^4} \sum_{(i_1,i_2,i_3,i_4) \in [n]^4} \hin(X_{i_1},X_{i_2},X_{i_3},X_{i_4}) = 4\mmd^2(\PX,\PY \times \PZ).
\end{talign}
By the triangle inequality, we have the following sure inequalities for each $k\in[n]$:
\begin{talign}
&|\mmd(\pjnt, \P\times\Q) - \mmd(\PX, \PY\times\PZ)| 
    \leq
\mmd(\pjnt, \PX) + \mmd(\P\times\Q, \PY\times\PZ) \\
    &\qquad\defeq 
\Delta(\X) + \Delta'(\X),  \\ 
&|\Delta(\X)
    -
\Delta((\X_{-k},\tilde{X}))|
    \leq
\frac{1}{n}\mmd(\delta_{X_k},\delta_{\tilde{X}}) 
    \leq
\frac{2\sqrt{K}}{n}, 
    \qtext{and} \\
&|\Delta'(\X)
    -
\Delta'((\X_{-k},\tilde{X}))|
    \\
    &\qquad\leq
\frac{1}{n}\mmd[g_Y](\delta_{Y_k},\delta_{\tilde{Y}}) \sqrt{(\PZ\times\PZ)g_Z}
    +
\frac{1}{n}\mmd[g_Z](\delta_{Z_k},\delta_{\tilde{Z}}) \sqrt{\sup_y g_Y(y,y)}
    \leq
\frac{4\sqrt{K}}{n}.
\end{talign}
Furthermore, by Jensen's inequality, we have
\begin{talign}\label{eq:mean-DeltaX}
\E[\Delta(\X)]
    &\leq
\sqrt{\E[\mmd(\pjnt, \PX)^2]}  
    = 
\sqrt{\E[ \frac{1}{n^2} \sum_{i=1}^{n} ((\delta_{X_i}-\pjnt)\times (\delta_{X_i}-\pjnt))\kernel]} \\
    &\leq
\sqrt{\E[ \frac{1}{n^2} \sum_{i=1}^{n} \kernel(X_i,X_i)]}
    \leq
\sqrt{\frac{K}{n}}.
\end{talign}
The triangle inequality, Cauchy-Schwarz, and analogous reasoning also yield
\begin{talign}
\begin{split} \label{eq:mean-DeltaXprime}
\E[\Delta'(\X)]
    &\leq
\E[\mmd(\P\times\Q,\P\times\PZ)]
+
\E[\mmd(\P\times\PZ,\PY\times\PZ)]
    \\
    &=
\E[\sqrt{(\P\times\P)g_Y}\mmd[g_Z](\Q,\PZ)]
+
\E[\sqrt{(\PZ\times\PZ)g_Z}\mmd[g_Y](\P,\PY)] \\
&=
\sqrt{\E[(\P\times\P)g_Y]\E[\mmd[g_Z]^2(\Q,\PZ)]}
+
\sqrt{\E[(\PZ\times\PZ)g_Z\mmd[g_Y]^2(\P,\PY)]}  
    \leq
2\sqrt{\frac{K}{n}}.
\end{split} 
\end{talign}
Hence, by \cref{thm:mcdiarmid_bounded}, with probability at least $1-\beta$,
\begin{talign}
|2\mmd(\pjnt, \P\times\Q) - 2\mmd(\PX, \PY\times\PZ)| 
    <
\Lambda_{V}(\beta).
\end{talign}
Since $|a^2-b^2|=|a-b| |b+a| \leq |a-b| (|b-a|+2|a|)$ for all real $a,b$ the result follows.
\end{proof}

\subsection{Alternative representation of the independence V-statistic} 
\label{subsec:alternative_V_statistic}
We next define an alternative, equivalent way to represent the independence V-statistic (\cref{def:independence_v_statistic}) and its permuted versions, due to 
\citep{gretton2007akernel}. 
We will leverage this view of the V-statistic when deriving CQBs for permuted V-statistics in \cref{subsec:bounded_ind_fluctuation}.
For $i\in [n]$ and $j \in [n]$, we define
\begin{talign}
    \bar{K}_{ij} &\defeq g_Y(Y_i, Y_j) - \frac{1}{n} \sum_{j'=1}^{n} g_Y(Y_i, Y_{j'}) \qtext{and} 
    \bar{L}_{ij} \defeq g_Z(Z_i, Z_j) - \frac{1}{n} \sum_{j'=1}^{n} g_Z(Z_i, Z_{j'}).
\end{talign}
For $i, j \in [s]$, we also define 
\begin{talign}
\label{eq:Rijkl}
    (R_{ijji},\, R_{ij(\swaps(j))(\swaps(i))}) = \sum_{a\in\bin_i,b\in\bin_j} (\bar{K}_{ab} \bar{L}_{ba},\, \bar{K}_{ab} \bar{L}_{\swap(b)\swap(a)}) \\
    (R_{ij(\swaps(j))i},\, R_{ijj(\swaps(i))}) = \sum_{a\in\bin_i,b\in\bin_j} (\bar{K}_{ab} \bar{L}_{\swap(b)a},\, \bar{K}_{ab} \bar{L}_{b\swap(a)})
\end{talign}
where the bins $\bin_1, \ldots, \bin_s$ and the mapping $\swap$ are defined in \cref{algo:cheap_independence_permutation_test}. Note the similarity between $R_{ijkl}$ and the quantities $S_{ijkl}$ from \cref{algo:cheap_independence_permutation_test}.
Now define the V-statistic
\begin{talign}\label{eq:Vtilde}
\tilde{V}_n
    \defeq
    \frac{1}{n^2} \sum_{i,j=1}^{n} \bar{K}_{ij} \bar{L}_{ji} = \frac{1}{s^2} \sum_{i,j=1}^{s} R_{ijji} 
\end{talign}
and its cheaply permuted version for any  wild bootstrap permutation $\pi$ on $[s]$,  
\begin{talign}\label{eq:Vtildepi}
\tilde{V}^{\pi,s}_n
    \defeq
    \frac{1}{s^2} \sum_{i,j=1}^{s} R_{ij\pi_{j}\pi_{i}}.
\end{talign}
Our next result shows the close relationship between $V$ and $\tilde{V}$.
\begin{lemma}[V-statistic relationship]\label{v-stat-relationship}
In the notation of \cref{eq:Vtilde}, \cref{eq:Vtildepi}, and \cref{def:independence_v_statistic}  
\begin{talign}
\tilde{V}_n = \quarter V_n
    \qtext{and}
\tilde{V}^{\pi,s}_n = \quarter V^{\pi,s}_n
\end{talign} 
for any wild bootstrap permutation $\pi$ on $[s]$.
\end{lemma}
\begin{proof}
    It suffices to show the result for $\tilde{V}^{\pi,s}_n$. 
    We expand $V^{\pi,s}_n$:
    \begin{talign}
        V^{\pi,s}_n &= \frac{1}{n^4} \sum_{(i_1,i_2,i_3,i_4) \in [n]^4} 
        \big(g_Y(Y_{i_1},Y_{i_2}) + g_Y(Y_{i_3},Y_{i_4})  - g_Y(Y_{i_1},Y_{i_3}) - g_Y(Y_{i_2},Y_{i_4}) \big) \\
        &
        \qquad\qquad \times  \big(g_Z(Z_{\pi_{i_1}},Z_{\pi_{i_2}}) + g_Z(Z_{\pi_{i_3}},Z_{\pi_{i_4}}) - g_Z(Z_{\pi_{i_1}},Z_{\pi_{i_3}}) - g_Z(Z_{\pi_{i_2}},Z_{\pi_{i_4}}) \big) \\
        &= \frac{4}{n^2} \sum_{(i_1,i_2) \in [n]^2} \bigg( g_Y(Y_{i_1},Y_{i_2}) g_Z(Z_{\pi_{i_1}},Z_{\pi_{i_2}}) - g_Y(Y_{i_1},Y_{i_2}) \big( \frac{1}{n}\sum_{i_3 \in [n]} g_Z(Z_{\pi_{i_2}},Z_{\pi_{i_3}}) \big) \\ &\qquad\qquad\qquad - \big( \frac{1}{n} \sum_{i_4 \in [n]} g_Y(Y_{i_1},Y_{i_4}) \big) g_Z(Z_{\pi_{i_1}},Z_{\pi_{i_2}}) \\ &\qquad\qquad\qquad + \big( \frac{1}{n}\sum_{i_3 \in [n]} g_Z(Z_{\pi_{i_2}},Z_{\pi_{i_3}}) \big) \big( \frac{1}{n} \sum_{i_4 \in [n]} g_Y(Y_{i_1},Y_{i_4}) \big) \bigg) \\ &= \frac{4}{n^2} 
        \sum_{(i_1,i_2) \in [n]^2} \big( g_Y(Y_{i_1},Y_{i_2}) - \frac{1}{n} \sum_{i_4 \in [n]} g_Y(Y_{i_1},Y_{i_4}) \big) \\ &\qquad\qquad\qquad \times \big( g_Z(Z_{\pi_{i_2}},Z_{\pi_{i_1}}) - \frac{1}{n} \sum_{i_3 \in [n]} g_Z(Z_{\pi_{i_2}},Z_{\pi_{i_3}}) \big) \\ &= \frac{4}{n^2} 
        \sum_{(i_1,i_2) \in [n]^2} \bar{K}_{i_1,i_2} \bar{L}_{\pi_{i_2},\pi_{i_1}} = 4 \tilde{V}^{\pi,s}_n.
    \end{talign}
\end{proof}

\subsection{Bounding the fluctuations of the threshold} \label{subsec:bounded_ind_fluctuation}
We next identify CQBs $\Psi_{\X,s}$ and $\Psi_{n,s}$ that satisfy the properties
\begin{talign}
    &\Pr \left( V^{\pi,s}_{n} %
        \geq \Psi_{\X,s}(\alpha) \mid \X \right) \leq \alpha, \qquad \forall \alpha \in (0,1),
        \qtext{and} \label{eq:Psi_X_n_ts_app_ind} \\
        &\Pr ( %
        \Psi_{\X,s}(\alpha) \geq \Psi_{n,s}(\alpha,\beta)) \leq \beta, \qquad \forall \alpha, \beta \in (0,1) \label{eq:Psi_P_n_ts_app_ind}
\end{talign}
for the cheaply permuted V-statistic $V^{\pi,s}_{n}$ \cref{eq:cheaply-permuted-independence-V}.
Since the alternative V-statistic representation $\tilde{V}^{\pi,s}_{n}$ \cref{eq:Vtildepi}, satisfies $\tilde{V}^{\pi,s}_{n} = \frac{1}{4} V^{\pi,s}_{n}$ by \cref{v-stat-relationship}, we will find $\tilde{\Psi}_{n,s}$ and $\tilde{\Psi}_{\X,s}$ satisfying
\begin{talign}
    &\mathrm{Pr} ( \tilde{V}^{\pi,s}_{n} 
    \geq \tilde{\Psi}_{\X,s}(\alpha) \mid \X ) \leq \alpha, \qquad \forall \alpha \in (0,1), \qtext{and}\label{eq:Psi_X_n_ts_app_ind_2} \\
    &\mathrm{Pr} ( %
    \tilde{\Psi}_{\X,s}(\alpha) \geq \tilde{\Psi}_{n,s}(\alpha,\beta)) \leq \beta, \qquad \forall \alpha, \beta \in (0,1), \label{eq:Psi_P_n_ts_app_ind_2}
\end{talign}
and set $\Psi_{\X,s}(\alpha) = 4 \tilde{\Psi}_{\X,s}(\alpha)$ and $\Psi_{n,s}(\alpha,\beta) = 4 \tilde{\Psi}_{n,s}(\alpha,\beta)$.
Our derivation makes use of four lemmas. 
The first, proved in \cref{proof:coefficients_A_b}, rewrites $\tilde{V}^{\pi,s}_{n}$ in terms of independent Rademacher random variables.
Below,

\begin{lemma}[Rademacher representation of $\tilde{V}_n^{\pi,s}$] \label{prop:coefficients_A_b}
Instantiate the assumptions of \cref{thm:final_ind}, and let 
$\fronorm{\cdot}$ represent  the Frobenius norm, $\boldone\in\reals^s$ be a vector of ones, and \begin{talign}
\swap_s(i) \defeq i + \frac{s}{2} - s\indic{i  > \frac{s}{2}}   
\end{talign}
for each $i\in[s]$.
If $\pi$ is a wild bootstrap permutation on $[s]$ drawn independently of $\X$, then the alternative V-statistic representation 
\cref{eq:Vtildepi} satisfies 
    \begin{talign}
        \tilde{V}_n^{\pi,s} = \eps^{\top} A \eps + b^{\top} \eps + c,
    \end{talign}
    where $\eps\dist\Unif(\{\pm1\}^{s/2})$ 
    is independent of $\X$, and %
    $A \in \R^{s/2 \times s/2}$, $b \in \R^{s/2}$, and $c \in \R$ satisfy 
    \begin{talign}
    &\big\|A- \frac{\mmd^2(\pjnt, \P \times \Q)}{s^2} \boldone \boldone^{\top} \big\|_{F}^2 
    		\leq 
    \frac{(\mathcal{A}^{(1)})^2 (\mathcal{A}^{(2)})^2}{4 s^4}
    + \frac{\mmd^2(\pjnt, \P \times \Q) ( (\mathcal{A}^{(1)})^2 + (\mathcal{A}^{(2)})^2 )}{4 s^3},
     \\
     &\mathrm{Tr}\big( A- \frac{\mmd^2(\pjnt, \P \times \Q)}{s^2} \boldone \boldone^{\top} \big) 
     	\leq 
    \frac{\mathcal{A}^{(1)} \mathcal{A}^{(2)}}{2s^2}
    + \frac{\mmd(\pjnt, \P \times \Q) ( \mathcal{A}^{(1)} + \mathcal{A}^{(2)} )}{4s^{3/2}}, \\
        &\big|c - \frac{\mmd^2(\pjnt, \P \times \Q)}{4} \big| 
        		\leq 
		\frac{1}{4} \big( \frac{(\mathcal{A}^{(3)})^2}{s} + \frac{2 \mathcal{A}^{(3)} \mmd(\pjnt, \P \times \Q)}{\sqrt{s}}  \big), \stext{and} 
	\\
    &\big\|b - \frac{\mmd^2(\pjnt, \P \times \Q)}{s} \boldone \big\|_{2}^2 
    		\leq 
        \frac{3(\mathcal{A}^{(2)})^2 (\mathcal{A}^{(3)})^2 + (\mathcal{A}^{(1)})^2 (\mathcal{A}^{(4)})^2}{4 s^3}
			\\ &\qquad\qquad\qquad\qquad\qquad + 
        \frac{3\mmd^2(\pjnt, \P \times \Q)( (\mathcal{A}^{(1)})^2 +(\mathcal{A}^{(2)})^2 + (\mathcal{A}^{(3)})^2 + (\mathcal{A}^{(4)})^2 )}{4s^2}
    \end{talign}
    for the random variables
    \begin{talign} 
        \mathcal{A}^{(1)} &\defeq  \sqrt{\sum_{i=1}^{s} \| (\mu_i - \pjnt + \P \times \Q)\kernel \|_{\kernel}^2}, 
        \quad
        \mathcal{A}^{(2)} \defeq  \sqrt{\sum_{i=1}^{s} \| (\tilde{\mu}_i - \pjnt + \P \times \Q)\kernel \|_{\kernel}^2}, \\
        \mathcal{A}^{(3)} &\defeq  \sqrt{s} \| (\mu - \pjnt + \P \times \Q)\kernel \|_{\kernel}, 
        \qtext{and}
        \mathcal{A}^{(4)} \defeq  \sqrt{s} \| (\tilde{\mu} - \pjnt + \P \times \Q)\kernel \|_{\kernel}
	\label{eq:A_tilde_A_B_tilde_B}
    \end{talign}
    defined in terms of the random signed measures
    \begin{talign} 
        \mu_i \defeq  \delta_{\mathbb{X}^{(i)}} - \delta_{\mathbb{Y}} \times \delta_{\mathbb{Z}^{(i)}} - \delta_{\mathbb{Y}^{(i)}} \times \delta_{\mathbb{Z}^{(\swaps(i))}} + \delta_{\mathbb{Y}} \times \delta_{\mathbb{Z}^{(\swaps(i))}}, 
        \ \ 
        \tilde{\mu}_i \defeq  \delta_{\mathbb{X}^{(i)}} - \delta_{\mathbb{Y}^{(i)}} \times \delta_{\mathbb{Z}^{(\swaps(i))}}, \\
        \mu \defeq {\delta_{\mathbb{X}}} - 2{\delta_{\mathbb{Y}}} \times{\delta_{\mathbb{Z}}} + \frac{1}{s} \sum_{j=1}^{s}{\delta_{\mathbb{Y}^{(j)}}} \times {\delta_{\mathbb{Z}^{(\swaps(j))}}}, 		
        \sstext{and}
        \tilde{\mu} \defeq  {\delta_{\mathbb{X}}} - \frac{1}{s} \sum_{j=1}^{s} {\delta_{\mathbb{Y}^{(j)}}} \times {\delta_{\mathbb{Z}^{(\swaps(j))}}}.
    \label{eq:chi_defs}
    \end{talign}
\end{lemma}
Our second lemma, proved in \cref{proof:expectation_A_i}, bounds the expectation of each quantity $\mathcal{A}^{(i)}$ introduced in \cref{prop:coefficients_A_b}.
\begin{lemma}[$\mathcal{A}^{(i)}$ mean bounds] \label{lem:expectation_A_i}
Under the assumptions and notation of \cref{prop:coefficients_A_b} with $m\defeq\frac{n}{s}$, 
\begin{talign}
\mathbb{E}\big[\mathcal{A}^{(1)} \big] 
    &\leq 
D_1 
    \defeq  
\sqrt{Ks} \big(\frac{5}{\sqrt{m}} + \frac{2}{\sqrt{n}}\big),
    \quad
\mathbb{E}\big[\mathcal{A}^{(2)}\big] 
    \leq 
D_2 
    \defeq  
\sqrt{K s}\, \frac{3}{\sqrt{m}}, 
\\
\mathbb{E}\big[\mathcal{A}^{(3)} \big] 
    &\leq 
D_3 
    \defeq  
\sqrt{K s}\, \frac{5+\sqrt{8}}{\sqrt{n}},
    \qtext{and}
\mathbb{E}\big[\mathcal{A}^{(4)} \big] 
    \leq 
D_4 
    \defeq  
\sqrt{K s} \, \frac{1+\sqrt{8}}{\sqrt{n}} .
\end{talign}
\end{lemma}

Our third lemma, proved in \cref{proof:hp_bounds_A_i}, uses \cref{lem:expectation_A_i} and the bounded differences inequality (\cref{thm:mcdiarmid_bounded}) to provide CQBs for each quantity $\mathcal{A}^{(i)}$ introduced in \cref{prop:coefficients_A_b}.
\begin{lemma}[$\mathcal{A}^{(i)}$ complementary quantile bounds] \label{prop:hp_bounds_A_i}
Instantiate the assumptions and notation of \cref{lem:expectation_A_i}, and fix any $i\in[4]$. 
With probability at least $1-\delta$, 
\begin{talign}
\mathcal{A}^{(i)} 
    &\leq 
D_i + C_i \sqrt{\frac{\log(1/\delta)}{2}} 
    \qtext{for}
(C_1,C_2,C_3,C_4)
    =
(\frac{10\sqrt{Ks}}{\sqrt{m}}
    +
\frac{2\sqrt{Ks}}{\sqrt{n}},
\frac{6\sqrt{Ks}}{\sqrt{m}},
\frac{14 \sqrt{Ks}}{\sqrt{n}},
\frac{6\sqrt{Ks}}{\sqrt{n}}).
    \end{talign}
\end{lemma}

Our fourth lemma, proved in \cref{proof:ind_perm_thresh_quantile_bound}, uses the estimates of \cref{prop:hp_bounds_A_i} to verify the CQB properties \cref{eq:Psi_X_n_ts_app_ind_2,eq:Psi_P_n_ts_app_ind_2} for a particular pairing of $(\tilde{\Psi}_{\X,s},\tilde{\Psi}_{n,s})$.
\newcommand{\psitildeb}[2]{\frac{\sqrt{K}}{\sqrt{n}}
    \big( 30 + 22 \sqrt{\log(#2/\beta)} \big) 
    \big( \half + \frac{1+12\log(3/#1)}{6\sqrt{s}}
    + \sqrt{\frac{3}{2}\log(3/#1)} 
    \big)}
\newcommand{\psitildec}[2]{\frac{K}{n}(196+148\sqrt{\log(#2/\beta)}+48\log(#2/\beta))
    \big(\frac{3}{4}+
    \frac{5\log(3/#1)}{3} +
    \sqrt{\frac{3}{2}\log(3/#1)}\big)}
\begin{lemma}[Quantile bound for independence permutation threshold]\label{ind_perm_thresh_quantile_bound} 
Instantiate the assumptions and notation of \cref{prop:hp_bounds_A_i}.
Then, for all $\alpha, \beta \in (0,1)$, 
\begin{talign}
\quarter\Psi_{n,s}(\alpha,\beta) 
    &\defeq 
\tilde{\Psi}_{n,s}(\alpha,\beta) \\
    &\defeq
\psitildec{\alpha}{4} \\ 
    & + 
\mmd(\pjnt,\P\times\Q)\psitildeb{\alpha}{4}
    \\ 
    &+ 
\mmd^2(\pjnt, \P \times \Q) \big( 
    \frac{1}{4} + \frac{\log (6/\alpha)}{s} + \sqrt{\frac{\log (6/\alpha)}{s}} 
    \big)
\end{talign}
satisfies $ \Pr (\tilde{\Psi}_{\X,s}(\alpha) \geq \tilde{\Psi}_{n,s}(\alpha,\beta)) \leq \beta$ for $\tilde{\Psi}_{\X,s}(\alpha)$ defined in equation \eqref{eq:tilde_V_psi_def} and satisfying $\Pr ( \tilde{V}^{\pi,s}_{n} \geq \tilde{\Psi}_{\X,s}(\alpha) \mid \X ) \leq \alpha$.
\end{lemma}
\subsection{Putting everything together}
Fix any $\beta\in(0,1)$, and suppose that 
\begin{talign}\label{eq:separation_mmd_bounded}
\tconst 
=
4\mmd^2(\Pi,\P \times \Q) \geq \Psi_{n,s}(\astar,\beta/3) + \Phi_{n}(\beta/3),
\end{talign}
for $\Psi_{n,s}$ and $\Phi_{n}$ defined in \cref{ind_perm_thresh_quantile_bound,bounded_independence_quantile_bound}. 
By \cref{ind_perm_thresh_quantile_bound,v-stat-relationship,bounded_independence_quantile_bound,Lemma: Generic Two Moments Method}, a cheap wild bootstrap test based on $V_n$ has power at least $1-\beta$. 
We now show that the assumed condition of \cref{thm:final_ind} is sufficient to imply \cref{eq:separation_mmd_bounded}.
Indeed, we see that \cref{eq:separation_mmd_bounded} holds whenever $a x^2 - bx - c \geq 0$ for $x = \mmd(\pjnt, \P \times \Q),$
\begin{talign}
a 
    &= 
\frac{3}{4} - \frac{\log (6/\astar)}{s} - \sqrt{\frac{\log (6/\astar)}{s}},
\\ 
b 
    &=
\frac{\sqrt{K}}{\sqrt{n}}
    \big( 30 + 22 \sqrt{\log(12/\beta)} \big) 
    \big( 3 + \frac{1+12\log({3}{/\astar})}{6\sqrt{s}}
    + \sqrt{\frac{3}{2}\log(3/\astar)} 
    \big)
    \\
    &\geq
\psitildeb{\astar}{12} 
    +
\frac{6\sqrt{K}(1+\sqrt{2\log(3/\beta)})}{\sqrt{n}}
\\
c   
    &=
\frac{K}{n}(14+7\sqrt{\log(12/\beta)})^2
    \big(\sqrt{\frac{7}{8}}+
    \sqrt{\frac{5}{3}\log(3/\astar)}\big)^2
\\
    &\geq
\frac{K}{n}(14+7\sqrt{\log(12/\beta)})^2
    \big(\frac{7}{8}+
    \frac{5\log({3}{/\astar})}{3} +
    \sqrt{\frac{3}{2}\log(3/\astar)}\big)
\\
    &\geq 
\psitildec{\astar}{12} \\ 
&\qquad + \frac{3K}{n}(1+\sqrt{2\log(3/\beta)})^2
\end{talign}
Since $a \leq 1$, it suffices to have 
$x \geq \frac{1}{a} (b+\sqrt{c}) \geq \frac{b}{a}+\sqrt{\frac{c}{a}} \geq \frac{b+\sqrt{b^2+4ac}}{2a}$.
The advertised result now follows as 
$3+\half\sqrt{\frac{7}{8}}\leq \frac{7}{2}$ and $\sqrt{\frac{3}{2}}+\half\sqrt{\frac{5}{3}}\leq 2$. 
\subsection{\pcref{prop:coefficients_A_b}}\label{proof:coefficients_A_b}
Recall the definition \cref{eq:Rijkl} of $R_{ijkl}$, and, for each $i\in[s]$, define the Rademacher vector $\teps_i = 2\indic{i=\pi_i} - 1$. Observe that by construction, $\teps_i = \teps_{\swaps(i)}$ for all $i \in [s/2]$. Then
\begin{talign}
    R_{i j \pi_j \pi_i} &= (\frac{1+\teps_i}{2})(\frac{1+\teps_j}{2}) R_{i j j i} + (\frac{1+\teps_i}{2})(\frac{1-\teps_j}{2}) R_{i j (\swaps(j)) i} \\ &\quad + (\frac{1-\teps_i}{2})(\frac{1+\teps_j}{2}) R_{i j j (\swaps(i))} + (\frac{1-\teps_i}{2})(\frac{1-\teps_j}{2}) R_{i j (\swaps(j)) (\swaps(i))} \\ &= \frac{\teps_i \teps_j}{4} \big( R_{i j j i} - R_{i j (\swaps(j)) i} - R_{i j j (\swaps(i))} + R_{i j (\swaps(j)) (\swaps(i))} \big) \\ &\quad + \frac{\teps_i}{4} \big( R_{i j j i} + R_{i j (\swaps(j)) i} - R_{i j j (\swaps(i))} - R_{i j (\swaps(j)) (\swaps(i))} \big) \\ &\quad + \frac{\teps_j}{4} \big( R_{i j j i} - R_{i j (\swaps(j)) i} + R_{i j j (\swaps(i))} - R_{i j (\swaps(j)) (\swaps(i))} \big) \\ &\quad + \frac{1}{4} \big( R_{i j j i} + R_{i j (\swaps(j)) i} + R_{i j j (\swaps(i))} + R_{i j (\swaps(j)) (\swaps(i))} \big).
\end{talign}
Hence,
\begin{talign}
    \tilde{V}_n^{\pi,s} &= \frac{1}{n^2} \sum_{i,j=1}^{s} \big( \frac{\teps_i \teps_j}{4} \big( R_{i j j i} - R_{i j (\swaps(j)) i} - R_{i j j (\swaps(i))} + R_{i j (\swaps(j)) (\swaps(i))} \big) \\ &\qquad\qquad + \frac{\teps_i}{4} \big( R_{i j j i} + R_{i j (\swaps(j)) i} - R_{i j j (\swaps(i))} - R_{i j (\swaps(j)) (\swaps(i))} \big) \\ &\qquad\qquad + \frac{\teps_j}{4} \big( R_{i j j i} - R_{i j (\swaps(j)) i} + R_{i j j (\swaps(i))} - R_{i j (\swaps(j)) (\swaps(i))} \big) \\ &\qquad\qquad + \frac{1}{4} \big( R_{i j j i} + R_{i j (\swaps(j)) i} + R_{i j j (\swaps(i))} + R_{i j (\swaps(j)) (\swaps(i))} \big) \big) 
    \\ &\defeq  \teps^{\top} A \teps + b^{\top} \teps + c,
\end{talign}
where $\tilde{A} \in \R^{s \times s}$, $\tilde{b} \in \R^{s}$, and $\tilde{c} \in \R$ have the components
\begin{talign}
\begin{split} \label{eq:a_ij_b_i_c}
    \tilde{a}_{ij} &= \frac{1}{4 n^2} \big( R_{i j j i} - R_{i j (\swaps(j)) i} - R_{i j j (\swaps(i))} + R_{i j (\swaps(j)) (\swaps(i))} \big), \\
    \tilde{b}_{i} &= \frac{1}{4 n^2} \sum_{j=1}^{n} \big( R_{i j j i} + R_{i j (\swaps(j)) i} - R_{i j j (\swaps(i))} - R_{i j (\swaps(j)) (\swaps(i))} \\ &\qquad\qquad\quad + R_{j i i j} + R_{j i (\swaps(i)) j} - R_{j i i (\swaps(j))} - R_{j i (\swaps(i)) (\swaps(j))} \big), \\
    \tilde{c} &= \frac{1}{4 n^2} \sum_{i,j=1}^{s} \big( R_{i j j i} + R_{i j (\swaps(j)) i} + R_{i j j (\swaps(i))} + R_{i j (\swaps(j)) (\swaps(i))} \big).
\end{split}
\end{talign}

We now rewrite the terms $R_{ijji}$, $R_{ijqi}$, $R_{ijjr}$, and $R_{ijqr}$ in terms of kernel expectations: %
\begin{talign} 
\begin{split} \label{eq:hat_M_1}
    \frac{1}{m^2} R_{i j j i} &= \frac{1}{m^2} \sum_{k,l = 1}^{m} \big( g_Y(Y_{m(i-1)+k},Y_{m(j-1)+l}) - n^{-1} \sum_{j'=1}^{n} g_Y(Y_{m(i-1)+k},Y_{j'}) \big) \\ &\qquad\qquad \times \big( g_Z(Z_{m(j-1)+k},Z_{m(i-1)+l}) - n^{-1} \sum_{i'=1}^{n} g_Z(Z_{m(j-1)+k},Z_{i'}) \big) \\ 
    &= 
((\delta_{\mathbb{X}^{(j)}} - \delta_{\mathbb{Y}} \times \delta_{\mathbb{Z}^{(j)}} )\times (\delta_{\mathbb{X}^{(i)}} - \delta_{\mathbb{Y}^{(i)}} \times \delta_{\mathbb{Z}} ))\kernel,
\end{split}
\end{talign}
\begin{talign} 
\begin{split} \label{eq:hat_M_2}
\frac{1}{m^2} R_{i j q i} 
    &= 
\frac{1}{m^2} \sum_{k,l = 1}^{m} \big( g_Y(Y_{m(i-1)+k},Y_{m(j-1)+l}) - n^{-1} \sum_{j'=1}^{n} g_Y(Y_{m(i-1)+k},Y_{j'}) \big) \\ 
    &\qquad\qquad \times 
\big( g_Z(Z_{m(q-1)+k},Z_{m(i-1)+l}) - n^{-1} \sum_{i'=1}^{n} g_Z(Z_{m(q-1)+k},Z_{i'}) \big) \\ 
    &= 
((\delta_{\mathbb{Y}^{(j)}} \times \delta_{\mathbb{Z}^{(q)}} - \delta_{\mathbb{Y}} \times \delta_{\mathbb{Z}^{(q)}} )\times(\delta_{\mathbb{X}^{(i)}} - \delta_{\mathbb{Y}^{(i)}} \times \delta_{\mathbb{Z}} ))\kernel,
\end{split}
\end{talign}
\begin{talign} 
\begin{split} \label{eq:hat_M_3}
    \frac{1}{m^2} R_{i j j r} &= \frac{1}{m^2} \sum_{k,l = 1}^{m} \big( g_Y(Y_{m(i-1)+k},Y_{m(j-1)+l}) - n^{-1} \sum_{j'=1}^{n} g_Y(Y_{m(i-1)+k},Y_{j'}) \big) \\ &\qquad\qquad \times \big( g_Z(Z_{m(j-1)+k},Z_{m(r-1)+l}) - n^{-1} \sum_{i'=1}^{n} g_Z(Z_{m(j-1)+k},Z_{i'}) \big) \\ 
    &= 
((\delta_{\mathbb{X}^{(j)}} - \delta_{\mathbb{Y}} \times \delta_{\mathbb{Z}^{(j)}} )\times (\delta_{\mathbb{Y}^{(i)}} \times \delta_{\mathbb{Z}^{(r)}} - \delta_{\mathbb{Y}^{(i)}} \times \delta_{\mathbb{Z}} ))\kernel, 
    \qtext{and}
\end{split}
\end{talign}
\begin{talign} 
\begin{split} \label{eq:hat_M_4}
    \frac{1}{m^2} R_{i j q r} &= \frac{1}{m^2} \sum_{k,l = 1}^{m} \big( g_Y(Y_{m(i-1)+k},Y_{m(j-1)+l}) - n^{-1} \sum_{j'=1}^{n} g_Y(Y_{m(i-1)+k},Y_{j'}) \big) \\ &\qquad\qquad \times \big( g_Z(Z_{m(q-1)+k},Z_{m(r-1)+l}) - n^{-1} \sum_{i'=1}^{n} g_Z(Z_{m(q-1)+k},Z_{i'}) \big) \\ 
    &= 
((\delta_{\mathbb{Y}^{(j)}} \times \delta_{\mathbb{Z}^{(q)}} - \delta_{\mathbb{Y}} \times \delta_{\mathbb{Z}^{(q)}} )\times  (\delta_{\mathbb{Y}^{(i)}} \times \delta_{\mathbb{Y}^{(r)}} - \delta_{\mathbb{Y}^{(i)}} \times \delta_{\mathbb{Z}} ))\kernel.
\end{split}
\end{talign}
Plugging $q = \swaps(j)$ and $r= \swaps(i)$ in \cref{eq:hat_M_1}, \cref{eq:hat_M_2}, \cref{eq:hat_M_3}, \cref{eq:hat_M_4}, we obtain
\begin{talign}
\begin{split}
    \tilde{a}_{ij} &= \frac{1}{4s^2} \big( \frac{R_{i j j i}}{m^2} - \frac{R_{i j (\swaps(j)) i}}{m^2} - \frac{R_{i j j (\swaps(i))}}{m^2} + \frac{R_{i j (\swaps(j)) (\swaps(i))}}{m^2} \big)  
    =  
\frac{1}{4s^2}  (\mu_j\times\tilde{\mu}_i) \kernel.
\end{split}
\end{talign}
We bound the following quantity using the Cauchy-Schwarz inequality:
\begin{talign}
\begin{split} \label{eq:bar_a_ij}
    &\big|\tilde{a}_{ij} - \frac{1}{4s^2} \mmd^2(\pjnt, \P \times \Q) \big| \\ &= \frac{1}{4s^2} \big| \langle \mu_j \kernel, \tilde{\mu}_i \kernel \rangle_{\kernel} - \langle (\pjnt - \P \times \Q )\kernel, (\pjnt - \P \times \Q )\kernel \rangle_{\kernel} \big| \\ &= \frac{1}{4s^2} \big| \langle \mu_j \kernel - (\pjnt - \P \times \Q )\kernel, \tilde{\mu}_i \kernel \rangle_{\kernel} \\ &\qquad\qquad - \langle (\pjnt - \P \times \Q )\kernel - \tilde{\mu}_i \kernel, (\pjnt - \P \times \Q )\kernel \rangle_{\kernel} \big| \\ &= \frac{1}{4s^2} \big| \langle \mu_j \kernel - (\pjnt - \P \times \Q )\kernel, \tilde{\mu}_i \kernel - (\pjnt - \P \times \Q )\kernel \rangle_{\kernel} \\ &\qquad\qquad + \langle \mu_j \kernel - (\pjnt - \P \times \Q )\kernel, (\pjnt - \P \times \Q )\kernel \rangle \\ &\qquad\qquad - \langle (\pjnt - \P \times \Q )\kernel - \tilde{\mu}_i \kernel, (\pjnt - \P \times \Q )\kernel \rangle_{\kernel} \big| \\ &\leq \frac{1}{4s^2} \big( \| (\mu_j - \pjnt + \P \times \Q )\kernel \|_{\kernel} \| (\tilde{\mu}_i - \pjnt + \P \times \Q )\kernel \|_{\kernel} \\ &+ \mmd(\pjnt, \P \times \Q) \big( \| (\mu_j - \pjnt + \P \times \Q )\kernel \|_{\kernel} + \| (\tilde{\mu}_i - \pjnt + \P \times \Q )\kernel \|_{\kernel} \big) \big).
\end{split}
\end{talign}

Plugging $q = \swaps(j)$ and $r= \swaps(i)$ into \cref{eq:hat_M_1}, \cref{eq:hat_M_2}, \cref{eq:hat_M_3}, \cref{eq:hat_M_4}, we find that
\begin{talign}
\begin{split}
    \tilde{b}_i &= \frac{1}{4 s^2} \sum_{j=1}^{s} \big( \frac{R_{i j j i}}{m^2} + \frac{R_{i j (\swaps(j)) i}}{m^2} - \frac{R_{i j j (\swaps(i))}}{m^2} - \frac{R_{i j (\swaps(j)) (\swaps(i))}}{m^2} \\ &\qquad\qquad\quad + \frac{R_{j i i j}}{m^2} + \frac{R_{j i (\swaps(i)) j}}{m^2} - \frac{R_{j i i (\swaps(j))}}{m^2} - \frac{R_{j i (\swaps(i)) (\swaps(j))}}{m^2} \big) 
     \\ &= \frac{1}{4 s^2} \big(  \big\langle \sum_{j=1}^{s} (\delta_{\mathbb{X}^{(j)}} - \delta_{\mathbb{Y}} \times \delta_{\mathbb{Z}^{(j)}} + \delta_{\mathbb{Y}^{(j)}} \times \delta_{\mathbb{Z}^{(\swaps(j))}} - \delta_{\mathbb{Y}} \times \delta_{\mathbb{Z}^{(\swaps(j))}} )\kernel,  \\ &\qquad\qquad\qquad (\delta_{\mathbb{X}^{(i)}} - \delta_{\mathbb{Y}^{(i)}} \times \delta_{\mathbb{Z}} - \delta_{\mathbb{Y}^{(i)}} \times \delta_{\mathbb{Z}^{(\swaps(i))}} + \delta_{\mathbb{Y}^{(i)}} \times \delta_{\mathbb{Z}} )\kernel \big\rangle_{\kernel} \\ &\qquad\qquad\quad +  \big\langle (\delta_{\mathbb{X}^{(i)}} - \delta_{\mathbb{Y}} \times \delta_{\mathbb{Z}^{(i)}} + \delta_{\mathbb{Y}^{(i)}} \times \delta_{\mathbb{Z}^{(\swaps(i))}} - \delta_{\mathbb{Y}} \times \delta_{\mathbb{Z}^{(\swaps(i))}} )\kernel,  \\ &\qquad\qquad\qquad \sum_{j=1}^{s} (\delta_{\mathbb{X}^{(j)}} - \delta_{\mathbb{Y}^{(j)}} \times \delta_{\mathbb{Z}} - \delta_{\mathbb{Y}^{(j)}} \times \delta_{\mathbb{Z}^{(\swaps(j))}} + \delta_{\mathbb{Y}^{(j)}} \times \delta_{\mathbb{Z}} )\kernel \big\rangle_{\kernel} \big) \\ 
     &= \frac{1}{4s} ( \langle \mu \kernel, \tilde{\mu}_i \kernel \rangle_{\kernel} + \langle \mu_i \kernel, \tilde{\mu} \kernel \rangle_{\kernel} ).
\end{split}
\end{talign}
We also bound the following using the Cauchy-Schwarz inequality:
\begin{talign}
\begin{split} \label{eq:bar_b_i}
    &\big| \tilde{b}_i - \frac{1}{2s} \mmd^2(\pjnt, \P \times \Q) \big| \\ &= \frac{1}{4s} \big| \langle \mu \kernel, \tilde{\mu}_i \kernel \rangle_{\kernel} + \langle \mu_i \kernel, \tilde{\mu} \kernel \rangle_{\kernel} - 2 \langle (\pjnt - \P \times \Q )\kernel, (\pjnt - \P \times \Q )\kernel \rangle_{\kernel} \big| \\ &\leq \frac{1}{4s} \big( \| (\mu - \pjnt + \P \times \Q )\kernel \|_{\kernel} \| (\tilde{\mu}_i - \pjnt + \P \times \Q )\kernel \|_{\kernel} \\ &\qquad\qquad + \| (\mu_i - \pjnt + \P \times \Q )\kernel \|_{\kernel} \| (\tilde{\mu} - \pjnt + \P \times \Q)\kernel \|_{\kernel} \\ &+ \mmd(\pjnt, \P \times \Q) \big( \| (\mu - \pjnt + \P \times \Q)\kernel \|_{\kernel} + \| (\tilde{\mu}_i - \pjnt + \P \times \Q)\kernel \|_{\kernel} \\ &\qquad\qquad\qquad\qquad + \| (\mu_i - \pjnt + \P \times \Q)\kernel \|_{\kernel} + \| (\tilde{\mu} - \pjnt + \P \times \Q)\kernel \|_{\kernel} \big) \big).
\end{split}
\end{talign}

Plugging $q = \swaps(j)$ and $r= \swaps(i)$ into \cref{eq:hat_M_1}, \cref{eq:hat_M_2}, \cref{eq:hat_M_3}, \cref{eq:hat_M_4}, we find that
\begin{talign}
\begin{split}
    \tilde{c} &= \frac{1}{4 s^2} \sum_{i,j=1}^{s} \big( \frac{R_{i j j i}}{m^2} + \frac{R_{i j (\swaps(j)) i}}{m^2} + \frac{R_{i j j (\swaps(i))}}{m^2} + \frac{R_{i j (\swaps(j)) (\swaps(i))}}{m^2} \big) \\ 
    &= \frac{1}{4} \big\| \big(\delta_{\mathbb{X}} - 2 \delta_{\mathbb{Y}} \times \delta_{\mathbb{Z}} + \sum_{i=1}^{s} \delta_{\mathbb{Y}^{(i)}} \times \delta_{\mathbb{Z}^{(\swaps(i))}} \big)\kernel \big\|_{\kernel}^2 = \frac{1}{4} \| \mu \kernel \|_{\kernel}^2
\end{split}
\end{talign}
We also bound the scalar $\tilde{c} - \frac{1}{4} \mmd^2(\pjnt, \P \times \Q)$ using Cauchy-Schwarz:
\begin{talign}
\begin{split} \label{eq:bar_c}
    &\big|\tilde{c} - \frac{1}{4} \mmd^2(\pjnt, \P \times \Q) \big| = \big| \frac{1}{4} \| \mu \kernel \|_{\kernel}^2 - \frac{1}{4} \langle (\pjnt - \P \times \Q)\kernel, (\pjnt - \P \times \Q)\kernel \rangle_{\kernel} \big| \\ &= \frac{1}{4} \big| \langle (\mu - \pjnt + \P \times \Q)\kernel, \mu \kernel \rangle_{\kernel} - \langle (\pjnt - \P \times \Q - \mu)\kernel, (\pjnt - \P \times \Q)\kernel \rangle_{\kernel} \big|
    \\ &= \frac{1}{4} \big| \| (\mu - \pjnt + \P \times \Q)\kernel \|^2_{\kernel} - 2 \langle (\pjnt - \P \times \Q - \mu)\kernel, (\pjnt - \P \times \Q)\kernel \rangle_{\kernel} \big| \\ &\leq \frac{1}{4} \big( \| (\mu - \pjnt + \P \times \Q)\kernel \|_{\kernel}^2 + 2 \| (\mu - \pjnt + \P \times \Q)\kernel \|_{\kernel} \mmd(\pjnt, \P \times \Q) \big) \\ &= \frac{1}{4} \big( \frac{(\mathcal{A}^{(3)})^2}{s} + \frac{2 \mathcal{A}^{(3)} \mmd(\pjnt, \P \times \Q)}{\sqrt{s}}  \big).
\end{split}
\end{talign}
At this point, we further rewrite the expression $\tilde{V}_n^{\pi,s} = \teps^{\top} \tilde{A} \teps + \tilde{b}^{\top} \teps + \tilde{c}$ in terms of $\teps_{1:s/2}$. Since $\teps_{s/2:s} = \teps_{1:s/2}$, we have that
\begin{talign}
    \tilde{V}_n^{\pi,s} &= \teps_{1:s/2}^{\top} \tilde{A}_{1:s/2,1:s/2} \teps_{1:s/2} + \teps_{s/2:s}^{\top} \tilde{A}_{s/2:s,1:s/2} \teps_{1:s/2} \\ &\quad + \teps_{1:s/2}^{\top} \tilde{A}_{1:s/2,s/2:s} \teps_{s/2:s} + \teps_{s/2:s}^{\top} \tilde{A}_{s/2:s,s/2:s} \teps_{s/2:s} \\ &\quad + \tilde{b}_{1:s/2}^{\top} \teps_{1:s/2} + \tilde{b}_{s/2:s}^{\top} \teps_{s/2:s} + \tilde{c} \\ &= \teps_{1:s/2}^{\top} \big( \tilde{A}_{1:s/2,1:s/2} + \tilde{A}_{s/2:s,1:s/2} + \tilde{A}_{1:s/2,s/2:s} + \tilde{A}_{s/2:s,s/2:s} \big) \teps_{1:s/2} \\ &\quad + \big( \tilde{b}_{1:s/2} + \tilde{b}_{s/2:s} \big)^{\top} \teps_{1:s/2} + \tilde{c}.
\end{talign}
Thus, if we define 
\begin{talign} \label{eq:eps_A_b_c}
\eps &= \teps_{1:s/2} \dist\Unif(\{\pm1\}^{s/2}), \\
A &= \tilde{A}_{1:s/2,1:s/2} + \tilde{A}_{s/2:s,1:s/2} + \tilde{A}_{1:s/2,s/2:s} + \tilde{A}_{s/2:s,s/2:s} \in \mathbb{R}^{s/2 \times s/2}, \\
b &= \tilde{b}_{1:s/2} + \tilde{b}_{s/2:s} \in \mathbb{R}^{s/2}, \qtext{and}
c = \tilde{c} \in \mathbb{R}, 
\end{talign}
we have that
\begin{talign}
    \tilde{V}_n^{\pi,s} = \eps^{\top} A \eps + b^{\top} \eps + c.
\end{talign}
If we let $a_{ij}$ and $b_i$ be the components of $A$ and $b$, we have that
\begin{talign}
    a_{ij} &= \tilde{a}_{ij} + \tilde{a}_{(\swaps(i))j} + \tilde{a}_{i(\swaps(j))} + \tilde{a}_{(\swaps(i))(\swaps(j))}
    \qtext{and}
    b_{i} = \tilde{b}_{i} + \tilde{b}_{\swaps(i)}.
\end{talign}
Next, define 
\begin{talign} 
\begin{split} \label{eq:bar_A_b_c}
    \bar{A} &= A - \frac{1}{s^2} \mmd^2(\pjnt, \P \times \Q) \boldone \boldone^{\top}, \\ 
    \bar{b} &= b- \frac{\mmd^2(\pjnt, \P \times \Q)}{s} \boldone \boldone^{\top}, \qtext{and} \\
    \bar{c} &= c - \frac{1}{4} \mmd^2(\pjnt, \P \times \Q). 
\end{split}
\end{talign}
If we let $\bar{a}_{ij}$ and $\bar{b}_i$ be the components of $\bar{A}$ and $\bar{b}$, by equations \cref{eq:bar_a_ij} and \cref{eq:bar_b_i}, we have that
\begin{talign}
\big| \bar{a}_{ij} \big| 
    &\leq \big| \tilde{a}_{ij} - \frac{1}{4s^2} \mmd^2(\pjnt, \P \times \Q) \big| + \big| \tilde{a}_{(\swaps(i))j} - \frac{1}{4s^2} \mmd^2(\pjnt, \P \times \Q) \big| \\ 
    & + \big| \tilde{a}_{i(\swaps(j))} - \frac{1}{4s^2} \mmd^2(\pjnt, \P \times \Q) \big| + \big| \tilde{a}_{(\swaps(i))(\swaps(j))} - \frac{1}{4s^2} \mmd^2(\pjnt, \P \times \Q) \big| 
    \\ 
    &\leq \frac{1}{4s^2} \big( \big( \| (\mu_j - \pjnt + \P \times \Q )\kernel \|_{\kernel} + \| (\mu_{\swaps(j)} - \pjnt + \P \times \Q )\kernel \|_{\kernel} \big) \\ 
    &\qquad\quad \times \big( \| (\tilde{\mu}_i - \pjnt + \P \times \Q )\kernel \|_{\kernel} + \| (\tilde{\mu}_{\swaps(i)} - \pjnt + \P \times \Q )\kernel \|_{\kernel} \big) \\ &\quad + \mmd(\pjnt, \P \times \Q) \big( \| (\mu_j - \pjnt + \P \times \Q )\kernel \|_{\kernel} + \| (\tilde{\mu}_i - \pjnt + \P \times \Q )\kernel \|_{\kernel} \\ 
    &\qquad\quad + \| (\mu_{\swaps(j)} - \pjnt + \P \times \Q )\kernel \|_{\kernel} + \| (\tilde{\mu}_{\swaps(i)} - \pjnt + \P \times \Q )\kernel \|_{\kernel} \big) \big) 
\end{talign}
and%
\begin{talign}
    \big| \bar{b}_{i} \big| 
    &\leq \big| \tilde{b}_i - \frac{1}{2s} \mmd^2(\pjnt, \P \times \Q) \big| \big| + \big| \tilde{b}_{\swaps(i)} - \frac{1}{2s} \mmd^2(\pjnt, \P \times \Q) \big| \\ 
    &\leq \frac{1}{4s} \big( \| (\mu - \pjnt + \P \times \Q )\kernel \|_{\kernel} \big( \| (\tilde{\mu}_i - \pjnt + \P \times \Q )\kernel \|_{\kernel} + \| (\tilde{\mu}_{\swaps(i)} - \pjnt + \P \times \Q )\kernel \|_{\kernel} \big) \\ 
    &\quad + \big( \| (\mu_i - \pjnt + \P \times \Q )\kernel \|_{\kernel} + \| (\mu_{\swaps(i)} - \pjnt + \P \times \Q )\kernel \|_{\kernel} \big) \| (\tilde{\mu} - \pjnt + \P \times \Q)\kernel \|_{\kernel} \\ &\quad + \mmd(\pjnt, \P \times \Q) \big( 2 \| (\mu - \pjnt + \P \times \Q)\kernel \|_{\kernel} + 2 \| (\tilde{\mu} - \pjnt + \P \times \Q)\kernel \|_{\kernel} \\ &\qquad\qquad\qquad\qquad + \| (\tilde{\mu}_i - \pjnt + \P \times \Q)\kernel \|_{\kernel} + \| (\tilde{\mu}_{\swaps(i)} - \pjnt + \P \times \Q)\kernel \|_{\kernel} \\ &\qquad\qquad\qquad\qquad + \| (\mu_i - \pjnt + \P \times \Q)\kernel \|_{\kernel} + \| (\mu_{\swaps(i)} - \pjnt + \P \times \Q)\kernel \|_{\kernel} \big) \big).
\end{talign}
We now turn to bounding  $\|\bar{A}\|^2_{F}$ in terms of $\mathcal{A}^{(1)}$ and $\mathcal{A}^{(2)}$ using Jensen's inequality:
\begin{talign}
\begin{split}
    &16s^4 \sum_{i,j=1}^{s/2} \bar{a}_{ij}^2 \\ &\leq \sum_{i,j=1}^{s/2} 
    \big( \big( \| (\mu_j - \pjnt + \P \times \Q )\kernel \|_{\kernel} + \| (\mu_{\swaps(j)} - \pjnt + \P \times \Q )\kernel \|_{\kernel} \big) \\ &\qquad\qquad\quad \times \big( \| (\tilde{\mu}_i - \pjnt + \P \times \Q )\kernel \|_{\kernel} + \| (\tilde{\mu}_{\swaps(i)} - \pjnt + \P \times \Q )\kernel \|_{\kernel} \big) \\ &\quad + \mmd(\pjnt, \P \times \Q) \big( \| (\mu_j - \pjnt + \P \times \Q )\kernel \|_{\kernel} + \| (\tilde{\mu}_i - \pjnt + \P \times \Q )\kernel \|_{\kernel} \\ &\qquad\qquad\qquad\qquad\quad + \| (\mu_{\swaps(j)} - \pjnt + \P \times \Q )\kernel \|_{\kernel} + \| (\tilde{\mu}_{\swaps(i)} - \pjnt + \P \times \Q )\kernel \|_{\kernel} \big) \big)^2
    \\ &\leq \sum_{i,j=1}^{s/2}  \big( 2 \big( 2 \| (\mu_j - \pjnt + \P \times \Q )\kernel \|_{\kernel}^2 + 2 \| (\mu_{\swaps(j)} - \pjnt + \P \times \Q )\kernel \|_{\kernel}^2 \big) \\ &\qquad\qquad\quad \times \big( 2 \| (\tilde{\mu}_i - \pjnt + \P \times \Q )\kernel \|_{\kernel}^2 + 2 \| (\tilde{\mu}_{\swaps(i)} - \pjnt + \P \times \Q )\kernel \|_{\kernel}^2 \big) \\ &\quad + 2 \mmd(\pjnt, \P \times \Q) \big( 4 \| (\mu_j - \pjnt + \P \times \Q )\kernel \|_{\kernel}^2 + 4 \| (\tilde{\mu}_i - \pjnt + \P \times \Q )\kernel \|_{\kernel}^2 \\ &\qquad\qquad\qquad\qquad\quad + 4 \| (\mu_{\swaps(j)} - \pjnt + \P \times \Q )\kernel \|_{\kernel}^2 + 4 \| (\tilde{\mu}_{\swaps(i)} - \pjnt + \P \times \Q )\kernel \|_{\kernel}^2 \big) \big)
    \\ &\leq 4 \big( \sum_{j=1}^{s} \| (\mu_j - \pjnt + \P \times \Q )\kernel \|_{\kernel}^2 \big)  \big( \sum_{i=1}^{s} \| (\tilde{\mu}_i - \pjnt + \P \times \Q )\kernel \|_{\kernel}^2 \big) \\ &\qquad + 
    4 s \mmd(\pjnt, \P \times \Q) \big( \sum_{j=1}^{s} \| (\mu_j - \pjnt + \P \times \Q )\kernel \|_{\kernel}^2 + \sum_{i=1}^{s} \| (\tilde{\mu}_i - \pjnt + \P \times \Q )\kernel \|_{\kernel}^2 \big)
    \\ 
    &=
    4 (\mathcal{A}^{(1)})^2 (\mathcal{A}^{(2)})^2 + 4 s \mmd^2(\pjnt, \P \times \Q) \big( (\mathcal{A}^{(1)})^2 + (\mathcal{A}^{(2)})^2 \big).
\end{split}
\end{talign}
Next, we bound $\mathrm{Tr}(\bar{A})$ in terms of $\mathcal{A}^{(1)}$ and $\mathcal{A}^{(2)}$ using Cauchy-Schwarz:
\begin{talign}
\begin{split}
    &4 s^2 \sum_{i=1}^{s/2} \bar{a}_{ii} \leq \sum_{i=1}^{s/2} 
    \big( \big( \| (\mu_i - \pjnt + \P \times \Q )\kernel \|_{\kernel} + \| (\mu_{\swaps(i)} - \pjnt + \P \times \Q )\kernel \|_{\kernel} \big) \\ &\qquad\qquad\quad \times \big( \| (\tilde{\mu}_i - \pjnt + \P \times \Q )\kernel \|_{\kernel} + \| (\tilde{\mu}_{\swaps(i)} - \pjnt + \P \times \Q )\kernel \|_{\kernel} \big) \\ &\quad + \mmd(\pjnt, \P \times \Q) \big( \| (\mu_i - \pjnt + \P \times \Q )\kernel \|_{\kernel} + \| (\tilde{\mu}_i - \pjnt + \P \times \Q )\kernel \|_{\kernel} \\ &\qquad\qquad\qquad\qquad\quad + \| (\mu_{\swaps(i)} - \pjnt + \P \times \Q )\kernel \|_{\kernel} + \| (\tilde{\mu}_{\swaps(i)} - \pjnt + \P \times \Q )\kernel \|_{\kernel} \big) \big)
    \\ &\leq 2 \big( \sum_{i=1}^{s} \| (\mu_i - \pjnt + \P \times \Q)\kernel \|_{\kernel}^2 \big)^{1/2} \big( \sum_{i=1}^{s} \| (\tilde{\mu}_i - \pjnt + \P \times \Q)\kernel \|_{\kernel} \big)^{1/2} \\ &+ \sqrt{s} \mmd(\pjnt, \P \times \Q) \big( \sqrt{ \sum_{i=1}^{s} \| (\mu_i \! - \! \pjnt \! + \! \P \times \Q)\kernel \|_{\kernel}^2 } \! + \! \sqrt{ \sum_{i=1}^{s} \| (\tilde{\mu}_i \! - \! \pjnt \! + \! \P \times \Q)\kernel \|_{\kernel}^2 } \big) \\ &= 2 \mathcal{A}^{(1)} \mathcal{A}^{(2)} +  \sqrt{s} \mmd(\pjnt, \P \times \Q) \big( \mathcal{A}^{(1)} + \mathcal{A}^{(2)} \big).
\end{split}
\end{talign}

Finally, we bound $\statictwonorm{\bar{b}}^2$ in terms of $\mathcal{A}^{(1)}$, $\mathcal{A}^{(2)}$, $\mathcal{A}^{(3)}$ and $\mathcal{A}^{(4)}$:
\begin{talign}
\begin{split}
    &16 s^2 \sum_{i=1}^{s/2} \bar{b}_i^2 
    \\ &\leq \sum_{i=1}^{s/2} \big( \| (\mu - \pjnt + \P \times \Q )\kernel \|_{\kernel} \big( \| (\tilde{\mu}_i - \pjnt + \P \times \Q )\kernel \|_{\kernel} + \| (\tilde{\mu}_{\swaps(i)} - \pjnt + \P \times \Q )\kernel \|_{\kernel} \big) \\ &\qquad\quad + \big( \| (\mu_i - \pjnt + \P \times \Q )\kernel \|_{\kernel} + \| (\mu_{\swaps(i)} - \pjnt + \P \times \Q )\kernel \|_{\kernel} \big) \| (\tilde{\mu} - \pjnt + \P \times \Q)\kernel \|_{\kernel} \\ &\quad + \mmd(\pjnt, \P \times \Q) \big( 2 \| (\mu - \pjnt + \P \times \Q)\kernel \|_{\kernel} + 2 \| (\tilde{\mu} - \pjnt + \P \times \Q)\kernel \|_{\kernel} \\ &\qquad\qquad\qquad\qquad + \| (\tilde{\mu}_i - \pjnt + \P \times \Q)\kernel \|_{\kernel} + \| (\tilde{\mu}_{\swaps(i)} - \pjnt + \P \times \Q)\kernel \|_{\kernel} \\ &\qquad\qquad\qquad\qquad + \| (\mu_i - \pjnt + \P \times \Q)\kernel \|_{\kernel} + \| (\mu_{\swaps(i)} - \pjnt + \P \times \Q)\kernel \|_{\kernel} \big) \big)^2
    \\ &\leq 12 \sum_{i=1}^{s/2} \big( \| (\mu - \pjnt + \P \times \Q )\kernel \|_{\kernel}^2 \big( \| (\tilde{\mu}_i - \pjnt + \P \times \Q )\kernel \|_{\kernel}^2 + \| (\tilde{\mu}_{\swaps(i)} - \pjnt + \P \times \Q )\kernel \|_{\kernel}^2 \big) \\ &\qquad\quad + \big( \| (\mu_i - \pjnt + \P \times \Q )\kernel \|_{\kernel}^2 + \| (\mu_{\swaps(i)} - \pjnt + \P \times \Q )\kernel \|_{\kernel}^2 \big) \| (\tilde{\mu} - \pjnt + \P \times \Q)\kernel \|_{\kernel}^2 \\ &\quad + \mmd(\pjnt, \P \times \Q)^2 \big( 2 \| (\mu - \pjnt + \P \times \Q)\kernel \|_{\kernel}^2 + 2 \| (\tilde{\mu} - \pjnt + \P \times \Q)\kernel \|_{\kernel}^2 \\ &\qquad\qquad\qquad\qquad + \| (\tilde{\mu}_i - \pjnt + \P \times \Q)\kernel \|_{\kernel}^2 + \| (\tilde{\mu}_{\swaps(i)} - \pjnt + \P \times \Q)\kernel \|_{\kernel}^2 \\ &\qquad\qquad\qquad\qquad + \| (\mu_i - \pjnt + \P \times \Q)\kernel \|_{\kernel}^2 + \| (\mu_{\swaps(i)} - \pjnt + \P \times \Q)\kernel \|_{\kernel}^2 \big) \big)
    \\ &= \frac{12 (\mathcal{A}^{(2)})^2 (\mathcal{A}^{(3)})^2}{s} + \frac{12 (\mathcal{A}^{(1)})^2 (\mathcal{A}^{(4)})^2}{s} \\ &\qquad+ 12 \mmd^2(\pjnt, \P \times \Q) \big( (\mathcal{A}^{(1)})^2 +(\mathcal{A}^{(2)})^2 + (\mathcal{A}^{(3)})^2 + (\mathcal{A}^{(4)})^2 \big).
\end{split}
\end{talign}
\subsection{\pcref{lem:expectation_A_i}} \label{proof:expectation_A_i}
\subsubsection{Bounding $\E[\mathcal{A}^{(1)}]$ and $\E[\mathcal{A}^{(2)}]$}
By the triangle inequality, %
\begin{talign}
\begin{split} \label{eq:expectation_A}
&\mathbb{E}\big[\mathcal{A}^{(1)} \big] 
    =
\mathbb{E}\big[\sqrt{\sum_{i=1}^{s} \| (\mu_i - \pjnt + \P \times \Q)\kernel \|_{\kernel}^2} \big] \\
    &\leq
\mathbb{E}\big[\sqrt{\sum_{i=1}^{s} \mmd(\delta_{\mathbb{X}^{(i)}},\pjnt)^2}\big]
    +
\mathbb{E}\big[\sqrt{\sum_{i=1}^{s} \mmd({\delta_{\mathbb{Y}^{(i)}}} \!\!\times\! {\delta_{\mathbb{Z}^{(\swaps(i))}}}, \P\!\times\!\Q)^2} \big]\\
    &+
\mathbb{E}\big[\sqrt{\sum_{i=1}^{s} \mmd({\delta_{\mathbb{Y}}} \times {\delta_{\mathbb{Z}^{(i)}}},  \P\times\Q)^2}\big]
    +
\mathbb{E}\big[\sqrt{\sum_{i=1}^{s} \mmd({\delta_{\mathbb{Y}}}\! \times\! {\delta_{\mathbb{Z}^{(\swaps(i))}}}, \P\!\times\!\Q)^2} \big].
\end{split}
\end{talign}
Cauchy-Schwarz and the expected squared MMD bound of \cref{eq:mean-DeltaX} now yield
\begin{talign}
\mathbb{E}\big[\sqrt{\sum_{i=1}^{s} \mmd(\delta_{\mathbb{X}^{(i)}},\pjnt)^2}\big]
    \leq
\sqrt{\mathbb{E}\big[\sum_{i=1}^{s} \mmd(\delta_{\mathbb{X}^{(i)}},\pjnt)^2\big]}
    \leq 
\sqrt{\frac{sK}{m}}.
\end{talign}
Similarly, the triangle inequality, Cauchy-Schwarz, and the expected squared MMD bounds of \cref{eq:mean-DeltaXprime} yield
\begin{talign}
&\mathbb{E}\big[\sqrt{\sum_{i=1}^{s} \mmd({\delta_{\mathbb{Y}}} \!\times\! {\delta_{\mathbb{Z}^{(i)}}},  \P\!\times\!\Q)^2}\big] \\
    &\leq
\mathbb{E}\big[\sqrt{\sum_{i=1}^{s} \mmd({\delta_{\mathbb{Y}}} \!\times\! {\delta_{\mathbb{Z}^{(i)}}},  \delta_{\mathbb{Y}}\!\times\!\Q)^2}\big] 
    +
\mathbb{E}\big[\sqrt{\sum_{i=1}^{s} \mmd({\delta_{\mathbb{Y}}} \!\times\! {\Q},  \P\!\times\!\Q)^2}\big] \\
    &\leq
\sqrt{\mathbb{E}\big[\sum_{i=1}^{s} \mmd({\delta_{\mathbb{Y}}} \!\times\! {\delta_{\mathbb{Z}^{(i)}}},  \delta_{\mathbb{Y}}\!\times\!\Q)^2\big]}
    +
\sqrt{\mathbb{E}\big[\sum_{i=1}^{s} \mmd({\delta_{\mathbb{Y}}} \!\times\! {\Q},  \P\!\times\!\Q)^2\big]}\\
    &\leq
\sqrt{\frac{sK}{m}}+\sqrt{\frac{sK}{n}}.
\end{talign}
Parallel reasoning yields the bounds
\begin{talign}
\mathbb{E}\big[\sqrt{\sum_{i=1}^{s} \mmd({\delta_{\mathbb{Y}}}\! \times\! {\delta_{\mathbb{Z}^{(\swaps(i))}}}, \P\!\times\!\Q)^2} \big]
    &\leq
\sqrt{\frac{sK}{m}}+\sqrt{\frac{sK}{n}} 
    \qtext{and} \\
\mathbb{E}\big[\sqrt{\sum_{i=1}^{s} \mmd(\delta_{\mathbb{Y}^{(i)}} \!\!\times\! \delta_{\mathbb{Z}^{(\swaps(i))}}, \P\!\times\!\Q)^2} \big]
    &\leq
2\sqrt{\frac{sK}{m}}. 
\end{talign}
Hence $\E[\mathcal{A}^{(1)}] \leq D_1$ as advertised, and $\E[\mathcal{A}^{(2)}] \leq D_2$ follows from parallel reasoning.

\subsubsection{Bounding $\E[\mathcal{A}^{(3)}]$ and $\E[\mathcal{A}^{(4)}]$}
Since $\delta_{\mathbb{Y}^{(j)}} \!\!\times\! \delta_{\mathbb{Z}^{(\swaps(j))}}=\P\times\Q$, independence and our prior MMD moment bounds imply that 
\begin{talign}
&\E[\mmd(\frac{2}{s}\sum_{j=1}^{s/2}\delta_{\mathbb{Y}^{(j)}} \!\!\times\! \delta_{\mathbb{Z}^{(\swaps(j))}},\P\!\times\!\Q)^2]
    =
\frac{4}{s^2}\sum_{j=1}^{s/2}\E[\mmd(\delta_{\mathbb{Y}^{(j)}} \!\!\times\! \delta_{\mathbb{Z}^{(\swaps(j))}},\P\!\times\!\Q)^2] \\
    &\leq
\frac{8}{s^2}\sum_{j=1}^{s/2}\E[\mmd(\delta_{\mathbb{Y}^{(j)}} \!\!\times\! \Q,\P\!\times\!\Q)^2]
    +
\frac{8}{s^2}\sum_{j=1}^{s/2}\E[\mmd(\delta_{\mathbb{Y}^{(j)}} \!\!\times\! \Q,\P\!\times\!\Q)^2]
    \leq
\frac{8K}{n}.
\end{talign}
Hence, by the triangle inequality, Cauchy-Schwarz, and our prior MMD moment bounds
\begin{talign}
&\mathbb{E}[\mathcal{A}^{(3)}] 
    = 
\sqrt{s} \mathbb{E}\big[\| (\mu - \pjnt + \P \times \Q)\kernel \|_{\kernel} \big] 
    \leq
\sqrt{s}\E[\mmd(\delta_{\X},\pjnt)
    +
2\mmd(\delta_{\Y}\times\delta_{\Z}, \P\times\Q) \\
    &+
\half\mmd(\frac{2}{s}\sum_{j=1}^{s/2}\delta_{\mathbb{Y}^{(j)}} \!\!\times\! \delta_{\mathbb{Z}^{(\swaps(j))}},\P\!\times\!\Q)
    +
\half\mmd(\frac{2}{s}\sum_{j=\frac{s}{2}+1}^{s}\delta_{\mathbb{Y}^{(j)}} \!\!\times\! \delta_{\mathbb{Z}^{(\swaps(j))}},\P\!\times\!\Q)
] \\
    &\leq
\sqrt{\frac{sK}{n}}
    +
4\sqrt{\frac{sK}{n}}
    +
\sqrt{\frac{8sK}{n}}.
\end{talign}
Analogous reasoning gives
$\mathbb{E}[\mathcal{A}^{(4)}] 
    \leq
D_4$, completing the proof.

\subsection{\pcref{prop:hp_bounds_A_i}} \label{proof:hp_bounds_A_i}

Fix any $k\in[n]$, let $\tilde{X}=(\tilde{Y},\tilde{Z})$ be an independent draw from $\pjnt$, and let $\tilde{\X} = (X_1,\dots,X_{k-1},\tilde{X},X_k,\dots,X_n) = (\tilde{\Y},\tilde{\Z})$.
Viewing $\mathcal{A}^{(1)}$ as a function of $\X$ and using the triangle inequality, we can write
\begin{talign}
&|\mathcal{A}^{(1)}(\X)
    -
\mathcal{A}^{(1)}(\tilde{\X})|
    \leq
\frac{1}{m}\mmd(\delta_{X_k},\delta_{\tilde{X}}) 
    +
\frac{2}{m}\mmd[g_Z](\delta_{Z_k},\delta_{\tilde{Z}})\sqrt{ (\delta_{\tilde{\Y}}\times\delta_{\tilde{\Y}})g_Y} \\
    &+
\frac{1}{m}\mmd[g_Y](\delta_{Y_k},\delta_{\tilde{Y}})\sqrt{\sup_z g_Z(z,z)}%
    +
\frac{1}{m}\mmd[g_Z](\delta_{Z_k},\delta_{\tilde{Z}})\sqrt{\sup_y g_Y(y,y)}\\ %
    &+
\frac{1}{n}\mmd[g_Y](\delta_{Y_k},\delta_{\tilde{Y}})\sqrt{\sum_{i=1}^s \mmd[g_Z]^2(\delta_{\Z^{(i)}},\delta_{\Z^{(\swaps(i))}})} 
    \leq
\frac{10\sqrt{K}}{m}+\frac{2\sqrt{sK}}{n}.
\end{talign}
Viewing $\mathcal{A}^{(3)}$ as a function of $\X$ and using the triangle inequality, we can also write
\begin{talign}
&\frac{1}{\sqrt{s}}|\mathcal{A}^{(3)}(\X)
    -
\mathcal{A}^{(3)}(\tilde{\X})|
    \leq
\frac{1}{n}\mmd(\delta_{X_k},\delta_{\tilde{X}}) 
    +
\frac{2}{n}\mmd[g_Z](\delta_{Z_k},\delta_{\tilde{Z}})\sqrt{ (\delta_{\tilde{\Y}}\times\delta_{\tilde{\Y}})g_Y} \\
    &+
\frac{2}{n}\mmd[g_Y](\delta_{Y_k},\delta_{\tilde{Y}})\sqrt{ (\delta_{\Z}\times\delta_{\Z})g_Y}
    +
\frac{1}{sm}\mmd[g_Y](\delta_{Y_k},\delta_{\tilde{Y}})\sqrt{\sup_zg_Z(z,z)}\\ %
    &+
\frac{1}{sm}\mmd[g_Z](\delta_{Z_k},\delta_{\tilde{Z}})\sqrt{\sup_yg_Y(y,y)} %
    \leq
\frac{14\sqrt{K}}{n}.
\end{talign}
Parallel reasoning additionally yields the bounds
\begin{talign}
&|\mathcal{A}^{(2)}(\X)
    -
\mathcal{A}^{(2)}(\tilde{\X})|
    \leq
\frac{6\sqrt{K}}{m}
    \qtext{and}
|\mathcal{A}^{(4)}(\X)
    -
\mathcal{A}^{(4)}(\tilde{\X})|
    \leq
\frac{6\sqrt{Ks}}{n}
.
\end{talign}
The advertised result now follows from the bounded differences inequality (\cref{thm:mcdiarmid_bounded}) and the mean estimates of \cref{lem:expectation_A_i}.

\subsection{\pcref{ind_perm_thresh_quantile_bound}} \label{proof:ind_perm_thresh_quantile_bound}
\subsubsection{Bounding $\tilde{V}^{\pi,s}$ in terms of $\tilde{\Psi}_{\X,s}(\alpha)$}
Fix any $\alpha,\beta\in(0,1)$.
By \cref{prop:coefficients_A_b}, %
\begin{talign}
\tilde{V}_n^{\pi,s} 
    &= 
\eps^{\top} A \eps + b^{\top} \eps + c \\ 
    &= 
\eps^{\top} \bar{A} \eps + 
\frac{\mmd^2(\pjnt, \P \times \Q)}{s^2} \eps^{\top} \boldone \boldone^{\top} \eps 
+ \bar{b}^{\top} \eps + 
\frac{\mmd^2(\pjnt, \P \times \Q)}{s} \boldone^{\top} \eps
    + 
\bar{c} + \frac{\mmd^2(\pjnt, \P \times \Q)}{4} \\ 
    &\leq 
|\eps^{\top} \bar{A} \eps -\mathrm{Tr}(\bar{A})| + \mathrm{Tr}(\bar{A}) + \bar{b}^{\top} \eps + \big|\bar{c} \big| 
    + 
\mmd^2(\pjnt, \P \times \Q) \big( \frac{1}{4} 
+ \frac{\eps^{\top} \boldone \boldone^{\top} \eps}{s^2} + \frac{\boldone^{\top} \eps}{s} 
\big).
\end{talign}
where $\eps\sim\Unif(\{\pm1\}^{s/2})$ is independent of $\X$, $A$, $b$, and $c$ were defined in \cref{eq:eps_A_b_c},
and $\bar{A}$, $\bar{b}$, and $\bar{c}$ were defined in \cref{eq:bar_A_b_c}.

Now fix any $\delta'\in(0,0.86)$ and $\delta'',\delta'''>0$. 
By equation \cref{eq:hanson_inverted_2} of the Hanson-Wright inequality (\cref{lem:hanson_wright}), with probability at least $1-\delta'$,
\begin{talign}
    |\eps^{\top} \bar{A} \eps - \mathrm{Tr}(\bar{A})| = |\eps^{\top} \bar{A} \eps - \E[\eps^{\top} \bar{A} \eps]| \leq \frac{ %
    20 \|\bar{A}\|_{\text{F}} \log(1/\delta')}{3}.
\end{talign}
Meanwhile, Hoeffding's inequality \citep[Thm.~2]{hoeffding1963probability} implies that 
\begin{talign}
\bar{b}^{\top} \eps 
    \leq 
\twonorm{\bar{b}}\sqrt{2 \log(1/\delta'')}
    &\qtext{with probability at least}
1-\delta'' 
    \qtext{and} \\
|\boldone^\top \eps| 
    \leq 
\sqrt{s \log(2/\delta''')}
    &\qtext{with probability at least}
1-\delta'''.
\end{talign}
Hence, with probability at least $1-\delta'''$,
\begin{talign} \label{eq:bound_rademacher_var}
    \frac{1}{4} 
    + \frac{\eps^{\top} \boldone \boldone^{\top} \eps}{s^2} + \frac{\boldone^{\top} \eps}{s}
    \leq \frac{1}{4} 
    + \frac{s \log (2/\delta''')}{s^2} + \sqrt{\frac{\log (1/\delta''')}{s}}
    = \frac{1}{4} + \frac{\log (2/\delta''')}{s} + \sqrt{\frac{\log (2/\delta''')}{s}}.
\end{talign}

Moreover, \cref{prop:coefficients_A_b} guarantees that
\begin{talign} 
\label{eq:bound_bar_c}
    &| \bar{c} | \leq \frac{1}{4} \big( \frac{(\mathcal{A}^{(3)})^2}{s} + \frac{2 \mathcal{A}^{(3)} \mmd(\pjnt, \P \times \Q)}{\sqrt{s}}  \big) \\
\begin{split} \label{eq:bounds_bar_A}
    &\|\bar{A} \|_{F}^2 \leq 
    \frac{(\mathcal{A}^{(1)})^2 (\mathcal{A}^{(2)})^2}{4 s^4}
    + \frac{\mmd^2(\pjnt, \P \times \Q) ( (\mathcal{A}^{(1)})^2 + (\mathcal{A}^{(2)})^2 )}{4 s^3},
    \\
    &\mathrm{Tr}(\bar{A}) \leq 
    \frac{\mathcal{A}^{(1)} \mathcal{A}^{(2)}}{2s^2} 
    + \frac{\mmd(\pjnt, \P \times \Q) ( \mathcal{A}^{(1)} + \mathcal{A}^{(2)} )}{4s^{3/2}}, 
    \qtext{and}
\end{split}\\ \label{eq:bound_bar_b}
    &\twonorm{\bar{b}}^2 \leq 
    \frac{3(\mathcal{A}^{(2)})^2 (\mathcal{A}^{(3)})^2 + (\mathcal{A}^{(1)})^2 (\mathcal{A}^{(4)})^2}{4 s^3}
    +
    \frac{3\mmd^2(\pjnt, \P \times \Q)( (\mathcal{A}^{(1)})^2 +(\mathcal{A}^{(2)})^2 + (\mathcal{A}^{(3)})^2 + (\mathcal{A}^{(4)})^2 )}{4s^2}.
\end{talign}
Hence, if we choose $\delta' = \delta'' = \delta''' = \alpha/3$, then, by the union bound and the triangle inequality, we have, with probability at least $1-\alpha$,
\begin{talign}
\begin{split} \label{eq:tilde_V_psi_def}
    \tilde{V}_n^{\pi,s} &\leq \tilde{\Psi}_{\X,s}(\alpha) \defeq \frac{\mathcal{A}^{(1)} \mathcal{A}^{(2)}}{s^2} \big( \frac{1}{2} + \frac{ 
    10\log(3/\alpha)}{3} \big) +  \frac{ \sqrt{\frac{3}{2}\log(3/\alpha)}(\mathcal{A}^{(2)} \mathcal{A}^{(3)} + \mathcal{A}^{(1)} \mathcal{A}^{(4)})}{s^{3/2}} + \frac{(\mathcal{A}^{(3)})^2}{4s} \\ & + 
    \mmd(\pjnt, \P \times \Q) \big( \frac{\mathcal{A}^{(1)} + \mathcal{A}^{(2)}}{s^{3/2}} \big( \frac{1}{4} \! + \! \frac{ 
    10\log(3/\alpha)}{3} \big) \! + \! \frac{\sqrt{\frac{3}{2}\log(3/\alpha)} ( \mathcal{A}^{(1)} \! + \! \mathcal{A}^{(2)} + \mathcal{A}^{(3)} \! + \! \mathcal{A}^{(4)})}{s} + \frac{\mathcal{A}^{(3)}}{2\sqrt{s}} \big)
    \\ & + \mmd^2(\pjnt, \P \times \Q) \big( 
    \frac{1}{4} + \frac{\log (6/\alpha)}{s} + \sqrt{\frac{\log (6/\alpha)}{s}} 
    \big).
\end{split}
\end{talign}

\subsubsection{Bounding $\tilde{\Psi}_{\X,s}(\alpha)$ in terms of $\tilde{\Psi}_{n,s}(\alpha,\beta)$}
Recall the definitions of $(D_i)_{i=1}^4$ and $(C_i)_{i=1}^4$  in \cref{lem:expectation_A_i,prop:hp_bounds_A_i} respectively.
We have 
\begin{talign} 
\begin{split} \label{eq:bounds_C_i}
\frac{C_1}{s} 
    &= 
\frac{10\sqrt{K}}{\sqrt{sm}}
    +
\frac{2\sqrt{K}}{\sqrt{sn}} 
    \leq 
\frac{12\sqrt{K}}{\sqrt{n}},
    \ \ 
\frac{C_2}{s} 
    = 
\frac{6\sqrt{K}}{\sqrt{sm}} 
    =
\frac{6\sqrt{K}}{\sqrt{n}}, 
    \ \ 
\frac{C_3}{\sqrt{s}} 
    = 
\frac{14 \sqrt{K}}{\sqrt{n}}, 
    \sstext{and}
\frac{C_4}{\sqrt{s}} 
    = 
\frac{6\sqrt{K}}{\sqrt{n}},
\end{split}
\end{talign}
\begin{talign} 
\begin{split} \label{eq:bounds_D_i}
\frac{D_1}{s} 
    &= 
\frac{5\sqrt{K}}{\sqrt{ns}} + \frac{2\sqrt{K}}{\sqrt{sm}}  
    \leq 
\frac{7\sqrt{K}}{\sqrt{n}}, 
    \qquad
\frac{D_2}{s} 
    = 
\frac{3\sqrt{K}}{\sqrt{n}}, \\
\frac{D_3}{\sqrt{s}} 
    &= 
\frac{(5+\sqrt{8})\sqrt{K}}{\sqrt{n}}
    \leq
\frac{8\sqrt{K}}{\sqrt{n}}, 
    \qtext{and}
\frac{D_4}{\sqrt{s}} 
    = 
\frac{(1+\sqrt{8})\sqrt{K}}{\sqrt{n}}
    \leq
\frac{4\sqrt{K}}{\sqrt{n}}.
\end{split}
\end{talign}
By \cref{prop:hp_bounds_A_i}, with probability at least $1-\beta$, the following $\mathcal{A}^{(i)}$-based  bounds all hold simultaneously.
First, 
\newcommand{\firstAibound}{\frac{K}{n} \big( 21 + 56 \sqrt{\log(4/\beta)} + 36 \log(4/\beta) \big)}
\begin{talign}
&\frac{\mathcal{A}^{(1)} \mathcal{A}^{(2)}}{s^2} 
    \leq 
\frac{1}{s^2} \big( D_1 + C_1 \sqrt{\frac{\log(4/\beta)}{2}} \big) \big( D_2 + C_2 \sqrt{\frac{\log(4/\beta)}{2}} \big) 
    \\ 
    &\leq 
\big( \frac{7\sqrt{K}}{\sqrt{n}} + \frac{12 \sqrt{K}}{\sqrt{2n}} \sqrt{\log(4/\beta)} \big) \times \big( \frac{3\sqrt{K}}{\sqrt{n}} + \frac{6 \sqrt{K}}{\sqrt{2n}} \sqrt{\log(4/\beta)} \big) \\ 
    &\leq 
\firstAibound.
\end{talign}
Second,
\newcommand{\secondAibound}{\frac{K}{n} \big( 84 + 148 \sqrt{\log (4/\beta)} + 48 \log (4/\beta) \big)}
\begin{talign}
\frac{\mathcal{A}^{(2)} \mathcal{A}^{(3)} + \mathcal{A}^{(1)} \mathcal{A}^{(4)}}{s^{3/2}}  
    &\leq 
\frac{\big( D_2 + C_2 \sqrt{\frac{\log(4/\beta)}{2}} \big) \big( D_3 + C_3 \sqrt{\frac{\log(4/\beta)}{2}} \big) + \big( D_1 + C_1 \sqrt{\frac{\log(4/\beta)}{2}} \big) \big( D_4 + C_4 \sqrt{\frac{\log(4/\beta)}{2}} \big)}{s^{3/2}} 
    \\ 
    &\leq 
\big( \frac{3\sqrt{K}}{\sqrt{n}} + \frac{6 \sqrt{K}}{\sqrt{n}} \sqrt{\frac{\log(4/\beta)}{2}} \big) 
    \times
\big( \frac{14\sqrt{K}}{\sqrt{n}} + \frac{8\sqrt{K}}{\sqrt{n}} \sqrt{\frac{\log(4/\beta)}{2}} \big) \\ 
    &\qquad 
+ \big( \frac{7\sqrt{K}}{\sqrt{n}} + \frac{12\sqrt{K}}{\sqrt{n}} \sqrt{\frac{\log(4/\beta)}{2}} \big) 
    \times    
\big( \frac{6\sqrt{K}}{\sqrt{n}} + \frac{4\sqrt{K}}{\sqrt{n}} \sqrt{\frac{\log(4/\beta)}{2}} \big) \\ 
    &\leq 
\secondAibound.
\end{talign}
Third,
\newcommand{\thirdAibounda}{\frac{\sqrt{K}}{\sqrt{n}} \big( 14 + 8 \sqrt{\frac{\log(4/\beta)}{2}} \big)}
\newcommand{\thirdAiboundb}{\frac{K}{n} \big(196 + 112 \sqrt{\log(4/\beta)} + 64 \log(4/\beta) \big)}
\begin{talign}
\frac{\mathcal{A}^{(3)}}{\sqrt{s}} 
    &\leq 
\frac{D_3}{\sqrt{s}} + \frac{C_3}{\sqrt{s}} \sqrt{\frac{\log(4/\beta)}{2}} 
    \leq 
\thirdAibounda 
    \qtext{and hence} \\
\frac{(\mathcal{A}^{(3)})^2}{s} 
    &\leq 
\thirdAiboundb.
\end{talign}
Fourth,
\newcommand{\fourthAibound}{\frac{\sqrt{K}}{\sqrt{n s}} \big( 10 + 13 \sqrt{\log(4/\beta)} \big)}
\begin{talign}
\frac{\mathcal{A}^{(1)} + \mathcal{A}^{(2)}}{s^{3/2}} 
    &\leq 
\frac{1}{\sqrt{s}} \big(\frac{D_1}{s} + \frac{C_1}{s} \sqrt{\frac{\log(4/\beta)}{2}} + \frac{D_2}{s} + \frac{C_2}{s} \sqrt{\frac{\log(4/\beta)}{2}} \big) 
    \leq 
\fourthAibound.
\end{talign}
Fifth,
\newcommand{\fifthAibound}{\frac{\sqrt{K}}{\sqrt{n}} \big( 30  + 22 \sqrt{\log(4/\beta)} \big)}
\begin{talign}
&\frac{\mathcal{A}^{(1)} + \mathcal{A}^{(2)} + \mathcal{A}^{(3)} + \mathcal{A}^{(4)}}{s} \\ 
    &\leq 
\frac{D_1}{s} + \frac{C_1}{s} \sqrt{\frac{\log(4/\beta)}{2}} + \frac{D_2}{s} + \frac{C_2}{s} \sqrt{\frac{\log(4/\beta)}{2}} 
    + 
\frac{1}{\sqrt{s}} \big( \frac{D_3}{\sqrt{s}} + \frac{C_s}{\sqrt{s}} \sqrt{\frac{\log(4/\beta)}{2}} + \frac{D_4}{\sqrt{s}} + \frac{C_4}{\sqrt{s}} \sqrt{\frac{\log(4/\beta)}{2}} \big) \\ 
    &= 
\frac{\sqrt{K}}{\sqrt{n}} \big( 30  + 22 \sqrt{\log(4/\beta)} \big).
\end{talign}
Hence, with probability at least $1-\beta$, the quantity $\tilde{\Psi}_{\X,s}(\alpha)$ \cref{eq:tilde_V_psi_def} satisfies 
\begin{talign}
\tilde{\Psi}_{\X,s}(\alpha) 
    &\leq 
\firstAibound
    \times
\big( \frac{1}{2} + \frac{10\log(3/\alpha)}{3} \big) \\    
    &\qquad +  
\secondAibound \times \sqrt{\frac{3}{2}\log(3/\alpha)} \\ 
    &\qquad + 
\thirdAiboundb \\ 
    &\quad + 
\mmd(\pjnt, \P \times \Q) \Big( \fourthAibound \times \big( \frac{1}{4} + \frac{10\log(3/\alpha)}{3} \big) \\ 
    &\qquad + 
\fifthAibound \sqrt{\frac{3}{2}\log(3/\alpha)}
    + 
\thirdAibounda \Big) 
    \\ &\quad + 
\mmd^2(\pjnt, \P \times \Q) \big( 
\frac{1}{4} + \frac{\log (6/\alpha)}{s} + \sqrt{\frac{\log (6/\alpha)}{s}} 
\big)
    \leq
\tilde{\Psi}_{n,s}(\alpha,\beta).
\end{talign}

\section{\pcref{Lemma: Generic Two Moments Method}} \label{Section: Proof of Lemma: Two Moments Method SubG}
Let ${(T_{(b)})}_{b=1}^{\numperm}$ represent the auxiliary values ${(T_{b})}_{b=1}^{\numperm}$ 
ordered increasingly.  
We begin by relating the random critical value 
$T_{(b_{\alpha})}$ to the conditional quantiles of $T_1$ given $\X$.
\begin{lemma}[High probability bound on critical value {\citep[Lem.~6, Cor.~1]{domingo2023compress}}] \label{prop:mathfrak_q_bound} 
Under the assumptions of \cref{Lemma: Generic Two Moments Method}, let $F_{\X}$ be the cumulative distribution function of $T_1$ given $\X$ and let 
\begin{talign} \label{eq:q_def}
q_{1-\alpha(\delta)}(\mbb X) 
    \defeq 
\inf \big\{x \in \R : 1-\alpha(\delta) \leq F_{\X}(x) \big\}
    \sstext{for}
\alpha(\delta)
    \defeq
(\frac{\delta}{{\numperm \choose \floor{\alpha(\numperm+1)}}} )^{\frac{1}{\floor{\alpha(\numperm+1)}}} \geq \frac{\alpha}{2e} \delta^{\frac{1}{\floor{\alpha(\numperm+1)}}}
\end{talign} 
be the $1-\alpha(\delta)$ quantile of $T_1$ given $\X$.
Then, with probability at least $1-\delta$, 
\begin{talign}
    T_{(b_{\alpha})} \leq q_{1-\alpha(\delta)}(\X).
\end{talign}
\end{lemma}

\begin{proof}
Lem.~6 and Cor.~1 of \citep{domingo2023compress} establish this result for the special case of permuted MMD coreset test statistics, but the proof for the general case is identical. 
\end{proof}
Since $\Psi_{\X}(\alpha)\geq q_{1-\alpha(\beta/3)}(\X)$ almost surely by assumption \cref{eq:Psi_X_n}, \cref{prop:mathfrak_q_bound} implies that
\begin{talign}
\Pr(\mathcal{A}_1^c) \leq \beta/3
    \qtext{for}
\mathcal{A}_1 \defeq \{ T_{(b_{\alpha})} \leq %
    \Psi_{\X}( \astar) \}.
\end{talign}
Moreover, our assumption \cref{eq:Psi_P_n} on $\Psi_n$ ensures that 
\begin{talign}
\Pr(\mathcal{A}_2^c) \leq \beta/3
    \qtext{for}
\mathcal{A}_2 \defeq 
\{ \Psi_{\X}(\astar) \leq 
\Psi_{n}(\astar,\beta/3) \}. 
\end{talign}
Since $\Delta(\X)$ rejects when $T(\X)>T_{(b_{\alpha})}$, the acceptance probability is upper-bounded by
\begin{talign}
\Pr (T(\X) \leq T_{(b_{\alpha})}) 
&=
\Pr (T(\X) \leq T_{(b_{\alpha})}, \ \mathcal{A}_1 %
\cap \mathcal{A}_2)  + \Pr (T(\X) \leq T_{(b_{\alpha})}, \ \mathcal{A}_1^c %
\cup \mathcal{A}_2^c) \\%[.5em]
&\leq \Pr (T(\X) \leq \Psi_{n}(\astar,\beta/3)) 
+  \Pr (\mathcal{A}_1^c \cup \mathcal{A}_2^c) \\
&\leq
\Pr(T(\X) \leq 
\tconst
- \Phi_{n}(\frac{\beta}{3})) 
+  \Pr (\mathcal{A}_1^c) + \Pr(\mathcal{A}_2^c) 
\leq \beta,
\end{talign}
where the final two inequalities used the union bound and our assumptions \cref{eq:separation,eq:Phi_P_n}. 
\section{\pcref{Lemma: Refined Two Moments Method}} 
\label{Section: Proof of Lemma: Refined Two Moments Method}
Our choices of $\Phi_n$ and $\Psi_{\X}$ satisfy the requirements \cref{eq:Phi_P_n,eq:Psi_X_n} by Cantelli's inequality \citep[(18)]{cantelli1929sui}.
Moreover, Markov's inequality \citep{markov1884certain} and Cantelli's inequality imply
\begin{talign}
    &\Pr ( \sqrt{(\alpha^{-1} - 1)\Var(T_1 | \X)} \geq  \sqrt{(\alpha^{-1} - 1)\beta^{-1} \mE[ \Var (T_1 | \X) ] } \mid \X ) \\ 
&\quad= \Pr ( (\alpha^{-1} -1) \Var(T_1 | \X) \geq  (\alpha^{-1} -1) \beta^{-1} \mE[ \Var (T_1 | \X) ] \mid \X ) \leq \beta 
\qtext{and} \\ 
&\Pr ( \mE[T_1 | \X] - \mE [T_1] \geq \sqrt{(2\beta^{-1} -1) \Var (\mE[T_1 | \X])}  ) \leq \beta/2
\end{talign}
for all $\beta\in(0,1]$. 
Hence, our $\Psi_n$ choice \cref{eq:refined_Psi_n_zero} satisfies \cref{eq:Psi_P_n} when $\E[T_1\mid\X]=0$ almost surely and, by the union bound, our $\Psi_n$ choice \cref{eq:refined_Psi_n_nonzero} satisfies \cref{eq:Psi_P_n} more generally.
\section{Proofs of minimax optimality results}

\subsection{\pcref{Proposition: Multinomial Two-Sample Testing}} \label{subsec:pf_minimax_homogeneity_multinomial}
Introduce the shorthand $t\defeq s\sqrt[3]{\nratio}$, $u\defeq  \frac{\beta\astar}{24(1-\astar)}$, and $v \defeq \sqrt{t^3u}-1$. 
\cref{thm: Full Two Sample Tests} and 
the facts that $\psi_{Y,1}, \psi_{Z,1} \leq 4 \sqrt{b_{(1)}} \twonorm{\P - \Q}^2$ and $\psi_{YZ,2} \leq b_{(1)}$ by \citep[App.~G]{kim2022minimax} imply that the cheap discrete homogeneity permutation test has power at least $1-\beta$ whenever
\begin{talign}
    \twonorm{\P - \Q}^2 \! = \! \mathbb{E}[U_{n_1,n_2}] &\geq \bigg( 
    \sqrt{\frac{16(3-\beta)\sqrt{b_{(1)}} \twonorm{\P - \Q}^2}{\beta}\big(\frac{1}{n_1} \! + \! \frac{1}{n_2} \big)}
    + \sqrt{\frac{1-\astar}{\astar}\frac{(36 n_1^2 + 36 n_2^2 + 198 n_1 n_2) b_{(1)}}{\beta n_1 (n_1 - 1) n_2 (n_2 - 1)}} \bigg) \\ &\qquad \times \! ( 1 \! + \! \frac{1}{v} ) 
    + \frac{
\sqrt{\frac{201s^2}{4n^2} b_{(1)} + {\frac{48 s}{n} \sqrt{b_{(1)}} \twonorm{\P - \Q}^2}{}}
}{ v}.
\end{talign}
By the triangle inequality, it suffices to have  $a x^2 - b x - c \geq 0$ for $x = \twonorm{\P - \Q}$, $a=1$,
\begin{talign}
    b &= \sqrt{\frac{16(3-\beta)\sqrt{b_{(1)}}}{\beta}\big(\frac{1}{n_1} \! + \! \frac{1}{n_2} \big)} \! \times \! ( 1 \! + \! \frac{1}{v} ) + \sqrt{\frac{48\sqrt{b_{(1)}}}{n\sqrt[3]{\nratio}}}\, 
    \frac{\sqrt{t}}{ v}, 
    \qtext{and} \\
    c &= \sqrt{\frac{1-\astar}{\astar}\frac{(36 n_1^2 + 36 n_2^2 + 198 n_1 n_2) b_{(1)}}{\beta n_1 (n_1 - 1) n_2 (n_2 - 1)}}\! \times \!(1+\frac{1}{v}) + \frac{\sqrt{201 b_{(1)}}}{2n\sqrt[3]{\nratio}}\, 
    \frac{t}{ v}.
\end{talign}
Hence, by the quadratic formula, it suffices to have $x\geq b+\sqrt{c} \geq (b+\sqrt{b^2+4ac})/(2a)$.
Finally, since $1/(n\sqrt[3]{\nratio}) = O(1/n_1)$, there exist universal constants $C,C'$ such that
\begin{talign}
b + \sqrt{c} 
    &\leq 
\frac{C' b_{(1)}^{1/4}}{\sqrt{\min(\beta,\astar)\, n_1}}
( 1 + \frac{\sqrt{t}}{
\min(v,\sqrt{v})}) 
    =
\frac{C' b_{(1)}^{1/4}}{\sqrt{\min(\beta,\astar)\, n_1}}
( 1 + \frac{(1+v)^{1/3}}{u^{1/6}
\min(v,\sqrt{v})}) \\
    &\leq
\frac{C b_{(1)}^{1/4}}{\sqrt[6]{\beta\astar}\sqrt{\min(\beta,\astar)\, n_1}}
( 1 + \frac{1}{v}).
\end{talign}

\subsection{\pcref{Proposition: Two-Sample Testing for Holder Densities}}
\label{subsec:pf_minimax_homogeneity_holder}
Let $\tilde{\P}$ be a discrete distribution over $\{B_1,\dots,B_K\}$ with probability $ \int_{B_k} d\P(x)$ of selecting bin $B_k$ for each $k\in[K]$. Define $\tilde{\Q}$ analogously based on $\Q$, and let $\kappa_{(1)} = \lfloor n_1^{2/(4\holderorder+d)} \rfloor$.
By \cite[Lem.~3]{arias2018remember}, 
\begin{talign} \label{eq:l2_norms_inequality}
    \statictwonorm{\tilde{\P}-\tilde{\Q}}^2 \geq C_1 \kappa_{(1)}^{-d} \twonorm{\P-\Q}^2,
\end{talign}
for a universal constant  $C_1$. 
Furthermore, by the argument in \citep[App.~I]{kim2022minimax}, 
\begin{talign} \label{eq:bound_b_1}
    b_{(1)} \defeq \max \{ \twonorm{\P}^2, \twonorm{\Q}^2 \} \leq L \kappa_{(1)}^{-d}.
\end{talign}
These bounds, together with \cref{Proposition: Multinomial Two-Sample Testing}, yield the sufficient power condition
\begin{talign}
\frac{C' L^{1/4} \kappa_{(1)}^{-d/4}}{\sqrt[6]{\beta\astar}\sqrt{\min(\beta,\astar)\, n_1}}
\bigg(1+\frac{1}{
\sqrt{s^3\frac{\beta\astar}{24(1-\astar)} \nratio}-1}\bigg)
    \leq
\sqrt{C_1}\kappa_{(1)}^{-d/2} \twonorm{\P-\Q} 
    \leq
\statictwonorm{\tilde{\P}-\tilde{\Q}}
\end{talign}
for a universal constant $C'$. 
The result now follows from the upper estimate $n_1^{2/(4\holderorder+d)}\geq \kappa_{(1)}$.

\subsection{\pcref{Proposition: Multinomial Independence Testing}}
\label{subsec:pf_minimax_independence_multinomial}
Note that the augmentation variables $(A_i)_{i=1}^n,(B_i)_{i=1}^n\distiid\Unif((0,1))$ are, with probability $1$, all distinct. 
Therefore, $g_Y((Y_i,A_i),(Y_i,A_i)) = g_Z((Z_i,B_i),(Z_i,B_i))= 0$ for all $i\in[n]$.
The  arguments of 
\citep[Prop.~5.3]{kim2022minimax} imply the variance component \cref{Eq: definition of psi prime functions} bounds 
\begin{talign} 
\label{eq:psi_prime_bounds_minimax}
\psi'_1 
    &\leq 
C_1 \sqrt{b_{(2)}} \twonorm{\Pi - \P \times \Q}^2,
    \quad
\psi'_2 
    \leq 
C_2 b_{(2)}, 
    \qtext{and} 
\tilde{\psi}'_1 
    \leq 
C_3 \sqrt{b_{(2)}} \twonorm{\Pi - \P \times \Q}^2
\end{talign}
for universal constants $C_1$, $C_2$, and $C_3$. Additionally, we can bound the mean component $\tilde{\xi}$  \cref{Eq: mean component} by bounding the following expression for arbitrary $i_1,\dots,i_4,i'_1,\dots,i'_4 \in [8]$.
\begin{talign}
&\big|\E\big[\hin\big(((Y_{i_1},A_{i_1}),(Z_{i'_1},B_{i'_1})),((Y_{i_2},A_{i_2}),(Z_{i'_2},B_{i'_2})), \\ 
    &\qquad\quad 
((Y_{i_3},A_{i_3}),(Z_{i'_3},B_{i'_3})),((Y_{i_4},A_{i_4}),(Z_{i'_4},B_{i'_4})\big)\big]\big| \\ 
    &= 
\big|\E\big[ \big( g_Y((Y_{i_1},A_{i_1}),(Y_{i_2},A_{i_2})) +g_Y((Y_{i_3},A_{i_3}),(Y_{i_4},A_{i_4})) \\  
    &\qquad 
-g_Y((Y_{i_1},A_{i_1}),(Y_{i_3},A_{i_3})) -g_Y((Y_{i_2},A_{i_2}),(Y_{i_4},A_{i_4})) \big) \\ 
    &\qquad\qquad 
\times \big( g_Z((Z_{i'_1},B_{i'_1}),(Z_{i'_2},B_{i'_2})) +g_Z((Z_{i'_3},B_{i'_3}),(Z_{i'_4},B_{i'_4})) \\ 
    &\qquad\qquad\quad 
-g_Z((Z_{i'_1},B_{i'_1}),(Z_{i'_3},B_{i'_3})) -g_Z((Z_{i'_2},B_{i'_2}),(Z_{i'_4},B_{i'_4})) \big) \big]\big| \\
    &\stackrel{(i)}{=} 
\E\big[ \big( \indic{Y_{i_1} = Y_{i_2}, i_1 \neq i_2} + \indic{Y_{i_3} = Y_{i_4}, i_3 \neq i_4} \\ 
    &\qquad 
- \indic{Y_{i_1} = Y_{i_3}, i_1 \neq i_3} -\indic{Y_{i_2} = Y_{i_4}, i_2 \neq i_4} \big) \\ 
    &\qquad\qquad 
\times \big( \indic{Z_{i'_1} = Z_{i'_2}, i'_1 \neq i'_2} + \indic{Z_{i'_3} = Z_{i'_4}, i'_3 \neq i'_4} \\ 
    &\qquad\qquad\quad 
-\indic{Z_{i'_1} = Z_{i'_3}, i'_1 \neq i'_3} -\indic{Z_{i'_2} = Z_{i'_2}, i'_3 \neq i'_4} \big) \big] 
    \\ 
    &\stackrel{(ii)}{\leq} 
16 \max \big\{ \E[\indic{Y_{1} = Y_{2}} \indic{Z_{i} = Z_{j}}] : (i,j) \in \{(1,2),(1,3),(3,4)\} \big\} \\ 
    &\stackrel{(iii)}{=} 
16 \max \big\{ \E[\indic{Y_{1} = Y_{2}, A_{1} \neq A_{2}} \indic{Z_{i} = Z_{j}, B_{i} \neq B_{j}}] : (i,j) \in \{(1,2),(1,3),(3,4)\} \big\} \\ 
    &= 
16 \max \big\{ \E[g_Y((Y_{1},A_{1}),(Y_{2},A_{2}))^2 g_Z((Z_{i},B_{i}),(Z_{j},B_{j}))^2 ] : (i,j) \in \{(1,2),(1,3),(3,4)\} \big\} \\ 
    &= 
\psi'_2.
\end{talign}
Above, equality (i) holds because, with probability 1, $A_{i} = A_{j} \iff i = j$ and $B_{i} = B_{j} \iff i = j$. Inequality (ii) holds because, after expanding the product, the only terms that remain are of the form $\indic{Y_{i} = Y_{j}} \indic{Z_{k} = Z_{l}}$ with $i \neq j$ and $k \neq l$. Equality (iii) holds because $A_1 \neq A_2$ and $B_i \neq B_j$ almost surely. We conclude that
\begin{talign}
    \tilde{\xi} \leq \psi'_2 \leq C_2 b_{(2)}.
\end{talign}

Since $\mathcal{U}=4\twonorm{\Pi - \P \times \Q}^2$, \cref{thm: Full Independence Tests}, our mean and variance component bounds, and the triangle inequality imply that the cheap discrete independence test has power at least $1-\beta$ whenever $a x^2 - b x - c \geq 0$ for $x = \twonorm{\Pi - \P \times \Q}$, $a=1 - \frac{s+1}{3s^2}$,
\begin{talign} 
\frac{3}{4}b 
    &= 
\sqrt{(\frac{3}{\beta} \!-\! 1) \frac{16 C_1 \sqrt{b_{(2)}}}{n}}
    \!+\! 
\sqrt{C_3 \sqrt{b_{(2)}}} 
    \big(
\sqrt{\big( \frac{12}{\astar \beta} \!- \!2\big) \frac{32}{n}}
    + 
\sqrt{\big( \frac{12}{\astar \beta} - 2\big)
\big( \frac{2304}{n s} + \frac{460800}{n s^2}
\big)}\big), 
    \text{ and}\\
\frac{3}{4}c 
    &= 
\sqrt{C_2 b_{(2)}}
    \big(
\sqrt{(\frac{3}{\beta} \!-\! 1) \big(\frac{144}{n^2} \big)}
    \!+\! 
\sqrt{\big( \frac{12}{\astar \beta} \!- \!2\big) \big(\!\frac{960}{n^2} \big)}
    + 
\frac{24}{n} 
    + 
\frac{16}{n} 
    +
\sqrt{\big( \frac{12}{\astar \beta} - 2\big)\frac{4147200}{n^2 s} }
    \big).
\end{talign}
Since $a\leq 1$, it therefore suffices to have $x \geq \frac{1}{a} (b+\sqrt{c}) \geq \frac{b}{a} +\sqrt{\frac{c}{a}} \geq \frac{b+\sqrt{b^2+4ac}}{2a}$ 
by the quadratic formula. The claim now follows as 
$\frac{1}{a} (b+\sqrt{c}) \leq \frac{C\,b_{(2)}^{1/4}}{\sqrt{\beta\astar n}}$ for a universal constant $C$.
\subsection{\pcref{Proposition: Independence Testing for Holder Densities}}
\label{subsec:pf_minimax_independence_holder}
Let $\tilde{\P}$ be a discrete distribution over $\{B_{1,1},\dots,B_{1,K_1}\}$ with probability $ \int_{B_{1,k}} d\P(y)$ of selecting bin $B_{1,k}$ for each $k\in[K_1]$. Define $\tilde{\Q}$ analogously as a discrete distribution over $\{B_{2,1},\dots,B_{2,K_2}\}$  based on $\Q$ and $\tilde{\pjnt}$ 
analogously as a discrete distribution over $\{(B_{1,k_1},B_{2,k_2}) : k_1\in[K_1], k_2\in[K_2]\}$ based on $\pjnt$.
Let $\kappa_{(2)} = \lfloor n^{2/(4\holderorder+d_1+d_2)} \rfloor$.
By \cite[Lem.~3]{arias2018remember},\begin{talign} \label{eq:l2_norms_inequality_ind}
    \statictwonorm{\tilde{\Pi} - \tilde{\P} \times \tilde{\Q}}^2 \geq C_1 \kappa_{(2)}^{-(d_1+d_2)} \twonorm{\Pi - \P \times \Q}^2,
\end{talign}
for a universal constant  $C_1$. 
Furthermore, by the argument in \citep[App.~M]{kim2022minimax}, 
\begin{talign} \label{eq:bound_b_2}
    b_{(2)} \defeq \max \{ \twonorm{\Pi}^2, \twonorm{\P \times \Q}^2 \} \leq L \kappa_{(2)}^{-(d_1 + d_2)},
\end{talign}
These bounds, together with \cref{Proposition: Multinomial Independence Testing}, yield the sufficient power condition
\begin{talign}
\frac{C'L^{1/4} \kappa_{(2)}^{-(d_1 + d_2)/4} }{\sqrt{\beta\astar n}}
    \leq
\sqrt{C_1}\kappa_{(2)}^{-(d_1+d_2)/2} \twonorm{\Pi - \P \times \Q}
    \leq
\statictwonorm{\tilde{\Pi} - \tilde{\P} \times \tilde{\Q}}
\end{talign}
for a universal constant $C'$. 
The result now follows from the  estimate $n^{2/(4\holderorder+d_1+d_2)}\geq \kappa_{(2)}$.

\section{Supplementary experiment details}
\label{sec:additional_plots}

All independence experiments were run with Python 3.13.0 on a single AMD EPYC 7V13 CPU, with operating system Ubuntu 22.04.5 LTS. 
All Higgs boson testing experiments were run with Python 3.11.10 on a single AMD EPYC 7V13 CPU, with operating system Ubuntu 22.04.5 LTS.
All other homogeneity experiments were run with Python 3.11.10 on a single Intel Xeon Platinum 8268 CPU, with operating system Red Hat Enterprise Linux 9.2 (Plow). 

\subsection{Selection of the number of permutations $\numperm$} \label{subsec:num_permutations_plot}
To select our common permutation count of $\numperm = 1279$ for the homogeneity testing experiments, we evaluated the power of the standard MMD test with total sample size $n = 16384$ and varying permutation count of the form $\numperm = 20\times 2^a-1$ for $a\geq 0$ in the \textsc{MMD setting} of \cref{subsec:homogeneity_experiments}.
We followed the same procedure to select the permutation count for  independence testing experiments, only using the \textsc{HSIC setting} of \cref{subsec:independence_experiments} with $n=2048$ in place of the MMD setting.
In both cases, we found that power saturated after $\numperm=1279$ as shown in \cref{fig:n_permutations_homogeneity}.

\begin{figure}[tb]
    \centering
    \includegraphics[width=0.49\textwidth]{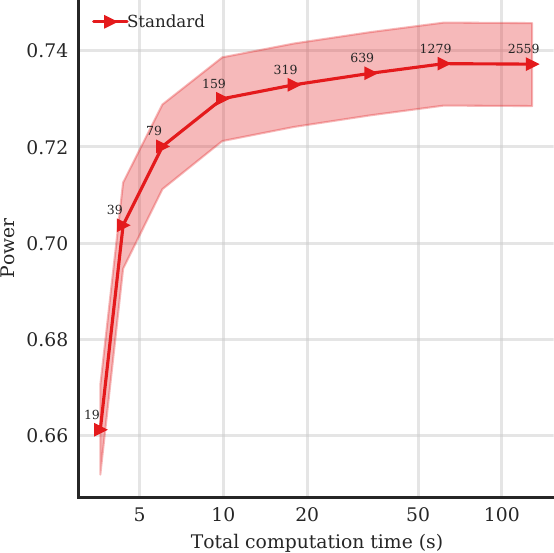}
    \includegraphics[width=0.49\textwidth]{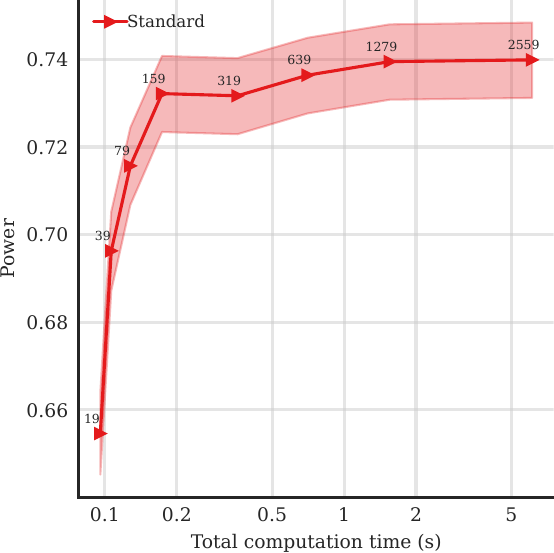}
    \caption{Power and runtime of standard MMD permutation tests of homogeneity (left) and standard HSIC wild bootstrap tests of independence (right) as number of permutations $\numperm$ varies.  %
    }
    \label{fig:n_permutations_homogeneity}
\end{figure}

\subsection{Cheap permutation tests for feature-based kernels}
\label{sec:feature_cheap}
When a homogeneity U-statistic (\cref{def:homogeneity_u_stat}) has a rank-$r$ base function of the form \begin{talign}\label{eq:homogeneity_low_rank}
g(y,z) = \sum_{j=1}^r \feat_j(y)\feat_j(z),
\end{talign} 
each test statistic can be computed in time $\Theta(c_\feat nr)$ where $c_\feat$ is the maximum time required to evaluate a single feature function $\feat_j$, and cheap permutation can be carried out in $\Theta(\numperm s r)$ additional time with $\Theta(sr)$ memory, as described in \cref{algo:cheap_permutation_wild_bootstrap_two_sample_features}.
\begin{algorithm2e}[h!]
\small{

    \KwIn{Samples $(Y_i)_{i=1}^{n_1}$, $(Z_i)_{i=1}^{n_2}$, bin count $s$, feature maps $f=(f_j)_{j=1}^r$ \cref{eq:homogeneity_low_rank}, level $\alpha$, permutation count $\numperm$}
   \vspace{5pt}
  
  Define $(X_i)_{i=1}^{n} \defeq (Y_1, \ldots, Y_{n_1}, Z_1, \ldots, Z_{n_2})$ for $n \gets n_1+n_2$, $m \gets n/s$, and $s_1 \gets {s\, n_1/}{n}$ \\ 

  \textit{// Compute sufficient statistics using $\Theta(c_\feat n r)$ time and $\Theta(sr)$ memory} \\
  \lFor{$i=1, 2, \ldots, s$}{
    $\Feat_{i} \gets \sum_{k=1}^{m} \feat(X_{(i-1)m + k})$
  }

    \textit{// Compute original and  permuted test statistics using $\Theta(\numperm s r)$ elementary operations}\\    
  \For{$b=0, 1, 2, \ldots, \numperm$}
    {
    $\pi \gets$ identity permutation if $b=0$ else 
    random uniform permutation of $[s]$ \\ 
   $T_{b} \gets \frac{1}{n^2}\twonorm{ \sum_{i=1}^{s} (2\,\indicator(\pi_b(i) \leq s_1) - 1) \Feat_{i}}^2$
   }
    \textit{// Return rejection probability} \\
   $R \gets 1+$ number of permuted statistics $(T_{b})_{b=1}^{\numperm}$
   smaller than $T_0$ if ties are broken at random 
   \\
    \KwRet{$\Delta(\X) \defeq \min(1,\max(0,R \!-\! (1\!-\!\alpha)(\numperm\!+\!1)))$}
}
  \caption{Cheap 
  homogeneity testing with feature kernels }
  \label{algo:cheap_permutation_wild_bootstrap_two_sample_features}
\end{algorithm2e}

\end{appendix}

\end{document}